\documentclass[titlepage,twoside]{report}
\usepackage{titlesec}
\usepackage{amsmath, amssymb, latexsym, amsthm, mathdots, tikz, caption, subcaption}
\usepackage{geometry}
\usetikzlibrary{cd, decorations.pathreplacing}
\usepackage{hyperref}
\usepackage{fancyhdr}
\usepackage{setspace}
\hypersetup{
    colorlinks=true,
     citecolor=blue
}

\theoremstyle{plain}
\newtheorem{thm}{Theorem}[subsection]
\newtheorem{cor}[thm]{Corollary}
\newtheorem{lem}[thm]{Lemma}
\newtheorem{prop}[thm]{Proposition}
\newtheorem{conj}[thm]{Conjecture}

\newtheorem*{thm*}{Theorem}
\newtheorem*{cor*}{Corollary}
\newtheorem*{lem*}{Lemma}
\newtheorem*{prop*}{Proposition}
\newtheorem*{conj*}{Conjecture}

\newtheorem{thm4}{Theorem}[chapter]
\newtheorem{cor4}[thm4]{Corollary}

\newtheorem{prop4}[thm4]{Proposition}

\theoremstyle{definition}
\newtheorem{defn}[thm]{Definition}
\newtheorem{ex}[thm]{Example}
\newtheorem{rem}[thm]{Remark}

\newtheorem*{defn*}{Definition}
\newtheorem*{ex*}{Example}
\newtheorem*{rem*}{Remark}
\newtheorem*{alg*}{Algorithm}
\newtheorem*{asi*}{Aside}

\newtheorem{ex4}[thm4]{Example}

\theoremstyle{remark}
\newtheorem*{claim}{Claim}

\makeatletter
\def\blfootnote{\xdef\@thefnmark{}\@footnotetext}
\makeatother

\title{ {}
{\includegraphics[scale=0.2]{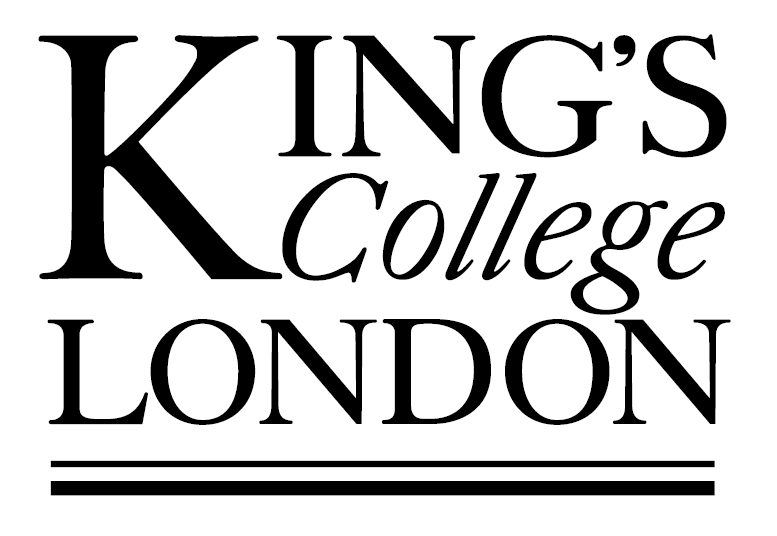}} \\
\vspace{0.5cm} Mirror symmetry, tropical geometry \\and representation theory \\
\vspace{2.3cm} \Large Teresa L\"udenbach \\
\vspace{0.5cm} \large Supervised by Prof.~Konstanze Rietsch \\
\vspace{2.3cm} A thesis submitted for the degree of \\
        Doctor of Philosophy in Pure Mathematics \\
\vspace{0.5cm} January 2023
}
\date{}
\author{}

\begin{document}

\pagenumbering{Alph}
\maketitle
\thispagestyle{empty}
\pagenumbering{arabic}

\fancyhead[L]{}
\fancyhead[R]{\nouppercase{\leftmark}}

\begin{abstract}
\thispagestyle{plain}
\setcounter{page}{2}

An ideal filling is a combinatorial object introduced by Judd %in \cite{Judd2018}
that amounts to expressing a dominant weight $\lambda$ of $SL_n$ as a rational sum of the positive roots in a canonical way, such that the coefficients satisfy a $\max$ relation.
He proved that whenever an ideal filling has integral coefficients it corresponds to a lattice point in the interior of the string polytope which parametrises the canonical basis of the representation with highest weight $\lambda$.
The work of Judd makes use of a construction of string polytopes via the theory of geometric crystals, and involves tropicalising the superpotential of the flag variety $SL_n/B$ in certain `string' coordinates.
He shows that each ideal filling relates to a positive critical point of the superpotential over the field of Puiseux series, through a careful analysis of the critical point conditions.

In this thesis we give a new interpretation of ideal fillings, together with a parabolic generalisation.
For every dominant weight $\lambda$ of $GL_n$, we also define a new family of polytopes in $\mathbb{R}^{R_+}$, where $R_+$ denotes the positive roots of $GL_n$, with one polytope for each reduced expression of the longest element of the Weyl group.
These polytopes are related by piecewise-linear transformations which fix the ideal filling associated to $\lambda$ as a point in the interior of each of these polytopes.

Our main technical tool is a new coordinate system in which to express the superpotential, which we call the `ideal' coordinates. We describe explicit transformations between these coordinates and string coordinates in the $GL_n/B$ case.

Finally, we demonstrate a close relation between our new interpretation of ideal fillings and factorisations of Toeplitz matrices into simple root subgroups.

\end{abstract}

\newpage
\setcounter{page}{3}
\tableofcontents
\listoffigures
\listoftables

\newpage
\chapter*{Acknowledgements}
\addcontentsline{toc}{chapter}{Acknowledgements}

I wish to thank my supervisor Konni Rietsch for suggesting this idea to me, and particularly for her helpful explanations and patience throughout. Her wisdom, generosity and good humour have helped make these years a time of growth for me, personally as well as mathematically, for which I will always be grateful.

I also wish to thank my family and friends for their unwavering support and encouragement, without which this thesis would not have been possible.
Particular thanks go to Susie Russell, Manuel Crepin, Yll Buzoku and Mar\'ia Fl\'orez Martin.
% The Sisters of Mary Morning Star
% Rhi Harrison, Iona Duncan, Fr Toby Lees, Jaheda Youssuff

This work was supported by the Engineering and Physical Sciences Research Council {[EP/S021590/1]}, the EPSRC Centre for Doctoral Training in Geometry and Number Theory {(The London School of Geometry and Number Theory)}, University College London, and King's College London.

\newpage

\chapter*{Introduction}
\addcontentsline{toc}{chapter}{Introduction}
\fancyhead[R]{Introduction}
\fancyhead[L]{}

Representations of Lie groups are often described in terms of their weights - the characters arising in the action of a maximal torus. A standard way to depict these weights is by embedding the character lattice into a real vector space and viewing the weights as lattice points in their convex hull, the so-called `weight polytope’ of the representation. For an irreducible representation the weights along the boundary of the weight polytope (including the highest weight) all have one-dimensional weight spaces. The weight spaces corresponding to interior points can be higher-dimensional. Accordingly, a better `picture’ of the representation may be given by a higher-dimensional polytope that projects onto the weight polytope, such that the lattice points in a fibre parametrise a basis of the corresponding weight space.

A famous example of such a construction is given by the Gelfand--Tsetlin polytope of a representation of $GL_n$, given first in \cite{GelfandTsetlin1950} (e.g. Figure \ref{fig A Gelfand-Tsetlin polytope}).
More recent examples relate to Lusztig's canonical basis and its combinatorial and geometric construction (\cite{Luszig1990}, \cite{Lusztig1990_2}), as well as Kashiwara's crystal basis operators (\cite{Kashiwara1991}). Of particular interest on the crystal basis side are the string polytopes introduced by Littelmann in \cite{Littelmann1998}. On the canonical basis side there is another parametrisation due to Lusztig (\cite{Lusztig1994}). His parametrisation uses coordinate charts on the Langlands dual flag variety and the notion of tropicalisation that he introduced.
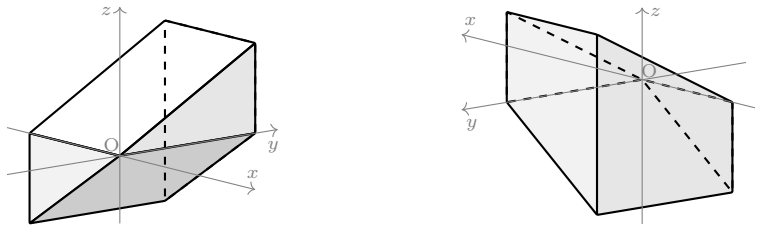
\begin{figure}[ht]
\centering
\begin{tikzpicture}[scale=0.6]
    \filldraw[fill=black!10, thick, rounded corners=0.5] (0,0) -- (3,2.5) -- (3,0.5) -- cycle; % $x=0$
    \filldraw[fill=black!5, thick, rounded corners=0.5] (0,0) -- (-2,-1.5) -- (-2,0.5) -- cycle; % $y=0$
    \filldraw[fill=black!20, thick, rounded corners=0.5] (0,0) -- (3,0.5) -- (1,-1) -- (-2,-1.5) -- cycle; % $x=z$
    \draw[thick, dashed, rounded corners=0.5] (3,0.5) -- (3,2.5) -- (1,3) -- (1,-1) -- cycle; % $y=1$
    \draw[thick, rounded corners=0.5] (0,0) -- (-2,0.5) -- (1,3) -- (3,2.5) -- cycle; % $y=z$
    \draw[black!50, ->] (-2.5,0.625) -- (3,-0.75); % x axis
        \node[black!50] at (2.95,-0.4) {\scriptsize{$x$}};
    \draw[black!50, ->] (-2.5,-0.417) -- (3.5,0.583); % y axis
        \node[black!50] at (3.4,0.2) {\scriptsize{$y$}};
    \draw[black!50, ->] (0,-1.5) -- (0,3.3); % z axis
        \node[black!50] at (-0.3,3.2) {\scriptsize{$z$}};
        \node[black!50] at (-0.18, 0.25) {\scriptsize{O}};
\end{tikzpicture}
\hspace{2cm}
\begin{tikzpicture}[scale=0.6]
    \filldraw[fill=black!5, thick, rounded corners=0.5] (0,-2) -- (0,2) -- (-2,2.5) -- (-2,0.5) -- cycle; % $y=1$
    \filldraw[fill=black!10, thick, rounded corners=0.5] (0,-2) -- (0,2) -- (3,0.5) -- (3,-1.5) -- cycle; % $x=-1$
    \draw[thick, dashed, rounded corners=0.5] (3,0.5) -- (1,1) -- (3,-1.5) -- cycle; % $y=0$
    \draw[thick, dashed, rounded corners=0.5] (1,1) -- (-2,0.5) -- (-2,2.5) -- cycle; % $x=0$
    \draw[black!50, ->] (3.5,0.375) -- (-3,2); % x axis
        \node[black!50] at (-2.8,2.3) {\scriptsize{$x$}};
    \draw[black!50, ->] (3.3,1.38) -- (-3,0.33); % y axis
        \node[black!50] at (-2.77,0) {\scriptsize{$y$}};
    \draw[black!50, ->] (1,-2.2) -- (1,2.6); % z axis
        \node[black!50] at (1.3,2.5) {\scriptsize{$z$}};
        \node[black!50] at (1.16, 1.2) {\scriptsize{O}};
\end{tikzpicture}
\caption{A Gelfand--Tsetlin polytope} \label{fig A Gelfand-Tsetlin polytope}
\end{figure}

Building on the work of Lusztig and Kashiwara, Berenstein and Kazhdan in \cite{BerensteinKazhdan2000}, \cite{BerensteinKazhdan2007}, `geometrised' such polytopes via their theory of geometric crystals.
Their construction includes a function that encodes all of the walls of these polytopes.
This function turns out to agree with the superpotential of a flag variety, which was later independently constructed by Rietsch in the the context of mirror symmetry (\cite{Rietsch2008}, see also Chhaibi's work in \cite{ChhaibiThesis13}).
Another example of this are the polytopes constructed using the mirror symmetry of Grassmannians by Rietsch and Williams in \cite{RietschWilliams2019}, which relate to fundamental representations of $GL_n$ and cluster duality (\cite{FockGoncharov2009}).

The main objects of study in this thesis are the superpotentials of full and partial flag varieties $GL_n/B$ and $GL_n/P$, and polytopes we can construct using them. We denote these superpotentials by $\mathcal{W}$ and $\mathcal{W}_P$ respectively.

Out of a superpotential function, say $\mathcal{W}$, together with a choice of dominant weight $\lambda$ of $GL_n$, we can construct a multitude of polytopes in $\mathbb{R}^{N}$, where $N=\binom{n}{2}$, which depend on the choice of torus chart. Such polytopes are related by subtraction-free, rational transformations, which preserve the integral lattice. Naturally, different choices of torus chart have different advantages, and, for example, can recover familiar polytopes, such as the string polytopes mentioned above.

All of these polytopes obtain an additional structure when constructed out of a function like the superpotential. Namely, a special point in the interior of the polytope, which we call the tropical critical point. This point arises from the valuations of the coordinates of a `positive' critical point of the function, when considered over the field of Puiseux series (see Judd's work in \cite{Judd2018}). There is a special relevance of the lattice points of these polytopes and so Judd looked when at when the tropical critical point is integral (\cite{Judd2018}). To do so, he introduced a combinatorial object called an ideal filling, and formed a close connection to the tropical critical point (compare Proposition \ref{prop GLn version of Jamie's 6.2}). An ideal filling is an assignment of rational coordinates $n_{ij}$ to boxes in an upper triangular arrangement, obeying the $\max$ relation $n_{ij}=\max\{n_{i+1, j}, n_{i, j-1} \} $ for $j-i\geq 2$. For example if $n=4$, the arrangement looks like Figure \ref{fig Ideal filling arrangement for n=4}.
\begin{figure}[ht!]
\centering
\begin{tikzpicture}
    %box lines
    %horizontal
    \draw (0,0) -- (2.4,0);
    \draw (0,-0.8) -- (2.4,-0.8);
    \draw (0.8,-1.6) -- (2.4,-1.6);
    \draw (1.6,-2.4) -- (2.4,-2.4);

    %vertical
    \draw (0,0) -- (0,-0.8);
    \draw (0.8,0) -- (0.8,-1.6);
    \draw (1.6,0) -- (1.6,-2.4);
    \draw (2.4,0) -- (2.4,-2.4);

    %nij labels
    \node at (0.4,-0.4) {\small{$n_{12}$}};

    \node at (1.2,-0.4) {\small{$n_{13}$}};
    \node at (1.2,-1.2) {\small{$n_{23}$}};

    \node at (2,-0.4) {\small{$n_{14}$}};
    \node at (2,-1.2) {\small{$n_{24}$}};
    \node at (2,-2) {\small{$n_{34}$}};
\end{tikzpicture}
\caption{Ideal filling arrangement for $n=4$} \label{fig Ideal filling arrangement for n=4}
\end{figure}
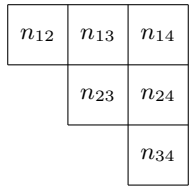

We find an alternative approach to ideal fillings via Toeplitz matrices, and illustrate it in the following example:

\begin{ex*} \label{ex Toeplitz mat and ideal fillings}
Let $m_1, m_2, m_3$  be Puiseux series with positive leading coefficients and non-negative valuations $\mu_i=\mathrm{Val}_{\mathbf{K}}(m_i)$ (defined in Section \ref{subsec The basics of tropicalisation} as the exponent of the first non-zero term). %and consider the following matrix
    % $$b(d,\boldsymbol{m})=\mathbf{y}_{\mathbf{i}_0}^{\vee}\left(\frac{1}{m_1}, \frac{1}{m_2}, \frac{1}{m_3} \right)
    %     \begin{pmatrix} d_3 m_2m_3 & & \\ & d_2 \frac{m_1}{m_3} & \\ & & d_1 \frac{1}{m_1m_2} \end{pmatrix}.
    % $$
Suppose that we are given the following matrix (which appears as a factor in Examples \ref{ex b and maps in ideal coords dim 3} and \ref{ex ideal coords in dim 3 run through})
    \begin{equation} \label{eqn intro simple rt subgps}
    \begin{pmatrix} 1 & & \\ \frac{1}{m_1} & 1 & \\ 0 & 0 & 1 \end{pmatrix}
        \begin{pmatrix} 1 & & \\ 0 & 1 & \\ 0 & \frac{1}{m_2} & 1 \end{pmatrix}
        \begin{pmatrix} 1 & & \\ \frac{1}{m_3} & 1 & \\ 0 & 0 & 1 \end{pmatrix}
    = \begin{pmatrix}  1 & & \\ \frac{1}{m_1}+\frac{1}{m_3} & 1 & \\ \frac{1}{m_2m_3} & \frac{1}{m_2} & 1 \end{pmatrix}
    \end{equation}
and ask for it to be a Toeplitz matrix, namely that the entries on any given diagonal take the same value. This imposes the following condition on our coordinates:
    $$\frac{1}{m_2}=\frac{1}{m_1}+\frac{1}{m_3}.
    $$
Applying the valuation $\mathrm{Val}_{\mathbf{K}}$ we obtain
    $$\mu_2 = \max \{ \mu_1, \mu_3 \}.
    $$
Indeed we see that when our matrix is a Toeplitz matrix, then the valuations $\mu_i$ form ideal filling, as given in Figure \ref{fig Ideal filling in dimension 3 intro ex}.
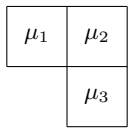
\begin{figure}[ht]
\centering
    \begin{tikzpicture}
        %box lines
        %horizontal
        \draw (0,0) -- (1.6,0);
        \draw (0,-0.8) -- (1.6,-0.8);
        \draw (0.8,-1.6) -- (1.6,-1.6);

        %vertical
        \draw (0,0) -- (0,-0.8);
        \draw (0.8,0) -- (0.8,-1.6);
        \draw (1.6,0) -- (1.6,-1.6);

        %nij labels
        \node at (0.4,-0.4) {\small{$\mu_1$}};

        \node at (1.2,-0.4) {\small{$\mu_2$}};
        \node at (1.2,-1.2) {\small{$\mu_3$}};
    \end{tikzpicture}
\caption{Ideal filling in dimension $3$} \label{fig Ideal filling in dimension 3 intro ex}
\end{figure}
\end{ex*}

The generalisation of this example is given in Theorem \ref{thm G/P Toeplitz mat and ideal fillings}, and is the main result of this thesis.
It gives an interpretation of ideal fillings using Toeplitz matrices (factored into simple root subgroups) over generalised Puiseux series.
The proof of the theorem makes use of the connection between critical points of the superpotential and Toeplitz matrices, given by Rietsch in \cite{Rietsch2008}.

% As a corollary we have the following theorem which gives an interpretation of ideal fillings using Toeplitz matrices over the field of Puiseux series:
%     $$\mathcal{K}= \bigcup_{n\in \mathbb{Z}_{>0}} \mathbb{C}((t^{\frac{1}{n}})).
%     $$
% This result is a special case of Theorem \ref{thm G/P Toeplitz mat and ideal fillings}.
The remaining results of this thesis are closely related to this theorem, and are the following; firstly, we construct a new torus coordinate chart for $B_-$.
This coordinate chart is very natural; it is obtained by multiplying simple root subgroups together in the style of Lusztig (\cite{Lusztig1994}), according to a choice of reduced expression for the longest element of the Weyl group (as in Equation \ref{eqn intro simple rt subgps}) together with a maximal torus factor. See Example \ref{ex b and maps in ideal coords dim 3} and Sections \ref{sec The ideal coordinates} and \ref{sec G/P The ideal coordinates} for the precise definitions in the $GL_n/B$ and $GL_n/P$ cases respectively.

We restrict the superpotential to this chart in order to obtain a polytope, and then show that the tropical critical point in this polytope exactly recovers the ideal filling given by Judd (see Proposition \ref{prop GLn version of Jamie's 6.2} and Corollary \ref{cor nu'_i coords ideal filling for lambda} in the $GL_n/B$ case, and Proposition \ref{prop G/P version of Jamie's 6.2} for the analogous result in the $GL_n/P$ setting). Consequently we call this chart the `ideal coordinate chart'.

We give explicit transformations between the ideal coordinate chart and previously studied charts, namely the coordinate system which gives rise to string polytopes, and a coordinate system arising from Givental-type quivers \cite{Givental1997}.
In the $GL_n/B$ case our main result here is Theorem \ref{thm coord change}. This theorem shows that the string and ideal tori, for a particular reduced expression, coincide. In particular, it gives an explicit transformation between the string and ideal coordinates. In its proof we use a sequence of lemmas to describe the transformations first from the string to the quiver coordinates, and then to the ideal coordinates. %Lemmas \ref{lem form of u_1 and b factorisations}, \ref{lem coord change for u_1 p_i in terms of z_i} and \ref{lem coord change for b m_i in terms of p_i}.
In the $GL_n/P$ setting, the main result here is Theorem \ref{thm b- in terms of y_i}. In this theorem we generalise the transformation between the quiver and ideal coordinates, using the quiver decoration detailed in section \ref{subsec G/P Quiver decoration}.

Finally, we show that our ideal coordinate torus is part of a family of tori, where we go from one to the next by changing reduced expression. If we index our coordinates in the appropriate way by the positive roots, $R_+$, of $GL_n$, each of these alternative tori, together with the superpotential function, gives a polytope in $\mathbb{R}^{R_+}$ (see Section \ref{subsec A family of ideal polytopes}). Moreover, we show that all of these transformations preserve the tropical critical point (Proposition \ref{prop trop crit pt independent of i}).

\section*{Structure of thesis}
\addcontentsline{toc}{section}{Structure of thesis}

This thesis is divided into two parts. The first, based on a stand-alone paper (\cite{Ludenbach2022}), treats the case of full flag varieties $GL_n/B$. In the second part we consider how this generalises to partial flag varieties $GL_n/P$.

\fancyhead[L]{Structure of thesis}

Both of these parts have a similar outlook; in each case we construct the respective `ideal coordinates', and consider quivers, polytopes and critical points of the superpotential. The setting of partial flag varieties is more intricate however, and as such, is not a like-for-like generalisation of the first part of the thesis, indeed the structures and content of the two parts are subtly different. Consequently, to aid clarity, we now outline the structures of the two parts.

The case of full flag varieties is structured as follows: in Section \ref{sec Mirror symmetry for G/B applied to representation theory} we introduce the mirror to the flag variety $GL_n/B$ and present the first of two coordinate systems. We do this since the key to constructing a polytope using the superpotential, is to express the superpotential in some torus coordinate chart. The first system, studied by Berenstein and Kazhdan in \cite{BerensteinKazhdan2000}, \cite{BerensteinKazhdan2007}, is an important coordinate system in this context since it gives rise to string polytopes. The second coordinate system, the `ideal' coordinate system, is new and will be best suited to the tropical critical point. We begin Section \ref{sec The ideal coordinates} by defining this system for the reduced expression given by %$\mathbf{i}_0$
    $$\mathbf{i}_0 = (i_1, \ldots, i_N) := (1,2, \ldots, n-1, 1,2 \ldots, n-2, \ldots, 1, 2, 1).
    $$
The body of the section is devoted to proving Theorem \ref{thm coord change} by constructing the ideal coordinates from the string coordinates. We do this through a sequence of transformations between the string, quiver and ideal coordinate charts, which are described explicitly in Lemmas \ref{lem form of u_1 and b factorisations}, \ref{lem coord change for u_1 p_i in terms of z_i} and \ref{lem coord change for b m_i in terms of p_i}.

Then, following Judd (\cite{Judd2018}), in Section \ref{sec Givental-type quivers and critical points} we introduce quivers as defined by Givental in \cite{Givental1997}, whose arrows and vertices can be decorated with our coordinate charts. From this decoration we can read off information such as the superpotential and the critical point conditions, giving us a combinatorial description of each. The main result of this section is Proposition \ref{prop crit points, sum at vertex is nu_i}, which gives a very simple formula for recovering the ideal coordinates of a critical point from the arrow decoration.

Finally, in Section \ref{sec The tropical viewpoint} we consider everything we have developed up until this point through the lens of tropical geometry. This is where we will discuss the polytopes mentioned above and prove, for a given highest weight, that the ideal filling and tropical critical point coincide (Proposition \ref{prop GLn version of Jamie's 6.2}).
We conclude by generalising the ideal coordinates to arbitrary reduced expressions and presenting our new family of polytopes.

The second half of this thesis, the case of partial flag varieties, is structured as follows: in Sections \ref{sec G/P Notation and definitions} and \ref{sec G/P Landau-Ginzburg models} we introduce notation and the mirror to the $G/P$, where $G=GL_n$. %We recall that the key to constructing a polytope using the superpotential, is to express the superpotential in some torus coordinate chart.
Since the `ideal coordinate system' in the $G/B$ case is best suited to the tropical critical point, we wish to generalise this to the $G/P$ setting. To do so, we first consider a generalisation due to Batyrev, Ciocan-Fontanine, Kim and van Straten (\cite{BatyrevEtAl2000}) of Givental-type quivers (defined in \cite{Givental1997}). Just as in the case of full flag varieties, these generalised quivers, which are the focus of Section \ref{sec Quivers for partial flag varieties}, are defined such that they will also succinctly hold the information of the toric charts, highest weight and superpotential, as well as the critical point conditions for a given highest weight.

In Section \ref{sec A conjecture on the form of elements of ZP} we further develop the connection between the mirror to $G/P$ and the generalised quivers. Namely we present a conjecture, together with supporting evidence, for a more complete interpretation of the quiver coordinates (see Conjecture \ref{conj b lies in Z}).

The goal of Section \ref{sec G/P The ideal coordinates} is to generalise the ideal coordinate system to the case of partial flag varieties. To do this we begin by defining a particular quiver decoration generalising the one from the full flag variety setting. In Proposition \ref{prop G/P crit points, sum at vertex is nu_i} we prove a generalisation of Proposition \ref{prop crit points, sum at vertex is nu_i} which shows that these coordinates are again very suited for studying critical points. The main result of this section is Theorem \ref{thm b- in terms of y_i} which gives an interpretation of these coordinates in terms of factorisations analogous to the definition of the ideal coordinates in the $G/B$ case. We devote the body of Section \ref{sec G/P The ideal coordinates} to the proof of this theorem.

Finally, in Section \ref{sec G/P The tropical viewpoint} we return to tropical geometry; we first construct polytopes using the ideal coordinate chart and then generalise the notion of ideal fillings to the setting of partial flag varieties. This enables us to extend our earlier result (Proposition \ref{prop GLn version of Jamie's 6.2}) that for a given highest weight $\lambda$, the ideal filling and tropical critical point coincide. We conclude with our main theorem which gives an interpretation of ideal fillings using Toeplitz matrices over generalised Puiseux series, as mentioned in the introduction (Theorem \ref{thm G/P Toeplitz mat and ideal fillings}).

In Appendix \ref{append Recovering our coordinates} we show how to recover our ideal coordinates from a general element in the $G/B$ case. Appendix \ref{append Example of complete quiver labelling} is the continuation of a running example in the setting of partial flag varieties.

\newpage

\chapter*{Full flag varieties}
\addcontentsline{toc}{chapter}{Full flag varieties}
\fancyhead[R]{Full flag varieties}
\fancyhead[L]{}

\section[Mirror symmetry for \texorpdfstring{$G/B$}{G/B} applied to representation theory]{Mirror symmetry for $G/B$ applied to representation theory} \label{sec Mirror symmetry for G/B applied to representation theory}

\subsection{Notation and definitions} \label{subsec Notation and definitions}
\fancyhead[L]{1.1 \ \ Notation and definitions}
% \fancyhead[L]{2.1 \ \ Notation and definitions}

Let $\mathbb{K}$ be a field of characteristic $0$, containing a positive semifield, that is, a subset closed under multiplication, addition and division by non-zero elements.
% To describe positive semifields we first recall that a semiring is an algebra $\mathcal{S}=\langle S, \oplus, \otimes, e, e' \rangle$ where the additive structure, $\langle S, \oplus, e \rangle$, is a commutative monoid and the multiplicative structure, $\langle S \setminus \{e\}, \otimes, e' \rangle$, is a monoid with multiplication distributing over addition (from both right and left) and such that multiplication by $e$ annihilates $S$. % J.S. Golan `Semirings and their Applications' 1999
% For our purposes we will assume that multiplication is commutative.
% For a semiring to be positive we require that it is zero sum-free, namely that if a sum is trivial then all summands are also trivial, and entire, namely that it $e$ has no non-trivial factors. % M. Gondran, M. Minoux `Graphs, Dioids and Semirings. New Models and Algorithms' 2008
% We say that a semiring is a semifield if we have a multiplicative element for every element in $S\setminus\{e\}$.
A classic example of a positive semifield
%(a zero sum-free, entire semiring with multiplicative inverse but not necessarily with additive inverse)
is the set $\mathbb{R}_{\geq0} \subset \mathbb{R}$ of non-negative real numbers with the standard operations. %$\mathbb{R}_{\geq0} = \langle [0,\infty),+,\cdot, -^{-1}, 0,1 \rangle$.

Unless otherwise stated we take $G=GL_n(\mathbb{K})$ with $B$, $B_-$ the Borel subgroups of upper and lower triangular matrices. Let $U$, $U_-$ be their respective unipotent radicals, that is the subgroups of upper and lower triangular matrices with all diagonal entries equal to $1$, and let $T=B \cap B_-$ be the diagonal matrices in $G$. The Langlands dual group to $G$ is denoted $G^{\vee}=GL_n(\mathbb{K})$, and may be taken together with the corresponding subgroups $B^{\vee}$, $B^{\vee}_-$, $U^{\vee}$, $U^{\vee}_-$ and $T^{\vee}$ in $G^{\vee}$.

For $i=1,\ldots , n$, we write $\epsilon_i$, $\epsilon^{\vee}_i$ for the standard characters and cocharacters of $T$, corresponding to diagonal matrix entries. Let $X^*(T) = \mathrm{Hom}(T,\mathbb{K}^*)$, $X_*(T)= \mathrm{Hom}(\mathbb{K}^*, T)$ be the respective character and cocharacter lattices, written additively. These are dually paired in the standard way by
    \begin{equation*} \label{eqn bilinear form pair lattices}
    \langle \ , \ \rangle : X^*(T) \times X_*(T) \to \mathrm{Hom}(\mathbb{K}^*, \mathbb{K}^*) \cong \mathbb{Z},
    \end{equation*}
with $\{\epsilon_i\}$ and $\{\epsilon^{\vee}_i\}$, $i=1,\ldots , n$, forming dual bases.

We take $\alpha_{ij} = \epsilon_i - \epsilon_j \in X^*(T)$, $\alpha_{ij}^{\vee} = \epsilon^{\vee}_i - \epsilon^{\vee}_j \in X^*(T^{\vee})= X_*(T)$. Additionally for each $i \in I=\{1, \ldots, n-1\}$ we write $\alpha_i=\alpha_{i, i+1}$, $\alpha_i^{\vee}=\alpha_{i, i+1}^{\vee}$. Then the roots and positive roots of $G$ are
    $$R = \{\alpha_{ij} \ \vert \ i \neq j\} \text{ and } R_+ = \{\alpha_{ij} \ \vert \ i < j\}
    $$
respectively and the simple roots of $G$ are $\{ \alpha_i \ \vert \ i \in I\}$. The Cartan matrix is $A=(a_{ij})$ defined by $a_{ij}=\langle \alpha_j, \alpha_i^{\vee}\rangle$.

The fundamental weights of $G$ are given by $\omega_i = \epsilon_1 + \cdots + \epsilon_i$. Additionally we denote the set of dominant integral weights by
    $${X^*(T)}^+=\{\lambda \in X^*(T) \ \vert \ \langle \lambda, \alpha_{ij}^{\vee} \rangle \geq 0 \ \forall i <j\}.
    $$
For $\lambda \in {X^*(T)}^+$, let $V_{\lambda}$ denote the irreducible representation with highest weight $\lambda$.

Note that we may identify $X^*(T)\otimes \mathbb{R} = X_*(T^{\vee})\otimes \mathbb{R}$ with the dual, $\mathfrak{h}^*_{\mathbb{R}}$, of the Lie algebra of the split real torus of $G$.

The Weyl group of $G$ is the symmetric group, $W=N_G(T)/T = S_n$, generated by the simple reflections $s_i$ for $i \in I$.
% It acts on $X^*(T)$ by permuting the roots and we denote the action of $w \in W$ on $\alpha \in X^*(T)$ by $w\alpha$.
Each simple reflection $s_i\in W$ acts as a reflection on the character lattice $X^*(T)$ as follows:
    $$s_i\gamma = \gamma - \langle \gamma, \alpha^{\vee}_i \rangle \alpha_i \quad \text{for} \ \ \gamma \in X^*(T).
    $$
% where $\alpha^{\vee}=2\alpha / \langle \alpha, \alpha \rangle$.
In particular, the action on the simple roots $\alpha_j$, simple coroots $\alpha^{\vee}_j$ and fundamental weights $\omega_j$ is given by
    $$s_i\alpha_j = \alpha_j - a_{ij} \alpha_i, \quad s_i\alpha^{\vee}_j = \alpha^{\vee}_j - a_{ji} \alpha^{\vee}_i, \quad s_i\omega_j = \omega_j - \delta_{ij} \alpha_i.
    $$
Concretely, in terms of the standard characters $\epsilon_i$, $s_i$ permutes $\epsilon_i$ and $\epsilon_{i+1}$, and fixes all other $\epsilon_j$, thus action preserves $R$, the roots of $G$.

Associated to each simple root $\alpha_i$ there is a homomorphism $\phi_i: SL_2 \to G$, explicitly
    $$\phi_i : \begin{pmatrix} a & b \\ c & d \end{pmatrix} \mapsto
    \begin{pmatrix} 1 & & & & & \\ & \ddots & & & & \\ & & a & b & & \\ & & c & d & & \\ & & & & \ddots & \\ & & & & & 1\end{pmatrix}
    \quad \text{with } a \text{ in position } (i,i),
    $$
and we have a number of $1$-parameter subgroups of $G$ respectively defined by
    $$\begin{aligned}
    & \mathbf{x}_i(z) = \phi_i\begin{pmatrix} 1 & z \\ 0 & 1 \end{pmatrix}, \quad &&
        \mathbf{y}_i(z) = \phi_i\begin{pmatrix} 1 & 0 \\ z & 1 \end{pmatrix}, \\
    & \mathbf{x}_{-i}(z) = \phi_i\begin{pmatrix} z^{-1} & 0 \\ 1 & z \end{pmatrix}, \quad &&
        \mathbf{t}_i(t) = \phi_i\begin{pmatrix} t & 0 \\ 0 & t^{-1} \end{pmatrix}
    \end{aligned}$$
for $z \in \mathbb{K}$, $t \in \mathbb{K}^*$ and $i \in I$.
The simple reflections in the Weyl group, $s_i \in W$, are given explicitly by $s_i=\bar{s}_i T$ where
    $$\bar{s}_i = \mathbf{x}_i(-1)\mathbf{y}_i(1)\mathbf{x}_i(-1) = \phi_i\begin{pmatrix} 0 & -1 \\ 1 & 0 \end{pmatrix}.
    $$
More generally we may write each $w \in W$ as a product with a minimal number of factors, $w=s_{i_1}\cdots s_{i_m}$. We get a representative of $w$ in $N_G(T)$ by taking $\bar{w}=\bar{s}_{i_1}\cdots \bar{s}_{i_m}$. Here $m$ is called the length of $w$, denoted $l(w)$, and the choice of expression $s_{i_1}\cdots s_{i_m}$ is said to be reduced. In particular it is well known that $\bar{w}$ is independent of this choice, \cite{BourbakiCh1-3}. For ease of notation we will often let $\mathbf{i}=(i_1, \ldots, i_m)$ stand for the reduced expression ${s}_{i_1}\cdots {s}_{i_m}$.

In a similar way for $G^{\vee}$ we have, for each $i \in I$, a homomorphism $\phi_i^{\vee}: SL_2 \to G^{\vee}$ and $\mathbf{x}^{\vee}_i(z)$, $\mathbf{y}^{\vee}_i(z)$, $\mathbf{x}^{\vee}_{-i}(z)$, $\mathbf{t}^{\vee}_i(t)$ defined analogously. The Weyl group of $G^{\vee}$ is again the symmetric group and we use the same notation as above.

With this in mind, we make an observation which will be used frequently; given a reduced expression, say $s_{i_1}\cdots s_{i_m}$, we can construct matrices in $G^{\vee}$ which are indexed by $\mathbf{i}=(i_1,\ldots,i_m)$. We do this by taking products of the matrices defined above. An explicit example is given by the following map:
    $$\mathbf{x}^{\vee}_{\mathbf{i}} : (\mathbb{K}^*)^N \to U^{\vee} \cap B_-^{\vee} \bar{w}_0 B_-^{\vee}, \quad (z_1, \ldots, z_N)\mapsto \mathbf{x}^{\vee}_{i_1}(z_1) \cdots \mathbf{x}^{\vee}_{i_N}(z_N).
     $$

\subsection{Landau--Ginzburg models} \label{subsec Landau-Ginzburg models}
\fancyhead[L]{1.2 \ \ Landau--Ginzburg models}
% \fancyhead[L]{2.2 \ \ Landau--Ginzburg models}

The mirror to the flag variety $G/B$ is a pair $(Z,\mathcal{W})$, called a Landau--Ginzburg model, where $Z \subset G^{\vee}$ is an affine variety and $\mathcal{W}:Z\to \mathbb{K}^*$ is a holomorphic function called the superpotential. In order to give a more precise description we first recall Bruhat decomposition, namely that $G$ may be written as a disjoint union of Bruhat cells $B \bar{w} B$ (see \cite[Theorem 8.3.8]{SpringerLAG}):
    $$G= \bigsqcup_{w \in W} B \bar{w} B.
    $$
Similarly we may write $G/B$ as
    $$G/B= \bigsqcup_{w \in W} B \bar{w} B/B \quad \text{with} \quad \mathrm{dim}(B \bar{w} B/B)= l(w).
    $$
We note that the cells $B \bar{w} B$ do not depend on the choice of representative $\bar{w}$.

These Bruhat cells give rise to a partial ordering of Weyl group elements, known as the Bruhat order (see \cite[Theorem 8.5.4]{SpringerLAG}); for $v,w \in W$ we say $v\leq w$ if $B \bar{v} B \subseteq \overline{B \bar{w} B}$. With respect to this ordering there is a unique maximal element $w_0 \in W$ and we set $N=l(w_0)$.

Additionally, we use the Bruhat order to define open Richardson varieties. These are given by intersecting opposite Bruhat cells; for $v,w \in W$ such that $v\leq w$ we have
    $$\mathcal{R}_{v,w}:=(B_- \bar{v}B \cap B \bar{w} B)/B \subset G/B.
    $$
It is well known that $\mathcal{R}_{v,w}$ is smooth, irreducible and has dimension $l(w)-l(v)$, \cite{KazhdanLustig1980}. On the dual side we have
    $$\mathcal{R}^{\vee}_{v,w}:= (B^{\vee}_- \bar{v}B^{\vee} \cap B^{\vee} \bar{w} B^{\vee})/B^{\vee} \subset G^{\vee}/B^{\vee}.
    $$

We now return to  $(Z, \mathcal{W})$, the Landau--Ginzburg model for $G/B$, and define the subvariety
    $$Z:= B_-^{\vee}\cap B^{\vee}\bar{w}_0B^{\vee}\subset G^{\vee}.
    $$
In order to define the superpotential $\mathcal{W}$, we let $\chi:U^{\vee} \to \mathbb{K}$ be the sum of above-diagonal elements
    $$\chi(u) := \sum_{i=1}^{n-1} u_{i\, i+1}, \quad u=(u_{ij})\in U^{\vee}.
    $$
Then the superpotential is given by
    $$\mathcal{W}:Z \to \mathbb{K}^*, \quad u_1d\bar{w}_0u_2 \mapsto \chi(u_1)+ \chi(u_2)
    $$
where $u_1, u_2 \in U^{\vee}$ and $d \in T^{\vee}$.
This map will appear frequently in subsequent sections.

The motivation for introducing the Landau--Ginzburg model is to study the representation theory of $G$ using the mirror to $G/B$. It is natural then to equip $Z$ with highest weight and weight maps. The highest weight map recovers the original torus factor, $d$, as follows:
    $$\mathrm{hw}:Z \to T^{\vee}, \quad u_1d\bar{w}_0u_2 \mapsto d.
    $$
For the weight map we first note that each element $b \in Z$ may be written as $b=[b]_-[b]_0$ with $[b]_- \in U_-^{\vee}$, $[b]_0\in T^{\vee}$. Then the weight map is given by the projection
    $$
    \mathrm{wt}:Z \to T^{\vee}, \quad b \mapsto [b]_0.
    $$
We will often write the above decomposition of $b$ as $b=[b]_-t_R$ to remind us that the torus factor is taken on the right.
While $\mathrm{hw}$ is defined on all of $B^{\vee} \bar{w}_0 B^{\vee}$ and $\mathrm{wt}$ is defined on all of $B_-^{\vee}$, these maps will only be of relevance to us as maps on $Z$.

% \newpage
\subsection{The string coordinates} \label{subsec The string coordinates}
\fancyhead[L]{1.3 \ \ The string coordinates}
% \fancyhead[L]{2.3 \ \ The string coordinates}

In order to make the connection with representation theory we restrict our attention to various toric charts $T^{\vee} \times (\mathbb{K}^*)^N \to Z$, indexed by reduced expressions for $w_0$. The first chart we want to consider is useful for reconstructing the string polytope via the superpotential.
These `string coordinates', which we introduce in this section, were used by Judd \cite{Judd2018} following Chhaibi \cite{ChhaibiThesis13}, who was in turn inspired by the work of Berenstein and Kazhdan \cite{BerensteinKazhdan2000}, \cite{BerensteinKazhdan2007}.

The toric chart in question is defined by the composition of a number of maps which we write here for overview and then define in detail:
\begin{center}
\begin{tikzcd}[column sep=1.25cm]
    T^{\vee} \times (\mathbb{K}^*)^N \arrow[r, " {\left(id, \, \mathbf{x}_{-\mathbf{i}}^{\vee}\right)} "] & T^{\vee}\times (B_-^{\vee} \cap U^{\vee} \bar{w}_0 U^{\vee}) \arrow[r, " \tau "] & T^{\vee}\times (U^{\vee} \cap B_-^{\vee} \bar{w}_0 B_-^{\vee}) \arrow[r, " \Phi "] & Z \\
\end{tikzcd}
\end{center}

\vspace{-0.7cm}
The first map constructs a matrix parametrised by a torus. This parametrisation is dependent on $\mathbf{i}=(i_1, \ldots, i_N)$, which we take to stand for a reduced expression $s_{i_1} \cdots s_{i_N}$ for $w_0$. The map is given as follows:
% If we let $\mathbf{i}=(i_1, \ldots, i_N)$ stand for a reduced expression $s_{i_1} \cdots s_{i_N}$ for $w_0$, then first map constructs a matrix indexed by $\mathbf{i}$ as follows:
    $$\mathbf{x}_{-\mathbf{i}}^{\vee} : (\mathbb{K}^*)^N \to B_-^{\vee}\cap U^{\vee}\bar{w}_0 U^{\vee}\,, \quad (z_1, \ldots, z_N)\mapsto \mathbf{x}_{-i_1}^{\vee}(z_1) \cdots \mathbf{x}_{-i_N}^{\vee}(z_N).
    $$

The second map, $\tau$, may be written as the composition of a twist map $\eta^{w_0,e}$ and an involution $\iota$.  We present $\tau$ in this way since the involution will be helpful later. The twist map is defined to be
    $$\eta^{w_0,e}:B_-^{\vee}\cap U^{\vee}\bar{w}_0 U^{\vee} \to U^{\vee} \cap B_-^{\vee} \bar{w}_0 B_-^{\vee}\,, \quad b \mapsto [(\bar{w}_0b^T)^{-1}]_+.
    $$
Here $b^T$ is the transpose of $b$ and $[g]_+$ is given by the LDU decomposition of $g$, namely $g=[g]_- [g]_0 [g]_+ $ where $[g]_- \in U_-^{\vee}$, $[g]_0\in T^{\vee} $ and $[g]_+ \in U^{\vee}$. The involution is given by
    $$\iota:G^{\vee} \to G^{\vee}, \quad g\mapsto (\bar{w}_0g^{-1}\bar{w}_0^{-1})^T.
    $$
We note that this map preserves $U^{\vee}$. It remains to define $\tau$ by applying the composition $\iota \circ \eta^{w_0,e}$ to the second factor:
    $$\tau : T^{\vee}\times (B_-^{\vee} \cap U^{\vee} \bar{w}_0 U^{\vee}) \to T^{\vee}\times (U^{\vee} \cap B_-^{\vee} \bar{w}_0 B_-^{\vee}), \ (d,u) \mapsto \left(d,\iota(\eta^{w_0,e}(u))\right).
    $$

The final map in the definition of the string toric chart is an isomorphism which allows us to factorise elements of $Z$:
    $$\Phi :T^{\vee}\times (U^{\vee} \cap B_-^{\vee} \bar{w}_0 B_-^{\vee}) \to Z, \quad (d,u_1) \mapsto u_1 d \bar{w}_0 u_2.$$
Here $u_2 \in U^{\vee}$ is the unique element such that $u_1 d\bar{w}_0 u_2 \in Z$.
%, that is $u_2=\left[ u_1 d \bar{w}_0 \right]_+^{-1}$.
To see that such a $u_2$ exists, we take $u_1 \in U^{\vee}$, $d \in T^{\vee}$ and consider the following LDU decomposition:
    $$u_1 d \bar{w}_0 = \left[ u_1 d \bar{w}_0 \right]_- \left[ u_1 d \bar{w}_0 \right]_0 \left[ u_1 d \bar{w}_0 \right]_+.
    $$
Since $\left[ u_1 d \bar{w}_0 \right]_+ \in U^{\vee}$, it has a well-defined inverse in $U^{\vee}$ and thus $u_1 d \bar{w}_0 \left[ u_1 d \bar{w}_0 \right]_+^{-1} \in Z$, as desired.
% To see that $u_2=\left[ u_1 d \bar{w}_0 \right]_+^{-1}$ is unique, we first note that for any $A \in U^{\vee}$ we have $u_1 d \bar{w}_0 A \left[ u_1 d \bar{w}_0 A \right]_+^{-1} \in Z$ by a similar argument to the above, thus we must show that
%     $$ u_2 = A \left[ u_1 d \bar{w}_0 A \right]_+^{-1}.
%     $$
% Now, using LDU decompositions, may write $u_1 d \bar{w}_0 A$ in two ways:
%     $$\begin{aligned}
%     u_1 d \bar{w}_0 A &= \left[ u_1 d \bar{w}_0 A \right]_- \left[ u_1 d \bar{w}_0 A \right]_0 \left[ u_1 d \bar{w}_0 A \right]_+ \\
%         &= \left[ u_1 d \bar{w}_0 \right]_- \left[ u_1 d \bar{w}_0 \right]_0 \left[ u_1 d \bar{w}_0 \right]_+ A.
%     \end{aligned}
%     $$
% In particular, $\left[ u_1 d \bar{w}_0 A \right]_+ = \left[ u_1 d \bar{w}_0 \right]_+ A$ by the uniqueness of LDU decompositions and since $A\in U^{\vee}$. Consequently
%     $$ A \left[ u_1 d \bar{w}_0 A \right]_+^{-1} = A \left(\left[ u_1 d \bar{w}_0 \right]_+ A\right)^{-1} = A A^{-1} \left[ u_1 d \bar{w}_0 \right]_+^{-1} = u_2.
%     $$
To see that $u_2=\left[ u_1 d \bar{w}_0 \right]_+^{-1}$ is unique, we take some non-trivial $u' \in U^{\vee}$ and observe that since $u_1 d\bar{w}_0 u_2 \in B_-^{\vee}$ we must have $u_1 d\bar{w}_0 u_2 u' \notin B_-^{\vee}$, and consequently $u_1 d\bar{w}_0 u_2 u' \notin Z$.

With the above notation, the string toric chart on $Z$ (corresponding to $\mathbf{i}$) is defined to be
    \begin{equation} \label{eqn string coord chart defn}
    \varphi_{\mathbf{i}}:T^{\vee} \times (\mathbb{K}^*)^N \to Z, \quad \varphi_{\mathbf{i}}(d,\boldsymbol{z}) = \Phi\circ \tau\left(d, \mathbf{x}_{-\mathbf{i}}^{\vee}(\boldsymbol{z})\right) = \Phi\left(d, \iota \left(\eta^{w_0,e}\left(\mathbf{x}_{-\mathbf{i}}^{\vee}(\boldsymbol{z})\right)\right)\right).
    \end{equation}

Later in this work we will need the specific toric chart corresponding to
    $$\mathbf{i}_0 = (i_1, \ldots, i_N) := (1,2, \ldots, n-1, 1,2 \ldots, n-2, \ldots, 1, 2, 1)
    $$
so, unless otherwise stated, from now on we will take
    $$
    u:=\mathbf{x}_{-\mathbf{i}_0}^{\vee}(\boldsymbol{z}), \quad u_1 := \iota(\eta^{w_0,e}(u)), \quad b:= u_1 d \bar{w}_0 u_2.
    $$
Using this, the composition of maps defining $\varphi_{\mathbf{i}_0}$, the string toric chart corresponding to $\mathbf{i}_0$, may be visualised as follows:
% \vspace{-0.3cm}
\begin{center}
\begin{tikzcd}[row sep=0.3em, column sep=1.35cm]
    T^{\vee} \times (\mathbb{K}^*)^N \arrow[r, " {\left(id, \, \mathbf{x}_{-\mathbf{i}_0}^{\vee}\right)} "] & T^{\vee}\times (B_-^{\vee} \cap U^{\vee} \bar{w}_0 U^{\vee}) \arrow[r, " \tau "] & T^{\vee}\times (U^{\vee} \cap B_-^{\vee} \bar{w}_0 B_-^{\vee}) \arrow[r, " \Phi "] & Z \\
    (d,\boldsymbol{z}) \arrow[r, mapsto] & (d,u) \arrow[r, mapsto] & (d,u_1) \arrow[r, mapsto] & b \\
\end{tikzcd}
\end{center}

\begin{ex}[Dimension 3] \label{ex dim 3 z coords} The reduced expression is $\mathbf{i}_0=(1,2,1)$.
We will start with $\left(d, (z_1, z_2, z_3 )\right)$ and apply the sequence of maps defined above.

Applying $\mathbf{x}_{-\mathbf{i}_0}^{\vee}$ to $(z_1,z_2,z_3)$ gives the matrix $u$:
    $$u=\mathbf{x}_{-\mathbf{i}_0}^{\vee}(z_1, z_2, z_3) = \begin{pmatrix}
        \frac{1}{z_1z_3} & & \\ \frac{1}{z_3}+\frac{z_1}{z_2} & \frac{z_1z_3}{z_2} & \\ 1 & z_3 & z_2
    \end{pmatrix}, \quad
    \bar{w}_0=\begin{pmatrix} & & 1 \\ & -1 & \\ 1 & & \end{pmatrix}.
    $$
We recall the definition of second map: $\tau\left(d, (z_1, z_2, z_3 )\right) = \left(d, \iota \circ \eta^{w_0,e}(u)\right)$. To see this in action we first evaluate $\eta^{w_0,e}(u)$ and then apply $\iota$ to obtain the matrix $u_1$:
    $$\begin{aligned}
    \eta^{w_0,e}(u)
        &= \left[ \left( \bar{w}_0 \begin{pmatrix} \frac{1}{z_1z_3} & \frac{z_1}{z_2}+\frac{1}{z_3} & 1 \\ & \frac{z_1z_3}{z_2} & z_3 \\ & & z_2  \end{pmatrix} \right)^{-1} \right]_+
        = \left[ \begin{pmatrix} 1 & z_1 +\frac{z_2}{z_3} & z_1z_3 \\ -\frac{1}{z_1} & -\frac{z_2}{z_1z_3} & \\ \frac{1}{z_2} & & \end{pmatrix} \right]_+ \\
        &= \begin{pmatrix} 1 & z_1 +\frac{z_2}{z_3} & z_1z_3 \\ 0 & 1 & z_3 \\ 0 & 0 & 1 \end{pmatrix}
    \end{aligned}
    $$
    $$u_1=\iota\left(\eta^{w_0,e}(\mathbf{x}_{-\mathbf{i}_0}^{\vee}(\boldsymbol{z}))\right) = \left(\bar{w}_0 \begin{pmatrix} 1 & z_1 +\frac{z_2}{z_3} & z_1z_3 \\ 0 & 1 & z_3 \\ 0 & 0 & 1 \end{pmatrix}^{-1} \bar{w}_0^{-1}\right)^T
        = \begin{pmatrix} 1 & z_3 & z_2 \\ & 1 & z_1 +\frac{z_2}{z_3} \\ & & 1\end{pmatrix}
    $$
Of note, we can factorise this matrix using a different reduced expression; $\mathbf{i}'_0=(2,1,2)$. We obtain
    \begin{equation*}
    \tau\left(d, \mathbf{x}_{-\mathbf{i}_0}^{\vee}\left(z_1, z_2, z_3 \right)\right)
    = \left(d, u_1\right) = \left(d, \mathbf{x}_{\mathbf{i}'_0}^{\vee}\left(z_1, z_3, \frac{z_2}{z_3}\right)\right).
    \end{equation*}

It remains to apply the map $\Phi$. We recall that $u_2 \in U^{\vee}$ will be the unique element such that $b=u_1 d\bar{w}_0 u_2 \in Z$.
$$\begin{aligned}
    b&= \Phi\left(d, u_1\right) = u_1 d \bar{w}_0u_2 \\
    &= \begin{pmatrix} 1 & z_3 & z_2 \\ & 1 & z_1 +\frac{z_2}{z_3} \\ & & 1\end{pmatrix}
        \begin{pmatrix} d_1& &  \\ & d_2 & \\ & & d_3 \end{pmatrix}
        \begin{pmatrix} & & 1 \\ & -1 & \\ 1 & & \end{pmatrix}
        \begin{pmatrix} 1 & \frac{d_2}{d_3} \frac{z_3}{z_2} & \frac{d_1}{d_3}\frac{1}{z_1z_3} \\ & 1 & \frac{d_1}{d_2}\left(\frac{1}{z_3}+\frac{z_2}{z_1z_3^2}\right) \\ & & 1\end{pmatrix} \\
    &=\begin{pmatrix} d_3 z_2 & & \\ d_3 \left(z_1+\frac{z_2}{z_3} \right) & d_2 \frac{z_1z_3}{z_2} & \\ d_3 & d_2 \frac{z_3}{z_2} & d_1 \frac{1}{z_1z_3} \end{pmatrix} \\
    \end{aligned}$$
From our construction of $b$ we see that the superpotential is
    $$\mathcal{W}(d,\boldsymbol{z}) = z_1 +\frac{z_2}{z_3} + z_3 + \frac{d_1}{d_2}\left(\frac{1}{z_3} + \frac{z_2}{z_1z_3^2}\right) + \frac{d_2}{d_3}\frac{z_3}{z_2}
    $$
and the weight matrix is
    \begin{equation} \label{eqn dim 3 wt matrix z coords}
    \mathrm{wt}(b)=\begin{pmatrix} d_3z_2 & & \\ & d_2\frac{z_1z_3}{z_2} & \\ & & d_1\frac{1}{z_1z_3} \end{pmatrix}.
    \end{equation}
\end{ex}

\newpage
\subsection{The form of the weight matrix} \label{subsec The form of the weight matrix}
\fancyhead[L]{1.4 \ \ The form of the weight matrix}
% \fancyhead[L]{2.4 \ \ The form of the weight matrix}

In this section we generalise the formula for the weight matrix $\mathrm{wt}(b)=t_R$ in terms of the string coordinates, $(d, \boldsymbol{z})$, as in the above example, Equation (\ref{eqn dim 3 wt matrix z coords}). We begin with the two factorisations of $b$ that we have already seen:
    $$b=u_1d \bar{w}_0 u_2 = [b]_- t_R \quad \text{where} \quad [b]_- \in U^{\vee}_-.
    $$
Now recalling the involution $\iota$, we see this acts on elements $b\in Z$ as
    $$\iota(b)=\iota(u_1) \bar{w}_0 d^{-1} \iota(u_2).
    $$
Defining $\tilde{u}_i:=\iota(u_i)$ and $\tilde{b}:=\iota(b)$ gives $\tilde{b}=\tilde{u}_1 \bar{w}_0 d^{-1} \tilde{u}_2$. Moreover we can write
    $$\tilde{b}= \left[\tilde{b}\right]_-\tilde{t}_R \quad \text{where} \quad \left[\tilde{b}\right]_- \in U^{\vee}_-, \ \tilde{t}_R \in T^{\vee}.
    $$
Then
    $$\tilde{b}=\iota(b) = \left(\bar{w}_0 b^{-1}\bar{w}_0^{-1}\right)^T
        = \left(\bar{w}_0 \left([b]_- t_R \right)^{-1}\bar{w}_0^{-1}\right)^T
        = \left(\bar{w}_0[b]_-^{-1}\bar{w}_0^{-1}\right)^T \bar{w}_0t_R^{-1}\bar{w}_0^{-1}.
    $$
Thus $\tilde{t}_R = \bar{w}_0t_R^{-1}\bar{w}_0^{-1}$ or equivalently, $t_R = \bar{w}_0\tilde{t}_R^{-1}\bar{w}_0^{-1}$.

\begin{prop} \label{prop form of t_R using d,u}
Let $N=\binom{n}{2}$ and recall the definition $u:=\mathbf{x}_{-\mathbf{i}_0}^{\vee}(\boldsymbol{z})$ where
    $$\mathbf{i}_0 = (i_1, \ldots, i_N) := (1,2, \ldots, n-1, 1,2 \ldots, n-2, \ldots, 1, 2, 1).
    $$
Then $t_R = \bar{w}_0d[u]_0\bar{w}_0^{-1}$.
\end{prop}

The proof of Proposition \ref{prop form of t_R using d,u} will require the following lemma:
\begin{lem} \label{lem [(bar w_0u^T)^-1]_0=I}
Let $\mathbf{i}=(i_1, \ldots, i_N)$ stand for a reduced expression $s_{i_1} \cdots s_{i_N}$ for $w_0$ and take $\boldsymbol{z}\in \mathbb{K}^N$. If $A= \mathbf{x}_{-i_1}^{\vee}(z_1) \cdots \mathbf{x}_{-i_N}^{\vee}(z_N)$ and $I$ denotes the identity matrix, then
    $$\left[\left(\bar{w}_0A^T\right)^{-1}\right]_0=I.$$
\end{lem}

\begin{proof}[Proof of Proposition \ref{prop form of t_R using d,u}]

Since $t_R = \bar{w}_0\tilde{t}_R^{-1}\bar{w}_0^{-1}$ we will in fact prove that $\tilde{t}_R= d^{-1}[u]_0^{-1}$.
Recalling the definition
        $$u_1 := \iota(\eta^{w_0,e}(u))
        $$
we see that
        $$\tilde{u}_1 := \iota(u_1)=\eta^{w_0,e}(u)=\left[\left(\bar{w}_0u^T\right)^{-1}\right]_+.$$
We have
    $$\left(u^T\right)^{-1}  = \left(\bar{w}_0u^T\right)^{-1}\bar{w}_0= \left[\left(\bar{w}_0 u^T\right)^{-1}\right]_- \left[\left(\bar{w}_0u^T\right)^{-1}\right]_0 \tilde{u}_1 \bar{w}_0.
    $$
In particular by Lemma \ref{lem [(bar w_0u^T)^-1]_0=I} we see that
    $$\left(u^T\right)^{-1} \in U^{\vee}_- \tilde{u}_1\bar{w}_0 \quad
    \Rightarrow \quad u^{-1} \in \bar{w}_0^{-1} \tilde{u}_1^T U^{\vee}_+.
    $$

Thus if $\lambda\in {X^*(T)}^+$ is a dominant integral weight and we denote a corresponding highest weight vector by $v^+_{\lambda} \in V_{\lambda}$, then
    $$ u^{-1} \cdot v^+_{\lambda} = \bar{w}_0^{-1} \tilde{u}_1^T \cdot v^+_{\lambda}.
    $$
This expression allows for two computations of the coefficient of the highest weight vector in $u^{-1} \cdot v^+_{\lambda}$. Firstly, since $u=[u]_0[u]_-$ we have
    $$\left\langle u^{-1}\cdot v^+_{\lambda}, v^+_{\lambda} \right\rangle = \lambda\left([u]_0^{-1}\right).
    $$
Secondly, we rewrite $\bar{w}_0^{-1} \tilde{u}_1^T$ using $\tilde{b}=\tilde{u}_1 \bar{w}_0 d^{-1} \tilde{u}_2=\left[\tilde{b}\right]_-\tilde{t}_R$;
    $$\bar{w}_0^{-1} \tilde{u}_1^T = \bar{w}_0^{-1} \left(\tilde{b}\tilde{u}_2^{-1}d\bar{w}_0^{-1}\right)^T = \bar{w}_0^{-1} \bar{w}_0 d^T \left(\tilde{u}_2^{-1}\right)^T \tilde{b}^T = d \left(\tilde{u}_2^{-1}\right)^T \tilde{t}_R \left[\tilde{b}\right]_-^T.
    $$
Then we see that the result follows from the second computation;
    $$\begin{aligned} \lambda\left([u]_0^{-1}\right)=\left\langle u^{-1}\cdot v^+_{\lambda}, v^+_{\lambda} \right\rangle =
    \left\langle \bar{w}_0^{-1}\tilde{u}_1^T\cdot v^+_{\lambda}, \ v^+_{\lambda} \right\rangle
        &=\big\langle d \left(\tilde{u}_2^{-1}\right)^T \tilde{t}_R \left[\tilde{b}\right]_-^T  \cdot v^+_{\lambda}, \ v^+_{\lambda} \big\rangle \\
        &=\lambda(\tilde{t}_R) \lambda(d) \big\langle \left(\tilde{u}_2^{-1}\right)^T \left[\tilde{b}\right]_-^T\cdot v^+_{\lambda}, \ v^+_{\lambda} \big\rangle =\lambda(\tilde{t}_R) \lambda(d).
    \end{aligned}
    $$
\end{proof}

\begin{proof}[Proof of Lemma \ref{lem [(bar w_0u^T)^-1]_0=I}]
Taking the decomposition
    $$\left(\bar{w}_0A^T\right)^{-1} = \left[\left(\bar{w}_0 A^T\right)^{-1}\right]_- \left[\left(\bar{w}_0A^T\right)^{-1}\right]_0 \left[\left(\bar{w}_0 A^T\right)^{-1}\right]_+
    $$
and defining
    $$U:=\left[\left(\bar{w}_0 A^T\right)^{-1}\right]_+^{-1} \in U^{\vee}_+ \ , \quad D:= \left[\left(\bar{w}_0A^T\right)^{-1}\right]_0^{-1} \in T^{\vee} \ , \quad L:= \left[\left(\bar{w}_0 A^T\right)^{-1}\right]_-^{-1} \in U^{\vee}_-
    $$
we see that
    $$
    \left(\bar{w}_0A^T\right)^{-1} = L^{-1} D^{-1} U^{-1} \quad \Rightarrow \quad A^T = \bar{w}_0^{-1} UDL.
    $$
We will show that $D=I$.

We recall
    $$\mathbf{x}_{-i}^{\vee}(z)^T = \phi^{\vee}_i\begin{pmatrix} z^{-1} & 1 \\ 0 & z \end{pmatrix} \ , \quad \bar{s}_i^{-1} = \phi_i\begin{pmatrix} 0 & 1 \\ -1 & 0 \end{pmatrix}.
    $$

For any fundamental representation $V_{\omega_i}=\bigwedge^i\mathbb{C}^n$, applying $\bar{w}_0^{-1}$ to a lowest weight vector $v^-_{\omega_i} = v_{n-i+1}\wedge \cdots \wedge v_n $ gives the corresponding highest weight vector $v^+_{\omega_i} = v_1\wedge \cdots \wedge v_i$, so
    $$\left\langle \bar{w}_0^{-1} \cdot v^-_{\omega_i} , \ v^+_{\omega_i} \right\rangle = 1.
    $$
Similarly applying $A^T=\mathbf{x}_{-i_N}^{\vee}(z_N)^T \cdots \mathbf{x}_{-i_1}^{\vee}(z_1)^T $ gives
    $$\left\langle A^T \cdot v^-_{\omega_i} , \ v^+_{\omega_i} \right\rangle = \left\langle \mathbf{x}_{-i_N}^{\vee}(z_N)^T \cdots \mathbf{x}_{-i_1}^{\vee}(z_1)^T \cdot v^-_{\omega_i} , \ v^+_{\omega_i} \right\rangle = 1.
    $$

We may also evaluate the coefficient of the highest weight vector in $A^T \cdot v^-_{\lambda}$ using the expression $A^T = \bar{w}_0^{-1} UDL$;
    $$\begin{aligned}
    \left\langle A^T \cdot v^-_{\omega_i} , \ v^+_{\omega_i} \right\rangle
    = \left\langle \bar{w}_0^{-1} UDL \cdot v^-_{\omega_i} , \ v^+_{\omega_i} \right\rangle
    = \left\langle \bar{w}_0^{-1} UD \cdot v^-_{\omega_i} , \ v^+_{\omega_i} \right\rangle
        &= \omega_i(D) \left\langle \bar{w}_0^{-1} U \cdot v^-_{\omega_i} , \ v^+_{\omega_i} \right\rangle \\
        &= \omega_i(D) \left\langle \bar{w}_0^{-1} \cdot v^-_{\omega_i} , \ v^+_{\omega_i} \right\rangle = \omega_i(D).
    \end{aligned}
    $$
Thus $\omega_i(D)=1$ for all $i$, so $D=I$.
\end{proof}

Now that we better understand the weight matrix $\mathrm{wt}(b) = t_R$, we complete this section by expressing it directly in terms of $(d,\boldsymbol{z})$.

\begin{cor} \label{cor wt matrix in z string coords}
The weight matrix $t_R$, given in terms of the string coordinates $(d,\boldsymbol{z})$, is the diagonal matrix with entries
\begin{equation} \label{eqn general formula wt matrix z coords}
    \left( t_R \right)_{n-j+1, n-j+1} =  \frac{d_j \prod_{\substack{1\leq m \leq N \\ i_m=j-1}} z_m}{\prod_{\substack{1\leq m \leq N \\ i_m=j}} z_m}, \quad j=1, \ldots, n.
    \end{equation}
\end{cor}

\begin{proof}
We recall the matrix $u = \mathbf{x}_{-i_1}^{\vee}(z_1) \cdots \mathbf{x}_{-i_N}^{\vee}(z_N)$ where
    $$\mathbf{i}_0 = (i_1, \ldots, i_N) := (1,2, \ldots, n-1, 1,2 \ldots, n-2, \ldots, 1, 2, 1)
    $$
and
    $$\mathbf{x}_{-i}^{\vee}(z) = \phi_i\begin{pmatrix} z^{-1} & 0 \\ 1 & z \end{pmatrix}.
    $$
We see that
    \begin{align}
    \left( [u]_0 \right)_{jk}
    & = \begin{cases}
        \prod_{m=1}^{N} \left( \mathbf{x}_{-i_m}^{\vee}(z_m) \right)_{jj} & \text{if} \ j=k\\
        0 & \text{if} \ j\neq k
    \end{cases} \nonumber \\
    & = \begin{cases}
        \frac{\prod_{\substack{1\leq m \leq N \\ i_m=j-1}} z_m}{\prod_{\substack{1\leq m \leq N \\ i_m=j}} z_m} & \text{if} \ j=k\\
        0 & \text{if} \ j\neq k \label{eqn wt matrix z coords}
    \end{cases}
    \end{align}
noting that $\prod_{\substack{1\leq m \leq N \\ i_m=0}} z_m = 1 = \prod_{\substack{1\leq m \leq N \\ i_m=n}} z_m$.
Thus since $t_R = \bar{w}_0d[u]_0\bar{w}_0^{-1}$ we see that it is the diagonal matrix with entries
    $$\left( t_R \right)_{n-j+1, n-j+1} =  \frac{d_j \prod_{\substack{1\leq m \leq N \\ i_m=j-1}} z_m}{\prod_{\substack{1\leq m \leq N \\ i_m=j}} z_m}, \quad j=1, \ldots, n.
    $$
\end{proof}

\section{The ideal coordinates} \label{sec The ideal coordinates}
\fancyhead[L]{2 \ \ The ideal coordinates}
% \fancyhead[L]{3 \ \ The ideal coordinates}

Our preferred coordinate system on $Z$, which we will call the ideal coordinate system, is far more natural than the string toric chart since it is easier to define. We begin by recalling the reduced expression
    $$\mathbf{i}_0 = (i_1, \ldots, i_N) := (1,2, \ldots, n-1, 1,2 \ldots, n-2, \ldots, 1, 2, 1)$$
and consider the map
    \begin{equation*}
    \left(\mathbb{K}^*\right)^N \times T^{\vee} \to Z\,, \quad \left((m_1, \ldots, m_N), t_R\right) \mapsto \mathbf{y}_{i_1}^{\vee}\left(\frac{1}{m_1}\right) \cdots \mathbf{y}_{i_N}^{\vee}\left(\frac{1}{m_N}\right) t_R.
    \end{equation*}
We recall that $Z$ has two projections to $T^{\vee}$, given by the highest weight and weight maps. In the previous coordinate system the highest weight map was obvious due to the form of $b=u_1d\bar{w}_0 u_2$, whereas the weight required more effort to compute. In this new system the weight is much more straightforward.

We could consider this system with coordinates $(\boldsymbol{m},t_R)$, but instead, for ease of later application, we wish to work with coordinates $(d, \boldsymbol{m})$ which we call the ideal coordinates\footnote{
    The choice to work with the inverted coordinates $\frac{1}{m_i}$ is motivated by the main result of Section \ref{sec Givental-type quivers and critical points}, Proposition \ref{prop crit points, sum at vertex is nu_i}, which also gives rise to the name for this coordinate system.
}:
    $$\psi : T^{\vee} \times \left(\mathbb{K}^*\right)^N \to Z\,, \quad \left(d, (m_1, \ldots, m_N)\right) \mapsto \mathbf{y}_{i_1}^{\vee}\left(\frac{1}{m_1}\right) \cdots \mathbf{y}_{i_N}^{\vee}\left(\frac{1}{m_N}\right) t_R(d,\boldsymbol{m}).
    $$
Here $t_R(d,\boldsymbol{m})$ is the weight matrix given now in terms of the coordinates $(d, \boldsymbol{m})$. An explicit description of the map
    $$T^{\vee} \times \left(\mathbb{K}^*\right)^N \to T^{\vee} \,, \quad \left(d,\boldsymbol{m}\right) \mapsto t_R(d,\boldsymbol{m})
    $$
will be given shortly.

First, we present the main theorem of this section. For a given highest weight matrix, $d$, this theorem describes the coordinate change from the string coordinates to the ideal coordinates, allowing us to move freely between the two systems:
\begin{thm} \label{thm coord change}
To change from the string coordinates, $(d,\boldsymbol{z})$, to the ideal coordinates, $(d,\boldsymbol{m})$, we first let
    $$s_k := \sum_{j=1}^{k-1}(n-j)
    $$
then for $k=1,\ldots, n-1$, $a=1,\ldots, n-k$ we have
    $$m_{s_k+a} = \begin{cases}
            z_{1+s_{n-a}} &\text{if } k=1, \\
            \frac{z_{k+s_{n-k-a+1}}}{z_{k-1+s_{n-k-a+1}}} &\text{otherwise}.
            \end{cases}
    $$
\end{thm}

The string and ideal coordinate systems are related by repeated application of a theorem known as the Chamber Ansatz, which we will discuss in Section \ref{subsec The Chamber Ansatz}. We will then develop our understanding of the relation between the string and ideal coordinates in Sections \ref{subsec Chamber Ansatz minors} and \ref{subsec The coordinate change}, culminating in the proof of Theorem \ref{thm coord change}. To complete the current section we present an example followed by a further application of this theorem.

\begin{ex} \label{ex b and maps in ideal coords dim 3}
In dimension 3 the coordinate change is
    $$m_1=z_3, \quad m_2 = z_1, \quad m_3= \frac{z_2}{z_1}   \quad \qquad z_1=m_2, \quad z_2=m_2m_3, \quad z_3=m_1.
    $$
In Example \ref{ex dim 3 z coords} we saw that the matrix $b$ was given by
    $$b= \Phi\left(d, \mathbf{x}_{\mathbf{i}'_0}^{\vee}\left(z_1, z_3, \frac{z_2}{z_3}\right)\right) = \begin{pmatrix} d_3 z_2 & & \\ d_3 \left(z_1+\frac{z_2}{z_3} \right) & d_2 \frac{z_1z_3}{z_2} & \\ d_3 & d_2 \frac{z_3}{z_2} & d_1 \frac{1}{z_1z_3} \end{pmatrix}.
    $$
With the new coordinates we have
    $$b =\begin{pmatrix} d_3 m_2m_3 & & \\ d_3 \left(m_2+\frac{m_2m_3}{m_1} \right) & d_2 \frac{m_1}{m_3} & \\ d_3  & d_2 \frac{m_1}{m_2m_3} & d_1 \frac{1}{m_1m_2} \end{pmatrix}
    = \mathbf{y}_{\mathbf{i}_0}^{\vee}\left(\frac{1}{m_1}, \frac{1}{m_2}, \frac{1}{m_3} \right)
        \begin{pmatrix} d_3 m_2m_3 & & \\ & d_2 \frac{m_1}{m_3} & \\ & & d_1 \frac{1}{m_1m_2} \end{pmatrix}.
    $$
Additionally the superpotential is now given by
    $$\mathcal{W}(d,\boldsymbol{m}) = m_1 + m_2 + \frac{m_2m_3}{m_1} + \frac{d_2}{d_3} \frac{m_1}{m_2m_3} + \frac{d_1}{d_2} \left(\frac{m_3}{{m_1}^2}+\frac{1}{m_1}\right).
    $$
We again also give the weight matrix:
    $$\mathrm{wt}(b) = \begin{pmatrix} d_3 m_2m_3 & & \\ & d_2 \frac{m_1}{m_3} & \\ & & d_1 \frac{1}{m_1m_2} \end{pmatrix}.
    $$
\end{ex}

We now present an application of Theorem \ref{thm coord change} on the weight matrix, namely we describe $\mathrm{wt}(b)=t_R$ in terms of the ideal coordinates.

\begin{cor} \label{cor wt matrix in m ideal coords}
The weight matrix $t_R$, given in terms of the ideal coordinates $(d, \boldsymbol{m})$, is the diagonal matrix with entries
\begin{equation} \label{eqn general formula wt matrix m coords}
    \left( t_R \right)_{n-j+1, n-j+1} =  \frac{d_j \prod\limits_{k=1, \ldots,j-1} m_{s_k+(j-k)}}{\prod\limits_{r=1, \ldots, n-j} m_{s_j+r}}, \quad j=1, \ldots, n, \ \text{with} \ r=k-j.
    \end{equation}
\end{cor}

\begin{proof}
Recalling Corollary \ref{cor wt matrix in z string coords}, we see that we need to show
    \begin{equation} \label{eqn t_R quotients of z and nu'}
    \frac{\prod_{\substack{1\leq m \leq N \\ i_m=j-1}} z_m}{\prod_{\substack{1\leq m \leq N \\ i_m=j}} z_m}
        = \frac{\prod\limits_{k=1, \ldots,j-1} m_{s_k+(j-k)}}{\prod\limits_{r=1, \ldots, n-j} m_{s_j+r}}, \quad \text{where} \ r=k-j.
    \end{equation}

The denominator of the left hand side of (\ref{eqn t_R quotients of z and nu'}) is
    $$\prod_{\substack{1\leq m \leq N \\ i_m=j}} z_m = \prod_{\substack{m=j+s_r \\ r=1, \ldots, n-j}} z_m = \prod_{r=1, \ldots, n-j} z_{j+s_{n-j-r+1}}.
    $$
Similarly the numerator is
    $$\begin{aligned}
    \prod_{\substack{1\leq m \leq N \\ i_m=j-1}} z_m
        &= \begin{cases}
            1 & \text{if } j=1 \\
            \prod\limits_{\substack{m=j-1+s_r \\ r=1, \ldots, n-j+1}} z_m & \text{otherwise}
        \end{cases} \\
        &= \begin{cases}
            1 & \text{if } j=1 \\
            \prod\limits_{r=0, \ldots, n-j} z_{j-1+s_{n-j-r+1}} & \text{otherwise}.
        \end{cases}
    \end{aligned}
    $$
Consequently if $j=1$ the left hand side of (\ref{eqn t_R quotients of z and nu'}) becomes
    $$\frac{\prod_{\substack{1\leq m \leq N \\ i_m=j-1}} z_m}{\prod_{\substack{1\leq m \leq N \\ i_m=j}} z_m} = \frac{1}{\prod\limits_{r=1, \ldots, n-j} z_{j+s_{n-j-r+1}}} = \frac{1}{\prod\limits_{r=1, \ldots, n-j} m_{s_j+r}}
    $$
as desired.

If $j\geq 2$ then the left hand side of (\ref{eqn t_R quotients of z and nu'}) is
    $$\frac{\prod_{\substack{1\leq m \leq N \\ i_m=j-1}} z_m}{\prod_{\substack{1\leq m \leq N \\ i_m=j}} z_m} = \frac{\prod\limits_{r=0, \ldots, n-j} z_{j-1+s_{n-j-r+1}}}{\prod\limits_{r=1, \ldots, n-j} z_{j+s_{n-j-r+1}}} \\
        = z_{j-1+s_{n-j+1}} \prod_{r=1, \ldots, n-j} \frac{1}{m_{s_j+r}}.
    $$
It remains to show that
    $$\prod\limits_{k=1, \ldots,j-1} m_{s_k+(j-k)} = z_{j-1+s_{n-j+1}}.
    $$
Indeed from the coordinate change formula in Theorem \ref{thm coord change}, for $j\geq 2$ we see
    $$m_{s_k+(j-k)} = \frac{z_{k+s_{n-k-(j-k)+1}}}{z_{k-1+s_{n-k-(j-k)+1}}} = \frac{z_{k+s_{n-j+1}}}{z_{k-1+s_{n-j+1}}}.
    $$
So the product becomes telescopic and, as desired, we obtain
    $$\prod\limits_{k=1, \ldots,j-1} m_{s_k+(j-k)} = z_{1+s_{n-j+1}}\prod\limits_{k=2, \ldots,j-1} \frac{z_{k+s_{n-j+1}}}{z_{k-1+s_{n-j+1}}}
    = z_{j-1+s_{n-j+1}}.
    $$
\end{proof}

\subsection{The Chamber Ansatz} \label{subsec The Chamber Ansatz}
\fancyhead[L]{2.1 \ \ The Chamber Ansatz}
% \fancyhead[L]{3.1 \ \ The Chamber Ansatz}

In order to prove Theorem \ref{thm coord change} we require a sequence of lemmas, the first of which (Lemma \ref{lem form of u_1 and b factorisations}) makes use of the afore mentioned Chamber Ansatz. In this section we introduce the Chamber Ansatz and then state and prove Lemma \ref{lem form of u_1 and b factorisations}.

\begin{defn}
Let $J=\{j_1 < \cdots < j_l\}\subseteq [1,n]$ and $K=\{k_1 < \cdots < k_l\}\subseteq [1,n]$. The pair $(J,K)$ is called admissible if $j_s\leq k_s$ for $s=1, \ldots, l$.
\end{defn}

For such an admissible pair $(J,K)$, we denote by $\Delta^{J}_K$ the $l \times l$ minor with row and column sets defined by $J$ and $K$ respectively.

We now state a specific case of the Generalised Chamber Ansatz presented by Marsh and Rietsch in \cite[Theorem 7.1]{MarshRietsch2004}:
\begin{thm}[Chamber Ansatz] \label{thm chamber ansatz}
Consider $z\bar{w}_0 B^{\vee}_+\in \mathcal{R}^{\vee}_{e,w_0}$, where $z \in U^{\vee}_+$. Let $\mathbf{w}=(w_{(0)}, w_{(1)}, \ldots, w_{(n)})$ be a sequence of partial products for $w_0$ defined by its sequence of factors
    $$\left(w_{(1)}, w_{(1)}^{-1}w_{(2)}, \ldots, w_{(N-1)}^{-1}w_{(N)}\right) = (s_{i_1}, \ldots, s_{i_N}).
    $$
Then there is an element
    $$g= \mathbf{y}_{i_1}(t_1) \mathbf{y}_{i_2}(t_2)\cdots \mathbf{y}_{i_N}(t_N) \in U^{\vee}_- \cap B^{\vee}_+ \bar{w}_0 B^{\vee}_+$$
such that $z\bar{w}_0 B^{\vee}_+ = g B^{\vee}_+$. Moreover for $k=1, \ldots, N$ we have
    $$t_k = \frac{\prod_{j\neq i_k} \Delta^{\omega^{\vee}_j}_{w_{(k)}\omega^{\vee}_j}(z)^{-a_{j,i_k}}}{\Delta^{\omega^{\vee}_{i_k}}_{w_{(k)}\omega^{\vee}_{i_k}}(z) \Delta^{\omega^{\vee}_{i_k}}_{w_{(k-1)}\omega^{\vee}_{i_k}}(z)}.
    $$
\end{thm}

Each $\Delta^{\omega^{\vee}_{i_k}}_{w_{(k)}\omega^{\vee}_{i_k}}$, where $\omega^{\vee}_{i_k}$ ranges through the set of fundamental weights, is called a (standard) chamber minor. As above, it is given by the $i_k\times i_k$ minor with $\omega^{\vee}_{i_k}$ encoding the row set and $w_{(k)}\omega^{\vee}_{i_k}$ encoding the column set. We note that these row and column sets form admissible pairs.

Much of the information in the Chamber Ansatz may be read from an associated pseudoline arrangement; it may be viewed as a singular braid diagram and is called an ansatz arrangement. In dimension $n$, for the case we are considering, the ansatz arrangement consists of $n$ pseudolines. These are numbered from bottom to top on the left side of the arrangement.

Each factor $g_k = \mathbf{y}_{i_k}^{\vee}(t_k)$ of $g$ gives rise to a crossing of the pseudolines at level $i_k$. We label each chamber in the diagram with the labels of the strands passing below it and associate to the chamber with label $S$ the flag minor $\Delta^{\left[1,|S|\right]}_{S}$.

If $A_k$, $B_k$, $C_k$ and $D_k$ are the minors corresponding to the chambers surrounding the $k$-th singular point, with $A_k$ and $D_k$ above and below it and $B_k$ and $C_k$ to the left and right, then the Chamber Ansatz gives
    $$
\begin{tikzpicture}[baseline=0.43cm]
    %dots and stars
    \node at (1,0.5) {$\bullet$};

    \draw (0,1) -- (0.75,1) -- (1.25,0) -- (2,0);
    \draw (0,0) -- (0.75,0) -- (1.25,1) -- (2,1);
    %\labels
    \node at (1,-0.25) {\scriptsize{$D_k$}};
    \node at (0.3,0.5) {\scriptsize{$B_k$}};
    \node at (1.7,0.5) {\scriptsize{$C_k$}};
    \node at (1,1.25) {\scriptsize{$A_k$}};
\end{tikzpicture}
    \qquad \quad
    t_k = \frac{A_k D_k}{B_k C_k}.
    $$

Let $N=\binom{n}{2}$ and again take
    $$\begin{gathered}
    u:=\mathbf{x}_{-\mathbf{i}_0}^{\vee}(\boldsymbol{z}), \quad u_1 := \iota(\eta^{w_0,e}(u)), \quad b:= u_1 d \bar{w}_0 u_2, \\
    \mathbf{i}_0 = (i_1, \ldots, i_N) := (1,2, \ldots, n-1, 1,2 \ldots, n-2, \ldots, 1, 2, 1).
    \end{gathered}
    $$
We will also need
    $$\mathbf{i}'_0 = (i'_1, \ldots, i'_N) := (n-1, n-2, \ldots, 1, n-1, n-2, \ldots, 2, \ldots, n-1, n-2, n-1)
    $$
and we use the superscript `$\mathrm{op}$' to denote taking such an expression in reverse, for example
    $$ {\mathbf{i}'_0}^{\mathrm{op}}:=(i'_N, \ldots, i'_1) = (n-1, n-2, n-1, \ldots, n-3, n-2, n-1).
    $$

\begin{ex}[Ansatz arrangements for {$\mathbf{i}_0$}, {${\mathbf{i}'_0}^{\mathrm{op}}$} in dimension {$4$}] Since $n=4$, we have $N=6$. For $\mathbf{i}_0 = (1,2,3,1,2,1)$, the sequence of partial products for $w_0$ is given by
    $$\mathbf{w} = (w_{(0)}, w_{(1)}, \ldots, w_{(6)}) = (e, s_1, s_1s_2, s_1s_2s_3, s_1s_2s_3s_1, s_1s_2s_3s_1s_2, s_1s_2s_3s_1s_2s_1).
    $$
The ansatz arrangement for $\mathbf{i}_0$ is given in Figure \ref{fig ansatz arrangement i_0 dim 4}.

For ${\mathbf{i}'_0}^{\mathrm{op}} = (3,2,3,1,2,3)$, the sequence of partial products for $w_0$ is
    $$\mathbf{w} = (w_{(0)}, w_{(1)}, \ldots, w_{(6)}) = (e, s_3, s_3s_2, s_3s_2s_3, s_3s_2s_3s_1, s_3s_2s_3s_1s_2, s_3s_2s_3s_1s_2s_3).
    $$
The ansatz arrangement for ${\mathbf{i}'_0}^{\mathrm{op}}$ is given in Figure \ref{fig ansatz arrangement i'_0^op dim 4}.

\begin{figure}[hb]
\centering
\begin{minipage}[b]{0.47\textwidth}
    \centering
\begin{tikzpicture}[scale=0.85]
    % pseudolines
    \draw (0,4) -- (2.75,4) -- (3.25,3) -- (4.75,3) -- (5.25,2) -- (5.75,2) -- (6.25,1) -- (7,1);
    \draw (0,3) -- (1.75,3) -- (2.25,2) -- (3.75,2) -- (4.25,1) -- (5.75,1) -- (6.25,2) -- (7,2);
    \draw (0,2) -- (0.75,2) -- (1.25,1) -- (3.75,1) -- (4.25,2) -- (4.75,2) -- (5.25,3) -- (7,3);
    \draw (0,1) -- (0.75,1) -- (1.25,2) -- (1.75,2) -- (2.25,3) -- (2.75,3) -- (3.25,4) -- (7,4);

    % dots and stars
    \node at (1,1.5) {$\bullet$};
    \node at (2,2.5) {$\bullet$};
    \node at (3,3.5) {$\bullet$};
    \node at (4,1.5) {$\bullet$};
    \node at (5,2.5) {$\bullet$};
    \node at (6,1.5) {$\bullet$};

    % pseudoline labels
    \node at (-0.3,4) {$4$};
    \node at (-0.3,3) {$3$};
    \node at (-0.3,2) {$2$};
    \node at (-0.3,1) {$1$};

    \node at (7.3,4) {$1$};
    \node at (7.3,3) {$2$};
    \node at (7.3,2) {$3$};
    \node at (7.3,1) {$4$};

    % chamber labels
    \node at (1.5,3.5) {$123$};
    \node at (5,3.5) {$234$};

    \node at (1,2.5) {$12$};
    \node at (3.5,2.5) {$23$};
    \node at (6,2.5) {$34$};

    \node at (0.5,1.5) {$1$};
    \node at (2.5,1.5) {$2$};
    \node at (5,1.5) {$3$};
    \node at (6.5,1.5) {$4$};
\end{tikzpicture}
\caption{The ansatz arrangement for $\mathbf{i}_0$ in dimension 4} \label{fig ansatz arrangement i_0 dim 4}
\end{minipage}
    \hfill
\begin{minipage}[b]{0.47\textwidth}
    \centering
\begin{tikzpicture}[scale=0.85]
    % pseudolines
    \draw (0,4) -- (0.75,4) -- (1.25,3) -- (1.75,3) -- (2.25,2) -- (3.75,2) -- (4.25,1) -- (7,1);
    \draw (0,3) -- (0.75,3) -- (1.25,4) -- (2.75,4) -- (3.25,3) -- (4.75,3) -- (5.25,2) -- (7,2);
    \draw (0,2) -- (1.75,2) -- (2.25,3) -- (2.75,3) -- (3.25,4) -- (5.75,4) -- (6.25,3) -- (7,3);
    \draw (0,1) -- (3.75,1) -- (4.25,2) -- (4.75,2) -- (5.25,3) -- (5.75,3) -- (6.25,4) -- (7,4);

    % dots and stars
    \node at (1,3.5) {$\bullet$};
    \node at (2,2.5) {$\bullet$};
    \node at (3,3.5) {$\bullet$};
    \node at (4,1.5) {$\bullet$};
    \node at (5,2.5) {$\bullet$};
    \node at (6,3.5) {$\bullet$};

    % pseudoline labels
    \node at (-0.3,4) {$4$};
    \node at (-0.3,3) {$3$};
    \node at (-0.3,2) {$2$};
    \node at (-0.3,1) {$1$};

    \node at (7.3,4) {$1$};
    \node at (7.3,3) {$2$};
    \node at (7.3,2) {$3$};
    \node at (7.3,1) {$4$};

    % chamber labels
    \node at (0.5,3.5) {$123$};
    \node at (2,3.5) {$124$};
    \node at (4.5,3.5) {$134$};
    \node at (6.5,3.5) {$234$};

    \node at (1,2.5) {$12$};
    \node at (3.5,2.5) {$14$};
    \node at (6,2.5) {$34$};

    \node at (2,1.5) {$1$};
    \node at (5.5,1.5) {$4$};
\end{tikzpicture}
\caption{The ansatz arrangement for ${\mathbf{i}'_0}^{\mathrm{op}}$ in dimension 4} \label{fig ansatz arrangement i'_0^op dim 4}
\end{minipage}
\end{figure}

\end{ex}

We are now ready to state the first lemma needed for the proof of Theorem \ref{thm coord change}:

\begin{lem} \label{lem form of u_1 and b factorisations}
We can factorise $u_1$ and $b$ as follows:
        \begin{align}
        u_1 &= \mathbf{x}_{i'_1}^{\vee}(p_1) \cdots \mathbf{x}_{i'_N}^{\vee}(p_N) \label{eqn form of u_1}, \\
        b &=\mathbf{y}_{i_1}^{\vee}\left(\frac{1}{m_1}\right) \cdots \mathbf{y}_{i_N}^{\vee}\left(\frac{1}{m_N}\right) t_R. \label{eqn form of b}
        \end{align}
The $p_i$ and $m_i$ are given by the Chamber Ansatz in terms of chamber minors:
    \begin{align}
    p_{N-k+1} &= \frac{\prod_{j\neq i'_{N-k+1}} \Delta^{\omega^{\vee}_j}_{w_{(k)}\omega^{\vee}_j}(u^T)^{-a_{j,i'_{N-k+1}}}}{\Delta^{\omega^{\vee}_{i'_{N-k+1}}}_{w_{(k)}\omega^{\vee}_{i'_{N-k+1}}}(u^T) \Delta^{\omega^{\vee}_{i'_{N-k+1}}}_{w_{(k-1)}\omega^{\vee}_{i'_{N-k+1}}}(u^T)}, \quad k=1, \ldots, N, \label{eqn p_i in terms of chamber minors} \\
    \frac{1}{m_k} &= \frac{\prod_{j\neq i_k} \Delta^{\omega^{\vee}_j}_{w_{(k)}\omega^{\vee}_j}(u_1)^{-a_{j,i_k}}}{\Delta^{\omega^{\vee}_{i_k}}_{w_{(k)}\omega^{\vee}_{i_k}}(u_1) \Delta^{\omega^{\vee}_{i_k}}_{w_{(k-1)}\omega^{\vee}_{i_k}}(u_1)}, \quad k=1, \ldots, N. \label{eqn 1/m_i in terms of chamber minors}
    \end{align}
\end{lem}

\begin{proof}
We will use the Chamber Ansatz to prove that $u_1$ and $b$ have the factorisations given in (\ref{eqn form of u_1}) and (\ref{eqn form of b}) respectively.

To show (\ref{eqn form of u_1}) we first note that if $x= \mathbf{x}_{\mathbf{i}}^{\vee}(z_1, \ldots, z_N)$ then $x^T= \mathbf{y}_{\mathbf{i}^{\mathrm{op}}}^{\vee}(z_N, \ldots, z_1)$. Thus to apply the Chamber Ansatz we need a matrix $A \in U^{\vee}_+$ such that $u_1^T B^{\vee}_+ =  A \bar{w}_0 B^{\vee}_+$. We will extract this matrix from the definition of $u_1$:
    $$u_1 = \iota(\eta^{w_0,e}(u)) = \iota\left([(\bar{w}_0u^T)^{-1}]_+\right).
    $$
After applying the involution $\iota$ we see that
    $$B^{\vee}_-\left(\bar{w}_0u^T\right)^{-1} = B^{\vee}_- \eta^{w_0,e}(u) = B^{\vee}_- \iota(u_1) = B^{\vee}_- \left(\bar{w}_0u_1^{-1}\bar{w}_0^{-1}\right)^T = B^{\vee}_- \left(\bar{w}_0u_1^T\bar{w}_0^{-1} \right)^{-1}.
    $$
Taking the inverse gives
    $$\bar{w}_0u_1^T\bar{w}_0^{-1} B^{\vee}_- = \bar{w}_0 u^T B^{\vee}_- \quad \Rightarrow \quad u_1^T\bar{w}_0^{-1} B^{\vee}_- =  u^T B^{\vee}_-.
    $$
Then using the relation $\bar{w}_0 B^{\vee}_+ \bar{w}_0^{-1} = B^{\vee}_-$ we obtain the desired form:
    $$ u_1^T\bar{w}_0^{-1} \bar{w}_0 B^{\vee}_+ \bar{w}_0^{-1} =  u^T \bar{w}_0 B^{\vee}_+ \bar{w}_0^{-1} \quad \Rightarrow \quad u_1^T B^{\vee}_+ =  u^T \bar{w}_0 B^{\vee}_+.
    $$

We take the reduced expression for $w_0$ defined by ${\mathbf{i}'_0}^{\mathrm{op}}$ and let $\mathbf{w} = (w_{(0)}, w_{(1)}, \ldots, w_{(N)}) $ be the sequence of partial products for $w_0$ given by its sequence of factors
    $$\left(w_{(1)}, w_{(1)}^{-1}w_{(2)}, \ldots, w_{(N-1)}^{-1}w_{(N)}\right) = (s_{i'_N}, \ldots, s_{i'_1}).
    $$
Then by the Chamber Ansatz we have $u_1^T B^{\vee}_+ = u^T \bar{w}_0 B^{\vee}_+ = \mathbf{y}_{i'_N}^{\vee}(p_N) \cdots \mathbf{y}_{i'_1}^{\vee}(p_1) B^{\vee}_+$ with
    $$p_{N-k+1} = \frac{\prod_{j\neq i'_{N-k+1}} \Delta^{\omega^{\vee}_j}_{w_{(k)}\omega^{\vee}_j}(u^T)^{-a_{j,i'_{N-k+1}}}}{\Delta^{\omega^{\vee}_{i'_{N-k+1}}}_{w_{(k)}\omega^{\vee}_{i'_{N-k+1}}}(u^T) \Delta^{\omega^{\vee}_{i'_{N-k+1}}}_{w_{(k-1)}\omega^{\vee}_{i'_{N-k+1}}}(u^T)}, \quad k=1, \ldots, N
    $$
which is exactly in the form of (\ref{eqn p_i in terms of chamber minors}). Moreover since both $u_1^T \in U^{\vee}_- $ and $\mathbf{y}_{i'_N}^{\vee}(t'_1) \cdots \mathbf{y}_{i'_1}^{\vee}(t'_N) \in U^{\vee}_-$ this determines $u_1$ completely; $u_1 = \mathbf{x}_{i'_1}^{\vee}(p_1) \cdots \mathbf{x}_{i'_N}^{\vee}(p_N)$ as desired.

To give the factorisation in (\ref{eqn form of b}) we note that by definition $b B^{\vee}_+ = u_1 \bar{w}_0 B^{\vee}_+$ with $u_1 \in U^{\vee}_+$. For this second application of the Chamber Ansatz we let $w_0$ be described by $\mathbf{i}_0$. We again take $\mathbf{w} = (w_{(0)}, w_{(1)}, \ldots, w_{(N)})$ to be the respective sequence of partial products for $w_0$ defined by its sequence of factors
    $$\left(w_{(1)}, w_{(1)}^{-1}w_{(2)}, \ldots, w_{(N-1)}^{-1}w_{(N)}\right) = (s_{i_1}, \ldots, s_{i_N}).
    $$
Then by the Chamber Ansatz we have $b B^{\vee}_+ = u_1 \bar{w}_0 B^{\vee}_+ = \mathbf{y}_{i_1}^{\vee}\left(\frac{1}{m_1}\right) \cdots \mathbf{y}_{i_N}^{\vee}\left(\frac{1}{m_N}\right)B^{\vee}_+$ with
    $$\frac{1}{m_k} = \frac{\prod_{j\neq i_k} \Delta^{\omega^{\vee}_j}_{w_{(k)}\omega^{\vee}_j}(u_1)^{-a_{j,i_k}}}{\Delta^{\omega^{\vee}_{i_k}}_{w_{(k)}\omega^{\vee}_{i_k}}(u_1) \Delta^{\omega^{\vee}_{i_k}}_{w_{(k-1)}\omega^{\vee}_{i_k}}(u_1)}, \quad k=1, \ldots, N
    $$
which proves (\ref{eqn 1/m_i in terms of chamber minors}).
\end{proof}

\subsection{Chamber Ansatz minors} \label{subsec Chamber Ansatz minors}
\fancyhead[L]{2.2 \ \ Chamber Ansatz minors}
% \fancyhead[L]{3.2 \ \ Chamber Ansatz minors}

We wish to further describe the coordinate changes defined by our two applications of the Chamber Ansatz. In this section we show they are monomial by considering the required minors of $u^T$ and $u_1$. Similar to how the ansatz arrangement tells us which quotients of minors to take when applying Chamber Ansatz, we may use a planar acyclic directed graph to easily compute these minors, and in particular to confirm that they are all monomial (see \cite[Proposition 4.2]{FominZelevinsky1999}, generalising \cite[Theorem 2.4.4]{BerensteinFominZelevinsky1996}). Note that the following description is slightly different to that given by Fomin and Zelevinsky in \cite{FominZelevinsky1999}, since we do not need the same level of generality.

Let $\mathbf{i}= (i_1, \ldots, i_N)$ define some reduced expression for $w_0$ and consider
    $$\mathbf{x}_{\mathbf{i}}^{\vee}(\boldsymbol{z}) = \mathbf{x}_{i_1}^{\vee}(z_1) \cdots \mathbf{x}_{i_N}^{\vee}(z_N).
    $$
For particular choices of admissible pairs $(J,K)$, we wish to compute the %(flag)
minors
    $$\Delta^J_K\left(\mathbf{x}_{\mathbf{i}}^{\vee}(\boldsymbol{z})\right) \text{ or } \ \Delta^J_K\left(\mathbf{x}_{-\mathbf{i}}^{\vee}(\boldsymbol{z})^T\right).
    $$
In the second case it will be helpful to express $\mathbf{x}_{-\mathbf{i}}^{\vee}(\boldsymbol{z})^T$ as a product of matrices $\mathbf{x}_i$ and $\mathbf{t}_i$. To do this we notice
    $$\mathbf{x}_{-i}^{\vee}(z)^T = \phi_i^{\vee}\begin{pmatrix} z^{-1} & 1 \\ 0 & z \end{pmatrix} = \phi_i^{\vee} \left(\begin{pmatrix} z^{-1} & 0 \\ 0 & z \end{pmatrix} \begin{pmatrix} 1 & z \\ 0 & 1 \end{pmatrix} \right)
    = \mathbf{t}_{i}^{\vee}(z^{-1}) \mathbf{x}_{i}^{\vee}(z).
    $$
In particular we may rewrite $\mathbf{x}_{-\mathbf{i}}^{\vee}(\boldsymbol{z})^T $ as follows:
    $$\mathbf{x}_{-\mathbf{i}}^{\vee}(\boldsymbol{z})^T =  \mathbf{x}_{-i_N}^{\vee}(z_N)^T \cdots \mathbf{x}_{-i_1}^{\vee}(z_1)^T
    = \mathbf{t}_{i_N}^{\vee}(z_N^{-1}) \mathbf{x}_{i_N}^{\vee}(z_N) \cdots \mathbf{t}_{i_1}^{\vee}(z_1^{-1}) \mathbf{x}_{i_1}^{\vee}(z_1).
    $$

To construct the graph, $\Gamma$, corresponding to a matrix $x= \mathbf{x}_{\mathbf{i}}^{\vee}(\boldsymbol{z})$ or $\mathbf{x}_{-\mathbf{i}}^{\vee}(\boldsymbol{z})^T$ we begin with $n$ parallel horizontal lines. We add vertices to the ends of each line and number them from bottom to top on both sides. Then for each factor $\mathbf{t}_{i_k}^{\vee}(z_k^{-1})$, $\mathbf{x}_{i_k}^{\vee}(z_k)$ of $x$ we include a labelled line segment and vertices at height $i_k$ defined in Figure \ref{Labelled line segments in graphs for computing Chamber Ansatz minors}.
\begin{figure}[hb!]
\centering
\begin{minipage}[b]{0.3\textwidth}
\centering
    \begin{tikzpicture}[scale=0.85]
        \node at (0.75,2) {$\bullet$};
        \node at (1.25,2) {$\bullet$};
        \node at (0.75,1) {$\bullet$};
        \node at (1.25,1) {$\bullet$};
        \draw (0,3) -- (2,3);
        \draw (0,2) -- (2,2);
        \draw (0,1) -- (2,1);
        \draw (0,0) -- (2,0);
        \node at (1,1.35) {$\frac{1}{z_k}$};
        \node at (1,2.25) {$z_k$};
        \node at (1,-0.6) {For factors $\mathbf{t}_{i_k}^{\vee}(z_k^{-1})$};
    \end{tikzpicture}
\end{minipage}
\begin{minipage}[b]{0.3\textwidth}
\centering
    \begin{tikzpicture}[scale=0.85]
        \node at (0.75,1) {$\bullet$};
        \node at (1.25,2) {$\bullet$};
        \draw (0,3) -- (2,3);
        \draw (0,2) -- (2,2);
        \draw (0,1) -- (2,1);
        \draw (0,0) -- (2,0);
        \draw (0.75,1) -- (1.25,2);
        \node at (1.35,1.5) {$z_k$};
        \node at (1,-0.6) {For factors $\mathbf{x}_{i_k}^{\vee}(z_k)$};
    \end{tikzpicture}
\end{minipage}
\caption{Labelled line segments in graphs for computing Chamber Ansatz minors} \label{Labelled line segments in graphs for computing Chamber Ansatz minors}
\end{figure}
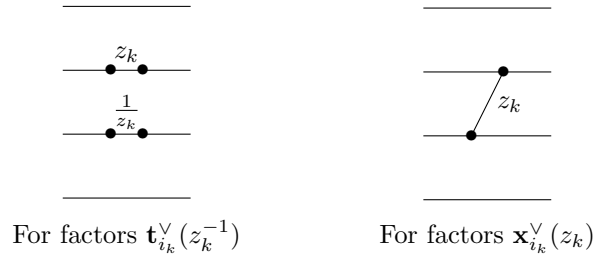

Each line segment is viewed as a labelled edge of $\Gamma$, oriented left to right. For an edge $e$, the labelling, called the weight of $e$ and denoted $w(e)$, is given by the diagrams above and taken to be $1$ if left unspecified. The weight $w(\pi)$ of an oriented path $\pi$ is defined to be the product of weights $w(e)$ taken over all edges $e$ in $\pi$.

The set of vertices of the graph $\Gamma$ is given by the endpoints of all line segments. Those vertices appearing as the leftmost (resp. rightmost) endpoints of the horizontal lines are the sources (resp. sinks) of $\Gamma$.

With this notation, \cite[Proposition 4.2]{FominZelevinsky1999} becomes the following:
\begin{thm} \label{thm minors from paths in graph}
For an admissible pair $(J,K)$ of size $l$
    $$ \Delta^J_K\left(\mathbf{x}_{\mathbf{i}'}^{\vee}(\boldsymbol{z})\right) = \sum_{\pi_1,\ldots, \pi_l} w(\pi_1) \cdots w(\pi_l)
    $$
where the sum is taken over all families of l vertex-disjoint paths $\{\pi_1,\ldots, \pi_l\}$ connecting the sources labelled by $J$ with the sinks labelled by $K$.
\end{thm}

To prove that the minors of $u^T$, $u_1$ appearing in our applications of the Chamber Ansatz are monomial we must show that in each case there is only one possible family of paths $\{\pi_1,\ldots, \pi_l\}$. Before showing this we give two examples to clarify the above construction.

\begin{ex}
We take $n=4$, so $N=6$, and we wish to compute the minors $\Delta^J_K\left(u_1\right)$ where
    $$u_1=\mathbf{x}_{\mathbf{i}'_0}^{\vee}\left(z_1, z_4, z_6, \frac{z_2}{z_4}, \frac{z_5}{z_6}, \frac{z_3}{z_5}\right) = \begin{pmatrix}
        1 & z_6 & z_5 & z_3 \\
         & 1 & z_4+\frac{z_5}{z_6} & z_2+\frac{z_3z_4}{z_5}+\frac{z_3}{z_6} \\
         &  & 1 & z_1+\frac{z_2}{z_4}+\frac{z_3}{z_5} \\
         &  &  & 1
    \end{pmatrix}
    $$
with $\mathbf{i}'_0 = (3,2,1,3,2,3)$. The graph for $u_1$ is given in Figure \ref{The graph for u_1 when n=4}.
\begin{figure}[ht!]
\centering
\begin{tikzpicture}
    % pseudolines
    \draw (0,4) -- (7,4);
    \draw (0,3) -- (7,3);
    \draw (0,2) -- (7,2);
    \draw (0,1) -- (7,1);

    \draw (0.75,3) -- (1.25,4);
    \draw (1.75,2) -- (2.25,3);
    \draw (2.75,1) -- (3.25,2);
    \draw (3.75,3) -- (4.25,4);
    \draw (4.75,2) -- (5.25,3);
    \draw (5.75,3) -- (6.25,4);
    % dots and stars
    \node at (0,1) {$\bullet$};
    \node at (0,2) {$\bullet$};
    \node at (0,3) {$\bullet$};
    \node at (0,4) {$\bullet$};
    \node at (7,1) {$\bullet$};
    \node at (7,2) {$\bullet$};
    \node at (7,3) {$\bullet$};
    \node at (7,4) {$\bullet$};

    \node at (0.75,3) {$\bullet$};
    \node at (1.25,4) {$\bullet$};
    \node at (1.75,2) {$\bullet$};
    \node at (2.25,3) {$\bullet$};
    \node at (2.75,1) {$\bullet$};
    \node at (3.25,2) {$\bullet$};

    \node at (3.75,3) {$\bullet$};
    \node at (4.25,4) {$\bullet$};
    \node at (4.75,2) {$\bullet$};
    \node at (5.25,3) {$\bullet$};
    \node at (5.75,3) {$\bullet$};
    \node at (6.25,4) {$\bullet$};

    % pseudoline labels
    \node at (-0.3,4) {$4$};
    \node at (-0.3,3) {$3$};
    \node at (-0.3,2) {$2$};
    \node at (-0.3,1) {$1$};

    \node at (7.3,4) {$4$};
    \node at (7.3,3) {$3$};
    \node at (7.3,2) {$2$};
    \node at (7.3,1) {$1$};

    % weight labels
    \node at (1.3,3.5) {$z_1$};
    \node at (2.3,2.5) {$z_4$};
    \node at (3.3,1.5) {$z_6$};
    \node at (4.3,3.5) {$\frac{z_2}{z_4}$};
    \node at (5.3,2.5) {$\frac{z_5}{z_6}$};
    \node at (6.3,3.5) {$\frac{z_3}{z_5}$};
\end{tikzpicture}
\caption{The graph for $u_1$ when $n=4$} \label{The graph for u_1 when n=4}
\end{figure}
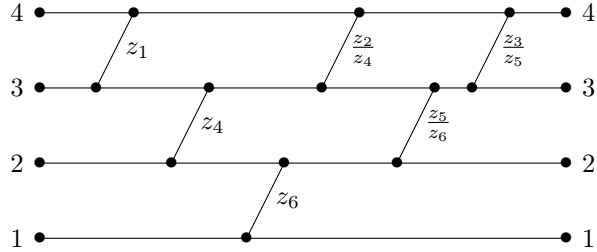

Computing the minor $\Delta^1_4(u_1)$ directly from the matrix $u_1$ is trivial. If we use the graph instead, we see that it is given by the weight of one path with three diagonal edges:
    $$\Delta^1_4(u_1) = z_6 \frac{z_5}{z_6} \frac{z_3}{z_5} = z_3.
    $$
The minor $\Delta^{\{1,2,3\}}_{\{2,3,4\}}(u_1)$ takes far more effort to compute directly, but using the graph makes the computation simple. This minor is the product of weights of three paths, each with one diagonal edge. We highlight the paths with parentheses:
    $$\Delta^{\{1,2,3\}}_{\{2,3,4\}}(u_1) = (z_6) (z_4) (z_1) = z_1z_4z_6.
    $$
\end{ex}

\begin{ex}
Again we take $n=4$, so $N=6$ and we wish to compute the minors $\Delta^J_K\left(u^T\right)$ where
    $$u^T= \left(\mathbf{x}_{-\mathbf{i}_0}^{\vee}(z_1, \ldots, z_6)\right)^T
        = \begin{pmatrix} \frac{1}{z_1z_4z_6} & \frac{z_1z_4(z_4z_6+z_5) + z_2z_5}{z_2z_4z_5z_6} & \frac{z_3z_5 +z_6(z_2z_5+z_3z_4)}{z_3z_5z_6} & 1 \\  & \frac{z_1z_4z_6}{z_2z_5} & \frac{z_6(z_2z_5+z_3z_4)}{z_3z_5} & z_6 \\ & & \frac{z_2z_5}{z_3} & z_5 \\ & & & z_3  \end{pmatrix}
    $$
with $\mathbf{i}_0 = (1,2,3,1,2,1)$. The graph for
    $$u^T = \mathbf{t}_{1}^{\vee}(z_6^{-1})\mathbf{x}_{1}^{\vee}(z_6) \mathbf{t}_{2}^{\vee}(z_5^{-1})\mathbf{x}_{2}^{\vee}(z_5) \mathbf{t}_{1}^{\vee}(z_4^{-1})\mathbf{x}_{1}^{\vee}(z_4) \mathbf{t}_{3}^{\vee}(z_3^{-1})\mathbf{x}_{3}^{\vee}(z_3) \mathbf{t}_{2}^{\vee}(z_2^{-1})\mathbf{x}_{2}^{\vee}(z_2) \mathbf{t}_{1}^{\vee}(z_1^{-1}) \mathbf{x}_{1}^{\vee}(z_1)
    $$
is given in Figure \ref{The graph for u^T when n=4}.

\begin{figure}[ht!]
\centering
\begin{tikzpicture}
    \draw (0,4) -- (10,4);
    \draw (0,3) -- (10,3);
    \draw (0,2) -- (10,2);
    \draw (0,1) -- (10,1);

    \draw (1.25,1) -- (1.75,2);
    \draw (2.75,2) -- (3.25,3);
    \draw (4.25,1) -- (4.75,2);
    \draw (5.75,3) -- (6.25,4);
    \draw (7.25,2) -- (7.75,3);
    \draw (8.75,1) -- (9.25,2);
    % dots and stars
    \node at (0,1) {$\bullet$};
    \node at (0,2) {$\bullet$};
    \node at (0,3) {$\bullet$};
    \node at (0,4) {$\bullet$};
    \node at (10,1) {$\bullet$};
    \node at (10,2) {$\bullet$};
    \node at (10,3) {$\bullet$};
    \node at (10,4) {$\bullet$};

    \node at (0.75,1) {$\bullet$};
    \node at (1.25,1) {$\bullet$};
    \node at (0.75,2) {$\bullet$};
    \node at (1.25,2) {$\bullet$};
    \node at (1.75,2) {$\bullet$};
    \node at (2.25,2) {$\bullet$};
    \node at (2.75,2) {$\bullet$};
    \node at (2.25,3) {$\bullet$};
    \node at (2.75,3) {$\bullet$};
    \node at (3.25,3) {$\bullet$};

    \node at (3.75,1) {$\bullet$};
    \node at (4.25,1) {$\bullet$};
    \node at (3.75,2) {$\bullet$};
    \node at (4.25,2) {$\bullet$};
    \node at (4.75,2) {$\bullet$};
    \node at (5.25,3) {$\bullet$};
    \node at (5.75,3) {$\bullet$};
    \node at (5.25,4) {$\bullet$};
    \node at (5.75,4) {$\bullet$};
    \node at (6.25,4) {$\bullet$};

    \node at (6.75,2) {$\bullet$};
    \node at (7.25,2) {$\bullet$};
    \node at (6.75,3) {$\bullet$};
    \node at (7.25,3) {$\bullet$};
    \node at (7.75,3) {$\bullet$};
    \node at (8.25,1) {$\bullet$};
    \node at (8.75,1) {$\bullet$};
    \node at (8.25,2) {$\bullet$};
    \node at (8.75,2) {$\bullet$};
    \node at (9.25,2) {$\bullet$};

    % pseudoline labels
    \node at (-0.3,4) {$4$};
    \node at (-0.3,3) {$3$};
    \node at (-0.3,2) {$2$};
    \node at (-0.3,1) {$1$};

    \node at (10.3,4) {$4$};
    \node at (10.3,3) {$3$};
    \node at (10.3,2) {$2$};
    \node at (10.3,1) {$1$};

    % weight labels
    \node at (1,1.3) {$\frac{1}{z_6}$};
    \node at (1,2.2) {$z_6$};
    \node at (1.8,1.5) {$z_6$};
    \node at (2.5,2.3) {$\frac{1}{z_5}$};
    \node at (2.5,3.2) {$z_5$};
    \node at (3.3,2.5) {$z_5$};
    \node at (4,1.3) {$\frac{1}{z_4}$};
    \node at (4,2.2) {$z_4$};
    \node at (4.8,1.5) {$z_4$};
    \node at (5.5,3.3) {$\frac{1}{z_3}$};
    \node at (5.5,4.2) {$z_3$};
    \node at (6.3,3.5) {$z_3$};
    \node at (7,2.3) {$\frac{1}{z_2}$};
    \node at (7,3.2) {$z_2$};
    \node at (7.8,2.5) {$z_2$};
    \node at (8.5,1.3) {$\frac{1}{z_1}$};
    \node at (8.5,2.2) {$z_1$};
    \node at (9.3,1.5) {$z_1$};
\end{tikzpicture}
\caption{The graph for $u^T$ when $n=4$} \label{The graph for u^T when n=4}
\end{figure}
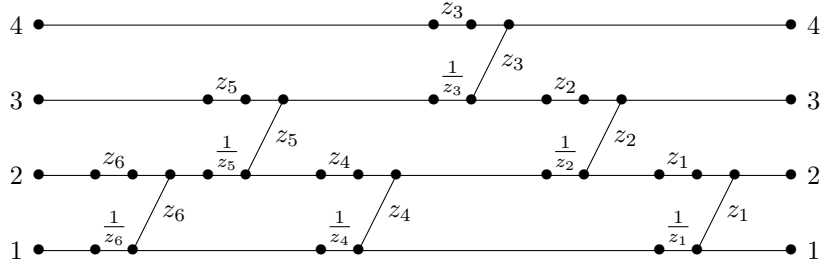

We again give example computations, with parentheses highlighting products of multiple paths:
    $$\begin{aligned}
    &\Delta^{\{1,2\}}_{\{1,4\}}(u^T) = \left( \frac{1}{z_6} \frac{1}{z_4} \frac{1}{z_1} \right) \left( z_6 \frac{1}{z_5} z_5 \frac{1}{z_3} z_3 \right) = \frac{1}{z_1z_4},  \\
    &\Delta^{\{1,2\}}_{\{3,4\}}(u^T) = \left( \frac{1}{z_6} \frac{1}{z_4} z_4 \frac{1}{z_2} z_2 \right) \left( z_6 \frac{1}{z_5} z_5 \frac{1}{z_3} z_3 \right) = 1.
    \end{aligned}
    $$
\end{ex}

\begin{lem} \label{lem our Chamber Ansatz minors are monomial}
All minors in both applications of the Chamber Ansatz in the proof of Lemma \ref{lem form of u_1 and b factorisations} are monomial and consequently the resulting coordinate changes are monomial.
\end{lem}

\begin{proof}
In each application of the Chamber Ansatz, the relevant minors are those flag minors with column sets given by the chamber labels of the corresponding ansatz arrangements.
\begin{claim}
Let
    $$\begin{gathered}
    \mathbf{i}_0 = (i_1, \ldots, i_N) := (1,2, \ldots, n-1, 1,2 \ldots, n-2, \ldots, 1, 2, 1), \\
    \mathbf{i}'_0 = (i'_1, \ldots, i'_N) := (n-1, n-2, \ldots, 1, n-1, n-2, \ldots, 2, \ldots, n-1, n-2, n-1), \\
    {\mathbf{i}'_0}^{\text{op}} := (i'_N, \ldots, i'_1).
    \end{gathered}
    $$
Then
\begin{enumerate}
    \item Chamber labels of the ansatz arrangement for $\mathbf{i}_0$ are of the form $\{a, \ldots, b\}$.

    \item Chamber labels of the ansatz arrangement for ${\mathbf{i}'_0}^{\text{op}}$ are of the form $\{1, \ldots, a\}$, $\{b, \ldots, n\}$ or $\{1, \ldots, a\} \cup \{b, \ldots, n\}$.
\end{enumerate}
\end{claim}

Note that flag minors of $u_1$ and $u^T$ correspond to chamber labels of the ansatz arrangement for $\mathbf{i}_0$ and ${\mathbf{i}'_0}^{\text{op}}$ respectively.

\begin{proof}[Proof of Claim]
We may construct the reduced expression ${\mathbf{i}'_0}^{\text{op}}$ from $\mathbf{i}_0$ in two steps; first replace each $i_k$ in $\mathbf{i}_0$ with $n-i_k$ (this gives ${\mathbf{i}'_0}$) and then reverse the order. In terms of the ansatz arrangement, both of these operations result in reflections. Viewing the ansatz arrangement for $\mathbf{i}_0$ in the the plane, with the origin at the centre of the arrangement, we see that the first step above reflects the ansatz arrangement in the horizontal axis. Note that this causes each chamber label $S$ to change in the following way
    $$S \mapsto \{1, \ldots, n\}\setminus S.
    $$
In the second step we reverse the order of the reduced expression which gives a refection of the arrangement in the vertical axis and in particular there is no further change to the chamber labels. It follows that the two statements in the claim are equivalent and so we will consider only the $\mathbf{i}_0$ case.

Given the ansatz arrangement for $\mathbf{i}_0$ in dimension $n$, if we ignore the first $n-1$ crossings and the top pseudoline after this point (i.e. we remove the 1-string), the remaining graph has the form of the ansatz arrangement in dimension $n-1$, with labelling $2, \ldots n$ rather than $1, \ldots, n-1$. This is a consequence of the form of $\mathbf{i}_0$, namely that the reduced expression $\mathbf{i}_0$ in dimension $n-1$ is given by the last $\binom{n-1}{2}$ entries of the expression $\mathbf{i}_0$ in dimension $n$.

Since, by definition of the ansatz arrangement, the leftmost chamber labels are always given by sets of consecutive integers, it follows by induction that all chamber labels are of this form.
\end{proof}

We now use the graphs for $u_1$ and $u^T$ to see that the relevant flag minors are all monomial, namely by using Theorem \ref{thm minors from paths in graph} and showing that there is only one possible family of paths in each case.
\begin{enumerate}
    \item Minors of $u_1$:
    \begin{enumerate}
        \item Column sets of the form $\{1, \ldots, b\}$: Since $u_1 \in U^{\vee}_+$ these minors always equal $1$. We can see this from the graph for $u_1$ since the paths must be horizontal and the lack of non-trivial torus factors means that all horizontal edges have weight $1$.

        \item Column sets of the form $\{a, \ldots, b\}$ with $a>1$: Note that there is only one edge connecting the bottom two horizontal lines. After travelling up this edge there is only one possible path to the third line and so on. Thus there is only one path from the source $1$ to the sink $a$.

        In order for the paths in our family to remain vertex disjoint, the path from the source $2$ to the sink $a+1$ must take the first opportunity to travel upwards and indeed every possible opportunity to travel upwards until it reaches the line at height $a+1$. This imposes the same restriction on the path from the source $3$ to sink $a+2$ and so on for all paths in this family. In particular there is only one possible family of paths, thus these minors are monomial.
        For example see Figure \ref{Example family of paths for the proof of Lem CA minors monomial, u_1 case}.
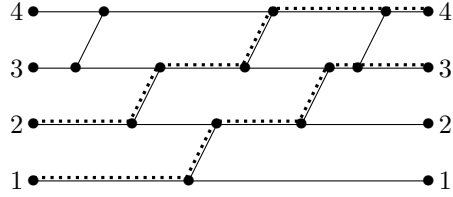
\begin{figure}[ht!]
\centering
\begin{tikzpicture}[scale=0.75]
    % paths
    \draw[very thick, dotted] (0,2.06) -- (1.71,2.06) -- (2.21,3.06) -- (3.71,3.06) -- (4.21,4.06) -- (7,4.06);
    \draw[very thick, dotted] (0,1.06) -- (2.71,1.06) -- (3.21,2.06) -- (4.71,2.06) -- (5.21,3.06) -- (7,3.06);

    % pseudolines
    \draw (0,4) -- (7,4);
    \draw (0,3) -- (7,3);
    \draw (0,2) -- (7,2);
    \draw (0,1) -- (7,1);

    \draw (0.75,3) -- (1.25,4);
    \draw (1.75,2) -- (2.25,3);
    \draw (2.75,1) -- (3.25,2);
    \draw (3.75,3) -- (4.25,4);
    \draw (4.75,2) -- (5.25,3);
    \draw (5.75,3) -- (6.25,4);
    % dots and stars
    \node at (0,1) {$\bullet$};
    \node at (0,2) {$\bullet$};
    \node at (0,3) {$\bullet$};
    \node at (0,4) {$\bullet$};
    \node at (7,1) {$\bullet$};
    \node at (7,2) {$\bullet$};
    \node at (7,3) {$\bullet$};
    \node at (7,4) {$\bullet$};

    \node at (0.75,3) {$\bullet$};
    \node at (1.25,4) {$\bullet$};
    \node at (1.75,2) {$\bullet$};
    \node at (2.25,3) {$\bullet$};
    \node at (2.75,1) {$\bullet$};
    \node at (3.25,2) {$\bullet$};

    \node at (3.75,3) {$\bullet$};
    \node at (4.25,4) {$\bullet$};
    \node at (4.75,2) {$\bullet$};
    \node at (5.25,3) {$\bullet$};
    \node at (5.75,3) {$\bullet$};
    \node at (6.25,4) {$\bullet$};

    % pseudoline labels
    \node at (-0.3,4) {$4$};
    \node at (-0.3,3) {$3$};
    \node at (-0.3,2) {$2$};
    \node at (-0.3,1) {$1$};

    \node at (7.3,4) {$4$};
    \node at (7.3,3) {$3$};
    \node at (7.3,2) {$2$};
    \node at (7.3,1) {$1$};

    % \node at (3.5,4.6) {$\vdots$}
\end{tikzpicture}
\caption{Example family of paths for the proof of Lemma \ref{lem our Chamber Ansatz minors are monomial}, $u_1$ case} \label{Example family of paths for the proof of Lem CA minors monomial, u_1 case}
\end{figure}
    \end{enumerate}

\newpage
    \item Minors of $u^T$:
    \begin{enumerate}
        \item Column sets of the form $\{1, \ldots, a\}$: These minors are monomial since they correspond to horizontal paths in the graph for $u^T$.

        \item Column sets of the form $\{b, \ldots, n\}$:
        We first note that there is only one edge from the line at height $n-1$ to the $n$-th horizontal line. Before this point there is only one edge from the line at height $n-2$ to the line at height $n-1$. Working backwards in this way we see there is only one possible path from each source which ends at the sink $n$, and in particular only one such path from the source $n-b+1$.

        Similarly, in order to have vertex distinct paths there is now only one possible way to reach the sink $n-1$ from the source $n-b$. Continuing in this way we see that there is only one possible family of paths and so these minors are monomial.
        For example see Figure \ref{Example family of paths for the proof of Lem CA minors monomial, u^T case}.
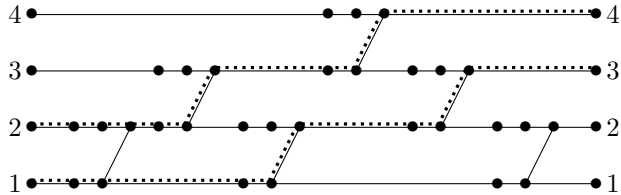
\begin{figure}[ht!]
\centering
\begin{tikzpicture}[scale=0.75]
    % paths
    \draw[very thick, dotted] (0,2.06) -- (2.71,2.06) -- (3.21,3.06) -- (5.71,3.06) -- (6.21,4.06) -- (10,4.06);
    \draw[very thick, dotted] (0,1.06) -- (4.21,1.06) -- (4.71,2.06) -- (7.21,2.06) -- (7.71,3.06) -- (10,3.06);

    % pseudolines
    \draw (0,4) -- (10,4);
    \draw (0,3) -- (10,3);
    \draw (0,2) -- (10,2);
    \draw (0,1) -- (10,1);

    \draw (1.25,1) -- (1.75,2);
    \draw (2.75,2) -- (3.25,3);
    \draw (4.25,1) -- (4.75,2);
    \draw (5.75,3) -- (6.25,4);
    \draw (7.25,2) -- (7.75,3);
    \draw (8.75,1) -- (9.25,2);
    % dots and stars
    \node at (0,1) {$\bullet$};
    \node at (0,2) {$\bullet$};
    \node at (0,3) {$\bullet$};
    \node at (0,4) {$\bullet$};
    \node at (10,1) {$\bullet$};
    \node at (10,2) {$\bullet$};
    \node at (10,3) {$\bullet$};
    \node at (10,4) {$\bullet$};

    \node at (0.75,1) {$\bullet$};
    \node at (1.25,1) {$\bullet$};
    \node at (0.75,2) {$\bullet$};
    \node at (1.25,2) {$\bullet$};
    \node at (1.75,2) {$\bullet$};
    \node at (2.25,2) {$\bullet$};
    \node at (2.75,2) {$\bullet$};
    \node at (2.25,3) {$\bullet$};
    \node at (2.75,3) {$\bullet$};
    \node at (3.25,3) {$\bullet$};

    \node at (3.75,1) {$\bullet$};
    \node at (4.25,1) {$\bullet$};
    \node at (3.75,2) {$\bullet$};
    \node at (4.25,2) {$\bullet$};
    \node at (4.75,2) {$\bullet$};
    \node at (5.25,3) {$\bullet$};
    \node at (5.75,3) {$\bullet$};
    \node at (5.25,4) {$\bullet$};
    \node at (5.75,4) {$\bullet$};
    \node at (6.25,4) {$\bullet$};

    \node at (6.75,2) {$\bullet$};
    \node at (7.25,2) {$\bullet$};
    \node at (6.75,3) {$\bullet$};
    \node at (7.25,3) {$\bullet$};
    \node at (7.75,3) {$\bullet$};
    \node at (8.25,1) {$\bullet$};
    \node at (8.75,1) {$\bullet$};
    \node at (8.25,2) {$\bullet$};
    \node at (8.75,2) {$\bullet$};
    \node at (9.25,2) {$\bullet$};

    % pseudoline labels
    \node at (-0.3,4) {$4$};
    \node at (-0.3,3) {$3$};
    \node at (-0.3,2) {$2$};
    \node at (-0.3,1) {$1$};

    \node at (10.3,4) {$4$};
    \node at (10.3,3) {$3$};
    \node at (10.3,2) {$2$};
    \node at (10.3,1) {$1$};
\end{tikzpicture}
\caption{Example family of paths for the proof of Lemma \ref{lem our Chamber Ansatz minors are monomial}, $u^T$ case} \label{Example family of paths for the proof of Lem CA minors monomial, u^T case}
\end{figure}

        \item Column sets of the form $\{1, \ldots, a\} \cup \{b, \ldots, n\}$: We see that these minors are all monomial by combining the previous two cases.
    \end{enumerate}
\end{enumerate}
\end{proof}

\subsection{The coordinate change} \label{subsec The coordinate change}
\fancyhead[L]{2.3 \ \ The coordinate change}
% \fancyhead[L]{3.3 \ \ The coordinate change}

In this section we address two final lemmas needed for the proof of Theorem \ref{thm coord change}, both detailing coordinate changes. We then recall and prove this theorem.

\begin{lem} \label{lem coord change for u_1 p_i in terms of z_i}
For $k=1,\ldots, n-1$, $a=1,\ldots, n-k$ we have
    $$p_{s_k+a} =\begin{cases}
        z_{1+s_{a}} & \text{if } k=1 \\
        \frac{z_{k+s_{a}}}{z_{k-1+s_{a+1}}} & \text{otherwise}
        \end{cases}
    \qquad \text{where} \ \
    s_k := \sum_{j=1}^{k-1}(n-j).
    $$
\end{lem}

\begin{proof}
In the proof of Lemma \ref{lem form of u_1 and b factorisations} we used the Chamber Ansatz with ${\mathbf{i}'_0}^{\mathrm{op}}$ to show that $u_1 = \mathbf{x}_{i'_1}^{\vee}(p_1) \cdots \mathbf{x}_{i'_N}^{\vee}(p_N)$ where
    $$
    p_k = \frac{\prod_{j\neq i'_k} \Delta^{\omega^{\vee}_j}_{w_{(N-k+1)}\omega^{\vee}_j}(u^T)^{-a_{j,i'_k}}}{\Delta^{\omega^{\vee}_{i_k}}_{w_{(N-k+1)}\omega^{\vee}_{i'_k}}(u^T) \Delta^{\omega^{\vee}_{i'_k}}_{w_{(N-k)}\omega^{\vee}_{i'_k}}(u^T)}, \quad k=1, \ldots, N.
    $$
Since this expression is quite unpleasant, we instead work diagrammatically. In particular, we will use the ansatz arrangement for ${\mathbf{i}'_0}^{\mathrm{op}}$ and the graph for $u^T$.

To give a visual aid we recall Figure \ref{fig ansatz arrangement i'_0^op dim 4}, the dimension $4$ example of the ansatz arrangement for ${\mathbf{i}'_0}^{\mathrm{op}} = (i'_N, \ldots, i'_1)=(3,2,3,1,2,3)$:
\begin{center}
    \begin{tikzpicture}[scale=0.85]
            % pseudolines
            \draw (0,4) -- (0.75,4) -- (1.25,3) -- (1.75,3) -- (2.25,2) -- (3.75,2) -- (4.25,1) -- (7,1);
            \draw (0,3) -- (0.75,3) -- (1.25,4) -- (2.75,4) -- (3.25,3) -- (4.75,3) -- (5.25,2) -- (7,2);
            \draw (0,2) -- (1.75,2) -- (2.25,3) -- (2.75,3) -- (3.25,4) -- (5.75,4) -- (6.25,3) -- (7,3);
            \draw (0,1) -- (3.75,1) -- (4.25,2) -- (4.75,2) -- (5.25,3) -- (5.75,3) -- (6.25,4) -- (7,4);

            % dots and stars
            \node at (1,3.5) {$\bullet$};
            \node at (2,2.5) {$\bullet$};
            \node at (3,3.5) {$\bullet$};
            \node at (4,1.5) {$\bullet$};
            \node at (5,2.5) {$\bullet$};
            \node at (6,3.5) {$\bullet$};

            % pseudoline labels
            \node at (-0.3,4) {$4$};
            \node at (-0.3,3) {$3$};
            \node at (-0.3,2) {$2$};
            \node at (-0.3,1) {$1$};

            \node at (7.3,4) {$1$};
            \node at (7.3,3) {$2$};
            \node at (7.3,2) {$3$};
            \node at (7.3,1) {$4$};

            % chamber labels
            \node at (0.5,3.5) {$123$};
            \node at (2,3.5) {$124$};
            \node at (4.5,3.5) {$134$};
            \node at (6.5,3.5) {$234$};

            \node at (1,2.5) {$12$};
            \node at (3.5,2.5) {$14$};
            \node at (6,2.5) {$34$};

            \node at (2,1.5) {$1$};
            \node at (5.5,1.5) {$4$};
    \end{tikzpicture}
\end{center}

We now define a new labelling of the chambers of the ansatz arrangement for ${\mathbf{i}'_0}^{\mathrm{op}}$ in terms of pairs $(k,a)$. This is motivated by two facts:
\begin{enumerate}
    \item Any integer $1 \leq m \leq N:=\binom{n}{2}$ may be written as
        $$m = s_k +a
        $$
    for some unique pair $(k,a)$, with $k=1,\ldots, n-1$, $a=1,\ldots, n-k$.
    \item For such a pair $(k,a)$, the label of the chamber to the \emph{left} of the $\left(N-s_k-a+1\right)$-th crossing is given by
        $$\begin{cases}
        \{1, \ldots, k \} & \text{if } k+a=n, \\
        \{1, \ldots, k \}\cup\{k+a+1, \ldots, n \} & \text{if } k+a\neq n.
        \end{cases}
        $$
\end{enumerate}
It follows that we assign the pair $(k,a)$ to the chamber on the left of the $\left(N-s_k-a+1\right)$-th crossing. The rightmost chambers are labelled consistently, taking $k=0$. We leave the chambers above and below the pseudoline arrangement unlabelled.

Continuing with our dimension $4$ example, the chamber pairs $(k,a)$ are shown in Figure \ref{fig ansatz arrangement i'_0^op dim 4 (k,a) labels}. In general we consider the $(N-s_k-a+1)$-th crossing. The pairs $(k,a)$ for each surrounding chamber are given, where they are defined, by the diagram in Figure \ref{fig chamber pair labels i'_0^op ansatz arrangement}.
\begin{figure}[hb]
\centering
\begin{minipage}[b]{0.47\textwidth}
    \centering
\begin{tikzpicture}[scale=0.85]
    % pseudolines
    \draw (0,4) -- (0.75,4) -- (1.25,3) -- (1.75,3) -- (2.25,2) -- (3.75,2) -- (4.25,1) -- (7,1);
    \draw (0,3) -- (0.75,3) -- (1.25,4) -- (2.75,4) -- (3.25,3) -- (4.75,3) -- (5.25,2) -- (7,2);
    \draw (0,2) -- (1.75,2) -- (2.25,3) -- (2.75,3) -- (3.25,4) -- (5.75,4) -- (6.25,3) -- (7,3);
    \draw (0,1) -- (3.75,1) -- (4.25,2) -- (4.75,2) -- (5.25,3) -- (5.75,3) -- (6.25,4) -- (7,4);

    % dots and stars
    \node at (1,3.5) {$\bullet$};
    \node at (2,2.5) {$\bullet$};
    \node at (3,3.5) {$\bullet$};
    \node at (4,1.5) {$\bullet$};
    \node at (5,2.5) {$\bullet$};
    \node at (6,3.5) {$\bullet$};

    % pseudoline labels
    \node at (-0.3,4) {$4$};
    \node at (-0.3,3) {$3$};
    \node at (-0.3,2) {$2$};
    \node at (-0.3,1) {$1$};

    \node at (7.3,4) {$1$};
    \node at (7.3,3) {$2$};
    \node at (7.3,2) {$3$};
    \node at (7.3,1) {$4$};

    % chamber labels
    \node at (0.3,3.5) {$(3,1)$};
    \node at (2,3.5) {$(2,1)$};
    \node at (4.5,3.5) {$(1,1)$};
    \node at (6.7,3.5) {$(0,1)$};

    \node at (1,2.5) {$(2,2)$};
    \node at (3.5,2.5) {$(1,2)$};
    \node at (6,2.5) {$(0,2)$};

    \node at (2,1.5) {$(1,3)$};
    \node at (5.5,1.5) {$(0,3)$};
\end{tikzpicture}
\caption{The ansatz arrangement for ${\mathbf{i}'_0}^{\mathrm{op}}$ in dimension 4 with $(k,a)$ labelling} \label{fig ansatz arrangement i'_0^op dim 4 (k,a) labels}
\end{minipage}
    \hfill
\begin{minipage}[b]{0.47\textwidth}
    \centering
\begin{tikzpicture}
    %dots and stars
    \node at (1,0.5) {$\bullet$};

    \draw[<-, >=stealth, densely dashed, rounded corners=15] (1.08,0.52) -- (1.7, 0.9) -- (2.2,1.85);
    \draw (-0.3,1) -- (0.75,1) -- (1.25,0) -- (2.3,0);
    \draw (-0.3,0) -- (0.75,0) -- (1.25,1) -- (2.3,1);
    %\labels
    \node at (1,-0.3) {{$(k-1, a+1)$}};
    \node at (0.2,0.5) {{$(k, a)$}};
    \node at (2.1,0.5) {{$(k-1, a)$}};
    \node at (1,1.3) {{$(k, a-1)$}};
    \node at (1,2) {{$\left(N-s_k-a+1\right)$-th crossing}};
\end{tikzpicture}
\caption{Labelling of chamber pairs $(k,a)$ for ${\mathbf{i}'_0}^{\mathrm{op}}$ ansatz arrangement} \label{fig chamber pair labels i'_0^op ansatz arrangement}
\end{minipage}
\end{figure}

With this notation we return to the coordinate change given by the Chamber Ansatz. It requires us to compute the minors corresponding to the chamber labels, which we will do in terms of the pairs $(k,a)$. To help us with this we recall Figure \ref{The graph for u^T when n=4}, namely that in dimension $4$ the graph for
    $$u^T = \mathbf{t}_{1}^{\vee}(z_6^{-1})\mathbf{x}_{1}^{\vee}(z_6) \mathbf{t}_{2}^{\vee}(z_5^{-1})\mathbf{x}_{2}^{\vee}(z_5) \mathbf{t}_{1}^{\vee}(z_4^{-1})\mathbf{x}_{1}^{\vee}(z_4) \mathbf{t}_{3}^{\vee}(z_3^{-1})\mathbf{x}_{3}^{\vee}(z_3) \mathbf{t}_{2}^{\vee}(z_2^{-1})\mathbf{x}_{2}^{\vee}(z_2) \mathbf{t}_{1}^{\vee}(z_1^{-1}) \mathbf{x}_{1}^{\vee}(z_1)
    $$
is given by
\begin{center}
\begin{tikzpicture}[scale=0.85]
    \draw (0,4) -- (10,4);
    \draw (0,3) -- (10,3);
    \draw (0,2) -- (10,2);
    \draw (0,1) -- (10,1);

    \draw (1.25,1) -- (1.75,2);
    \draw (2.75,2) -- (3.25,3);
    \draw (4.25,1) -- (4.75,2);
    \draw (5.75,3) -- (6.25,4);
    \draw (7.25,2) -- (7.75,3);
    \draw (8.75,1) -- (9.25,2);
    % dots and stars
    \node at (0,1) {$\bullet$};
    \node at (0,2) {$\bullet$};
    \node at (0,3) {$\bullet$};
    \node at (0,4) {$\bullet$};
    \node at (10,1) {$\bullet$};
    \node at (10,2) {$\bullet$};
    \node at (10,3) {$\bullet$};
    \node at (10,4) {$\bullet$};

    \node at (0.75,1) {$\bullet$};
    \node at (1.25,1) {$\bullet$};
    \node at (0.75,2) {$\bullet$};
    \node at (1.25,2) {$\bullet$};
    \node at (1.75,2) {$\bullet$};
    \node at (2.25,2) {$\bullet$};
    \node at (2.75,2) {$\bullet$};
    \node at (2.25,3) {$\bullet$};
    \node at (2.75,3) {$\bullet$};
    \node at (3.25,3) {$\bullet$};

    \node at (3.75,1) {$\bullet$};
    \node at (4.25,1) {$\bullet$};
    \node at (3.75,2) {$\bullet$};
    \node at (4.25,2) {$\bullet$};
    \node at (4.75,2) {$\bullet$};
    \node at (5.25,3) {$\bullet$};
    \node at (5.75,3) {$\bullet$};
    \node at (5.25,4) {$\bullet$};
    \node at (5.75,4) {$\bullet$};
    \node at (6.25,4) {$\bullet$};

    \node at (6.75,2) {$\bullet$};
    \node at (7.25,2) {$\bullet$};
    \node at (6.75,3) {$\bullet$};
    \node at (7.25,3) {$\bullet$};
    \node at (7.75,3) {$\bullet$};
    \node at (8.25,1) {$\bullet$};
    \node at (8.75,1) {$\bullet$};
    \node at (8.25,2) {$\bullet$};
    \node at (8.75,2) {$\bullet$};
    \node at (9.25,2) {$\bullet$};

    % pseudoline labels
    \node at (-0.3,4) {$4$};
    \node at (-0.3,3) {$3$};
    \node at (-0.3,2) {$2$};
    \node at (-0.3,1) {$1$};

    \node at (10.3,4) {$4$};
    \node at (10.3,3) {$3$};
    \node at (10.3,2) {$2$};
    \node at (10.3,1) {$1$};

    % weight labels
    \node at (1,1.3) {$\frac{1}{z_6}$};
    \node at (1,2.2) {$z_6$};
    \node at (1.8,1.5) {$z_6$};
    \node at (2.5,2.3) {$\frac{1}{z_5}$};
    \node at (2.5,3.2) {$z_5$};
    \node at (3.3,2.5) {$z_5$};
    \node at (4,1.3) {$\frac{1}{z_4}$};
    \node at (4,2.2) {$z_4$};
    \node at (4.8,1.5) {$z_4$};
    \node at (5.5,3.3) {$\frac{1}{z_3}$};
    \node at (5.5,4.2) {$z_3$};
    \node at (6.3,3.5) {$z_3$};
    \node at (7,2.3) {$\frac{1}{z_2}$};
    \node at (7,3.2) {$z_2$};
    \node at (7.8,2.5) {$z_2$};
    \node at (8.5,1.3) {$\frac{1}{z_1}$};
    \node at (8.5,2.2) {$z_1$};
    \node at (9.3,1.5) {$z_1$};
\end{tikzpicture}
\end{center}

To compute the minors corresponding to the chamber labels, we first use (\ref{eqn wt matrix z coords}) from Section \ref{subsec The form of the weight matrix} to see that for $k=1,\ldots, n-1$ we have
    $$\Delta^{\{1, \ldots, k\}}_{\{1, \ldots, k\}}(u^T) = \left(\prod_{\substack{1\leq m \leq N \\ i_m \in \{1, \ldots, k\}}} \frac{1}{z_m} \right) \left( \prod_{\substack{1\leq m \leq N \\ i_m \in \{1, \ldots, k-1\}}} z_m \right)
        = \prod_{\substack{1\leq m \leq N \\ i_m = k}} \frac{1}{z_m}.
    $$
Note that if $a=1$ then the chamber in the ansatz arrangement above the relevant crossings has label $\{1, \ldots, n\}$, with corresponding minor
    $$\Delta^{\{1, \ldots, n\}}_{\{1, \ldots, n\}}(u^T) = \left(\prod_{\substack{1\leq m \leq N \\ i_m \in \{1, \ldots, n-1\}}} \frac{1}{z_m} \right) \left( \prod_{\substack{1\leq m \leq N \\ i_m \in \{1, \ldots, n-1\}}} z_m \right)
        = 1.
    $$

For the remaining minors, using the graph for $u^T$ and Theorem \ref{thm minors from paths in graph} we see that
    $$\Delta^{\{1, \ldots, n-a\}}_{\{1, \ldots, k\}\cup\{k+a+1, \ldots, n \}}(u^T) = \Delta^{\{1, \ldots, k\}}_{\{1, \ldots, k\}}(u^T) \Delta^{\{k+1, \ldots, n-a\}}_{\{k+a+1, \ldots, n \}}(u^T).
    $$

The minors $\Delta^{\{k+1, \ldots, n-a\}}_{\{k+a+1, \ldots, n \}}(u^T)$ correspond to paths in the graph for $u^T$ which form `staircases' so their weights have contributions from both horizontal and diagonal edges.

The proof of Lemma \ref{lem our Chamber Ansatz minors are monomial} implies that on each path there are no horizontal edges with non-trivial weight after the last diagonal edge has been traversed. Additionally there is only one horizontal edge with non-trivial weight between each diagonal `step', namely the edge directly preceding the diagonal in Figure \ref{Labelled line segments in graphs in the proof of Lem coord change for u_1 p_i in terms of z_i}.
\begin{figure}[ht!]
\centering
\begin{tikzpicture}[scale=0.85]
    \draw (0,4) -- (3,4);
    \draw (0,3) -- (3,3);
    \draw (0,2) -- (3,2);
    \draw (0,1) -- (3,1);

    \draw (1.5,2) -- (2,3);

    % dots and stars
    \node at (1,2) {$\bullet$};
    \node at (1.5,2) {$\bullet$};
    \node at (1,3) {$\bullet$};
    \node at (1.5,3) {$\bullet$};
    \node at (2,3) {$\bullet$};

    % weight labels
    \node at (1.25,2.35) {$\frac{1}{z_m}$};
    \node at (1.25,3.25) {$z_m$};
    \node at (2.1,2.5) {$z_m$};
\end{tikzpicture}
\caption{Labelled line segments in graphs in the proof of Lemma \ref{lem coord change for u_1 p_i in terms of z_i}} \label{Labelled line segments in graphs in the proof of Lem coord change for u_1 p_i in terms of z_i}
\end{figure}
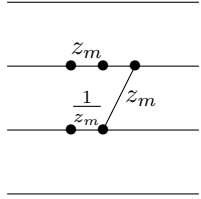

In particular, on each path the contributions from the diagonal edges and the horizontal edges directly preceding them will always cancel. So, roughly speaking, for each path we only need to consider the contributions from the horizontal edges which occur well before the first diagonal `step'; that is, if the first diagonal edge in the path occurs with weight $z_p$ then the only factors in the evaluation of the minor will be $\frac{1}{z_m}$, ${z_m}$ with $m>p$ (note that the ordering of the $i_m$ is reversed in the graph of $u^T$ due to taking the transpose):
    $$\begin{aligned}
    \Delta^{\{k+1, \ldots, n-a\}}_{\{k+a+1, \ldots, n \}}(u^T)
        &= \left( \prod_{\substack{m > i_m+\sum_{j=1}^{a-1} (n-j) \\ i_m\in \{k+1, \ldots, n-a\}}} \frac{1}{z_m} \right) \left( \prod_{\substack{m > i_m+\sum_{j=1}^{a-1} (n-j) \\  i_m\in \{k, \ldots, n-a-1\}}} z_m \right) \\
        &= \left( \prod_{\substack{m > n-a+s_a \\ i_m=n-a}} \frac{1}{z_m} \right) \left( \prod_{\substack{m > k+s_a \\ i_m=k}} z_m \right).
    \end{aligned}
    $$
Noticing that the last occurrence of $i_m=n-a$ in the reduced expression $\mathbf{i}_0$, is exactly when $m=n-a+s_{a}$, we obtain
    $$\Delta^{\{k+1, \ldots, n-a\}}_{\{k+a+1, \ldots, n \}}(u^T)
        = \prod_{\substack{m > k+s_{a} \\ i_m=k}} z_m.
    $$
Additionally we see that this minor is equal to $1$ if $k=0$, since $i_m\geq 1$ by definition.

Combining the above we obtain
    $$\begin{aligned}
    \Delta^{\{1, \ldots, n-a\}}_{\{1, \ldots, k\}\cup\{k+a+1, \ldots, n \}}(u^T)
        &= \Delta^{\{1, \ldots, k\}}_{\{1, \ldots, k\}}(u^T) \Delta^{\{k+1, \ldots, n-a\}}_{\{k+a+1, \ldots, n \}}(u^T) \\
        &= \left( \prod_{\substack{1\leq m \leq N \\ i_m = k}} \frac{1}{z_m} \right) \left( \prod_{\substack{m > k+s_{a} \\ i_m=k}} z_m \right) =  \prod_{\substack{m \leq k+s_{a} \\ i_m=k}} \frac{1}{z_m}.
    \end{aligned}
    $$

We now use the Chamber Ansatz to compute the $p_{s_k+a}$ coordinates, considering separately the case when $a=1$.

For $k=1,\ldots, n-1$, $a=2,\ldots, n-k$ we have
    \begin{equation}
    p_{s_k+a}
        = \frac{\left(\prod\limits_{\substack{m\leq k+s_{a-1} \\ i_{m}=k}} \frac{1}{z_{m}}\right) \left(\prod\limits_{\substack{m\leq k-1+s_{a+1} \\ i_{m}=k-1}} \frac{1}{z_{m}}\right)}{\left(\prod\limits_{\substack{m\leq k+s_{a} \\ i_m=k}} \frac{1}{z_m}\right) \left(\prod\limits_{\substack{ m\leq k-1+s_{a} \\ i_{m}=k-1}} \frac{1}{z_{m}}\right)}
        = \frac{\prod\limits_{\substack{k+s_{a-1} < m\leq k+s_{a} \\ i_m=k}} z_m }{\prod\limits_{\substack{k-1+s_{a} < m\leq k-1+s_{a+1} \\ i_{m}=k-1}} z_m}, \label{eqn t' coord general case}
    \end{equation}
We note that each of the products on the last line has exactly one term, namely when $m$ is equal to the upper bound. Thus we obtain
    $$p_{s_k+a} = \begin{cases}
        z_{k+s_{a}} & \text{if } k=1, \\
        \frac{z_{k+s_{a}}}{z_{k-1+s_{a+1}}} & \text{otherwise}.
        \end{cases}
    $$

We apply a similar argument in the case when $a=1$. For $k=1,\ldots, n-1$ the minor corresponding to the chamber above the pseudoline arrangement is equal to $1$ so we have
    $$\begin{aligned}
    p_{s_k+a}
        = \frac{\prod\limits_{\substack{m\leq k-1+s_{a+1} \\ i_{m}=k-1}} \frac{1}{z_{m}}}{\left(\prod\limits_{\substack{m\leq k+s_{a} \\ i_m=k}} \frac{1}{z_m}\right) \left(\prod\limits_{\substack{ m\leq k-1+s_{a} \\ i_{m}=k-1}} \frac{1}{z_{m}}\right)}
        &= \frac{\prod\limits_{\substack{m\leq k \\ i_m=k}} z_m }{\prod\limits_{\substack{k-1< m\leq k-1+(n-1) \\ i_{m}=k-1}} z_m} \\
        &= \begin{cases}
            z_{k} & \text{if } k=1 \\
            \frac{z_{k}}{z_{k-1+(n-1)}} & \text{otherwise}
            \end{cases}\\
        &= \begin{cases}
            z_{k+s_{a}} & \text{if } k=1 \\
            \frac{z_{k+s_{a}}}{z_{k-1+s_{a+1}}} & \text{otherwise}
            \end{cases}
    \end{aligned}
    $$
since $s_{a+1}=s_2=n-1$ and $s_{a}=s_1=0$.

Note that the need to consider the $a=1$ case separately is highlighted in (\ref{eqn t' coord general case}); if $a=1$ the numerator would be
    $$\prod\limits_{\substack{k+s_{0} < m\leq k+s_{1} \\ i_m=k}} z_m = \prod\limits_{\substack{k < m\leq k \\ i_m=k}} z_m =1
    $$
whereas in fact we should have $z_k=z_{k+s_{a}}$ in the numerator.

Combining the above we obtain the desired coordinate change:
    $$p_{s_k+a} =\begin{cases}
        z_{k+s_{a}} & \text{if } k=1, \\
        \frac{z_{k+s_{a}}}{z_{k-1+s_{a+1}}} & \text{otherwise}.
    \end{cases}
    $$
\end{proof}

\begin{lem} \label{lem coord change for b m_i in terms of p_i}
For $k=1,\ldots, n-1$, $a=1,\ldots, n-k$ we have
    $$\frac{1}{m_{s_k+a}}= \begin{cases}
        \prod\limits_{\substack{r=1,\ldots, k}} \frac{1}{p_{s_{r+1}-a+1}} &\text{if } k=1 \\
        \frac{\prod\limits_{\substack{r=1,\ldots, k-1}} p_{s_{r+1}-a} }{\prod\limits_{\substack{r=1,\ldots, k}} p_{s_{r+1}-a+1}} &\text{otherwise}
    \end{cases}
    \qquad \text{where} \ \
    s_k := \sum_{j=1}^{k-1}(n-j).
    $$
\end{lem}

\begin{proof}

In the proof of Lemma \ref{lem form of u_1 and b factorisations} we used the Chamber Ansatz with $\mathbf{i}_0$ to show that  $b B^{\vee}_+ = \mathbf{y}_{i_1}\left(\frac{1}{m_1}\right) \cdots \mathbf{y}_{i_N}\left(\frac{1}{m_N}\right)B^{\vee}_+$ where
    $$
    \frac{1}{m_k} = \frac{\prod_{j\neq i_k} \Delta^{\omega_j}_{w_{(k)}\omega_j}(u_1)^{-a_{j,i_k}}}{\Delta^{\omega_{i_k}}_{w_{(k)}\omega_{i_k}}(u_1) \Delta^{\omega_{i_k}}_{w_{(k-1)}\omega_{i_k}}(u_1)}, \quad k=1, \ldots, N.
    $$
Similar to the proof of the first coordinate change, since this expression is unpleasant we again work diagrammatically. In particular, we will use the ansatz arrangement for $\mathbf{i}'_0$ and the graph for $u_1$.

As before, to give a visual aid we recall Figure \ref{fig ansatz arrangement i_0 dim 4}, the dimension $4$ example of the ansatz arrangement for $\mathbf{i}_0=(1,2,3,1,2,1)$:
\begin{center}
    \begin{tikzpicture}[scale=0.85]
        % pseudolines
        \draw (0,4) -- (2.75,4) -- (3.25,3) -- (4.75,3) -- (5.25,2) -- (5.75,2) -- (6.25,1) -- (7,1);
        \draw (0,3) -- (1.75,3) -- (2.25,2) -- (3.75,2) -- (4.25,1) -- (5.75,1) -- (6.25,2) -- (7,2);
        \draw (0,2) -- (0.75,2) -- (1.25,1) -- (3.75,1) -- (4.25,2) -- (4.75,2) -- (5.25,3) -- (7,3);
        \draw (0,1) -- (0.75,1) -- (1.25,2) -- (1.75,2) -- (2.25,3) -- (2.75,3) -- (3.25,4) -- (7,4);

        % dots and stars
        \node at (1,1.5) {$\bullet$};
        \node at (2,2.5) {$\bullet$};
        \node at (3,3.5) {$\bullet$};
        \node at (4,1.5) {$\bullet$};
        \node at (5,2.5) {$\bullet$};
        \node at (6,1.5) {$\bullet$};

        % pseudoline labels
        \node at (-0.3,4) {$4$};
        \node at (-0.3,3) {$3$};
        \node at (-0.3,2) {$2$};
        \node at (-0.3,1) {$1$};

        \node at (7.3,4) {$1$};
        \node at (7.3,3) {$2$};
        \node at (7.3,2) {$3$};
        \node at (7.3,1) {$4$};

        % chamber labels
        \node at (1.5,3.5) {$123$};
        \node at (5,3.5) {$234$};

        \node at (1,2.5) {$12$};
        \node at (3.5,2.5) {$23$};
        \node at (6,2.5) {$34$};

        \node at (0.5,1.5) {$1$};
        \node at (2.5,1.5) {$2$};
        \node at (5,1.5) {$3$};
        \node at (6.5,1.5) {$4$};
    \end{tikzpicture}
\end{center}

Similar to the $u^T$ case, we give new labelling of the chambers of the ansatz arrangement in terms of pairs $(k,a)$. This time, for $k=1,\ldots, n-1$, $a=1,\ldots, n-k$, the label of the chamber to the \emph{right} of the $(s_k+a)$-th crossing is given by
    $$\{ k+1, \ldots, k+a \}.
    $$
It is then natural to assign the pair $(k,a)$ to the chamber on the right of the $(s_k+a)$-th crossing. The leftmost chambers will be labelled consistently, taking $k=0$. Again we leave the chambers above and below the pseudoline arrangement unlabelled.

Continuing our dimension $4$ example, the chamber pairs $(k,a)$ are shown in Figure \ref{fig ansatz arrangement i_0 dim 4 (k,a) labels}. In general, the pairs $(k,a)$ for each chamber surrounding the $(s_k+a)$-th crossing, where there are defined, are given by the diagram in Figure \ref{fig chamber pair labels i_0 ansatz arrangement}.
\begin{figure}[hb]
\centering
\begin{minipage}[b]{0.47\textwidth}
    \centering
\begin{tikzpicture}[scale=0.85]
    % pseudolines
    \draw (0,4) -- (2.75,4) -- (3.25,3) -- (4.75,3) -- (5.25,2) -- (5.75,2) -- (6.25,1) -- (7,1);
    \draw (0,3) -- (1.75,3) -- (2.25,2) -- (3.75,2) -- (4.25,1) -- (5.75,1) -- (6.25,2) -- (7,2);
    \draw (0,2) -- (0.75,2) -- (1.25,1) -- (3.75,1) -- (4.25,2) -- (4.75,2) -- (5.25,3) -- (7,3);
    \draw (0,1) -- (0.75,1) -- (1.25,2) -- (1.75,2) -- (2.25,3) -- (2.75,3) -- (3.25,4) -- (7,4);

    % dots and stars
    \node at (1,1.5) {$\bullet$};
    \node at (2,2.5) {$\bullet$};
    \node at (3,3.5) {$\bullet$};
    \node at (4,1.5) {$\bullet$};
    \node at (5,2.5) {$\bullet$};
    \node at (6,1.5) {$\bullet$};

    % pseudoline labels
    \node at (-0.3,4) {$4$};
    \node at (-0.3,3) {$3$};
    \node at (-0.3,2) {$2$};
    \node at (-0.3,1) {$1$};

    \node at (7.3,4) {$1$};
    \node at (7.3,3) {$2$};
    \node at (7.3,2) {$3$};
    \node at (7.3,1) {$4$};

    % chamber labels
    \node at (1.5,3.5) {$(0,3)$};
    \node at (5,3.5) {$(1,3)$};

    \node at (1,2.5) {$(0,2)$};
    \node at (3.5,2.5) {$(1,2)$};
    \node at (6,2.5) {$(2,2)$};

    \node at (0.3,1.5) {$(0,1)$};
    \node at (2.5,1.5) {$(1,1)$};
    \node at (5,1.5) {$(2,1)$};
    \node at (6.7,1.5) {$(3,1)$};
\end{tikzpicture}
\caption{The ansatz arrangement for $\mathbf{i}_0$ in dimension 4 with $(k,a)$ labelling} \label{fig ansatz arrangement i_0 dim 4 (k,a) labels}
\end{minipage}
    \hfill
\begin{minipage}[b]{0.47\textwidth}
    \centering
\begin{tikzpicture}
    %dots and stars
    \node at (1,0.5) {$\bullet$};

    \draw[<-, >=stealth, densely dashed, rounded corners=15] (1.08,0.52) -- (2, 0.9) -- (2.2,1.85);
    \draw (-0.3,1) -- (0.75,1) -- (1.25,0) -- (2.3,0);
    \draw (-0.3,0) -- (0.75,0) -- (1.25,1) -- (2.3,1);
    %\labels
    \node at (1,-0.3) {{$(k, a-1)$}};
    \node at (-0.2,0.5) {{$(k-1, a)$}};
    \node at (1.9,0.5) {{$(k, a)$}};
    \node at (1,1.3) {{$(k-1, a+1)$}};
    \node at (1,2) {{$(s_k+a)$-th crossing}};
\end{tikzpicture}
\caption{Labelling of chamber pairs $(k,a)$ for $\mathbf{i}_0$ ansatz arrangement} \label{fig chamber pair labels i_0 ansatz arrangement}
\end{minipage}
\end{figure}

Before applying the Chamber Ansatz we again compute minors in terms of the pairs $(k,a)$. We have just seen that in dimension $4$ the matrix $u_1$ is given by
    $$u_1=\mathbf{x}_{\mathbf{i}'_0}^{\vee}\left(z_1, z_4, z_6, \frac{z_2}{z_4}, \frac{z_5}{z_6}, \frac{z_3}{z_5}\right).
    $$
This is quite messy however, so for ease of computation we will continue to work for with the $p_j$ coordinates; $u_1=\mathbf{x}_{\mathbf{i}'_0}^{\vee}(p_1, \ldots, p_6)$. The corresponding graph for $u_1$ is given in Figure \ref{The graph for u_1 when n=4, p coords}.
\begin{figure}[ht!]
\centering
    \begin{tikzpicture}[scale=0.85]
        % pseudolines
        \draw (0,4) -- (7,4);
        \draw (0,3) -- (7,3);
        \draw (0,2) -- (7,2);
        \draw (0,1) -- (7,1);

        \draw (0.75,3) -- (1.25,4);
        \draw (1.75,2) -- (2.25,3);
        \draw (2.75,1) -- (3.25,2);
        \draw (3.75,3) -- (4.25,4);
        \draw (4.75,2) -- (5.25,3);
        \draw (5.75,3) -- (6.25,4);
        % dots and stars
        \node at (0,1) {$\bullet$};
        \node at (0,2) {$\bullet$};
        \node at (0,3) {$\bullet$};
        \node at (0,4) {$\bullet$};
        \node at (7,1) {$\bullet$};
        \node at (7,2) {$\bullet$};
        \node at (7,3) {$\bullet$};
        \node at (7,4) {$\bullet$};

        \node at (0.75,3) {$\bullet$};
        \node at (1.25,4) {$\bullet$};
        \node at (1.75,2) {$\bullet$};
        \node at (2.25,3) {$\bullet$};
        \node at (2.75,1) {$\bullet$};
        \node at (3.25,2) {$\bullet$};

        \node at (3.75,3) {$\bullet$};
        \node at (4.25,4) {$\bullet$};
        \node at (4.75,2) {$\bullet$};
        \node at (5.25,3) {$\bullet$};
        \node at (5.75,3) {$\bullet$};
        \node at (6.25,4) {$\bullet$};

        % pseudoline labels
        \node at (-0.3,4) {$4$};
        \node at (-0.3,3) {$3$};
        \node at (-0.3,2) {$2$};
        \node at (-0.3,1) {$1$};

        \node at (7.3,4) {$4$};
        \node at (7.3,3) {$3$};
        \node at (7.3,2) {$2$};
        \node at (7.3,1) {$1$};

        % weight labels
        \node at (1.3,3.5) {$p_1$};
        \node at (2.3,2.5) {$p_2$};
        \node at (3.3,1.5) {$p_3$};
        \node at (4.3,3.5) {$p_4$};
        \node at (5.3,2.5) {$p_5$};
        \node at (6.3,3.5) {$p_6$};
    \end{tikzpicture}
\caption{The graph for $u_1$ when $n=4$, $\boldsymbol{p}$ coordinates} \label{The graph for u_1 when n=4, p coords}
\end{figure}

From the proof of Lemma \ref{lem our Chamber Ansatz minors are monomial} we recall that there is only one family of vertex disjoint paths in the graph for $u_1$ from the set of sources $\{1, \ldots, a\}$ to the set of sinks $\{k+1, \ldots, k+a\}$. Within this family, the weight of the path from $b \in \{1,\ldots,a\}$ to $k+b$ is given by
    $$\prod_{\substack{m= (n-b)+\sum_{j=2}^r (n-j) \\ r=1,\ldots, k}} p_m = \prod_{\substack{r=1,\ldots, k}} p_{s_{r+1}-b+1}.
    $$

Taking the product over this family, that is, the product of the weights of the paths from $\{1, \ldots, a \}$ to $\{k+1, \ldots, k+a\}$, we obtain the desired minor
    $$\Delta^{\{1, \ldots, a \}}_{\{k+1, \ldots, k+a\}}(u_1)
        = \prod_{\substack{r=1,\ldots, k \\ b=1, \ldots, a}} p_{s_{r+1}-b+1}.
    $$

Note that since $u_1 \in U^{\vee}_+$, the minors corresponding to the leftmost chambers and the chamber above the pseudoline arrangement are all equal to $1$.

We now use the Chamber Ansatz to compute the $\frac{1}{m_j}$ coordinates.
For $k=1,\ldots, n-1$, $a=1,\ldots, n-k$ we have
    \begin{align}
    \frac{1}{m_{s_k+a}}
        &= \frac{\left(\prod\limits_{\substack{r=1,\ldots, k-1 \\ b=1, \ldots, a+1}} p_{s_{r+1}-b+1}\right) \left(\prod\limits_{\substack{r=1,\ldots, k \\ b=1, \ldots, a-1}} p_{s_{r+1}-b+1}\right)}{\left(\prod\limits_{\substack{r=1,\ldots, k-1 \\ b=1, \ldots, a}} p_{s_{r+1}-b+1}\right) \left(\prod\limits_{\substack{r=1,\ldots, k \\ b=1, \ldots, a}} p_{s_{r+1}-b+1}\right)} \notag \\
        &= \begin{cases}
        \prod\limits_{\substack{r=1,\ldots, k}} \frac{1}{p_{s_{r+1}-a+1}} &\text{if } k=1, \\
        \frac{\prod\limits_{\substack{r=1,\ldots, k-1}} p_{s_{r+1}-a} }{\prod\limits_{\substack{r=1,\ldots, k}} p_{s_{r+1}-a+1}} &\text{otherwise.}
    \end{cases} \label{eqn m coords in terms of p}
    \end{align}

\end{proof}

We now ready to prove Theorem \ref{thm coord change}, the statement of which we recall here:
\begin{thm*}
For $k=1,\ldots, n-1$, $a=1,\ldots, n-k$ we have
    $$m_{s_k+a} = \begin{cases}
            z_{1+s_{n-a}} &\text{if } k=1 \\
            \frac{z_{k+s_{n-k-a+1}}}{z_{k-1+s_{n-k-a+1}}} &\text{otherwise}
            \end{cases}
    \qquad \text{where} \ \
    s_k := \sum_{j=1}^{k-1}(n-j).
    $$
\end{thm*}

\begin{proof}

To obtain the $\frac{1}{m_{j}}$ coordinates in terms of the $z_j$, we need to compose the two coordinate transformations from Lemmas \ref{lem coord change for u_1 p_i in terms of z_i} and \ref{lem coord change for b m_i in terms of p_i}. To do this we write the factors $p_{s_{r+1}-a}$, $p_{s_{r+1}-a+1}$ of the products in (\ref{eqn m coords in terms of p}) in the form $p_{s_{k'}+{a'}}$ for a suitable pairs $(k',a')$;
\begin{enumerate}
    \item Firstly, given a pair $(k,a)$ with $k>1$, for each $r=1,\ldots, k-1$, we wish to find the pair $(k',a')$, $k'\in \{1, \ldots, n-1\}$, $a' \in \{1,\ldots, n-k'\}$ such that
        $$s_{k'}+a'=s_{r+1}-a.
        $$
    Indeed since
        $$s_{r+1}-a = s_{r}+n-r-a
        $$
    we take $k'=r$, $a'=n-r-a$, noting that $k'$ and $a'$ satisfy the necessary conditions
        $$ \begin{aligned}
        k'&=r \in \{1, \ldots, k-1\}\subseteq \{1, \ldots, n-1\}, \\
        a'&=n-r-a = n-k'-a \in \{k-k', \ldots, n-k'-1\} \subseteq \{1, \ldots, n-k'\}.
        \end{aligned}
        $$

    \item Similarly given a pair $(k,a)$ for any $k$, i.e. $1 \leq k \leq n-1$, for each $r=1,\ldots, k$, we wish to find these pairs $(k',a')$, $k'\in \{1, \ldots, n-1\}$, $a' \in \{1,\ldots, n-k'\}$ such that
        $$s_{k'}+a'=s_{r+1}-a+1 = s_{r}+n-r-a+1.
        $$
    We take $k'=r$, $a'=n-r-a+1$, again noting that $k'$ and $a'$ satisfy the necessary conditions
        $$ \begin{aligned}
        k'&=r \in \{1, \ldots, k\}\subseteq \{1, \ldots, n-1\}, \\
        a'&=n-r-a+1 = n-k'-a+1 \in \{k-k'+1, \ldots, n-k'\} \subseteq \{1, \ldots, n-k'\}.
        \end{aligned}
        $$
\end{enumerate}

In order to evaluate the products in (\ref{eqn m coords in terms of p}) we need to consider what happens to the pair $(k',a')$ when we increase $r$ by $1$.
\begin{enumerate}
    \item In the first case: for $r\in\{1, \ldots, k-2\}$
        $$s_{r+2}-a = s_{r+1}+n-(r+1)-a = s_{k'}+a'+ n-k'-1 = s_{k'+1}+a'-1.
        $$
    Note that $k'+1$ and $a'-1$ satisfy the necessary conditions
        $$ \begin{aligned}
        k'+1 &=r+1 \in \{2, \ldots, k-1\}\subseteq \{1, \ldots, n-1\}, \\
        a'-1 & =n-r-a-1 \in \{k-(k'+1), \ldots, n-(k'+1)-1\} \subseteq \{1, \ldots, n-(k'+1)\}.
        \end{aligned}
        $$

    \item Similarly in the second case: for $r \in\{1, \ldots, k-1\}$
        $$s_{r+2}-a+1 = s_{r+1}+n-(r+1)-a+1 = s_{k'}+a'+ n-k'-1 = s_{k'+1}+a'-1.
        $$
    Again we note that $k'+1$ and $a'-1$ satisfy the necessary conditions
        $$ \begin{aligned}
        k'+1 &=r+1 \in \{2, \ldots, k\}\subseteq \{1, \ldots, n-1\}, \\
        a'-1 &=n-r-a \in \{k-(k'+1)+1, \ldots, n-(k'+1)\} \subseteq \{1, \ldots, n-(k'+1)\}.
        \end{aligned}
        $$
\end{enumerate}

Now recalling
    $$p_{s_k+a} =\begin{cases}
        z_{k+s_{a}} & \text{if } k=1 \\
        \frac{z_{k+s_{a}}}{z_{k-1+s_{a+1}}} & \text{otherwise}
    \end{cases}
    $$
we have, for $r>1$ (so that $k'>1$)
    $$\begin{aligned}
    p_{s_{r+1}-a}p_{s_{r+2}-a}
        &= p_{s_{k'}+a'}p_{s_{k'+1}+a'-1} &\quad \text{where } k'=r,\ a'=n-r-a\\
        &= \frac{z_{k'+s_{a'}}}{z_{k'-1+s_{a'+1}}}\frac{z_{k'+1+s_{a'-1}}}{z_{k'+s_{a'}}} \\
        &= \frac{z_{k'+1+s_{a'-1}}}{z_{k'-1+s_{a'+1}}}.
    \end{aligned}
    $$
Consequently the following product becomes telescopic, so we have
    $$\prod\limits_{\substack{r=1,\ldots, k-1}} p_{s_{r+1}-a} = z_{k-1+s_{n-(k-1)-a}}.
    $$

Similarly in the second case
    $$\begin{aligned}
    p_{s_{r+1}-a+1}p_{s_{r+2}-a+1}
        &= p_{s_{k'}+a'}p_{s_{k'+1}+a'-1} &\quad \text{where } k'=r-1,\ a'=n-r-a+1\\
        &= \frac{z_{k'+1+s_{a'-1}}}{z_{k'-1+s_{a'+1}}}.
    \end{aligned}
    $$
Again we have a telescopic product, giving
    $$\prod\limits_{\substack{r=1,\ldots, k}} p_{s_{r+1}-a+1} = z_{k+s_{n-k-a+1}}.
    $$

Thus we have
    $$\frac{1}{m_{s_k+a}} = \begin{cases}
        \frac{1}{z_{k+s_{n-k-a+1}}} &\text{if } k=1 \\
        \frac{z_{k-1+s_{n-k-a+1}}}{z_{k+s_{n-k-a+1}}} &\text{otherwise}
    \end{cases}
    $$
which gives the desired coordinate change: for $k=1,\ldots, n-1$, $a=1,\ldots, n-k$
    $$m_{s_k+a} = \begin{cases}
            z_{1+s_{n-a}} &\text{if } k=1 \\
            \frac{z_{k+s_{n-k-a+1}}}{z_{k-1+s_{n-k-a+1}}} &\text{otherwise.}
            \end{cases}
    $$
\end{proof}

\section{Givental-type quivers and critical points} \label{sec Givental-type quivers and critical points}
\fancyhead[L]{3 \ \ Givental-type quivers and critical points}
% \fancyhead[L]{4 \ \ Givental-type quivers and critical points}

In this section we recall an earlier Landau--Ginzburg model for the full flag variety, defined on a torus by Givental \cite{Givental1997}. We relate our tori from Sections \ref{sec Mirror symmetry for G/B applied to representation theory} and \ref{sec The ideal coordinates}, as well as the superpotential defined in Section \ref{subsec Landau-Ginzburg models}, to Givental's torus and his formulation of the superpotential. We then use this to start describing the critical points of the superpotential. The key result in this section is Proposition \ref{prop crit points, sum at vertex is nu_i}, describing the $\boldsymbol{m}$ coordinates of such a critical point. This formula was conjectured by Konstanze Rietsch and checked in a particular case by Zainab %Fadhil Hussein
Al-Sultani \cite{ZainabMastersThesis}.

\subsection{The Givental superpotential} \label{subsec The Givental superpotential}
\fancyhead[L]{3.1 \ \ The Givental superpotential}
% \fancyhead[L]{4.1 \ \ The Givental superpotential}

In this section we recall Givental's construction from \cite{Givental1997}. We begin by considering a quiver, which consists of $n(n+1)/2$ vertices in lower triangular form together with arrows going up and left. We label the vertices with $v_{ij}$ for $1\leq j \leq i \leq n$ in the same way as we would for matrix entries and denote the set of such vertices by $\mathcal{V}$. The vertices $v_{ii}$ are star vertices and all others are dot vertices. We denote the sets of star and dot vertices respectively by
    $$\mathcal{V}^* = \left\{ v_{ii} \ | \ 1 \leq i \leq n \right\}, \quad \mathcal{V}^{\bullet} = \left\{ v_{ij} \ | \ 1 \leq j < i \leq n \right\}.
    $$

The set of arrows of the quiver is denoted $\mathcal{A}=\mathcal{A}_{\mathrm{v}} \cup \mathcal{A}_{\mathrm{h}}$. The vertical arrows $a_{ij}\in \mathcal{A}_{\mathrm{v}}$ are labelled such that $h(a_{ij})=v_{ij}$ where $h(a) \in \mathcal{V}$ denotes the head of the arrow $a$. The horizontal arrows $b_{ij} \in \mathcal{A}_{\mathrm{h}}$ are labelled such that $t(b_{ij})=v_{ij}$ where $t(a)\in \mathcal{V}$ denotes the tail of the arrow $a$. For example when $n=4$ the quiver is given in Figure \ref{fig Quiver for n=4}.
\begin{figure}[ht]
\centering
\begin{minipage}[b]{0.47\textwidth}
    \centering
\begin{tikzpicture}
    %dots and stars
    \node (41) at (0,0) {$\bullet$};
    \node (31) at (0,1.5) {$\bullet$};
    \node (21) at (0,3) {$\bullet$};
    \node (11) at (0,4.5) {$\boldsymbol{*}$};
        \node at (0.2,4.8) {$v_{11}$};
    \node (42) at (1.5,0) {$\bullet$};
    \node (32) at (1.5,1.5) {$\bullet$};
    \node (22) at (1.5,3) {$\boldsymbol{*}$};
        \node at (1.7,3.3) {$v_{22}$};
    \node (43) at (3,0) {$\bullet$};
    \node (33) at (3,1.5) {$\boldsymbol{*}$};
        \node at (3.2,1.8) {$v_{33}$};
    \node (44) at (4.5,0) {$\boldsymbol{*}$};
        \node at (4.7,0.3) {$v_{44}$};

    %arrows
    \draw[->] (41) -- (31);
    \draw[->] (31) -- (21);
    \draw[->] (21) -- (11);
    \draw[->] (42) -- (32);
    \draw[->] (32) -- (22);
    \draw[->] (43) -- (33);

    \draw[->] (22) -- (21);
    \draw[->] (33) -- (32);
    \draw[->] (32) -- (31);
    \draw[->] (44) -- (43);
    \draw[->] (43) -- (42);
    \draw[->] (42) -- (41);

    %dot vertex labels
    \node at (-0.4,3) {$v_{21}$};
    \node at (-0.4,1.5) {$v_{31}$};
    \node at (-0.4,0) {$v_{41}$};

    \node at (1.1,1.8) {$v_{32}$};
    \node at (1.5,-0.3) {$v_{42}$};
    \node at (3,-0.3) {$v_{43}$};

    %arrow labels
    \node at (-0.3,3.75) {$a_{11}$};
    \node at (-0.3,2.25) {$a_{21}$};
    \node at (-0.3,0.75) {$a_{31}$};

    \node at (0.75,2.7) {$b_{22}$};
    \node at (0.75,1.2) {$b_{32}$};
    \node at (0.75,-0.3) {$b_{42}$};

    \node at (1.2,2.25) {$a_{22}$};
    \node at (1.2,0.75) {$a_{32}$};
    \node at (2.7,0.75) {$a_{33}$};

    \node at (2.25,1.2) {$b_{33}$};
    \node at (2.25,-0.3) {$b_{43}$};
    \node at (3.75,-0.3) {$b_{44}$};
\end{tikzpicture}
\caption{Quiver when $n=4$} \label{fig Quiver for n=4}
\end{minipage}
    \hfill
\begin{minipage}[b]{0.47\textwidth}
    \centering
\begin{tikzpicture}
    %dots and stars
    \node (41) at (0,0) {};
    \node (31) at (0,1.5) {};
    \node (42) at (1.5,0) {};
    \node (32) at (1.5,1.5) {};

    %arrows
    \draw[->] (41) -- (31);
    \draw[->] (42) -- (32);

    \draw[->] (32) -- (31);
    \draw[->] (42) -- (41);

    %arrow labels
    \node at (-0.4,0.75) {$a_2$};
    \node at (1.9,0.75) {$a_3$};

    \node at (0.75,1.75) {$a_4$};
    \node at (0.75,-0.3) {$a_1$};
\end{tikzpicture}
\caption{Subquiver for box relations} \label{fig quiver box relations}
\end{minipage}
\end{figure}

We consider three tori which are defined in terms of the quiver, of which we introduce the first two now. The first torus is $(\mathbb{K}^*)^{\mathcal{V}}$ with coordinates $x_v$ for $v \in \mathcal{V}$, we call this the vertex torus. The second torus, $\bar{\mathcal{M}} \subset (\mathbb{K}^*)^{\mathcal{A}}$, corresponds to the arrows of the quiver and so will be called the arrow torus. It is given by
    $$ \bar{\mathcal{M}}:= \left\{ (r_a)_{a \in \mathcal{A}} \in (\mathbb{K}^*)^{\mathcal{A}} \ | \ r_{a_1}r_{a_2}=r_{a_3}r_{a_4} \text{ when } a_1,a_2,a_3,a_4 \text{ form a square as in Figure \ref{fig quiver box relations}} \right\}.
    $$
These two tori are related by the following surjection, given coordinate-wise:
    \begin{equation} \label{eqn quiver tori relation map}
    (\mathbb{K}^*)^{\mathcal{V}} \to \bar{\mathcal{M}}, \quad r_a = \frac{x_{h(a)}}{x_{t(a)}}.
    \end{equation}
Note that we get a point in the preimage of $\boldsymbol{r}_{\mathcal{A}}:=(r_a)_{a \in \mathcal{A}}\in \bar{\mathcal{M}}$ by first setting $x_{v_{nn}}=1$. Then for $v\neq v_{nn}$ we take $x_v=\prod_{a \in P_v} r_a$ where $P_v$ is any path from $v_{nn}$ to $v$. This map is well-defined since the preimage of $(1)_{a \in \mathcal{A}}$ is the set $\{ (c)_{v\in \mathcal{V}} \ | \ c \in \mathbb{K}^* \}$.

An analogy of this surjection is the map
    $$T^{\vee} \to (\mathbb{K}^*)^{n-1}
    $$
given by the simple roots $\alpha^{\vee}_1, \ldots, \alpha^{\vee}_{n-1}$ of $G^{\vee}$.

It will be more convenient for our purposes to use arrow coordinates rather than vertex coordinates. If we were to work with $SL_n$ then the arrow torus, $\bar{\mathcal{M}}$, would suffice, however we wish to work with $GL_n$ and so need to keep track of which fibre of the map (\ref{eqn quiver tori relation map}) we are in. Consequently we now introduce our third torus, $\mathcal{M}$, which we call the quiver torus:
    $$\mathcal{M} := \left\{ ( \boldsymbol{x}_{\mathcal{V}^*} , \boldsymbol{r}_{\mathcal{A}_{\mathrm{v}}}) \in (\mathbb{K}^*)^{\mathcal{V}^*} \times (\mathbb{K}^*)^{\mathcal{A}_{\mathrm{v}}} \right\}
    \hookrightarrow (\mathbb{K}^*)^{\mathcal{V}} \times \bar{\mathcal{M}}.
    $$
The quiver torus is isomorphic to the vertex torus, $(\mathbb{K}^*)^{\mathcal{V}}$, and we will work with these tori interchangeably. This isomorphism is a consequence of the following observations:

We note that if we choose the star vertex coordinate $x_{v_{nn}}$, then we are taking a particular lift of $\boldsymbol{r}_{\mathcal{A}} \in \bar{\mathcal{M}}$ such that in this fibre the map (\ref{eqn quiver tori relation map}) restricts to an isomorphism. Moreover, if we describe the coordinates of the star vertices and vertical arrows, then the coordinates of the horizontal arrows, and thus also of the dot vertices, are given uniquely using (\ref{eqn quiver tori relation map}) and the box relations; namely we use the fact that
    \begin{equation} \label{eqn relation in quiver from stars}
    r_{b_{i+1,i+1}} = \frac{x_{v_{ii}}}{x_{v_{i+1, i+1}}}\frac{1}{r_{a_{ii}}}, \quad i=1, \ldots, n-1
    \end{equation}
together with the relations $r_{a_1} r_{a_2} = r_{a_3} r_{a_4}$ when the arrows $a_1$, $a_2$, $a_3$, $a_4$ form a square as in Figure \ref{fig quiver box relations}.

We now recall the definition of Givental's superpotential. On the vertex torus this can be defined as
    $$\mathcal{F} :(\mathbb{K}^*)^{\mathcal{V}} \to \mathbb{K}, \quad \boldsymbol{x}_{\mathcal{V}} \mapsto \sum_{a \in \mathcal{A}} \frac{x_{h(a)}}{x_{t(a)}}.
    $$
This factors naturally through the arrow torus via (\ref{eqn quiver tori relation map}) and the following map:
    $$\bar{\mathcal{F}} : \bar{\mathcal{M}} \to \mathbb{K}, \quad \boldsymbol{r}_{\mathcal{A}} \mapsto \sum_{a \in \mathcal{A}} r_a.
    $$

We can now define the highest weight and weight maps on the vertex torus. The highest weight map is given by
    $$\kappa :(\mathbb{K}^*)^{\mathcal{V}} \to T^{\vee}, \quad \boldsymbol{x}_{\mathcal{V}} \mapsto (x_{v_{ii}})_{i=1, \ldots, n}.
    $$
The weight map \cite{JoeKim2003} on the vertex torus is defined in two steps; firstly, for $i=1, \ldots,n$ we let $\mathcal{D}_i := \{ v_{i,1}, v_{i+1,2}, \ldots, v_{n, n-i+1} \}$ be the $i$-th diagonal and let
    \begin{equation} \label{eqn xi for wt map defn}
    \Xi_i := \prod_{v\in \mathcal{D}_i} x_v \quad \text{with} \quad \Xi_{n+1}:=1.
    \end{equation}
Then the weight map is given by
    \begin{equation}\label{eqn defn gamma and t}
    \gamma :(\mathbb{K}^*)^{\mathcal{V}}\to T^{\vee}, \quad \boldsymbol{x}_{\mathcal{V}} \mapsto (t_i)_{i=1, \ldots, n} \quad \text{where} \quad t_i = \frac{\Xi_i}{\Xi_{i+1}}.
    \end{equation}

Unlike the superpotential $\mathcal{F}$, the maps $\kappa$ and $\gamma$ do not directly factor through $\bar{\mathcal{M}}$ but thanks to the map (\ref{eqn quiver tori relation map}) we have commutative diagrams:
\begin{equation*}
\begin{tikzcd}
(\mathbb{K}^*)^{\mathcal{V}} \arrow[d] \arrow[r, "\kappa"] & T^{\vee} \arrow[d, "{(\alpha^{\vee}_1, \ldots, \alpha^{\vee}_{n-1})}"] \\
\bar{\mathcal{M}} \arrow[r, "q"] & (\mathbb{K}^*)^{n-1}
\end{tikzcd}
\hspace{1.5cm}
\begin{tikzcd}
(\mathbb{K}^*)^{\mathcal{V}} \arrow[d] \arrow[r, "\gamma"] & T^{\vee} \arrow[d, "{(\alpha^{\vee}_1, \ldots, \alpha^{\vee}_{n-1})}"] \\
\bar{\mathcal{M}} \arrow[r] & (\mathbb{K}^*)^{n-1}
\end{tikzcd}
\end{equation*}
The maps along the bottom making the diagrams commute exist and are unique. For example, the map on $\bar{\mathcal{M}}$ corresponding to $\kappa$ is the map called $q$ defined by Givental as follows:
    $$q : \bar{\mathcal{M}} \to (\mathbb{K}^*)^{n-1}, \quad \boldsymbol{r}_{\mathcal{A}} \mapsto \left( r_{a_{11}}r_{b_{22}}, \, \ldots, \, r_{a_{n-1, n-1}}r_{b_{nn}} \right).
    $$

\begin{rem}
In Givental's version of (Fano) mirror symmetry, the arrow torus $\bar{\mathcal{M}}$ (taken over $\mathbb{C}$), is viewed as a family of varieties via the map $q$. Each fibre $\bar{\mathcal{M}}_{\boldsymbol{q}}$, being a torus, comes equipped with a natural holomorphic volume form $\omega_{\boldsymbol{q}}$. He proves a version of mirror symmetry relating the A-model connection, built out of Gromov-Witten invariants of the flag variety, to period integrals $S(\boldsymbol{q})$ on the family $(\bar{\mathcal{M}}_{\boldsymbol{q}}, \bar{\mathcal{F}}_{\boldsymbol{q}}, \omega_{\boldsymbol{q}} )$, defined using the superpotential:
    $$S(\boldsymbol{q}):= \int_{\Gamma} e^{\bar{\mathcal{F}}_{\boldsymbol{q}}} \omega_{\boldsymbol{q}}.
    $$
\end{rem}

\subsection[The quiver torus as another toric chart on \texorpdfstring{$Z$}{Z}]{The quiver torus as another toric chart on $Z$} \label{subsec The quiver torus as another toric chart on Z}
\fancyhead[L]{3.2 \ \ The quiver torus as another toric chart on $Z$}
% \fancyhead[L]{4.2 \ \ The quiver torus as another toric chart on $Z$}

In this section we recall a map $\mathcal{M} \to Z$ from \cite{Rietsch2008} which allows us to relate the Givental superpotential on the arrow torus, as well as the highest weight and weight maps on the vertex torus, to their analogues on $Z$. In particular our toric charts from the previous sections factor through this map.

A nice way to describe how the string and ideal toric charts factor through $\mathcal{M}$ is by decorating the arrows and vertices of the quiver. Indeed, using our previous work relating the string and ideal coordinate systems, we may easily describe a monomial map from the torus $T^{\vee}\times (\mathbb{K}^*)^N$ of string coordinates to the quiver torus $\mathcal{M}$. Composing this with the map $\mathcal{M} \to Z$ will recover our string toric chart.
We then use the coordinate change given in Theorem \ref{thm coord change} to decorate the quiver with the ideal coordinates.

In order to make sense of the way the string toric chart factors through $\mathcal{M}$, we choose an ordering of the $N$ vertical arrow coordinates. Starting at the lower left corner of the quiver and moving up each column in succession we obtain
    $$\boldsymbol{r}_{\mathcal{A}_{\mathrm{v}}} := \left(r_{a_{n-1, 1}}, r_{a_{n-2, 1}}, \ldots, r_{a_{1, 1}}, r_{a_{n-1, 2}}, r_{a_{n-2, 2}}, r_{a_{2, 2}}, \ldots, r_{a_{n-1, n-2}}, r_{a_{n-2, n-2}}, r_{a_{n-1, n-1}}\right).
    $$
Similarly, we give an ordering of the star vertex coordinates:
    $$\boldsymbol{x}_{\mathcal{V}^*}:=\left(x_{v_{11}}, x_{v_{22}}, \ldots, x_{v_{nn}}\right).
    $$
Then for the reduced expression
    $$\mathbf{i}'_0 = (i'_1, \ldots, i'_N) := (n-1, n-2, \ldots, 1, n-1, n-2, \ldots, 2, \ldots, n-1, n-2, n-1)
    $$
we have a map $\theta_{\mathbf{i}'_0}:\mathcal{M} \to Z$ given by
    $$\left(\boldsymbol{x}_{\mathcal{V}^*} , \boldsymbol{r}_{\mathcal{A}_{\mathrm{v}}} \right)
    \mapsto
    \Phi\left(\kappa_{\vert_{\mathcal{V}^*}}\left(\boldsymbol{x}_{\mathcal{V}^*}\right) , \mathbf{x}^{\vee}_{\mathbf{i}'_0}\left(\boldsymbol{r}_{\mathcal{A}_{\mathrm{v}}}\right)\right)
    = \mathbf{x}^{\vee}_{\mathbf{i}'_0}\left(\boldsymbol{r}_{\mathcal{A}_{\mathrm{v}}}\right) \kappa_{\vert_{\mathcal{V}^*}}\left(\boldsymbol{x}_{\mathcal{V}^*}\right) \bar{w}_0 u_2
    $$
where $\kappa_{\vert_{\mathcal{V}^*}}$ is the restriction of $\tilde{\kappa}$ to the star vertex coordinates, $u_2 \in U^{\vee}$ is the unique element such that $\theta_{\mathbf{i}'_0}\left(\boldsymbol{r}_{\mathcal{A}_{\mathrm{v}}}\right) \in Z$ and $\mathbf{x}^{\vee}_{\mathbf{i}}$, $\Phi$ are the maps from Section \ref{subsec The string coordinates}.

Next we recall the following definitions, also from Section \ref{subsec The string coordinates}:
    $$b:=u_1d\bar{w}_0 u_2 \in Z, \quad u_1:=\iota(\eta^{w_0,e}(u)), \quad u:=\mathbf{x}^{\vee}_{-\mathbf{i}_0}(\boldsymbol{z}),
    $$
together with the fact that by Lemmas \ref{lem form of u_1 and b factorisations} and \ref{lem coord change for u_1 p_i in terms of z_i} we may factorise $u_1$ as
    $$
    u_1 = \mathbf{x}_{i'_1}^{\vee}(p_1) \cdots \mathbf{x}_{i'_N}^{\vee}(p_N) \\
    $$
where, for $k=1,\ldots, n-1$, $a=1,\ldots, n-k$, we have
    $$p_{s_k+a} =\begin{cases}
        z_{1+s_{a}} & \text{if } k=1 \\
        \frac{z_{k+s_{a}}}{z_{k-1+s_{a+1}}} & \text{otherwise}
        \end{cases}
    \qquad \text{where} \ \
    s_k := \sum_{j=1}^{k-1}(n-j).
    $$
Thus taking $\boldsymbol{r}_{\mathcal{A}_{\mathrm{v}}}=(p_1, \ldots, p_N)$ we see that $\mathbf{x}^{\vee}_{\mathbf{i}'_0}\left(\boldsymbol{r}_{\mathcal{A}_{\mathrm{v}}}\right)=u_1$ and letting the star vertex coordinates be given by
    $$x_{v_{ii}}=d_i, \quad i=1,\ldots, n
    $$
we obtain $\kappa_{\vert_{\mathcal{V}^*}}\left(\boldsymbol{x}_{\mathcal{V}^*}\right) = d$. This quiver decoration in dimension $4$ is given in Figure \ref{fig Quiver decoration, n=4, z coordinates}.

Due to the relations in the quiver, namely (\ref{eqn relation in quiver from stars}) and the box relations, this decoration extends to all the vertices and arrows of the quiver. In particular, we can extend $\theta_{\mathbf{i}'_0}$ to a map
    $$\bar{\theta}_{\mathbf{i}'_0}: (\mathbb{K}^*)^{\mathcal{V}} \times \bar{\mathcal{M}} \to Z
    $$
given by first taking the projection onto $\mathcal{M}$ and then applying $\theta_{\mathbf{i}'_0}$. Thus decorating the quiver in this way and then applying the map $\bar{\theta}_{\mathbf{i}'_0}$ (or $\theta_{\mathbf{i}'_0}$) gives the string toric chart factored through $(\mathbb{K}^*)^{\mathcal{V}} \times \bar{\mathcal{M}}$ (respectively $\mathcal{M}$), that is
    $$
    T^{\vee} \times (\mathbb{K}^*)^N \to (\mathbb{K}^*)^{\mathcal{V}} \times \bar{\mathcal{M}} \to Z.
    $$

We summarise the results from this section:
\begin{lem}[{\cite[Theorem 9.2 and Lemma 9.3]{Rietsch2008}}] \label{lem factor three maps through quiver torus}
With the above notation, we have the following:
    $$\mathcal{W} \circ \bar{\theta}_{\mathbf{i}'_0} = \mathcal{F} \circ \mathrm{pr}, \quad \mathrm{hw} \circ \bar{\theta}_{\mathbf{i}'_0} = \kappa \circ \mathrm{pr}, \quad \mathrm{wt} \circ \bar{\theta}_{\mathbf{i}'_0} = \gamma \circ \mathrm{pr}.
    $$
where $\mathrm{pr}$ is the projection of $(\mathbb{K}^*)^{\mathcal{V}} \times \bar{\mathcal{M}}$ onto the first factor.
\end{lem}

\begin{figure}%[hb]
\centering
\begin{minipage}[b]{0.47\textwidth}
    \centering
\begin{tikzpicture}[scale=0.85]
    %dots and stars
    \node (41) at (0,0) {$\bullet$};
    \node (31) at (0,2) {$\bullet$};
    \node (21) at (0,4) {$\bullet$};
    \node (11) at (0,6) {$\boldsymbol{*}$};
        \node at (0.3,6.3) {$d_{1}$};
    \node (42) at (2,0) {$\bullet$};
    \node (32) at (2,2) {$\bullet$};
    \node (22) at (2,4) {$\boldsymbol{*}$};
        \node at (2.3,4.3) {$d_{2}$};
    \node (43) at (4,0) {$\bullet$};
    \node (33) at (4,2) {$\boldsymbol{*}$};
        \node at (4.3,2.3) {$d_{3}$};
    \node (44) at (6,0) {$\boldsymbol{*}$};
        \node at (6.3,0.3) {$d_{4}$};

    %arrows
    \draw[->] (41) -- (31);
    \draw[->] (31) -- (21);
    \draw[->] (21) -- (11);
    \draw[->] (42) -- (32);
    \draw[->] (32) -- (22);
    \draw[->] (43) -- (33);

    \draw[->] (22) -- (21);
    \draw[->] (33) -- (32);
    \draw[->] (32) -- (31);
    \draw[->] (44) -- (43);
    \draw[->] (43) -- (42);
    \draw[->] (42) -- (41);

    %arrow labels
    \node at (-0.3,4.8) {\scriptsize{$z_6$}};
    \node at (-0.3,2.8) {\scriptsize{$z_4$}};
    \node at (-0.3,0.8) {\scriptsize{$z_1$}};

    \node at (1.7,2.8) {\scriptsize{$\frac{z_5}{z_6}$}};
    \node at (1.7,0.8) {\scriptsize{$\frac{z_2}{z_4}$}};
    \node at (3.7,0.8) {\scriptsize{$\frac{z_3}{z_5}$}};

    %arrow labels
    \node at (1,3.6) {\scriptsize{$\frac{d_1}{d_2} \frac{1}{z_6}$}};
    \node at (1,1.6) {\scriptsize{$\frac{d_1}{d_2} \frac{z_5}{z_4z_6^2}$}};
    \node at (1,-0.4) {\scriptsize{$\frac{d_1}{d_2} \frac{z_2z_5}{z_1z_4^2z_6^2}$}};

    \node at (3,1.6) {\scriptsize{$\frac{d_2}{d_3} \frac{z_6}{z_5}$}};
    \node at (3,-0.4) {\scriptsize{$\frac{d_2}{d_3} \frac{z_3z_4z_6}{z_2z_5^2}$}};
    \node at (5,-0.4) {\scriptsize{$\frac{d_3}{d_4} \frac{z_5}{z_3}$}};
\end{tikzpicture}
\caption{Quiver decoration when $n=4$, \\ string coordinates} \label{fig Quiver decoration, n=4, z coordinates}
\end{minipage}\hfill
\begin{minipage}[b]{0.47\textwidth}
    \centering
\begin{tikzpicture}[scale=0.85]
    %dots and stars
    \node (41) at (0,0) {$\bullet$};
    \node (31) at (0,2) {$\bullet$};
    \node (21) at (0,4) {$\bullet$};
    \node (11) at (0,6) {$\boldsymbol{*}$};
        \node at (0.3,6.3) {$d_1$};
    \node (42) at (2,0) {$\bullet$};
    \node (32) at (2,2) {$\bullet$};
    \node (22) at (2,4) {$\boldsymbol{*}$};
        \node at (2.3,4.3) {$d_2$};
    \node (43) at (4,0) {$\bullet$};
    \node (33) at (4,2) {$\boldsymbol{*}$};
        \node at (4.3,2.3) {$d_3$};
    \node (44) at (6,0) {$\boldsymbol{*}$};
        \node at (6.3,0.3) {$d_4$};

    %arrows
    \draw[->] (41) -- (31);
    \draw[->] (31) -- (21);
    \draw[->] (21) -- (11);
    \draw[->] (42) -- (32);
    \draw[->] (32) -- (22);
    \draw[->] (43) -- (33);

    \draw[->] (22) -- (21);
    \draw[->] (33) -- (32);
    \draw[->] (32) -- (31);
    \draw[->] (44) -- (43);
    \draw[->] (43) -- (42);
    \draw[->] (42) -- (41);

    %arrow labels
    \node at (-0.3,0.8) {\scriptsize{$m_3$}};
    \node at (-0.3,2.8) {\scriptsize{$m_2$}};
    \node at (-0.3,4.8) {\scriptsize{$m_1$}};

    \node at (1.5,0.8) {\scriptsize{$\frac{m_3m_5}{m_2}$}};
    \node at (1.5,2.8) {\scriptsize{$\frac{m_2m_4}{m_1}$}};
    \node at (3.3,0.8) {\scriptsize{$\frac{m_3m_5m_6}{m_2m_4}$}};

    \node at (1,-0.4) {\scriptsize{$\frac{d_1}{d_2} \frac{m_4m_5}{{m_1}^2m_2}$}};
    \node at (3,-0.4) {\scriptsize{$\frac{d_2}{d_3} \frac{m_1m_6}{m_2{m_4}^2}$}};
    \node at (5,-0.4) {\scriptsize{$\frac{d_3}{d_4} \frac{m_2m_4}{m_3m_5m_6}$}};
    \node at (1,1.6) {\scriptsize{$\frac{d_1}{d_2} \frac{m_4}{{m_1}^2}$}};
    \node at (3,1.6) {\scriptsize{$\frac{d_2}{d_3} \frac{m_1}{m_2m_4}$}};
    \node at (1,3.6) {\scriptsize{$\frac{d_1}{d_2} \frac{1}{m_1}$}};
\end{tikzpicture}
\caption{Quiver decoration when $n=4$, \\ ideal coordinates} \label{fig Quiver decoration, n=4, m coordinates}
\end{minipage}
\end{figure}

To decorate the quiver with the ideal coordinates, we begin by recalling Theorem \ref{thm coord change}, namely that for $k=1,\ldots, n-1$, $a=1,\ldots, n-k$ we have
    $$m_{s_k+a} = \begin{cases}
            z_{1+s_{n-a}} &\text{if } k=1, \\
            \frac{z_{k+s_{n-k-a+1}}}{z_{k-1+s_{n-k-a+1}}} &\text{otherwise}.
            \end{cases}
    $$

In the above quiver decoration we have $\boldsymbol{r}_{\mathcal{A}}=(p_1, \ldots, p_N)$, that is
    $$r_{a_{ij}}=p_{s_j+n-i}=\begin{cases}
        z_{1+s_{n-i}} & \text{if } j=1, \\
        \frac{z_{j+s_{n-i}}}{z_{j-1+s_{n-i+1}}} & \text{otherwise}.
    \end{cases}
    $$
We notice that if $j=1$ then
    $$r_{a_{i1}}=p_{s_1+n-i}=z_{1+s_{n-i}}=m_{s_1+i}.
    $$
We also want to write $r_{a_{ij}}$ in terms of the $m_{s_k+a}$ coordinates when $j>2$. To do this we first note that for $k=2$ we have
    $$m_{s_k+a}m_{s_{k-1}+a+1} = \frac{z_{k+s_{n-k-a+1}}}{z_{k-1+s_{n-k-a+1}}}z_{k-1+s_{n-(k-1)-(a+1)+1}}
    = z_{k+s_{n-k-a+1}}.
    $$
Similarly for $j>2$ we have
    $$m_{s_k+a}m_{s_{k-1}+a+1}
        = \frac{z_{k+s_{n-k-a+1}}}{z_{k-1+s_{n-k-a+1}}}\frac{z_{k-1+s_{n-k-a+1}}}{z_{k-2+s_{n-k-a+1}}} \\
        = \frac{z_{k+s_{n-k-a+1}}}{z_{k-2+s_{n-k-a+1}}}.
    $$
Thus the following product is telescopic:
    $$\prod_{r=0}^{k-1} m_{s_{k-r}+a+r} = z_{k+s_{n-k-a+1}}.
    $$
Now taking quotients of these products allows us to write the vertical arrow coordinates, $r_{a_{ij}}=p_{s_j+n-i}$, in terms of the $m_j$; for $j>2$ we have
    $$r_{a_{ij}}=p_{s_j+n-i}
        =\frac{z_{j+s_{n-i}}}{z_{j-1+s_{n-i+1}}}
        =\frac{z_{j+s_{n-j-(i-j+1)+1}}}{z_{j-1+s_{n-(j-1)-(i-j+1)+1}}}
        =\frac{\prod_{r=0}^{j-1} m_{s_{j-r}+i-j+1+r}}{\prod_{r=0}^{j-2} m_{s_{j-1-r}+i-j+1+r}}.
    $$
This doesn't seem particularly helpful at first glance, however it leads to an iterative description of the new quiver decoration which will be very useful in the proof of the main result in Section \ref{sec Givental-type quivers and critical points}, Proposition \ref{prop crit points, sum at vertex is nu_i}. Denoting the numerator of $r_{a_{ij}}$ by $n(r_{a_{ij}})$, for $j=1, \ldots, n-1$ we have
    \begin{equation} \label{eqn rai,j+1 in terms of m's}
    r_{a_{i,j+1}}
        =\frac{\prod_{r=0}^{j} m_{s_{j+1-r}+i-j+r}}{\prod_{r=0}^{j-1} m_{s_{j-r}+i-j+r}}
        =m_{s_{j+1}+i-j}\frac{\prod_{r=0}^{j-1} m_{s_{j-r}+i+1-j+r}}{\prod_{r=0}^{j-1} m_{s_{j-r}+i-j+r}}
        =m_{s_{j+1}+i-j}\frac{n(r_{a_{ij}})}{n(r_{a_{i-1,j}})}. \\
    \end{equation}
For example, in dimension $4$ the quiver decoration is given in Figure \ref{fig Quiver decoration, n=4, m coordinates}.

\subsection{Critical points of the superpotential} \label{subsec Critical points on the superpotential}
\fancyhead[L]{3.3 \ \ Critical points of the superpotential}
% \fancyhead[L]{4.3 \ \ Critical points of the superpotential}

We begin by recalling the highest weight map on the vertex torus:
    $$\kappa :(\mathbb{K}^*)^{\mathcal{V}} \to T^{\vee}, \quad \boldsymbol{x}_{\mathcal{V}} \mapsto \left(x_{v_{ii}}\right)_{i=1, \ldots, n}.
    $$
Now in the fibre over $d\in T^{\vee}$ we have
    $$x_{v_{ii}}=d_i \quad \text{for} \ i=1,\ldots, n.
    $$
The remaining $x_v$ for $v \in \mathcal{V}^{\bullet}$ form a coordinate system on this fibre. In particular we can use these coordinates to compute critical points of the superpotential, as follows:
    $$x_v \frac{\partial \mathcal{F}}{\partial x_v} = \sum_{a:h(a)=v} \frac{x_{h(a)}}{x_{t(a)}} - \sum_{a:t(a)=v} \frac{x_{h(a)}}{x_{t(a)}}.
    $$
Thus the critical point conditions are
    $$\sum_{a:h(a)=v} \frac{x_{h(a)}}{x_{t(a)}} = \sum_{a:t(a)=v} \frac{x_{h(a)}}{x_{t(a)}} \quad \text{for} \ v \in \mathcal{V}^{\bullet}.
    $$
Since we favour working with arrow coordinates, we rewrite these equations:
    \begin{equation}\label{eqn crit pt conditions} \sum_{a:h(a)=v} r_a = \sum_{a:t(a)=v} r_a \quad \text{for} \ v \in \mathcal{V}^{\bullet}.
    \end{equation}
This means we can now use the arrow coordinates in the quiver, and thus the ideal coordinates, to give simple descriptions of both the superpotential and the defining equations of its critical points. In fact, using the quiver decoration in terms of the ideal coordinates, we can take this a step further:

\begin{prop} \label{prop crit points, sum at vertex is nu_i}
If the critical point conditions hold at every dot vertex $v \in \mathcal{V}^{\bullet}$, then the sum of the outgoing arrow coordinates at each dot vertex $v_{ik}$ is given in terms of the ideal coordinates by
    $$ \varpi(v_{ik}) := \sum_{a:t(a)=v_{ik}} r_a = m_{s_k +i-k}.
    $$
\end{prop}

\begin{proof}
By construction of the quiver labelling there is only one outgoing arrow at each $v_{i1}$ for $i=2, \ldots, n$, namely $a_{i-1,1}$, and indeed $r_{a_{i-1,1}}= m_{i-1} = m_{s_1+i-1}$. Since we have the desired property for the vertices $v_{i1}$, $i=2, \ldots, n-1$, we proceed by an inductive argument increasing both vertex subscripts simultaneously.

We consider the subquiver given in Figure \ref{fig Critical point conditions in the quiver} and suppose the sum of outgoing arrows at $v_{ik}$ is $m_{s_k +i-k}$. Then by the critical point condition at this vertex we have $m_{s_k +i-k} = r_{a_{ik}}+r_{b_{i,k+1}}$.
\begin{figure}[ht]
\centering
\begin{tikzpicture}[scale=0.9]
    %dots and stars
    \node (31) at (0,0) {$\bullet$};
    \node (21) at (0,1.5) {$\bullet$};
    \node (32) at (1.5,0) {$\bullet$};
    \node (22) at (1.5,1.5) {$\bullet$};
    \node (11) at (0,3) {$\bullet$};
    \node (20) at (-1.5,1.5) {$\bullet$};

    %arrows
    \draw[->] (21) -- (11);
    \draw[->] (31) -- (21);
    \draw[->] (32) -- (22);

    \draw[->] (21) -- (20);
    \draw[->] (22) -- (21);
    \draw[->] (32) -- (31);

    %arrow labels
    \node at (-0.48,2.25) {\scriptsize{$a_{i-1,k}$}};
    \node at (-0.3,0.75) {\scriptsize{$a_{ik}$}};
    \node at (2,0.75) {\scriptsize{$a_{i,k+1}$}};

    \node at (0.75,-0.3) {\scriptsize{$b_{i+1,k+1}$}};
    \node at (0.75,1.7) {\scriptsize{$b_{i,k+1}$}};
    \node at (-0.75,1.7) {\scriptsize{$b_{ik}$}};

    \node at (-0.55,0) {\scriptsize{$v_{i+1,k}$}};
    \node at (0.35,1.25) {\scriptsize{$v_{ik}$}};
    \node at (2.19,0) {\scriptsize{$v_{i+1,k+1}$}};
    \node at (2.07,1.5) {\scriptsize{$v_{i,k+1}$}};
\end{tikzpicture}
\caption{Critical point conditions in the quiver} \label{fig Critical point conditions in the quiver}
\end{figure}
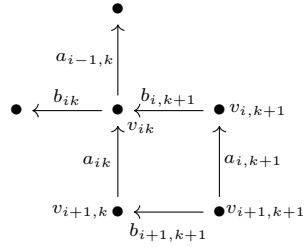

By the definition of the arrow coordinates we have
    $$r_{a_{i-1,k}} = m_{s_k +i-k} \frac{n(r_{a_{i-1,k-1}})}{n(r_{a_{i-2,k-1}})},  \quad
    r_{a_{i,k+1}} = m_{s_{k+1} +i-k} \frac{n(r_{a_{ik}})}{n(r_{a_{i-1,k}})}, \quad
    r_{a_{ik}} = m_{s_k +i+1-k} \frac{n(r_{a_{i,k-1}})}{n(r_{a_{i-1,k-1}})}.
    $$
In particular we see that
    \begin{equation} \label{eqn quotient of arrow coords, crit point in quiver}
    \frac{r_{a_{i,k+1}}}{r_{a_{ik}}} = m_{s_{k+1} +i-k} \frac{n(r_{a_{i-1,k-1}})}{n(r_{a_{i-1,k}})}.
    \end{equation}
The sum of outgoing arrows at the vertex $v_{i+1,k+1}$ is given by
    $$\begin{aligned}
    \varpi(v_{i+1,k+1}) :=\sum_{a:t(a)=v_{i+1,k+1}} r_a &=  r_{a_{i,k+1}} + r_{b_{i+1,k+1}} &\\
        &= r_{a_{i,k+1}} + \frac{r_{a_{i,k+1}}r_{b_{i,k+1}}}{r_{a_{ik}}} & \begin{aligned}[t]\text{by the box relation} \\ r_{a_{ik}}r_{b_{i+1,k+1}} = r_{a_{i,k+1}}r_{b_{i,k+1}} \end{aligned} \\
        &= r_{a_{i,k+1}} + \frac{r_{a_{i,k+1}}}{r_{a_{ik}}}(m_{s_k +i-k} - r_{a_{ik}}) & \text{by the inductive hypothesis}\\
        &= \frac{r_{a_{i,k+1}}}{r_{a_{ik}}}m_{s_k +i-k} & \\
        &= \frac{m_{s_{k+1} +i-k} n(r_{a_{i-1,k-1}})}{n(r_{a_{i-1,k}})} m_{s_k +i-k} & \text{using} \ (\ref{eqn quotient of arrow coords, crit point in quiver}) \\
        &= m_{s_{k+1} +i-k} & \text{by definition of} \ n(r_{a_{i-1,k}}) \\
        &= m_{s_{k+1}+(i+1)-(k+1)}.
    \end{aligned}$$
\end{proof}

\begin{rem}
At first glance, in the above theorem we seem to have lost the information about the highest weight element $d \in T^{\vee}$. However at a critical point this information can be partially recovered from the $\boldsymbol{m}$-coordinates; considering the dot vertices on the bottom wall of the quiver, for each $j=1, \ldots, n-1$ we have
    $$\begin{aligned}
    r_{b_{n, j+1}} &= \varpi(v_{nj}) & \text{by the critical point conditions} \\
        &= m_{s_j+n-j} & \text{by Proposition \ref{prop crit points, sum at vertex is nu_i}}.
    \end{aligned}
    $$
By the quiver decoration, each arrow coordinate $r_{b_{n, j+1}}$ is given by
    $$r_{b_{n, j+1}} = \frac{d_j}{d_{j+1}} \frac{\prod_{i \in I_{j_1}} m_i}{\prod_{i \in I_{j_2}} m_i}
    $$
for some multisets $I_{j_1}, I_{j_2}$ of the integers $1, \ldots, N$. Thus we have
    $$\alpha^{\vee}_j(d)=\frac{d_j}{d_{j+1}} = m_{s_j+n-j} \frac{\prod_{i \in I_{j_2}} m_i}{\prod_{i \in I_{j_1}} m_i}, \quad j=1, \ldots, n-1.
    $$
\end{rem}

We complete this section by tying together the quiver, critical points and our interest in the form of the weight matrix from Sections \ref{subsec The form of the weight matrix} and \ref{sec The ideal coordinates}. Namely it is natural to ask what happens to the weight matrix at critical points.

\begin{prop} \label{prop weight at crit point non trop}
At a critical point in the fibre over $d\in T^{\vee}$, the weight matrix is an $n\times n$ matrix $\mathrm{diag}(c, \ldots, c)$ where
    $$c^n = \prod_{i=1}^n d_i.
    $$
\end{prop}

In order to prove this we require the following lemma:
\begin{lem}[{\cite[Lemma 5.9]{Judd2018}}] \label{lem Jamie's lemma 5.9}
Suppose we have a quiver like the one given in Figure \ref{fig Diagonal subquiver}. We attach a variable $r_a$ to each arrow such that the box relations $r_{a_1}r_{a_2}=r_{a_3}r_{a_4}$ hold whenever $a_1, a_2, a_3, a_4$ form a square (see Figure \ref{fig Arrows forming a box}) and the critical point conditions hold at each black vertex. For each box $B_j$, $1\leq j \leq t$, let $O_j = r_{a_1}r_{a_3}$ and $I_j=r_{a_2}r_{a_4}$. Additionally let $K_t=\prod r_a$ where the product is over a (any) path from $v_t$ to $v_0$. Then we have
    $$\prod_{j=1}^t O_j \frac{r_{a_{\mathrm{out}}}}{r_{a_{\mathrm{in}}}} = K_t \quad \text{and} \quad \prod_{j=1}^t I_j \frac{r_{a_{\mathrm{in}}}}{r_{a_{\mathrm{out}}}} = K_t .
    $$
Note this agrees with $\prod_{j=1}^t O_j \prod_{j=1}^t I_j = K_t^2$.
\end{lem}

\begin{figure}[ht]
\centering
\begin{minipage}[b]{0.47\textwidth}
    \centering
    \begin{tikzpicture}[scale=0.9]
        %dots and stars
        \node (11) at (0,6) {$\boldsymbol{\circ}$};
            \node at (0.5,6) {$v_{-1}$};
        \node (22) at (1,5) {$\bullet$};
            \node at (0.6,4.9) {$v_0$};
        \node (33) at (2,4) {$\bullet$};
        \node (44) at (3,3) {$\bullet$};
            \node at (3.5,2.6) {$\ddots$};
        \node (55) at (4,2) {$\bullet$};
        \node (66) at (5,1) {$\bullet$};
            \node at (4.9,0.7) {$v_t$};
        \node (77) at (6,0) {$\boldsymbol{\circ}$};
            \node at (6.5,0) {$v_{t+1}$};

        \node (32) at (1,4) {$\boldsymbol{\circ}$};
        \node (43) at (2,3) {$\boldsymbol{\circ}$};
        \node (65) at (4,1) {$\boldsymbol{\circ}$};

        \node (23) at (2,5) {$\boldsymbol{\circ}$};
        \node (34) at (3,4) {$\boldsymbol{\circ}$};
        \node (56) at (5,2) {$\boldsymbol{\circ}$};

        %arrows
        \draw[->] (22) -- (11);
        \draw[->] (77) -- (66);

        \draw[->] (33) -- (32);
        \draw[->] (32) -- (22);
        \draw[->] (33) -- (23);
        \draw[->] (23) -- (22);

        \draw[->] (44) -- (43);
        \draw[->] (43) -- (33);
        \draw[->] (44) -- (34);
        \draw[->] (34) -- (33);

        \draw[->] (66) -- (65);
        \draw[->] (65) -- (55);
        \draw[->] (66) -- (56);
        \draw[->] (56) -- (55);

        %box labels
        \node at (1.5,4.5) {$B_1$};
        \node at (2.5,3.5) {$B_2$};
        \node at (4.5,1.5) {$B_t$};

        %arrow labels
        \node at (1,5.55) {$a_{\mathrm{out}}$};
        \node at (5.9,0.55) {$a_{\mathrm{in}}$};
    \end{tikzpicture}
\caption{Diagonal subquiver} \label{fig Diagonal subquiver}
\end{minipage} \hfill
\begin{minipage}[b]{0.47\textwidth}
    \centering
    \begin{tikzpicture}
        %dots and stars
        \node (41) at (0,0) {};
        \node (31) at (0,1) {};
        \node (42) at (1,0) {};
        \node (32) at (1,1) {};

        %arrows
        \draw[->] (41) -- (31);
        \draw[->] (42) -- (32);

        \draw[->] (32) -- (31);
        \draw[->] (42) -- (41);

        %arrow labels
        \node at (-0.4,0.5) {$a_2$};
        \node at (1.4,0.5) {$a_3$};

        \node at (0.5,1.2) {$a_4$};
        \node at (0.5,-0.3) {$a_1$};
    \end{tikzpicture}
\caption{Arrows forming a box} \label{fig Arrows forming a box}
\end{minipage}
\end{figure}

\begin{proof}[Proof of Proposition {\ref{prop weight at crit point non trop}}]
We consider the set of all arrows with either head or tail on the $i$th diagonal $\mathcal{D}_i$. These arrows form a subquiver like in Figure \ref{fig Diagonal subquiver}, with $v_t=v_{n, n-i+1}$. Moreover we have
    $$K_t=\frac{x_{v_0}}{x_{v_t}}, \quad r_{a_{\mathrm{in}}} = \frac{x_{v_t}}{x_{v_{t+1}}} , \quad r_{a_{\mathrm{out}}} = \frac{x_{v_{-1}}}{x_{v_0}}, \quad \text{so} \quad \frac{1}{K_t} \frac{r_{a_{\mathrm{out}}}}{r_{a_{\mathrm{in}}}} = \frac{x_{v_{-1}}x_{v_{t+1}}}{x_{v_0}^2}.
    $$
Thus, using Lemma \ref{lem Jamie's lemma 5.9}, we see that
    $$\frac{t_{i-1}}{t_i}=\frac{\Xi_{i+1}\Xi_{i-1}}{\Xi_i^2} = \frac{x_{v_{-1}}x_{v_{t+1}}}{x_{v_0}^2} \prod O_j = \frac{1}{K_t} \prod_{j=1}^t O_j \frac{r_{a_{\mathrm{out}}}}{r_{a_{\mathrm{in}}}} = 1.
    $$

So at a critical point, the weight matrix $\mathrm{diag}(t_1, t_2, \ldots, t_n)$ is given by $\mathrm{diag}(c,c, \ldots, c)$ for some $c$. By taking the determinant we obtain
    $$\prod_{i=1}^n t_i = \prod_{i=1}^n \frac{\Xi_i}{\Xi_{i+1}} = \frac{\Xi_1}{\Xi_{n+1}} = \Xi_1
    $$
recalling $\Xi_{n+1} =1$ by definition. This gives the desired value of $c$ as follows:
    $$c^n = \prod_{i=1}^n t_i = \Xi_1 = \prod_{i=1}^n d_i.
    $$
\end{proof}

\section{The tropical viewpoint} \label{sec The tropical viewpoint}
\fancyhead[L]{4 \ \ The tropical viewpoint}
% \fancyhead[L]{5 \ \ The tropical viewpoint}

In this section we recall how, by tropicalisation, we can use the superpotential to obtain polytopes associated to a given highest weight. These polytopes depend on the choice of positive toric chart. The goal in this section is to describe the polytope we get from the ideal coordinates. Additionally we show that for each choice of highest weight, the associated critical point of the superpotential gives rise to a point inside this polytope, which is Judd's tropical critical point \cite{Judd2018} and has a beautiful description in terms of so called ideal fillings.

\subsection{The basics of tropicalisation} \label{subsec The basics of tropicalisation}
\fancyhead[L]{4.1 \ \ The basics of tropicalisation}
% \fancyhead[L]{5.1 \ \ The basics of tropicalisation}

In this section we explain the concept of tropicalisation, following an original construction due to Lusztig \cite{Lusztig1994}. In order to do this we work over the field of Generalised Puiseux series, which we will denote by $\mathbf{K}$.

A generalised Puiseux series in a variable $t$ is a series with an exponent set $(\mu_k) = (\mu_0, \mu_1, \mu_2, \ldots ) \in \mathbb{R}$ which is strictly monotone and either finite, or countable and tending to infinity. That is,
    $$(\mu_k) \in \mathrm{MonSeq} = \left\{ A \subset \mathbb{R} \ | \ \mathrm{Cardinality}(A\cap \mathbb{R}_{\leq x}) < \infty \ \text{for arbitrarily large} \ x \in \mathbb{R} \right\} .
    $$
Thus we have
    $$\mathbf{K} = \left\{ c(t) = \sum_{(\mu_k) \in \mathrm{MonSeq}} c_{\mu_k}t^{\mu_k} \ | \ c_{\mu_k} \in \mathbb{C} \right\}.
    $$
The positive part of the field $\mathbf{K}$ is given by
    $$\mathbf{K}_{>0}:= \left\{ c(t)\in \mathbf{K} \ | \ c_{\mu_0} \in \mathbb{R}_{>0}\right\}
    $$
where we may assume that the lowest order term has a non-zero coefficient.

Given a torus $\mathcal{T}$, we note that we may identify $\mathcal{T}(\mathbf{K})=\mathrm{Hom}(M, \mathbf{K}^*)$, where we write $M$ for the character group of $\mathcal{T}$, $X^*(\mathcal{T})$, viewed as an abstract group and written additively. The positive part of $\mathcal{T}(\mathbf{K})$ is defined by those homomorphisms which take values in $\mathbf{K}_{>0}$, namely $\mathcal{T}(\mathbf{K}_{>0})=\mathrm{Hom}(M, \mathbf{K}_{>0})$. For $v \in M$ and $h \in \mathcal{T}(\mathbf{K})$ we will write $\chi^v(h)$ for the associated evaluation $h(v)$ in $\mathbf{K}^*$. We call $\chi^v$ the character associated to $v$.

We call a $\mathbf{K}$-linear combination of the characters $\chi^v$, a Laurent polynomial on $\mathcal{T}$. In addition, a Laurent polynomial is said to be positive if the coefficients of the characters lie in $\mathbf{K}_{>0}$. Now let $\mathcal{T}^{(1)}$, $\mathcal{T}^{(2)}$ be two tori over $\mathbf{K}$. We say that a rational map
    $$\psi: \mathcal{T}^{(1)} \dashrightarrow \mathcal{T}^{(2)}
    $$
is a positive rational map if, for any character $\chi$ of $\mathcal{T}^{(2)}$, the composition $\chi \circ \psi : \mathcal{T}^{(1)} \to \mathbf{K}$ is given by a quotient of positive Laurent polynomials on $\mathcal{T}^{(1)}$.

We now define the tropicalisation of these positive rational maps. Roughly speaking, it captures what happens to the leading term exponents. In order to define tropicalisation we use the natural valuation on $\mathbf{K}$ given by
    $$\mathrm{Val}_{\mathbf{K}} : \mathbf{K} \to \mathbb{R} \cup \{\infty\}, \quad
    \mathrm{Val}_{\mathbf{K}}\left(c(t)\right) = \begin{cases}
    \mu_0 & \text{if} \ c(t)=\sum_{(\mu_k) \in \mathrm{MonSeq}} c_{\mu_k}t^{\mu_k} \neq 0, \\
    \infty & \text{if} \ c(t)=0.
    \end{cases}
    $$

We define an equivalence relation $\sim$ on $\mathcal{T}(\mathbf{K}_{>0})$ using this valuation: we say $h \sim h'$ if and only if $\mathrm{Val}_{\mathbf{K}}(\chi(h))=\mathrm{Val}_{\mathbf{K}}(\chi(h'))$ for all characters $\chi$ of $\mathcal{T}$. Then the tropicalisation of the torus $\mathcal{T}$ is defined to be
    $$\mathrm{Trop}(\mathcal{T}):= T(\mathbf{K}_{>0})/\sim.
    $$
This set inherits the structure of an abelian group from the group structure of $\mathcal{\mathcal{T}}(\mathbf{K}_{>0})$, we denote this as addition.

In practical terms, when $\mathcal{T}=(\mathbf{K}^*)^r$, so that $\mathcal{T}(\mathbf{K}_{>0})=(\mathbf{K}_{>0})^r$, the valuation $\mathrm{Val}_{\mathbf{K}}$ on each coordinate gives an identification
    $$\mathrm{Trop}(\mathcal{T}) \to \mathbb{R}^r, \quad
    \left[ \left( c_1(t), \ldots, c_r(t) \right) \right] \mapsto \left( \mathrm{Val}_{\mathbf{K}}\left( c_1(t) \right), \ldots, \mathrm{Val}_{\mathbf{K}}\left( c_r(t) \right) \right).
    $$
To state this in a coordinate-free way, if $\mathcal{T}$ is a torus with cocharacter lattice $N:=X_*(\mathcal{T})$, then $\mathrm{Trop}(\mathcal{T})$ is identified with $N_{\mathbb{R}}=N\otimes \mathbb{R}$, see for example \cite{JuddRietsch2019}. We note that $N_{\mathbb{R}}$ is also identified with the Lie algebra of the torus taken over $\mathbb{R}$, for example $\mathrm{Trop}(T^{\vee})=\mathfrak{h}^*_{\mathbb{R}}$ (c.f. Section \ref{subsec Notation and definitions}).

We make the convention that if the coordinates of our torus $\mathcal{T}\cong(\mathbf{K}^*)^r$ are labelled by Roman letters, then the corresponding coordinates on $\mathrm{Trop}(\mathcal{T}) \cong \mathbb{R}^r$ are labelled by the associated Greek letters. In addition, by $(\mathbf{K}^*)^r_{\boldsymbol{b}}$ we mean $(\mathbf{K}^*)^r$ with coordinates $(b_1, \ldots, b_r)$, and similarly for $\mathbb{R}^N_{\boldsymbol{\zeta}}$, etc.

Suppose that $\mathcal{T}^{(1)}$, $\mathcal{T}^{(2)}$ are two tori over $\mathbf{K}$ and $\psi: \mathcal{T}^{(1)} \dashrightarrow \mathcal{T}^{(2)}$ is a positive rational map. The map
    $$\psi(\mathbf{K}_{>0}) : \mathcal{T}^{(1)}(\mathbf{K}_{>0}) \to \mathcal{T}^{(2)}(\mathbf{K}_{>0})
    $$
is well-defined and compatible with the equivalence relation $\sim$ (using the positivity of the leading terms). The tropicalisation $\mathrm{Trop}(\psi)$ is then defined to be the resulting map
    $$\mathrm{Trop}(\psi) : \mathrm{Trop}(\mathcal{T}^{(1)}) \to \mathrm{Trop}(\mathcal{T}^{(2)})
    $$
between equivalence classes. It is piecewise-linear with respect to the linear structures on the $\mathrm{Trop}(\mathcal{T}^{(i)})$.

In the case of a variety $X$ with a `positive atlas' consisting of torus charts related by positive birational maps (see \cite{FockGoncharov2006}, \cite{BerensteinKazhdan2007}), there is a well-defined positive part $X(\mathbf{K}_{>0})$ and tropical version $\mathrm{Trop}(X)$, which comes with a tropical atlas whose tropical charts $\mathrm{Trop}(X)\to \mathbb{R}^r$ are related by piecewise-linear maps. $\mathrm{Trop}(X)$ in this more general setting is a space with a piecewise-linear structure.

\begin{ex}
Let $\mathcal{T}=(\mathbf{K}^*)^2_{\boldsymbol{b}}$ and consider the following map\footnote{
This is the superpotential for $\mathbb{CP}^2$ (see \cite{EguchiHoriXiong1997}).
}:
    $$\psi:\mathcal{T} \to \mathbf{K}, \quad \psi(b_1,b_2) = b_1 + b_2 + \frac{t^3}{b_1b_2}.
    $$
We may consider $\psi$ as a positive birational map $\mathcal{T} \dashrightarrow \mathbf{K}^*$, and the corresponding map $\mathrm{Trop}(\psi) : \mathrm{Trop}(\mathcal{T}) \to \mathbb{R}$ is given in terms of the natural coordinates $\beta_1, \beta_2$ on $\mathrm{Trop}(\mathcal{T}) \cong \mathbb{R}^2$ by
    $$\mathrm{Trop}(\psi)(\beta_1, \beta_2) = \min\{ \beta_1, \beta_2, 3-\beta_1-\beta_2 \}.
    $$
In practice, we may think of tropicalisation as replacing addition by $\min$ and replacing multiplication by addition.
\end{ex}

\subsection{Constructing polytopes} \label{subsec Constructing polytopes}
\fancyhead[L]{4.2 \ \ Constructing polytopes}
% \fancyhead[L]{5.2 \ \ Constructing polytopes}

In this section we return to the Landau--Ginzburg model for $G/B$, defined in Section \ref{subsec Landau-Ginzburg models} as the pair $(Z, \mathcal{W})$. Working now over the field of generalised Puiseux series, $\mathbf{K}$, there is a well-defined notion of the totally positive part of $Z(\mathbf{K})$, denoted by $Z(\mathbf{K}_{>0})$. It is defined, for a given torus chart on $Z(\mathbf{K})$, by the subset where the characters take values in $\mathbf{K}_{>0}$. Moreover, each of the string, ideal and quiver torus charts mentioned in previous sections, gives an isomorphism
    \begin{equation} \label{eqn toric chart isom}
    T^{\vee}(\mathbf{K}_{>0}) \times (\mathbf{K}_{>0})^N \xrightarrow{\sim} Z(\mathbf{K}_{>0})
    \end{equation}
where we consider $T^{\vee}(\mathbf{K}_{>0})$ to be the highest weight torus.

We will now restrict our attention to a fibre of the highest weight map (see Section \ref{subsec Landau-Ginzburg models}). To do so, we observe that since a dominant integral weight $\lambda \in X^*(T)^+$ is a cocharacter of $T^{\vee}$, we can define $t^{\lambda} \in T^{\vee}(\mathbf{K}_{>0})$ via the condition $\chi(t^{\lambda}) = t^{\langle \chi, \lambda\rangle}$ for $\chi \in X^*(T^{\vee})$.
Extending %(\ref{eqn bilinear form pair lattices})
$\mathbb{R}$-bilinearly to the perfect pairing
    $$ \langle \ , \ \rangle :  X^*(T)_{\mathbb{R}} \times  X^*(T^{\vee})_{\mathbb{R}} \to \mathbb{R},
    $$
we have that $t^{\lambda}$ is well defined for all $\lambda \in X^*(T)_{\mathbb{R}}$, by the same formula. We therefore do not require that $\lambda$ be integral, though we continue to be interested in those $\lambda$ which are dominant, that is $\lambda\in X^*(T)^+_{\mathbb{R}}$.
This allows us, for a dominant weight $\lambda\in X^*(T)^+_{\mathbb{R}}$, to define
    $$Z_{t^\lambda}(\mathbf{K}):= \left\{b \in Z(\mathbf{K}) \ | \ \mathrm{hw}(b)=t^\lambda \right\}.
    $$
We denote the restriction of the superpotential to this fibre by
    $$\mathcal{W}_{t^{\lambda}}: Z_{t^\lambda}(\mathbf{K}) \to \mathbf{K}.
    $$

For a fixed element $t^{\lambda} \in T^{\vee}(\mathbf{K}_{>0})$ of the highest weight torus, the isomorphisms (\ref{eqn toric chart isom}) for the string and ideal toric charts restrict to
    $$\begin{aligned}
    \phi_{t^{\lambda},\boldsymbol{z}} &: (\mathbf{K}_{>0})^N \to Z_{t^{\lambda}}(\mathbf{K}_{>0}), \\
    \phi_{t^{\lambda},\boldsymbol{m}} &: (\mathbf{K}_{>0})^N \to Z_{t^{\lambda}}(\mathbf{K}_{>0}) \\
    \end{aligned}
    $$
respectively, with coordinates denoted by $(z_1, \ldots, z_N)$ and $(m_1, \ldots, m_N)$.
These toric charts may be considered as defining a positive atlas for $Z_{t^{\lambda}}(\mathbf{K}_{>0})$.
We denote the respective compositions of $\phi_{t^{\lambda},\boldsymbol{z}}$ and $\phi_{t^{\lambda},\boldsymbol{m}}$ with the superpotential $\mathcal{W}_{t^{\lambda}}$, by
    $$\begin{aligned}
    \mathcal{W}_{t^{\lambda},\boldsymbol{z}} &: (\mathbf{K}_{>0})^N \to \mathbf{K}_{>0}, \\
    \mathcal{W}_{t^{\lambda},\boldsymbol{m}} &: (\mathbf{K}_{>0})^N \to \mathbf{K}_{>0} \\
    \end{aligned}
    $$
and observe that both are positive rational maps. We denote their tropicalisations respectively by
    $$\begin{aligned}
    \mathrm{Trop}\left(\mathcal{W}_{t^{\lambda},\boldsymbol{z}}\right) &: \mathbb{R}^N_{\boldsymbol{\zeta}} \to \mathbb{R}, \\
    \mathrm{Trop}\left(\mathcal{W}_{t^{\lambda},\boldsymbol{m}}\right) &: \mathbb{R}^N_{\boldsymbol{\mu}} \to \mathbb{R}. \\
    \end{aligned}
    $$

We may associate convex polytopes to our tropical superpotentials, defined as follows:
    $$\begin{aligned}
    \mathcal{P}_{\lambda,\boldsymbol{\zeta}} &:= \left\{ \boldsymbol{\alpha} \in \mathbb{R}^N_{\boldsymbol{\zeta}} \ | \ \mathrm{Trop}\left(\mathcal{W}_{t^{\lambda},\boldsymbol{z}}\right)(\boldsymbol{\alpha}) \geq 0 \right\}, \\
    \mathcal{P}_{\lambda,\boldsymbol{\mu}} &:= \left\{ \boldsymbol{\alpha} \in \mathbb{R}^N_{\boldsymbol{\mu}} \ | \ \mathrm{Trop}\left(\mathcal{W}_{t^{\lambda},\boldsymbol{m}}\right)(\boldsymbol{\alpha}) \geq 0 \right\}. \\
    \end{aligned}
    $$

To motivate the definition of these polytopes, first recall the string toric chart, $\varphi_{\mathbf{i}}$, for an arbitrary reduced expression $\mathbf{i}$, defined by (\ref{eqn string coord chart defn}) in Section \ref{subsec The string coordinates}. Using this we have generalisations
    $$\begin{aligned}
    \phi^{\mathbf{i}}_{t^{\lambda},\boldsymbol{z}} &: (\mathbf{K}_{>0})^N \to Z_{t^{\lambda}}(\mathbf{K}_{>0}), \\
    \mathcal{W}^{\mathbf{i}}_{t^{\lambda},\boldsymbol{z}} &: (\mathbf{K}_{>0})^N \to \mathbf{K}_{>0}
    \end{aligned}
    $$
of the maps above, such that $\phi_{t^{\lambda},\boldsymbol{z}} = \phi^{\mathbf{i}_0}_{t^{\lambda},\boldsymbol{z}}$ and $\mathcal{W}_{t^{\lambda},\boldsymbol{z}}=\mathcal{W}^{\mathbf{i}_0}_{t^{\lambda},\boldsymbol{z}}$. With this notation we have the following theorem:

\begin{thm}[{\cite[Theorem 4.1]{Judd2018}}] \label{thm string polytope from string coords}
Consider a general reduced expression $\mathbf{i}$ for $\bar{w}_0$, and the superpotential for $GL_n/B$ written in the associated string coordinates, namely $\mathcal{W}^{\mathbf{i}}_{t^{\lambda},\boldsymbol{z}}$. Then the polytope
    $$\mathcal{P}^{\mathbf{i}}_{\lambda,\boldsymbol{\zeta}} = \left\{ \boldsymbol{\alpha} \in \mathbb{R}^N_{\boldsymbol{\zeta}} \ \big| \ \mathrm{Trop}\left(\mathcal{W}^{\mathbf{i}}_{t^{\lambda},\boldsymbol{z}}\right)(\boldsymbol{\alpha}) \geq 0 \right\}
    $$
is the string polytope associated to $\mathbf{i}$, $\mathrm{String}_{\mathbf{i}}(\lambda)$.
\end{thm}

\begin{rem} \label{rem string and ideal polytopes related for i_0}
The polytopes $\mathcal{P}_{\lambda,\boldsymbol{\zeta}}$ and $\mathcal{P}_{\lambda,\boldsymbol{\mu}}$ are simply linear transformations of each other.

If instead we were to take the toric chart given by the vertex torus, then the resulting polytope would be the respective Gelfand--Tsetlin polytope for $\lambda$. Since the quiver torus is so closely related to the vertex torus, the coordinates of the quiver toric chart provide a bridge between the string and Gelfand--Tsetlin polytopes. In particular, we see that the tropicalisation of the coordinate change from the string to the vertex coordinates defines an affine map between these two polytopes.
\end{rem}

\subsection{Tropical critical points and the weight map} \label{subsec Tropical critical points and the weight map}
\fancyhead[L]{4.3 \ \ Tropical critical points and the weight map}
% \fancyhead[L]{5.3 \ \ Tropical critical points and the weight map}

We recall the critical point conditions of the superpotential given in (\ref{eqn crit pt conditions}) which, working over $\mathbf{K}_{>0}$, define critical points of $\mathcal{W}$ in the fibres $Z_{t^\lambda}(\mathbf{K}_{>0})$. Judd showed in the $SL_n$ case that for each dominant weight $\lambda$, there is in fact only one critical point that lies in $Z_{t^{\lambda}}(\mathbf{K}_{>0})$ (see \cite[Section 5]{Judd2018}). We refer to this unique point as the positive critical point of $\mathcal{W}_{t^{\lambda}}$, denoted $p_{\lambda}$. Judd's statement extends to the $GL_n$ case that we are considering here, with the same proof. This also follows from the more general result of Judd and Rietsch in \cite{JuddRietsch2019}, moreover the assumption on $\lambda$ to be integral can be dropped, therefore we also have a unique positive critical point of $\mathcal{W}_{t^{\lambda}}$ in $Z_{t^{\lambda}}(\mathbf{K}_{>0})$ for any dominant $\lambda \in \mathbb{R}^n$. We use the same notation, $p_{\lambda}$, for this point. In addition we will use the term dominant weight loosely, to mean $\lambda \in \mathbb{R}^n$ such that $\lambda_1 \geq \lambda_2 \geq \cdots \geq \lambda_n$, and say dominant integral weight if, in addition, $\lambda \in \mathbb{Z}^n$.

This critical point $p_{\lambda} \in Z_{t^{\lambda}}(\mathbf{K}_{>0})$ defines a point $p_{\lambda}^{\mathrm{trop}} \in \mathrm{Trop}(Z_{t^{\lambda}})$, called the tropical critical point of $\mathcal{W}_{t^{\lambda}}$. Explicitly, using a positive chart (such as $\phi_{t^{\lambda},\boldsymbol{z}}$ or $\phi_{t^{\lambda},\boldsymbol{m}}$) we apply the valuation $\mathrm{Val}_{\mathbf{K}}$ to every coordinate of $p_{\lambda}$. This gives rise to the corresponding point ($p_{\lambda, \boldsymbol{\zeta}}^{\mathrm{trop}}$ or $p_{\lambda, \boldsymbol{\mu}}^{\mathrm{trop}}$ respectively) in the associated tropical chart $\mathrm{Trop}(Z_{t^{\lambda}}) \to \mathbb{R}^N$.
Moreover, for a choice of positive chart the tropical critical point lies in the interior of the respective superpotential polytope, for example, $p_{\lambda, \boldsymbol{\zeta}}^{\mathrm{trop}} \in \mathcal{P}_{\lambda,\boldsymbol{\zeta}}$ and $p_{\lambda, \boldsymbol{\mu}}^{\mathrm{trop}} \in \mathcal{P}_{\lambda,\boldsymbol{\mu}}$. This is implicit in Judd's work in \cite{Judd2018} but is also true more generally, with an explicit statement given by Judd and Rietsch in \cite[Theorem 1.2]{JuddRietsch2019}.

We also have the tropicalisation of the weight map, $\mathrm{wt}:Z_{t^{\lambda}}\to T^{\vee}$ defined in Section \ref{subsec Landau-Ginzburg models}, which can be interpreted as a kind of projection
    $$\mathrm{Trop}(\mathrm{wt}) : \mathrm{Trop}(Z_{t^{\lambda}}) \to \mathfrak{h}^*_{\mathbb{R}}.
    $$
In particular, in the case of integral $\lambda$, the image under this projection of either superpotential polytope, $\mathcal{P}_{\lambda,\boldsymbol{\zeta}}$ or $\mathcal{P}_{\lambda,\boldsymbol{\mu}}$, is exactly the weight polytope. We therefore generalise the standard definition of the weight polytope to be the projection of the superpotential polytope under $\mathrm{Trop}(\mathrm{wt})$. This extended definition holds for all dominant weights $\lambda$.

In the $SL_n$ case, Judd proved that $\mathrm{Trop}(\mathrm{wt})\left(p_{\lambda}^{\mathrm{trop}}\right)=0$ (see \cite[Theorem 5.1]{Judd2018}). Working more generally in the $GL_n$ case, we obtain that $\mathrm{Trop}(\mathrm{wt})\left(p_{\lambda}^{\mathrm{trop}}\right)$, the image of the tropical critical point under this weight projection, is in fact the centre of mass of the weight polytope:
\begin{cor}[Corollary of Proposition {\ref{prop weight at crit point non trop}}] \label{cor weight matrix at trop crit point}
Given a dominant weight $\lambda$, the weight matrix at the critical point in the fibre over $t^{\lambda} \in T^{\vee}(\mathbf{K}_{>0})$ is an $n\times n$ matrix $\mathrm{diag}\left(t^{\ell}, \ldots, t^{\ell}\right)$ where
    $$\ell = \frac{1}{n}\sum_{i=1}^n \lambda_i.
    $$
\end{cor}
\hfill\qedsymbol

\begin{ex}[Dimension 3] \label{ex nu'_i coords in dim 3} Recalling Example \ref{ex dim 3 z coords} and working now over the field of generalised Puiseux series, $\mathbf{K}$, we have
    $$\begin{aligned}
    b &= \begin{pmatrix} 1 & z_3 & z_2 \\ & 1 & z_1 +\frac{z_2}{z_3} \\ & & 1\end{pmatrix}
        \begin{pmatrix} d_1& &  \\ & d_2 & \\ & & d_3 \end{pmatrix}
        \begin{pmatrix} & & 1 \\ & -1 & \\ 1 & & \end{pmatrix}
        \begin{pmatrix} 1 & \frac{d_2}{d_3} \frac{z_3}{z_2} & \frac{d_1}{d_3}\frac{1}{z_1z_3} \\ & 1 & \frac{d_1}{d_2}\left(\frac{1}{z_3}+\frac{z_2}{z_1z_3^2}\right) \\ & & 1\end{pmatrix} \\
    &=\begin{pmatrix} d_3 z_2 & & \\ d_3 \left(z_1+\frac{z_2}{z_3} \right) & d_2 \frac{z_1z_3}{z_2} & \\ d_3 & d_2 \frac{z_3}{z_2} & d_1 \frac{1}{z_1z_3} \end{pmatrix}.
    \end{aligned}$$
Additionally, in reference to the previous section, we give the quiver for this coordinate system in Figure \ref{fig quiver n=3 z coords}.
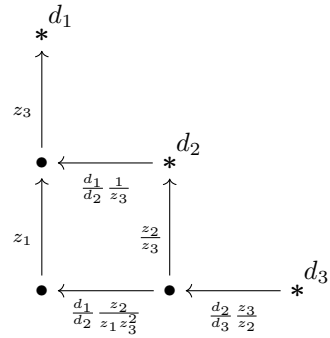
\begin{figure}[ht]
    \centering
    \begin{tikzpicture}[scale=0.85]
        %dots and stars
        \node (31) at (0,0) {$\bullet$};
        \node (21) at (0,2) {$\bullet$};
        \node (11) at (0,4) {$\boldsymbol{*}$};
            \node at (0.3,4.3) {$d_1$};
        \node (32) at (2,0) {$\bullet$};
        \node (22) at (2,2) {$\boldsymbol{*}$};
            \node at (2.3,2.3) {$d_2$};
        \node (33) at (4,0) {$\boldsymbol{*}$};
            \node at (4.3,0.3) {$d_3$};

        %arrows
        \draw[->] (31) -- (21);
        \draw[->] (21) -- (11);
        \draw[->] (32) -- (22);

        \draw[->] (22) -- (21);
        \draw[->] (33) -- (32);
        \draw[->] (32) -- (31);

        %arrow labels
        \node at (-0.3,0.8) {\scriptsize{$z_1$}};
        \node at (-0.3,2.8) {\scriptsize{$z_3$}};
        \node at (1.7,0.8) {\scriptsize{$\frac{z_2}{z_3}$}};

        \node at (1,-0.4) {\scriptsize{$\frac{d_1}{d_2}\frac{z_2}{z_1z_3^2}$}};
        \node at (3,-0.4) {\scriptsize{$\frac{d_2}{d_3}\frac{z_3}{z_2}$}};
        \node at (1,1.6) {\scriptsize{$\frac{d_1}{d_2}\frac{1}{z_3}$}};
    \end{tikzpicture}
\caption{Quiver decoration when $n=3$, string coordinates} \label{fig quiver n=3 z coords}
\end{figure}
We recall that we can use the quiver to read off the superpotential. It is the sum of the arrow coordinates and is the same map as given in Example \ref{ex dim 3 z coords}, obtained from $b$ via the formula in that section:
    $$\mathcal{W}(d,\boldsymbol{z}) = z_1 +\frac{z_2}{z_3} + z_3 + \frac{d_1}{d_2}\left(\frac{1}{z_3} + \frac{z_2}{z_1z_3^2}\right) + \frac{d_2}{d_3}\frac{z_3}{z_2}.
    $$

Now in order to obtain a polytope from $\mathcal{W}$, we need to tropicalise. To do so we take our highest weight torus element $d$ to be $t^{\lambda}\in T^{\vee}(\mathbf{K}_{>0})$, with $\lambda=(\lambda_1 \geq \lambda_2 \geq \lambda_3)$, that is $t^{\lambda}=\mathrm{diag}(t^{\lambda_1}, t^{\lambda_2}, t^{\lambda_3})$. Then our tropical superpotential is
    \begin{multline*}
    \mathrm{Trop}\left(\mathcal{W}_{t^{\lambda},\boldsymbol{z}}\right) (\zeta_1,\zeta_2,\zeta_3)
    = \min\left\{\zeta_1, \zeta_2-\zeta_3, \zeta_3, \lambda_1 - \lambda_2 - \zeta_3, \lambda_1 - \lambda_2 - \zeta_1 + \zeta_2 - 2\zeta_3, \right. \\ \left. \lambda_2 - \lambda_3 - \zeta_2 + \zeta_3 \right\}.
    \end{multline*}
The corresponding polytope, $\mathcal{P}_{\lambda,\boldsymbol{\zeta}} = \left\{ \boldsymbol{\zeta} \in \mathbb{R}^N \ | \ \mathrm{Trop}\left(\mathcal{W}_{t^{\lambda},\boldsymbol{z}}\right)(\boldsymbol{\zeta}) \geq 0 \right\}$, is then cut out by the following inequalities:
    $$\begin{aligned} 0 \leq &\zeta_1 \leq \lambda_1 - \lambda_2 + \zeta_2 - 2\zeta_3, \\
    \zeta_3 \leq &\zeta_2 \leq \lambda_2 - \lambda_3 + \zeta_3, \\
    0 \leq &\zeta_3 \leq \lambda_1 - \lambda_2.
    \end{aligned}$$
This polytope is given in Figure \ref{fig polytope n=3 z coords} for $\lambda=(2,1,-1)$.
\begin{figure}[ht]
\centering
\begin{minipage}[b]{0.46\textwidth}
    \centering
    \begin{tikzpicture}[scale=1]
        \draw[black!50, ->] (-0.4,0.2) -- (4,-2); % x axis
            \node[black!50] at (4,-1.72) {\scriptsize{$\zeta_1$}};
        \draw[black!50, ->] (-0.5,-0.1) -- (5,1); % y axis
            \node[black!50] at (4.9,1.21) {\scriptsize{$\zeta_2$}};
        \draw[black!50, ->] (0,-0.4) -- (0,2); % z axis
            \node[black!50] at (-0.28,1.9) {\scriptsize{$\zeta_3$}};
            % \node[black!60] at (-0.2, 0.3) {O};

        \draw[thick, dashed, rounded corners=0.5] (0,0) -- (2,0.4) -- (3,1.6) -- (1,1.2) -- cycle;
        \draw[thick, dashed, rounded corners=0.5] (2,0.4) -- (4.4,-0.8) -- (4.6,0.8) -- (3,1.6) --  cycle;
        \draw[thick, rounded corners=0.5] (1,1.2) -- (3,1.6) -- (4.6,0.8) -- cycle;
        \draw[thick, rounded corners=0.5] (0,0) -- (0.8,-0.4) -- (1,1.2) -- cycle;
        \draw[thick, rounded corners=0.5] (0.8,-0.4) -- (4.4,-0.8) -- (4.6,0.8) -- (1,1.2) -- cycle;

        \filldraw[blue] (0,0) circle (1pt);
        \filldraw (2,0.4) circle (1pt);
        \filldraw (3,1.6) circle (1pt);
        \filldraw (1,1.2) circle (1pt);
        \filldraw (0.8,-0.4) circle (1pt);
        \filldraw (2.8,0) circle (1pt);
        \filldraw (3.6,-0.4) circle (1pt);
        \filldraw (4.4,-0.8) circle (1pt);
        \filldraw[red] (4.6,0.8) circle (1pt);

        \node at (0,-1.19) {\tiny$\begin{aligned} &(0,0,0) \\ &(\lambda_1-\lambda_2,0,0) \\ &(0,\lambda_2-\lambda_3,0) \\ &(\lambda_1-\lambda_2,\lambda_2-\lambda_3,0) \\ &(\lambda_2-\lambda_3, \lambda_2-\lambda_3, 0) \\ &(\lambda_1-\lambda_3,\lambda_2-\lambda_3,0) \end{aligned}$};
            \draw[<-, >=stealth, densely dashed, black!50, rounded corners=4] (-0.03,-0.05) -- (-0.14,-0.3) -- (-0.35,-0.4);
            \draw[<-, >=stealth, densely dashed, black!50, rounded corners=5] (0.77,-0.45) -- (0.66,-0.65) -- (0.41,-0.72);
            \draw[<-, >=stealth, densely dashed, black!50, rounded corners=25] (1.98,0.35) -- (1.6,-0.7) -- (0.4,-1.03);
            \draw[<-, >=stealth, densely dashed, black!50, rounded corners=21] (2.77,-0.05) -- (2.2,-1) -- (1.15,-1.35);
            \draw[<-, >=stealth, densely dashed, black!50, rounded corners=26] (3.56,-0.45) -- (2.8,-1.25) -- (1.16,-1.66);
            \draw[<-, >=stealth, densely dashed, black!50, rounded corners=28] (4.34,-0.84) -- (3,-1.65) -- (1.16,-1.97);
        \node at (1.73,2) {\tiny$\begin{aligned} &(\lambda_2-\lambda_3,\lambda_1-\lambda_3,\lambda_1-\lambda_2) \\ &(0,\lambda_1-\lambda_3,\lambda_1-\lambda_2) \\ &(0,\lambda_1-\lambda_2,\lambda_1-\lambda_2) \end{aligned}$};
            \draw[<-, >=stealth, densely dashed, black!50, rounded corners=18] (4.57,0.86) -- (4,1.8) -- (3.25,2.3);
            \draw[<-, >=stealth, densely dashed, black!50, rounded corners=5] (2.97,1.65) -- (2.76,1.9) -- (2.5,2);
            \draw[<-, >=stealth, densely dashed, black!50, rounded corners=7] (0.95,1.22) -- (0.4,1.27) -- (-0.05,1.5) -- (0.18,1.66);
    \end{tikzpicture}
\caption{Superpotential polytope $\mathcal{P}_{\lambda,\boldsymbol{\zeta}}$ for $n=3$ and $\lambda=(2,1,-1)$ (string coordinates)} \label{fig polytope n=3 z coords}
\end{minipage} \hfill
\begin{minipage}[b]{0.47\textwidth}
    \centering
    \begin{tikzpicture}[scale=0.55]
        %horizontal
        \draw[black!50, densely dotted] (-3,0) -- (3,0);
        \draw[black!50, densely dotted] (-3,0.87) -- (3,0.87);
        \draw[black!50, densely dotted] (-3,-0.87) -- (3,-0.87);
        \draw[black!50, densely dotted] (-3,1.74) -- (3,1.74);
        \draw[black!50, densely dotted] (-3,-1.74) -- (3,-1.74);
        \draw[black!50, densely dotted] (-3,2.61) -- (3,2.61);
        \draw[black!50, densely dotted] (-3,-2.61) -- (3,-2.61);
        % up slant
        \draw[black!50, densely dotted] (-3,1.74) -- (-2.5,2.61);
        \draw[black!50, densely dotted] (-3,0) -- (-1.5,2.61);
        \draw[black!50, densely dotted] (-3,-1.74) -- (-0.5,2.61);
        \draw[black!50, densely dotted] (-2.5,-2.61) -- (0.5,2.61);
        \draw[black!50, densely dotted] (-1.5,-2.61) -- (1.5,2.61);
        \draw[black!50, densely dotted] (-0.5,-2.61) -- (2.5,2.61);
        \draw[black!50, densely dotted] (0.5,-2.61) -- (3,1.74);
        \draw[black!50, densely dotted] (1.5,-2.61) -- (3,0);
        \draw[black!50, densely dotted] (2.5,-2.61) -- (3,-1.74);
        % down slant
        \draw[black!50, densely dotted] (-3,-1.74) -- (-2.5,-2.61);
        \draw[black!50, densely dotted] (-3,0) -- (-1.5,-2.61);
        \draw[black!50, densely dotted] (-3,1.74) -- (-0.5,-2.61);
        \draw[black!50, densely dotted] (-2.5,2.61) -- (0.5,-2.61);
        \draw[black!50, densely dotted] (-1.5,2.61) -- (1.5,-2.61);
        \draw[black!50, densely dotted] (-0.5,2.61) -- (2.5,-2.61);
        \draw[black!50, densely dotted] (0.5,2.61) -- (3,-1.74);
        \draw[black!50, densely dotted] (1.5,2.61) -- (3,0);
        \draw[black!50, densely dotted] (2.5,2.61) -- (3,1.74);

            \node[black!60] at (0.15,0.09) {\tiny{$0$}};

        \draw (-2.5,-0.87) -- (-2.5,0.87);
        \draw (-2.5,-0.87) -- (2,1.74);
        \draw (-2.5,-0.87) -- (0.5,-2.61);
        \draw (-2.5,0.87) -- (0.5,2.61);
        \draw (-2.5,0.87) -- (2,-1.74);
        \draw (0.5,2.61) -- (0.5,-2.61);
        \draw (0.5,2.61) -- (2,1.74);
        \draw (2,1.74) -- (2,-1.74);
        \draw (0.5,-2.61) -- (2,-1.74);

        \filldraw[blue] (-2.5,-0.87) circle (1.5pt);
        \filldraw (-2.5,0.87) circle (1.5pt);
        \filldraw (-1,0) circle (1.5pt);
        \filldraw (0.5,0.87) circle (1.5pt);
        \filldraw (0.5,-0.87) circle (1.5pt);
        \filldraw (0.5,2.61) circle (1.5pt);
        \filldraw (0.5,-2.61) circle (1.5pt);
        \filldraw (2,-1.74) circle (1.5pt);
        \filldraw[red] (2,1.74) circle (1.5pt);

        \node at (5.25,0.5) {\tiny$(\lambda_2,\lambda_2,\lambda_1-\lambda_2+\lambda_3)$};
            \draw[<-, >=stealth, densely dashed, black!50, rounded corners=13] (0.6,0.8) -- (1.5,0.45) -- (3,0.5);
        \node at (5.25,0) {\tiny$(\lambda_1-\lambda_2+\lambda_3,\lambda_2,\lambda_2)$};
            \draw[<-, >=stealth, densely dashed, black!50, rounded corners=20] (-0.89,0.005) -- (0.5,-0.25) -- (3,0);
        \node at (5.25,-0.5) {\tiny$(\lambda_2,\lambda_1-\lambda_2+\lambda_3,\lambda_2)$};
            \draw[<-, >=stealth, densely dashed, black!50, rounded corners=13] (0.6,-0.8) -- (1.75,-0.45) -- (3,-0.5);

        \node at (4.2,2.5) {\tiny$(\lambda_2,\lambda_1,\lambda_3)$};
            \draw[<-, >=stealth, densely dashed, black!50, rounded corners=12] (0.7,2.55) -- (1.8,2.35) -- (3,2.5);
        \node at (4.2,2) {\tiny$(\lambda_1,\lambda_2,\lambda_3)$};
            \draw[<-, >=stealth, densely dashed, black!50, rounded corners=4] (2.1,1.8) -- (2.7,2) -- (2.9,2);
        \node at (4.2,1.5) {\tiny$(\lambda_3,\lambda_1,\lambda_2)$};
            \draw[<-, >=stealth, densely dashed, black!50, rounded corners=27] (-2.37,0.86) -- (0.5,1.6) -- (3,1.5);
        \node at (4.2,-1.5) {\tiny$(\lambda_3,\lambda_2,\lambda_1)$};
            \draw[<-, >=stealth, densely dashed, black!50, rounded corners=27] (-2.37,-0.86) -- (0.5,-1.6) -- (3,-1.5);
        \node at (4.2,-2) {\tiny$(\lambda_1,\lambda_3,\lambda_2)$};
            \draw[<-, >=stealth, densely dashed, black!50, rounded corners=4] (2.1,-1.8) -- (2.7,-2) -- (2.9,-2);
        \node at (4.2,-2.5) {\tiny$(\lambda_2,\lambda_3,\lambda_1)$};
            \draw[<-, >=stealth, densely dashed, black!50, rounded corners=12] (0.7,-2.55) -- (1.8,-2.35) -- (3,-2.5);
    \end{tikzpicture}
\caption{Projection of superpotential polytope $\mathcal{P}_{\lambda,\boldsymbol{\zeta}}$ onto weight lattice, $\lambda=(2,1,-1)$} \label{fig Projection of superpotential polytope onto weight lattice lambda=(2,1,-1)}
\end{minipage}
\end{figure}

Finally, recalling the weight matrix
    $$\begin{pmatrix} d_3z_2 & & \\ & d_2\frac{z_1z_3}{z_2} & \\ & & d_1\frac{1}{z_1z_3} \end{pmatrix}
    $$
we see that a point $\left(\zeta_1, \zeta_2, \zeta_3\right)$ in the polytope has weight
    $$\left(\lambda_3 + \zeta_2, \lambda_2 + \zeta_1 - \zeta_2 + \zeta_3, \lambda_1 - \zeta_1 - \zeta_3 \right).
    $$
In particular the weight projection, given for $\lambda=(2,1,-1)$ in Figure \ref{fig Projection of superpotential polytope onto weight lattice lambda=(2,1,-1)}, acts on the vertices and distinguished points as described in Table \ref{tab Vertices and distinguished points, and their corresponding weights, string coordinates}.
\begin{table}[ht!]
\centering
{\renewcommand{\arraystretch}{1.5}
\begin{tabular}{c | c}
    Regular vertices & Weight \\
    \hline\hline
    $(0, 0, 0)$
        & $(\lambda_3, \lambda_2, \lambda_1)$ \\ \hline
    $(0, \lambda_2-\lambda_3, 0)$
        & $(\lambda_2, \lambda_3, \lambda_1)$ \\ \hline
    $(\lambda_1-\lambda_2, 0, 0)$
        & $(\lambda_3, \lambda_1, \lambda_2)$ \\ \hline
    $(\lambda_1-\lambda_3, \lambda_2-\lambda_3, 0)$
        & $(\lambda_2, \lambda_1, \lambda_3)$ \\ \hline
    $(0, \lambda_1-\lambda_3, \lambda_1-\lambda_2)$
        & $(\lambda_1, \lambda_3, \lambda_2)$ \\ \hline \vspace{0.3cm}
    $(\lambda_2-\lambda_3, \lambda_1-\lambda_3, \lambda_1-\lambda_2)$
        & $(\lambda_1, \lambda_2, \lambda_3)$ \\
    Irregular vertex and distinguished points & Weight \\
    \hline\hline
    $(0, \lambda_1-\lambda_2, \lambda_1-\lambda_2)$
        & $(\lambda_1-\lambda_2+\lambda_3, \lambda_2, \lambda_2)$ \\ \hline
    $(\lambda_1-\lambda_2, \lambda_2-\lambda_3, 0)$
        & $(\lambda_2, \lambda_1-\lambda_2+\lambda_3, \lambda_2)$ \\ \hline
    $(\lambda_2-\lambda_3, \lambda_2-\lambda_3, 0)$
        & $(\lambda_2, \lambda_2, \lambda_1-\lambda_2+\lambda_3)$ \\ \hline
\end{tabular}}
\caption{Vertices and distinguished points, and their corresponding weights, string coordinates} \label{tab Vertices and distinguished points, and their corresponding weights, string coordinates}
\end{table}

\end{ex}

\bigskip

\begin{ex} \label{ex ideal coords in dim 3 run through}
For comparison, we now run through the previous example using the ideal coordinates instead. We begin by recalling the matrix $b$:
    $$b=\mathbf{y}_{\mathbf{i}_0}^{\vee}\left(\frac{1}{m_1}, \frac{1}{m_2}, \frac{1}{m_3} \right)
        \begin{pmatrix} d_3 m_2m_3 & & \\ & d_2 \frac{m_1}{m_3} & \\ & & d_1 \frac{1}{m_1m_2} \end{pmatrix}.
    $$
In this case, the quiver is given in Figure \ref{fig quiver n=3 m coords}.
\begin{figure}[ht]
\centering
\begin{tikzpicture}[scale=0.85]
    %dots and stars
    \node (31) at (0,0) {$\bullet$};
    \node (21) at (0,2) {$\bullet$};
    \node (11) at (0,4) {$\boldsymbol{*}$};
        \node at (0.3,4.3) {$d_1$};
    \node (32) at (2,0) {$\bullet$};
    \node (22) at (2,2) {$\boldsymbol{*}$};
        \node at (2.3,2.3) {$d_2$};
    \node (33) at (4,0) {$\boldsymbol{*}$};
        \node at (4.3,0.3) {$d_3$};

    %arrows
    \draw[->] (31) -- (21);
    \draw[->] (21) -- (11);
    \draw[->] (32) -- (22);

    \draw[->] (22) -- (21);
    \draw[->] (33) -- (32);
    \draw[->] (32) -- (31);

    %arrow labels
    \node at (-0.3,0.8) {\scriptsize{$m_2$}};
    \node at (-0.3,2.8) {\scriptsize{$m_1$}};
    \node at (1.5,0.8) {\scriptsize{$\frac{m_2m_3}{m_1}$}};

    \node at (1,-0.4) {\scriptsize{$\frac{d_1}{d_2}\frac{m_3}{m_1^2}$}};
    \node at (3,-0.4) {\scriptsize{$\frac{d_2}{d_3}\frac{m_1}{m_2m_3}$}};
    \node at (1,1.6) {\scriptsize{$\frac{d_1}{d_2}\frac{1}{m_1}$}};
\end{tikzpicture}
\caption{Quiver decoration when $n=3$, ideal coordinates} \label{fig quiver n=3 m coords}
\end{figure}

The superpotential is given by
    $$\mathcal{W}(d, \boldsymbol{m}) = m_1 + m_2 + \frac{m_2m_3}{m_1} + \frac{d_2}{d_3} \frac{m_1}{m_2m_3} + \frac{d_1}{d_2} \left(\frac{m_3}{{m_1}^2}+\frac{1}{m_1}\right).
    $$
We again take our torus element $d$ to be $t^{\lambda}\in T^{\vee}(\mathbf{K}_{>0})$, with $\lambda=(\lambda_1 \geq \lambda_2 \geq \lambda_3)$. Our tropical superpotential is given by
    \begin{multline*}\mathrm{Trop}\left(\mathcal{W}_{t^{\lambda},\boldsymbol{m}}\right) (\mu_1,\mu_2,\mu_3)
    = \min\left\{\mu_1, \mu_2, -\mu_1+\mu_2+\mu_3, \lambda_2-\lambda_3 +\mu_1-\mu_2-\mu_3, \right. \\ \left. \lambda_1-\lambda_2 -2\mu_1+\mu_3, \lambda_1-\lambda_2 -\mu_1 \right\}.
    \end{multline*}
The corresponding polytope $\mathcal{P}_{\lambda,\boldsymbol{\mu}}$ is cut out by
    $$\begin{aligned} 0 &\leq \mu_1 \leq \lambda_1-\lambda_2, \\
    0 &\leq \mu_2, \\
    0 &\leq  -\mu_1 +\mu_2 +\mu_3 \leq \lambda_2-\lambda_3, \\
    2\mu_1-\mu_3 &\leq \lambda_1-\lambda_2.
    \end{aligned}$$
This polytope is given in Figure \ref{fig polytope n=3 m coords} for $\lambda=(2,1,-1)$.
\begin{figure}[ht]
\centering
\begin{tikzpicture}[scale=1.3]
    \draw[black!50, ->] (-0.8,0.4) -- (2.8,-1.4); % x axis
        \node[black!50] at (2.8,-1.22) {\scriptsize{$\mu_1$}};
    \draw[black!50, ->] (-1,-0.2) -- (4,0.8); % y axis
        \node[black!50] at (4,0.6) {\scriptsize{$\mu_2$}};
    \draw[black!50, ->] (0,-1.3) -- (0,3); % z axis
        \node[black!50] at (-0.2,2.9) {\scriptsize{$\mu_3$}};
        % \node[black!60] at (-0.2, 0.3) {O};

    \draw[thick, dashed, rounded corners=0.5] (0,2) -- (3,-0.4);

    \draw[thick, rounded corners=0.5] (0,0) -- (0,2) -- (0.8,2.6) -- (2.8,1) -- (3,-0.4) --  (1,-0.8) -- cycle;
    \draw[thick, rounded corners=0.5] (0.8,0.6) -- (0,0) -- (1,-0.8) -- cycle;
    \draw[thick, rounded corners=0.5] (0.8,0.6) -- (2.8,1) -- (0.8,2.6) -- cycle;

    \filldraw[blue] (0,0) circle (1pt);
    \filldraw (0.8,0.6) circle (0.8pt);
    \filldraw (0,2) circle (0.8pt);
    \filldraw (0.8,2.6) circle (0.8pt);
    \filldraw (1,-0.8) circle (0.8pt);
    \filldraw (1,1.2) circle (0.8pt);
    \filldraw (2,0.4) circle (0.8pt);
    \filldraw (3,-0.4) circle (0.8pt);
    \filldraw[red] (2.8,1) circle (1pt);
    % \filldraw[yellow] (0.9,0.4) circle (1pt);

    \node at (-2.2,0.9) {\tiny$\begin{aligned} &(\lambda_1-\lambda_2, 0, \lambda_1-\lambda_3) \\ \\  &(0, 0, \lambda_2-\lambda_3) \\ \\ &(\lambda_1-\lambda_2, 0, \lambda_1-\lambda_2) \\ \\ &(0,0,0) \\ \\ &(0, \lambda_1-\lambda_2, -\lambda_1+\lambda_2) \end{aligned}$};
        \draw[<-, >=stealth, densely dashed, black!50, rounded corners=25] (0.73,2.6) -- (-0.41,2.35) -- (-1.36,1.9);
        \draw[<-, >=stealth, densely dashed, black!50, rounded corners=19] (-0.06,2) -- (-0.75,1.9) -- (-1.95,1.4);
        \draw[<-, >=stealth, densely dashed, black!50, rounded corners=26] (0.76,0.62) -- (-0.4,0.9) -- (-1.36,0.9);
        \draw[<-, >=stealth, densely dashed, black!50, rounded corners=27] (-0.1,0.015) -- (-1.4,0.1) -- (-2.5,0.4);
        \draw[<-, >=stealth, densely dashed, black!50, rounded corners=18] (0.95,-0.8) -- (-0.4,-0.5) -- (-1.2,-0.1);
    \node at (4.2,1.9) {\tiny$\begin{aligned} &(0,\lambda_1-\lambda_2, -\lambda_1+2\lambda_2-\lambda_3) \\ \\  &(0, \lambda_2-\lambda_3,0) \\ \\ &\hspace{1.2cm}(\lambda_1-\lambda_2, \lambda_2-\lambda_3, \lambda_1-\lambda_2) \\ \\ &\hspace{1.2cm}(0, \lambda_1-\lambda_3, -\lambda_1+\lambda_2) \end{aligned}$};
        \draw[<-, >=stealth, densely dashed, black!50, rounded corners=24] (1.04,1.25) -- (1.9,2.4) -- (2.55,2.62);
        \draw[<-, >=stealth, densely dashed, black!50, rounded corners=10] (2.02,0.45) -- (2.35,1.9) -- (2.55,2.12);
        \draw[<-, >=stealth, densely dashed, black!50, rounded corners=10] (2.82,1.04) -- (3.2,1.6) -- (3.45,1.65);
        \draw[<-, >=stealth, densely dashed, black!50, rounded corners=10] (3.02,-0.34) -- (3.3,0.9) -- (3.48,1.15);
\end{tikzpicture}
\caption{Superpotential polytope $\mathcal{P}_{\lambda,\boldsymbol{\mu}}$ for $n=3$ and $\lambda=(2,1,-1)$ (ideal coordinates)} \label{fig polytope n=3 m coords}
\end{figure}
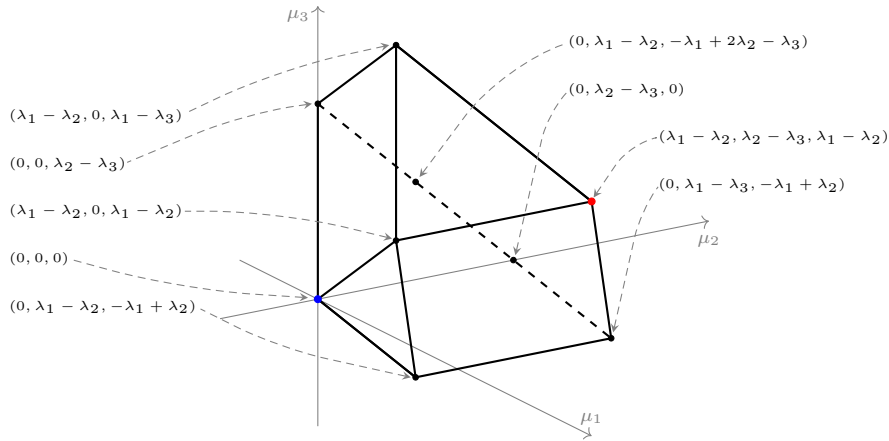

Similar to the previous example, using the weight matrix
    $$\begin{pmatrix} d_3 m_2m_3 & & \\ & d_2 \frac{m_1}{m_3} & \\ & & d_1 \frac{1}{m_1m_2} \end{pmatrix}
    $$
we see that a point $(\mu_1, \mu_2, \mu_3)$ in the polytope has weight
    $$\left(\lambda_3 + \mu_2 + \mu_3, \lambda_2 + \mu_1 - \mu_3, \lambda_1 - \mu_1 - \mu_2 \right).
    $$
The weight projection acts on the vertices and distinguished points as described in Table \ref{tab Vertices and distinguished points, and their corresponding weights, ideal coordinates}.
\begin{table}[ht!]
\centering
{\renewcommand{\arraystretch}{1.5}
\begin{tabular}{c | c}
    Regular vertices & Weight \\
    \hline\hline
    $(0,0,0)$ & $(\lambda_3, \lambda_2, \lambda_1)$ \\ \hline
    $(0, 0, \lambda_2-\lambda_3)$ & $(\lambda_2, \lambda_3, \lambda_1)$ \\ \hline
    $(0, \lambda_1-\lambda_2, -\lambda_1+\lambda_2)$ & $(\lambda_3, \lambda_1, \lambda_2)$ \\ \hline
    $(0, \lambda_1-\lambda_3, -\lambda_1+\lambda_2)$ & $(\lambda_2, \lambda_1, \lambda_3)$ \\ \hline
    $(\lambda_1-\lambda_2, 0, \lambda_1-\lambda_3)$ & $(\lambda_1, \lambda_3, \lambda_2)$ \\ \hline\vspace{0.3cm}
    $(\lambda_1-\lambda_2, \lambda_2-\lambda_3, \lambda_1-\lambda_2)$ & $(\lambda_1, \lambda_2, \lambda_3)$ \\
    Irregular vertex and distinguished points & Weight \\
    \hline\hline
    $(\lambda_1-\lambda_2, 0, \lambda_1-\lambda_2)$ & $(\lambda_1-\lambda_2+\lambda_3, \lambda_2, \lambda_2)$ \\ \hline
    $(0, \lambda_1-\lambda_2, -\lambda_1+2\lambda_2-\lambda_3)$ & $(\lambda_2, \lambda_1-\lambda_2+\lambda_3, \lambda_2)$ \\ \hline
    $(0, \lambda_2-\lambda_3, 0)$ & $(\lambda_2, \lambda_2, \lambda_1-\lambda_2+\lambda_3)$ \\
\end{tabular}}
\caption{Vertices and distinguished points, and their corresponding weights, ideal coordinates} \label{tab Vertices and distinguished points, and their corresponding weights, ideal coordinates}
\end{table}
\end{ex}

\subsection{Ideal fillings} \label{subsec Ideal fillings}

In the previous section we saw that we could tropicalise the critical point to obtain a unique point in the superpotential polytope. In \cite[Proposition 5.6]{Judd2018}, Judd shows that we obtain the same point by first tropicalising the critical point conditions and then looking for solutions of this new system. In the same paper he relates this point to a new combinatorial object he introduces: ideal fillings.

In this section we generalise this relation from the $SL_n$ case to the $GL_n$ case. In order to do so we first extend Judd's definition of ideal fillings to be suitable for working with $GL_n$ and then describe the tropical critical point conditions.

The benefit of considering ideal fillings will be a better description of the tropical critical point, and thus also the preimage of the weight polytope centre of mass under the weight projection.

\fancyhead[L]{4.4 \ \ Ideal fillings}
% \fancyhead[L]{5.4 \ \ Ideal fillings}

\begin{defn}
Take a grid of $n(n-1)/2$ boxes in upper triangular form and assign a non-negative real number to each box. This is called a filling and written as $\{n_{ij}\}_{1\leq i<j \leq n}$.

A filling is said to be ideal if $n_{ij}=\max\{n_{i+1, j}, n_{i, j-1} \} $ for $j-i\geq 2$ and is integral if all the $n_{ij}$ are integral.
\end{defn}

For an example with $n=4$, see Figure \ref{Ideal filling for n 4}. We note that an ideal filling is completely determined by the entries in the first diagonal since $n_{ij}= \max_{i\leq k\leq j-1}\{n_{k , k+1}\}$.

\begin{figure}[ht]
\centering
\begin{tikzpicture}
    %box lines
    %horizontal
    \draw (0,0) -- (2.4,0);
    \draw (0,-0.8) -- (2.4,-0.8);
    \draw (0.8,-1.6) -- (2.4,-1.6);
    \draw (1.6,-2.4) -- (2.4,-2.4);

    % \draw (3.2,0) -- (4,0);
    % \draw (3.2,-0.8) -- (4,-0.8);
    % \draw (3.2,-1.6) -- (4,-1.6);
    % \draw (3.2,-2.4) -- (4,-2.4);

    % \draw (3.2,-3.2) -- (4,-3.2);
    % \draw (3.2,-4) -- (4,-4);

    %vertical
    \draw (0,0) -- (0,-0.8);
    \draw (0.8,0) -- (0.8,-1.6);
    \draw (1.6,0) -- (1.6,-2.4);
    \draw (2.4,0) -- (2.4,-2.4);

    % \draw (3.2,0) -- (3.2,-2.4);
    % \draw (4,0) -- (4,-2.4);

    % \draw (3.2,-3.2) -- (3.2,-4);
    % \draw (4,-3.2) -- (4,-4);

    %nij labels
    \node at (0.4,-0.4) {\small{$n_{12}$}};

    \node at (1.2,-0.4) {\small{$n_{13}$}};
    \node at (1.2,-1.2) {\small{$n_{23}$}};

    \node at (2,-0.4) {\small{$n_{14}$}};
    \node at (2,-1.2) {\small{$n_{24}$}};
    \node at (2,-2) {\small{$n_{34}$}};

    % \node at (3.6,-0.4) {\small{$n_{1n}$}};
    % \node at (3.6,-1.2) {\small{$n_{2n}$}};
    % \node at (3.6,-2) {\small{$n_{3n}$}};
    % \node at (3.6,-3.6) {\scriptsize{$n_{n-1, n}$}};
\end{tikzpicture}
\caption{Ideal filling for $n=4$} \label{Ideal filling for n 4}
\end{figure}

For a dominant integral weight $\lambda$ of $SL_n$, Judd in \cite{Judd2018} defined an ideal filling for $\lambda$ to be an ideal filling $\{n_{ij}\}_{1\leq i<j \leq n}$ such that $\sum n_{ij}\alpha_{ij} = \lambda$. Unfortunately, although this definition is suitable when working with $SL_n$, it is not sufficient for $GL_n$. We take the following generalisation:

\begin{defn}\label{def ideal filling for lambda GLn version}
We say that $\{n_{ij}\}_{1\leq i<j \leq n}$ is an ideal filling for a dominant weight $\lambda$ of $GL_n$, if it is an ideal filling and $\sum n_{ij}\alpha_{ij} + \ell \sum \epsilon_i  = \lambda$, where $\ell := \frac1n \sum \lambda_i$.
\end{defn}

It is worth noting that this is the same $\ell$ comes up in the weight matrix at a critical point (see Corollary \ref{cor weight matrix at trop crit point}).

Returning our attention to critical points, we recall the conditions given in (\ref{eqn crit pt conditions}), which define critical points in the fibre over some $d\in T^{\vee}$:
    $$ \sum_{a:h(a)=v} r_a = \sum_{a:t(a)=v} r_a \quad \text{for} \ v \in \mathcal{V}^{\bullet}.
    $$
Working over the field of generalised Puiseux series we consider critical points of the superpotential in the fibre over some $t^{\lambda}\in T^{\vee}(\mathbf{K}_{>0})$ for a dominant weight $\lambda$. % \in {X^*(T)}^+$.
Tropicalising the above expression, and writing $\rho_a:=\mathrm{Val}_{\mathbf{K}}(r_a)$ following our notational convention, we obtain the tropical critical point conditions:
    \begin{equation} \label{eqn trop crit point conds}
    \min_{a:h(a)=v}\{\rho_a\} = \min_{a:t(a)=v}\{\rho_a\} \quad \text{for} \ v \in \mathcal{V}^{\bullet}.
    \end{equation}
For fixed $\lambda$ this system has a unique solution (\cite[Proposition 5.6]{Judd2018}), given by the valuation of the critical point. We will often refer to (\ref{eqn trop crit point conds}) as the tropical critical point conditions for $\lambda$, or with highest weight $\lambda$, to highlight the representation theoretic connection.

With these definitions in mind we now give the main result of this section. It is an adaptation of a proposition given by Judd in \cite{Judd2018}, extending it from the $SL_n$ case to the $GL_n$ case.

\begin{prop}[Generalisation of {\cite[Proposition 6.2]{Judd2018}}] \label{prop GLn version of Jamie's 6.2}
Let $\lambda$ be a dominant weight of $GL_n$ and $\ell := \frac1n \sum \lambda_i$.
For any solution $(\rho_a)$ to the tropical critical point conditions for $\lambda$, (\ref{eqn trop crit point conds}), the formula
    $$n_{ij}= \min_{a:h(a)=v_{ji}}\{\rho_a\} = \min_{a:t(a)=v_{ji}}\{\rho_a\}
    $$
defines an ideal filling for $\lambda$ (see Definition \ref{def ideal filling for lambda GLn version}), and every ideal filling arises in this way. In particular we see that for a given $\lambda$, the ideal filling for $\lambda$ exists and is unique.
% Then we have a bijective correspondence:
%     $$\begin{Bmatrix} \text{solutions to the tropical critical}\\ \text{conditions (\ref{eqn trop crit point conds}) with highest weight } \lambda \end{Bmatrix} \leftrightarrow \begin{Bmatrix} \text{ideal fillings } \{n_{ij}\}_{1\leq i<j \leq n} \text{ for }\lambda, \\ \text{i.e. such that } \sum n_{ij}\alpha_{ij} + \ell \sum \epsilon_i  = \lambda \end{Bmatrix}
%     $$
\end{prop}

\begin{proof} Judd, in his proof of \cite[Proposition 6.2]{Judd2018}, defines a pair of maps between two sets which are inverse to each other. We will follow the majority of his proof, in order to form a bijective correspondence:
    $$\begin{Bmatrix} \text{solutions to the tropical critical}\\ \text{conditions (\ref{eqn trop crit point conds}) with highest weight } \lambda \end{Bmatrix} \leftrightarrow \begin{Bmatrix} \text{ideal fillings } \{n_{ij}\}_{1\leq i<j \leq n} \text{ for }\lambda, \\ \text{i.e. such that } \sum n_{ij}\alpha_{ij} + \ell \sum \epsilon_i  = \lambda \end{Bmatrix}
    $$
Consequently, it suffices to simply give an outline of the proof which highlights the necessary generalisations. The exception to this is a new proof that the filling we construct (from the solutions to the tropical critical point conditions) is indeed an ideal filling.

\medskip

\noindent
\textbf{Map from ideal fillings for $\lambda$ to solutions to the tropical critical conditions.}

Let $\{n_{ij}\}_{1\leq i<j\leq n}$ be an ideal filling for $\lambda$. For each pair $(i,j)$ such that $1\leq i\leq j\leq n$, we define two sums of entries of the ideal filling; roughly speaking, those $n_{il}$ strictly to the right of $n_{ij}$ and those $n_{lj}$ strictly above $n_{ij}$ respectively:
    \begin{equation} \label{eqn Hh Hv definitions}
    H^h_{ij}:= \sum_{l>j}n_{il}, \quad H^v_{ij}:= \sum_{l<i}n_{lj}.
    \end{equation}
Making the first adaption of Judd's proof; for $\ell = \frac{1}{n} \sum \lambda_i$, we define a map from ideal fillings for $\lambda$ to tropical vertex coordinates of the quiver as follows:
    $$\delta_{v_{ji}}:=H^h_{ij}-H^v_{ij}+\ell.
    $$
We need to show that this defines a solution to the tropical critical conditions for $\lambda$.

Of note, the addition of $\ell$ in the above definition doesn't affect the tropical arrow coordinates. Indeed, computing the corresponding vertical arrow coordinates for $1\leq i \leq j < n$, and the horizontal arrow coordinates for $1\leq i < j \leq n$, we respectively obtain
    $$\delta_{v_{ji}}-\delta_{v_{j+1,i}}=H^v_{i+1,j+1} - H^v_{ij}
    \quad \text{and} \quad
    \delta_{v_{ji}}-\delta_{v_{j,i+1}}=H^h_{i,j-1} - H^h_{i+1,j}.
    $$
Both of these are $\geq 0$, so it follows that the point lies in $\left\{\mathrm{Trop}(\mathcal{W}_{t^{\lambda}})\geq 0\right\}$. Additionally, we see it will lie in the fibre over $\lambda$ as follows: for $\epsilon^{\vee}_k \in X^*(T^{\vee})$ we have
    $$\begin{aligned}
    \lambda_k = \left\langle \lambda, \epsilon^{\vee}_k \right\rangle
        &= \bigg\langle \sum_{1\leq i < j \leq n} n_{ij}(\epsilon_i-\epsilon_j) + \ell \sum_{1\leq i \leq n}\epsilon_i , \ \epsilon^{\vee}_k \ \bigg\rangle & \text{since } \{n_{ij}\} \text{ is an ideal filling for } \lambda \\
    &= H^h_{kk}-H^v_{kk}+\ell = \delta_{v_{kk}}.
    \end{aligned}
    $$

It remains to show that the point we have defined satisfies the tropical critical point conditions. Following Judd, we require a lemma:
\begin{lem}[{\cite[Lemma 6.7]{Judd2018}}] \label{lem Hh is Hh or Hv is Hv}
For $1 \leq i < j \leq n$, write $\bar{H}^h_{ij}:= H^h_{ij}+n_{ij}$, $\bar{H}^v_{ij}:= H^v_{ij} + n_{ij}$.
Then if $j-i\geq 1$, either
    $$\bar{H}^v_{i,j+1}=\bar{H}^v_{ij} \quad \text{or} \quad \bar{H}^h_{i,j+1}=\bar{H}^h_{i+1,j+1}
    $$
or both are true. Hence we have $\min\left\{ \bar{H}^v_{i,j+1} - \bar{H}^v_{ij}, \bar{H}^h_{i,j+1}-\bar{H}^h_{i+1,j+1} \right\} = 0$.
\end{lem}
We may use this lemma directly and so omit the proof. Let $v_{ji} \in \mathcal{V}^{\bullet}$ with $1<i<j<n$, that is, $v_{ji}$ doesn't lie on either wall of the quiver. Then the minimum over incoming arrows at $v_{ji}$ is
    $$\min\left\{ H^v_{i+1,j+1} - H^v_{ij}, H^h_{i,j-1}-H^h_{i+1,j} \right\} = n_{ij} + \min\left\{ \bar{H}^v_{i,j+1} - \bar{H}^v_{ij}, \bar{H}^h_{i,j+1}-\bar{H}^h_{i+1,j+1} \right\} = n_{ij}.
    $$
Similarly the minimum over outgoing arrows at $v_{ji}$ is
    $$\min\left\{ H^v_{i+1,j} - H^v_{i,j-1}, H^h_{i-1,j-1}-H^h_{ij} \right\} = n_{ij} + \min\left\{ \bar{H}^v_{i-1,j} - \bar{H}^v_{i-1,j-1}, \bar{H}^h_{i-1,j}-\bar{H}^h_{ij} \right\} = n_{ij}.
    $$
Thus the tropical critical point conditions are satisfied in this case. Finally, if $v_{ij}$ lies on the left wall there is only one outgoing arrow, $H^v_{2,j}-H^v_{1,j-1}=n_{1j}$, and if it lies on the bottom wall there is only one incoming arrow, $H^h_{i,n-1}-H^h_{i+1,n}=n_{in}$. Thus our point is indeed a tropical critical point for $\lambda$, as required.

\medskip

\noindent
\textbf{Map from solutions to the tropical critical conditions to ideal fillings.}

Suppose $(\rho_a)_{\ \in \mathcal{A}}$ is a solution to the tropical critical conditions for $\lambda$. Then for $v \in \mathcal{V}^{\bullet}$ we define the map
    $$\pi : \mathcal{V}^{\bullet} \to \mathbb{R}, \quad \pi(v):=\min_{a:h(a)=v}\{\rho_a\}
    .
    $$
We will first deviate from Judd's work to give an alternative proof that $\{ n_{ij}=\pi(v_{ij}) \}_{1\leq i < j \leq n}$ defines an ideal filling. Then we follow his proof to see that this is an ideal filling for $\lambda$.

\begin{lem} \label{lem trop crit pt is an ideal filling}
At a tropical critical point, the filling $\{n_{ij} = \pi(v_{ji})\}$ is an ideal filling. That is, if we have the following sub-diagram
\begin{center}
\begin{tikzpicture}
    %dots and stars
    \node (31) at (0,0) {$\bullet$};
        \node at (-0.2,-0.2) {\scriptsize{$v$}};
    \node (21) at (0,1) {$\bullet$};
        \node at (-0.2,1.2) {\scriptsize{$w$}};
    \node (32) at (1,0) {$\bullet$};
        \node at (1.2,-0.2) {\scriptsize{$u$}};

    %arrows
    \draw[->] (31) -- (21);
    \draw[->] (32) -- (31);

    %arrow labels
    \node at (-0.2,0.5) {\scriptsize{$a$}};
    \node at (0.5,-0.2) {\scriptsize{$b$}};
\end{tikzpicture}
\end{center}
then we must have $\pi(v)=\max\{ \pi(u), \pi(w)\}$.
\end{lem}

\begin{proof}
We show by induction that $\pi(t(b))\leq \pi(h(b))$ for each horizontal arrow $b$ and we have $\pi(h(a))\leq \pi(t(a))$ for each vertical arrow $a$.

First we consider the arrows in the bottom and left hand walls, described in Figure \ref{fig Arrows in the bottom and left hand walls of the quiver}.
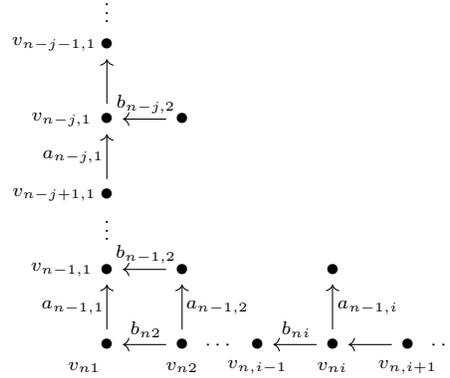
\begin{figure}[ht!]
\centering
\begin{tikzpicture}
    %dots and stars
    \node (n1) at (0,0) {$\bullet$};
        \node at (-0.3,-0.3) {\scriptsize{$v_{n1}$}};
    \node (n2) at (1,0) {$\bullet$};
        \node at (1,-0.3) {\scriptsize{$v_{n2}$}};
            \node at (1.5,0) {\scriptsize{$\cdots$}};
    \node (ni-1) at (2,0) {$\bullet$};
        \node at (2,-0.3) {\scriptsize{$v_{n,i-1}$}};
    \node (ni) at (3,0) {$\bullet$};
        \node at (3,-0.3) {\scriptsize{$v_{ni}$}};
    \node (ni+1) at (4,0) {$\bullet$};
        \node at (4,-0.3) {\scriptsize{$v_{n,i+1}$}};
        \node at (4.5,0) {\scriptsize{$\cdots$}};

    \node (n-11) at (0,1) {$\bullet$};
        \node at (-0.6,1) {\scriptsize{$v_{n-1,1}$}};
        \node at (0,1.6) {\scriptsize{$\vdots$}};
    \node (n-j+11) at (0,2) {$\bullet$};
        \node at (-0.7,2) {\scriptsize{$v_{n-j+1,1}$}};
    \node (n-j1) at (0,3) {$\bullet$};
        \node at (-0.6,3) {\scriptsize{$v_{n-j,1}$}};
    \node (n-j-11) at (0,4) {$\bullet$};
        \node at (-0.7,4) {\scriptsize{$v_{n-j-1,1}$}};
        \node at (0,4.5) {\scriptsize{$\vdots$}};

    \node (n-12) at (1,1) {$\bullet$};
    \node (n-1i) at (3,1) {$\bullet$};

    \node (n-j2) at (1,3) {$\bullet$};

    %arrows
    \draw[->] (ni+1) -- (ni);
    \draw[->] (ni) -- (ni-1);
    \draw[->] (n2) -- (n1);

    \draw[->] (n-12) -- (n-11);
    \draw[->] (n-j2) -- (n-j1);

    \draw[->] (n1) -- (n-11);
    \draw[->] (n-j+11) -- (n-j1);
    \draw[->] (n-j1) -- (n-j-11);

    \draw[->] (n2) -- (n-12);
    \draw[->] (ni) -- (n-1i);

    %arrow labels
    \node at (-0.45,0.5) {\scriptsize{$a_{n-1,1}$}};
    \node at (1.47,0.5) {\scriptsize{$a_{n-1,2}$}};
    \node at (3.47,0.5) {\scriptsize{$a_{n-1,i}$}};

    \node at (-0.45,2.5) {\scriptsize{$a_{n-j,1}$}};

    \node at (0.52,0.2) {\scriptsize{$b_{n2}$}};
    \node at (0.52,1.2) {\scriptsize{$b_{n-1,2}$}};
    \node at (0.52,3.2) {\scriptsize{$b_{n-j,2}$}};

    \node at (2.52,0.2) {\scriptsize{$b_{ni}$}};
\end{tikzpicture}
\caption{Arrows in the bottom and left hand walls of the quiver} \label{fig Arrows in the bottom and left hand walls of the quiver}
\end{figure}

We recall that the tropical critical point conditions hold, namely $\min_{a:h(a)=v}\{\rho_a\} = \min_{a:t(a)=v}\{\rho_a\}$ for all dot vertices $v \in \mathcal{V}^{\bullet}$. Then by considering the outgoing arrows at $v_{ni}$ for $i=2, \ldots, n-1$, we see that
    $$\pi(v_{ni}) = \min\{ \rho_{a_{n-1,i}}, \rho_{b_{ni}} \} \leq \rho_{b_{ni}} = \pi(v_{n,i-1}).
    $$
Similarly considering the incoming arrows at $v_{n-j,1}$ for $j=1, \ldots, n-2$, we have
    $$\pi(v_{n-j,1}) = \min\{ \rho_{a_{n-j,1}}, \rho_{b_{n-j,2}} \} \leq \rho_{a_{n-j,1}} = \pi(v_{n-j+1,1}).
    $$

For the inductive step we will show that if we have a sub-diagram like the one in Figure \ref{fig Subquiver for the proof of Lem trop crit pt is an ideal filling},
such that $\pi(u)\leq\pi(v)$ and $\pi(w) \leq \pi(v)$, then $\pi(x)\leq\pi(w)$ and $\pi(x) \leq \pi(u)$.
\begin{figure}[hb!]
\centering
\begin{tikzpicture}
    %dots and stars
    \node (31) at (0,0) {$\bullet$};
        \node at (-0.2,-0.2) {\scriptsize{$v$}};
    \node (21) at (0,1) {$\bullet$};
        \node at (-0.2,1.2) {\scriptsize{$w$}};
    \node (32) at (1,0) {$\bullet$};
        \node at (1.2,-0.2) {\scriptsize{$u$}};
    \node (22) at (1,1) {$\bullet$};
        \node at (1.2,1.2) {\scriptsize{$x$}};

    %arrows
    \draw[->] (31) -- (21);
    \draw[->] (32) -- (22);

    \draw[->] (22) -- (21);
    \draw[->] (32) -- (31);

    %arrow labels
    \node at (-0.2,0.5) {\scriptsize{$a$}};
    \node at (1.2,0.5) {\scriptsize{$d$}};

    \node at (0.5,-0.2) {\scriptsize{$b$}};
    \node at (0.5,1.2) {\scriptsize{$c$}};
\end{tikzpicture}
\caption{Subquiver for the proof of Lemma \ref{lem trop crit pt is an ideal filling}} \label{fig Subquiver for the proof of Lem trop crit pt is an ideal filling}
\end{figure}
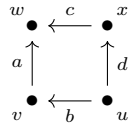

If $\rho_a \leq \rho_c$ then by the tropical box relation $\rho_a+\rho_b=\rho_c+\rho_d$ we have $\rho_d \leq \rho_b$. This means that $\pi(w)=\rho_a$ and $\pi(u)=\rho_d$, so
    \begin{equation} \label{eqn a leq c trop crit pt is ideal filling} \pi(x) \leq \rho_d = \pi(u) \leq \pi(v) \leq \rho_a = \pi(w).
    \end{equation}
Similarly if $\rho_a\geq \rho_c$ then we have $\rho_d \geq \rho_b$. This means that $\pi(w)=\rho_c$ and $\pi(u)=\rho_b$, so
    \begin{equation} \label{eqn a geq c trop crit pt is ideal filling} \pi(x) \leq \rho_c = \pi(w) \leq \pi(v) \leq \rho_b = \pi(u).
    \end{equation}
In both cases we have $\pi(x)\leq\pi(w)$ and $\pi(x) \leq \pi(u)$ as desired.

Finally we note that $\rho_a\leq \rho_c$ implies $\pi(v)=\pi(w) \geq \pi(u)$ since by (\ref{eqn a leq c trop crit pt is ideal filling}) and the inductive assumption we have
    $$ \pi(u)\leq \pi(v) \leq \pi(w) \leq \pi(v).
    $$
Similarly if $\rho_a \geq \rho_c$ then $\pi(v)=\pi(u) \geq \pi(w)$ as a consequence of (\ref{eqn a geq c trop crit pt is ideal filling}).
This completes the proof of Lemma \ref{lem trop crit pt is an ideal filling}.
\end{proof}

Now, again following Judd, we will show that $\{ n_{ij}=\pi(v_{ij}) \}_{1\leq i < j \leq n}$ is an ideal filling for $\lambda$. To do this, we need the vertex coordinates of the quiver at the tropical critical point, which we denote by $(\delta_v)_{v \in \mathcal{V}}$. In particular we notice that at the bottom left vertex we have
    \begin{align*}
    \delta_{v_{n1}}
    &= \mathrm{Val}_{\mathbf{K}}\left(\frac{\Xi_n}{\Xi_{n+1}}\right) &\text{by definition of } \Xi_i \text{ given in (\ref{eqn xi for wt map defn}), and noting } \Xi_{n+1}=1 \\
    &= \mathrm{Val}_{\mathbf{K}}(t_n) & \text{recalling the $t_i$ defined in (\ref{eqn defn gamma and t})} \\
    &= \ell & \text{by Corollary \ref{cor weight matrix at trop crit point}.}
    \end{align*}
We require a slight generalisation here:
\begin{lem}[Generalisation of {\cite[Lemma 6.9]{Judd2018}}] \label{lem Generalisation of Judd's Lemma 6.9}
For $v \in \mathcal{V}$ we write $\mathrm{bel}(v)$ and $\mathrm{lef}(v)$ for the sets of vertices directly below and directly to the left of $v$ respectively. Then at a tropical critical point we have
    $$\delta_v=\sum_{w \in \mathrm{bel}(v)} \pi(w) - \sum_{w \in \mathrm{lef}(v)} \pi(w) +\ell.
    $$
\end{lem}

\begin{proof}
The proof is by induction on the horizontal and vertical arrows. The initial case is the bottom left vertex, which we have already seen to take the value $\ell$ at critical points, as required. Since the horizontal and vertical inductive steps are similar it suffices to only consider the horizontal case; if we take the subquiver in Figure \ref{fig one arrow quiver}
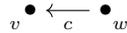
\begin{figure}[hb!]
\centering
\begin{tikzpicture}
    %dots and stars
    \node (31) at (0,0) {$\bullet$};
        \node at (-0.2,-0.2) {\scriptsize{$v$}};
    \node (32) at (1,0) {$\bullet$};
        \node at (1.2,-0.2) {\scriptsize{$w$}};

    %arrows
    \draw[->] (32) -- (31);

    %arrow labels
    \node at (0.5,-0.2) {\scriptsize{$c$}};
\end{tikzpicture}
\caption{Subquiver for induction, horizontal case} \label{fig one arrow quiver}
\end{figure}
such that the relation in the statement of the lemma holds for $v$, then we need to show it also holds for $w$. To do this, we consider the subquiver in Figure \ref{fig Subquiver for the proof of Lem Generalisation of Judd's Lemma 6.9}.
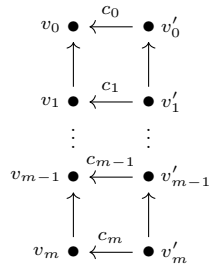
\begin{figure}[hb!]
\centering
\begin{tikzpicture}
    %dots and stars
    \node (n-j1) at (0,3) {$\bullet$};
        \node at (-0.3,3) {\scriptsize{$v_0$}};
    \node (n-j+11) at (0,2) {$\bullet$};
        \node at (-0.3,2) {\scriptsize{$v_1$}};
        \node at (0,1.6) {\scriptsize{$\vdots$}};
    \node (n-11) at (0,1) {$\bullet$};
        \node at (-0.5,1) {\scriptsize{$v_{m-1}$}};
    \node (n1) at (0,0) {$\bullet$};
        \node at (-0.35,0) {\scriptsize{$v_m$}};

    \node (n-j2) at (1,3) {$\bullet$};
        \node at (1.3,3) {\scriptsize{$v'_0$}};
    \node (n-j+12) at (1,2) {$\bullet$};
        \node at (1.3,2) {\scriptsize{$v'_1$}};
    \node at (1,1.6) {\scriptsize{$\vdots$}};
    \node (n-12) at (1,1) {$\bullet$};
        \node at (1.5,1) {\scriptsize{$v'_{m-1}$}};
    \node (n2) at (1,0) {$\bullet$};
        \node at (1.35,0) {\scriptsize{$v'_m$}};

    %horizontal arrows
    \draw[->] (n-j2) -- (n-j1);
    \draw[->] (n-j+12) -- (n-j+11);
    \draw[->] (n-12) -- (n-11);
    \draw[->] (n2) -- (n1);

    %vertical arrows
    \draw[->] (n-j+11) -- (n-j1);
    \draw[->] (n1) -- (n-11);

    \draw[->] (n-j+12) -- (n-j2);
    \draw[->] (n2) -- (n-12);

    %arrow labels
    \node at (0.5,3.2) {\scriptsize{$c_0$}};
    \node at (0.5,2.2) {\scriptsize{$c_1$}};
    \node at (0.5,1.2) {\scriptsize{$c_{m-1}$}};
    \node at (0.5,0.2) {\scriptsize{$c_m$}};
\end{tikzpicture}
\caption{Subquiver for the proof of Lemma \ref{lem Generalisation of Judd's Lemma 6.9}} \label{fig Subquiver for the proof of Lem Generalisation of Judd's Lemma 6.9}
\end{figure}

We suppose this is part of the full diagram which depicts a solution to the tropical critical conditions, and that $c_m$ lies on the bottom wall of this full quiver. Then, following Judd, we claim:
    \begin{equation}\label{eqn sum pi quiver vertices claim}
    \pi(v_0)+\pi(v_1)+ \cdots +\pi(v_m) = \rho_{c_0} +\pi(v'_1)+ \cdots +\pi(v'_m).
    \end{equation}
This can be proved by induction and, since it is unaffected by our addition of $\ell$ in the statement of Lemma \ref{lem Generalisation of Judd's Lemma 6.9}, we refer the reader to \cite{Judd2018} for the details.

Now, returning to Figure \ref{fig one arrow quiver}, we use the identity (\ref{eqn sum pi quiver vertices claim}) with $v_0=v$, $c_0=c$ and $v_0'=w$ to see that we have
    $$\pi(v)+ \sum_{u \in \mathrm{bel}(v)} \pi(u) = \rho_c + \sum_{u \in \mathrm{bel}(w)} \pi(u).
    $$
Using this we obtain the desired result, namely that the relation in the lemma holds for $w$:
    \begin{align*}
    \sum_{u \in \mathrm{bel}(w)} \pi(u) - \sum_{u \in \mathrm{lef}(w)} \pi(u) +\ell
    &= -\rho_c + \sum_{u \in \mathrm{bel}(v)} \pi(u) +\pi(v) - \sum_{u \in \mathrm{lef}(w)} \pi(u) +\ell \\
    &= -\rho_c + \sum_{u \in \mathrm{bel}(v)} \pi(u) - \sum_{u \in \mathrm{lef}(v)} \pi(u) +\ell \\
    &= -\rho_c + \delta_v \\
    &= \delta_w.
    \end{align*}
\end{proof}

Using this lemma we see that at a tropical critical point, the ideal filling $\{ n_{ij}=\pi(v_{ij}) \}$ is an ideal filling for $\lambda$:
    $$\begin{aligned}
    \bigg\langle \sum_{1\leq i < j \leq n} n_{ij}\alpha_{ij} +\ell \sum_{1\leq i \leq n}\epsilon_i, \ \epsilon^{\vee}_k \bigg\rangle
    &=  \sum_{k < j \leq n} n_{kj} - \sum_{1\leq i < k} n_{ik} +\ell \\
    &= \sum_{k < l \leq n} \pi(v_{lk}) - \sum_{1\leq l < k} \pi(v_{kl}) + \ell \\
    &= \sum_{w \in \mathrm{bel}(v_{kk})} \pi(w) - \sum_{w \in \mathrm{lef}(v_{kk})} \pi(w) + \ell \\
    &= \delta_{v_{kk}} = \lambda_k.
    \end{aligned}$$

To complete the proof of Proposition \ref{prop GLn version of Jamie's 6.2}, we note that the maps defined above are inverse to each other by construction.
\end{proof}

\begin{rem}
The correspondence in the above proposition (\ref{prop GLn version of Jamie's 6.2}) preserves integrality if $\lambda$ and $\ell$ are both integral.
This proposition also implies that there is a unique ideal filling for $\lambda$ due to the uniqueness of the tropical critical point.
\end{rem}

\begin{cor} \label{cor nu'_i coords ideal filling for lambda}
For a dominant weight $\lambda$, let the positive critical point $p_{\lambda} \in Z_{t^{\lambda}}(\mathbf{K}_{>0})$ of $\mathcal{W}_{t^{\lambda}}$ be written in the ideal coordinates $(m_1, \ldots, m_N)$. Then the valuations $\mu_k=\mathrm{Val}_{\mathbf{K}}(m_k)$ defining the tropical critical point, $p_{\lambda, \boldsymbol{\mu}}^{\mathrm{trop}}$, give rise to an ideal filling $\left\{n_{ij}= \mu_{s_i+j-i} \right\}_{1\leq i < j \leq n}$ for $\lambda$ (where we recall the definition of $s_i$ given in Section \ref{sec The ideal coordinates}).
\end{cor}

\begin{proof}
By Proposition \ref{prop crit points, sum at vertex is nu_i}, at a critical point we have
    $$ \varpi(v_{ji}) = \sum_{a:t(a)=v_{ji}} r_a = m_{s_i +j-i}.$$
Thus by Proposition \ref{prop GLn version of Jamie's 6.2} we see that
    $$n_{ij} = \pi(v_{ji}) =\mathrm{Val}_{\mathbf{K}}(\varpi(v_{ji})) = \mathrm{Val}_{\mathbf{K}}(m_{s_i +j-i}) = \mu_{s_i+j-i}
    $$
defines an ideal filling for $\lambda$.
\end{proof}

\begin{ex} \label{ex trop crit pt from ideal filling dim 3}

By Proposition \ref{prop GLn version of Jamie's 6.2} we have a one to one correspondence between solutions to the tropical critical conditions and ideal fillings for $\lambda = (\lambda_1\geq\lambda_2\geq\lambda_3)$. In this example we show that given an ideal filling in dimension 3, imposing the condition that this ideal filling is an ideal filling for $\lambda$ is the same as restricting our attention to those points with weight $(\ell, \ell, \ell)$ where $\ell=\frac13 \sum \lambda_i$ (see Corollary \ref{cor weight matrix at trop crit point}). Moreover, this will aid our geometric intuition.

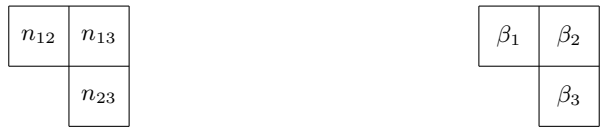
\begin{figure}[ht!]
\centering
    \begin{minipage}[b]{0.45\linewidth}
    \centering
    \begin{tikzpicture}
        %box lines
        %horizontal
        \draw (0,0) -- (1.6,0);
        \draw (0,-0.8) -- (1.6,-0.8);
        \draw (0.8,-1.6) -- (1.6,-1.6);

        %vertical
        \draw (0,0) -- (0,-0.8);
        \draw (0.8,0) -- (0.8,-1.6);
        \draw (1.6,0) -- (1.6,-1.6);

        %nij labels
        \node at (0.4,-0.4) {\small{$n_{12}$}};

        \node at (1.2,-0.4) {\small{$n_{13}$}};
        \node at (1.2,-1.2) {\small{$n_{23}$}};
    \end{tikzpicture}
    \end{minipage}
\hspace{-1cm}
    \begin{minipage}[b]{0.45\linewidth}
    \centering
    \begin{tikzpicture}
        %box lines
        %horizontal
        \draw (0,0) -- (1.6,0);
        \draw (0,-0.8) -- (1.6,-0.8);
        \draw (0.8,-1.6) -- (1.6,-1.6);

        %vertical
        \draw (0,0) -- (0,-0.8);
        \draw (0.8,0) -- (0.8,-1.6);
        \draw (1.6,0) -- (1.6,-1.6);

        %nij labels
        \node at (0.4,-0.4) {\small{$\beta_1$}};

        \node at (1.2,-0.4) {\small{$\beta_2$}};
        \node at (1.2,-1.2) {\small{$\beta_3$}};
    \end{tikzpicture}
    \end{minipage}
\caption{Fillings in dimension $3$} \label{fig Fillings in dimension 3}
\end{figure}

The filling in dimension $3$ is given in Figure \ref{fig Fillings in dimension 3}, we will write $\beta_1=n_{12}$, $\beta_2=n_{13}$, $\beta_3=n_{23}$.
The ideal filling condition $\beta_2 = \max\{\beta_1, \beta_3\}$ defines the following piecewise-linear subspace:
    \begin{equation} \label{eqn dim 3 ideal filling subpace}
    \{\beta_2=\beta_1\geq \beta_3\}\cup \{\beta_2=\beta_3\geq \beta_1\}.
    \end{equation}
The condition that the ideal filling is an ideal filling for $\lambda$ is the following:
    $$
    \lambda = \sum n_{ij}\alpha_{ij} + \ell \sum \epsilon_i \nonumber = \beta_1 \alpha_{12} + \beta_2 \alpha_{13} + \beta_3\alpha_{23} + \ell  \sum \epsilon_i \nonumber = \left( \ell + \beta_1+\beta_2, \ell -\beta_1+\beta_3, \ell -\beta_2-\beta_3 \right).
    $$
This gives a further set of constraints on the $c_i$. These additional constraints may also be obtained by setting the weight of a point $(\beta_1,\beta_2,\beta_3)$ equal to the weight of the positive critical point;
    $$(\lambda_3 + \beta_2 + \beta_3, \lambda_2 + \beta_1 - \beta_3, \lambda_1 - \beta_1 - \beta_2 )=( \ell, \ell, \ell ).$$
Intersecting this condition with (\ref{eqn dim 3 ideal filling subpace}) we find exactly two possibilities for the tropical critical point, depending on which of $\lambda_1-\lambda_2$ or $\lambda_2-\lambda_3$ is greater. By Proposition \ref{prop GLn version of Jamie's 6.2} this point lies within the superpotential polytope. We obtain:
\begin{itemize}
    \item If $\beta_1=\beta_2$ then $\lambda_1- 2\beta_1 = \ell$ and $\lambda_3+\beta_1+\beta_3 = \lambda_2 +\beta_1 - \beta_3$. So
    $$(\beta_1,\beta_2,\beta_3)=\left( \frac16 \left( 2\lambda_1 -\lambda_2-\lambda_3 \right), \frac16 \left( 2\lambda_1 -\lambda_2-\lambda_3 \right), \frac12 \left(\lambda_2-\lambda_3\right) \right).$$

    \item If $\beta_2=\beta_3$ then $\lambda_3 + 2\beta_2 = \ell $ and $ \lambda_2+\beta_1-\beta_2 = \lambda_1 -\beta_1 - \beta_2$. So
    $$(\beta_1,\beta_2,\beta_3)=\left( \frac12 \left(\lambda_1-\lambda_2\right), \frac16 \left( \lambda_1 +\lambda_2-2\lambda_3 \right), \frac16 \left( \lambda_1 +\lambda_2-2\lambda_3 \right) \right).
    $$
\end{itemize}

Computing the tropical critical point coordinates in this way is much quicker and easier than solving the simultaneous equations given by the tropical critical point conditions. Additionally it is now unsurprising that there is only one ideal filling for $\lambda$ in dimension $3$, since we are intersecting a the piecewise-linear 2-dimensional subspace and a line.

Alternatively we could obtain the same result using the proof of Proposition \ref{prop GLn version of Jamie's 6.2}.
The benefit is that, in addition to the tropical critical point, we would also gain vertex coordinates (and thus arrow coordinates) for the quiver. Unfortunately however, since this approach is algorithmic we lose some of the more visual interpretation.
\end{ex}

\subsection{A family of ideal polytopes} \label{subsec A family of ideal polytopes}
\fancyhead[L]{4.5 \ \ A family of ideal polytopes}
% \fancyhead[L]{5.5 \ \ A family of ideal polytopes}

In this section we define a family of polytopes which contains the ideal polytope $\mathcal{P}_{\lambda,\boldsymbol{\mu}}$. We do this by extending the definition of $\mathcal{P}_{\lambda,\boldsymbol{\mu}}$ to general reduced expressions $\mathbf{i}$ for $\bar{w}_0$.

Firstly, recalling the construction of the ideal coordinates defined in Section \ref{sec The ideal coordinates}, we begin our generalisation by by taking an arbitrary reduced expression $\mathbf{i}=(i_1, \ldots, i_N)$ for $w_0$ and considering the map
    \begin{equation*}
    \tilde{\psi}_{\mathbf{i}} : \left(\mathbb{K}^*\right)^N \times T^{\vee} \to Z\,, \quad \left(\left(m'_1, \ldots, m'_N\right), t_R\right) \mapsto \mathbf{y}_{i_1}^{\vee}\left(\frac{1}{m'_1}\right) \cdots \mathbf{y}_{i_N}^{\vee}\left(\frac{1}{m'_N}\right) t_R.
    \end{equation*}
We again wish to work with the highest weight as opposed to the weight, that is, coordinates $(d,\boldsymbol{m}')$ instead of $(\boldsymbol{m}',t_R)$, however this requires us to develop our description of the weight map. To do so it will be better to index the coordinates $\boldsymbol{m}'$ by positive roots as follows:

We recall that any reduced expression $\mathbf{i}=(i_1, \ldots, i_N)$ for $w_0$ determines an ordering on the set of positive roots $R_+$ by
    \begin{equation} \label{eqn ordering on R plus}
    \alpha^{\mathbf{i}}_j=\begin{cases}
    \alpha_{i_1} & \text{for } j=1, \\
    s_{i_1}\cdots s_{i_{j-1}}\alpha_{i_{j}} & \text{for } j=2, \ldots N.
    \end{cases}
    \end{equation}
This ordering has the property that whenever $\alpha,  \beta \in R_+$ are positive roots such that $\alpha + \beta \in R_+$, then $\alpha + \beta$ must occur in between $\alpha$ and $\beta$.
We use the ordering on $R_+$ defined by $\mathbf{i}$ to identify $\left(\mathbb{K}^*\right)^N$ with $\left(\mathbb{K}^*\right)^{R_+}$, namely $m'_j=m'_{\alpha^{\mathbf{i}}_j}$. We will write $\left(\mathbb{K}^*\right)^N_{\mathbf{i}}$, $\left(\mathbb{K}^*\right)^{R_+}_{\mathbf{i}}$ when we need to explicitly state which reduced expression we are using.

We also recall the classic arrangement of positive roots $\alpha_{ij}=\epsilon_i-\epsilon_j$, $i<j$, similar to a strictly upper triangular matrix:
    \begin{equation} \label{eqn alpha ij arrangement}
    \begin{matrix}
    \alpha_{12} & \alpha_{13} & \cdots & \alpha_{1,n-1} & \alpha_{1n} \\
    & \alpha_{23} & \cdots & \alpha_{2,n-1} & \alpha_{2n} \\
    & & \ddots & \vdots & \vdots \\
    & & & \alpha_{n-2,n-1} & \alpha_{n-2,n} \\
    & & & & \alpha_{n-1,n}
    \end{matrix}
    \end{equation}
In particular, this takes the same form as (ideal) fillings, for which we have the natural bijective correspondence
    \begin{equation} \label{eqn rel nij alpha ij}
    n_{ij} \leftrightarrow \alpha_{ij}.
    \end{equation}
We note that if $\alpha,  \beta \in R_+$ are positive roots such that $\alpha + \beta \in R_+$, then we must have $\alpha+\beta$ appearing either to the right of $\alpha$ and above $\beta$, or to the right of $\beta$ and above $\alpha$. This is a consequence of the fact that if $\alpha=\alpha_{ij}$ and $\beta=\alpha_{kl}$ then, in order for their sum $\alpha+\beta=\alpha_{ij}+\alpha_{kl}$ to be a positive root, we must have either $i=l$ or $j=k$ (resulting in $\alpha+\beta=\alpha_{kj}$ or $\alpha+\beta=\alpha_{il}$ respectively).

\begin{defn}[Universal weight map] \label{defn universal weight}
We define a map $t_R : T^{\vee} \times \left(\mathbb{K}^*\right)^{R_+} \to T^{\vee}$ by taking $t_R(d,\boldsymbol{m}')$ to be the $n\times n$ diagonal matrix with entries
    $$ \left(t_R(d,\boldsymbol{m}')\right)_{n-j+1,n-j+1} :
        = \frac{d_{j} \prod_{l=1}^{j-1} m'_{\alpha_{lj}}}{\prod_{l=j}^{n-1} m'_{\alpha_{j,l+1}}}.
    $$
We will refer to the matrix $t_R(d,\boldsymbol{m}')$ as the universal weight matrix.
\end{defn}

\begin{ex} \label{ex univ wt matrix n4}
When $n=4$ we have
    $$t_R(d,\boldsymbol{m}')=\begin{pmatrix}
    d_4 m'_{\alpha_{14}}m'_{\alpha_{24}}m'_{\alpha_{34}} & & & \\
    & d_3 \frac{m'_{\alpha_{13}}m'_{\alpha_{23}}}{m'_{\alpha_{34}}} & & \\
    & & d_2 \frac{m'_{\alpha_{12}}}{m'_{\alpha_{23}}m'_{\alpha_{24}}} & \\
    & & & d_1 \frac{1}{m'_{\alpha_{12}}m'_{\alpha_{13}}m'_{\alpha_{14}}}
    \end{pmatrix}.
    $$
\end{ex}

Now returning to our generalisation of the ideal coordinates, and recalling the construction given at the start Section \ref{sec The ideal coordinates}, we take a reduced expression $\mathbf{i}=(i_1, \ldots, i_N)$ for $w_0$ and define the ideal chart for $\mathbf{i}$ to be
    \begin{equation*} \label{eqn defn ideal chart arbitrary i}
    \psi_{\mathbf{i}} : T^{\vee} \times \left(\mathbb{K}^*\right)^{R_+} \to Z\,, \quad
    \left(d, \left(m'_{\alpha^{\mathbf{i}}_1}, \ldots, m'_{\alpha^{\mathbf{i}}_N} \right)\right) \mapsto
    \mathbf{y}_{i_1}^{\vee}\left(\frac{1}{m'_{\alpha^{\mathbf{i}}_1}}\right) \cdots \mathbf{y}_{i_N}^{\vee}\left(\frac{1}{m'_{\alpha^{\mathbf{i}}_N}}\right) t_R(d,\boldsymbol{m}')
    \end{equation*}
where $t_R(d,\boldsymbol{m}')$ is the universal weight matrix.

\begin{prop}
The universal weight map is independent of the choice of reduced expression $\mathbf{i}$ for $w_0$.
\end{prop}

\begin{proof}
Consider the reduced expression $\mathbf{i}_0= (1,2, \ldots, n-1, 1,2 \ldots, n-2, \ldots, 1, 2, 1)$. We will begin by showing that the description of the weight matrix given in Corollary \ref{cor wt matrix in m ideal coords} is the same as the universal weight matrix under the identification $m_j=m'_{\alpha^{\mathbf{i}_0}_j}$.

It is well known that the ordering (\ref{eqn ordering on R plus}) on $R_+$ given by $\mathbf{i}_0$ is
    $$\alpha_{12}, \ \alpha_{13}, \ \ldots, \  \alpha_{1n}, \ \alpha_{23}, \ \alpha_{24}, \ \ldots, \alpha_{2n}, \ \ldots, \ \alpha_{n-2,n-1}, \ \alpha_{n-2,n}, \ \alpha_{n-1,n}.
    $$
In particular, we see that $\alpha_{ij}$ appears in the $(s_i+j-i)$-th place in this sequence, where we recall the definition
    $$s_i := \sum_{k=1}^{i-1}(n-k).
    $$
This also follows from Corollary \ref{cor nu'_i coords ideal filling for lambda} and the correspondence (\ref{eqn rel nij alpha ij}). Thus $\alpha^{\mathbf{i}_0}_{s_i+(j-i)} = \alpha_{ij}$, and so $m'_{\alpha_{ij}} = m'_{\alpha^{\mathbf{i}_0}_{s_i+(j-i)}} = m_{s_i+(j-i)}$. Then recalling the description of the weight matrix from Corollary \ref{cor wt matrix in m ideal coords}, for $j=1, \ldots, n$ we have
    $$\begin{aligned}
    \left(t_R(d,\boldsymbol{m})\right)_{n-j+1, n-j+1}
        &= \frac{d_j \prod\limits_{k=1, \ldots,j-1} m_{s_k+(j-k)}}{\prod\limits_{r=1, \ldots, n-j} m_{s_j+r}} & \text{where } r=k-j \\
        &= \frac{d_j \prod_{k=1}^{j-1} m_{s_k+(j-k)} }{\prod_{k=j}^{n-1} m_{s_j+(k+1-j)}} \\
        &= \frac{d_j \prod_{k=1}^{j-1} m'_{\alpha_{kj}} }{\prod_{k=j}^{n-1} m'_{\alpha_{j,k+1}}} \\
        &= \left(t_R(d,\boldsymbol{m}')\right)_{n-j+1,n-j+1}
    \end{aligned}
    $$
and so the two descriptions agree for $\mathbf{i}=\mathbf{i}_0$.

Next we recall that any two reduced expressions $\mathbf{i}$ and $\mathbf{i}'$ for $w_0$, are related by a sequence of transformations
    \begin{alignat}{2} \label{eqn iji jij swap}
    i,j,i &\leftrightarrow j,i,j \qquad & \text{if } |i-j|=1, \\
    i,j &\leftrightarrow j,i & \text{if } |i-j|\geq 2. \nonumber
    \end{alignat}
It suffices to show that the form of the universal weight matrix is invariant under one of the transformations of the first type, (\ref{eqn iji jij swap}).

Suppose $\mathbf{i}$ and $\mathbf{i}'$ are two reduced expressions for $w_0$ which are related by a single transformation (\ref{eqn iji jij swap}) in positions $k-1, k, k+1$. Then their respective sequences of positive roots are
    \begin{equation} \label{eqn pos root sequences iji jij}
    \begin{aligned}
    \left( \alpha^{\mathbf{i}}_j \right) &= \left( \alpha^{\mathbf{i}}_1, \ldots, \alpha^{\mathbf{i}}_{k-2}, \alpha, \alpha+ \beta, \beta, \alpha^{\mathbf{i}}_{k+2}, \ldots, \alpha^{\mathbf{i}}_N  \right), \\
    \left( \alpha^{\mathbf{i}'}_j \right) &= \left( \alpha^{\mathbf{i}}_1, \ldots, \alpha^{\mathbf{i}}_{k-2}, \beta, \alpha+ \beta, \alpha, \alpha^{\mathbf{i}}_{k+2}, \ldots, \alpha^{\mathbf{i}}_N  \right).
    \end{aligned}
    \end{equation}
If $\boldsymbol{m}'$ and $\boldsymbol{m}''$ are such that $m'_j = m'_{\alpha^{\mathbf{i}}_j}$ and $m''_j = m''_{\alpha^{\mathbf{i}'}_j}$ respectively, then as a consequence of (\ref{eqn pos root sequences iji jij}) we must necessarily have
    $$\mathbf{y}_{\mathbf{i}_{k-1}}^{\vee}\left(\frac{1}{m'_{\alpha}}\right) \mathbf{y}_{\mathbf{i}_{k}}^{\vee} \left(\frac{1}{m'_{\alpha+\beta}}\right) \mathbf{y}_{\mathbf{i}_{k+1}}^{\vee}\left(\frac{1}{m'_{\beta}} \right) = \mathbf{y}_{\mathbf{i}_{k-1}}^{\vee}\left(\frac{1}{m''_{\beta}}\right) \mathbf{y}_{\mathbf{i}_{k}}^{\vee} \left(\frac{1}{m''_{\alpha+\beta}}\right) \mathbf{y}_{\mathbf{i}_{k+1}}^{\vee}\left(\frac{1}{m''_{\alpha}} \right).
    $$
Written explicitly this gives
    $$
    \begin{pmatrix} 1 & & & & & & \\ & \ddots & & & & & \\ & & 1 & & & & \\ & & \frac{m'_{\alpha}+m'_{\beta}}{m'_{\alpha}m'_{\beta}} & 1 & & & \\ & & \frac{1}{m'_{\alpha+\beta}m'_{\beta}} & \frac{1}{m'_{\alpha+\beta}} & 1 & & \\ & & & & & \ddots & \\ & & & & & & 1 \end{pmatrix}
    =\begin{pmatrix} 1 & & & & & & \\ & \ddots & & & & & \\ & & 1 & & & & \\ & & \frac{1}{m''_{\alpha+\beta}} & 1 & & & \\ & & \frac{1}{m''_{\beta}m''_{\alpha+\beta}} & \frac{m''_{\beta}+m''_{\alpha}}{m''_{\beta}m''_{\alpha}} & 1 & & \\ & & & & & \ddots & \\ & & & & & & 1 \end{pmatrix}
    $$
which defines the coordinate change:
    \begin{equation} \label{eqn alpha beta coord change iji jij}
    m''_{\alpha} = \frac{m'_{\alpha+\beta}(m'_{\alpha}+m'_{\beta})}{m'_{\beta}}, \quad m''_{\alpha+\beta} = \frac{m'_{\alpha}m'_{\beta}}{m'_{\alpha}+m'_{\beta}}, \quad m''_{\beta} = \frac{m'_{\alpha+\beta}(m'_{\alpha}+m'_{\beta})}{m'_{\alpha}},
    \end{equation}
with $m''_{\alpha^{\mathbf{i}}_j}=m'_{\alpha^{\mathbf{i}}_j}$ for all other coordinates.
In particular we have
    \begin{equation} \label{eqn prod and quo cood change m' m''}
    m''_{\alpha}m''_{\alpha+\beta}=m'_{\alpha}m'_{\alpha+\beta}, \quad m''_{\beta}m''_{\alpha+\beta} = m'_{\beta}m'_{\alpha+\beta}, \quad \frac{m''_{\alpha}}{m''_{\beta}}=\frac{m'_{\alpha}}{m'_{\beta}}.
    \end{equation}

It remains to show that the form of the universal weight matrix is unaltered by this coordinate change. Recalling the definition of the universal weight matrix (\ref{defn universal weight}) we notice that the product in the numerator (resp. denominator) has exactly one term $m'_{\alpha_{ij}}$ for every $\alpha_{ij}$ from the $(j-1)$-th column (resp. $j$-th row) of the arrangement (\ref{eqn alpha ij arrangement}).
We recall also that the root $\alpha+\beta \in R_+$ must lie either to the right of $\alpha$ and above $\beta$, or to the right of $\beta$ and above $\alpha$ in the arrangement (\ref{eqn alpha ij arrangement}). It follows then that for every $\left(t_R(d,\boldsymbol{m}')\right)_{ii}$, at most one of $m'_{\alpha}$ or $m'_{\beta}$ can appear in each of the two products in the description of this matrix entry, and the coordinate $m'_{\alpha+\beta}$ can appear in at most one of the two products.
We note that it is impossible for $m'_{\alpha}$ or $m'_{\beta}$ to appear in one of the products in $\left(t_R(d,\boldsymbol{m}')\right)_{11}$ or $\left(t_R(d,\boldsymbol{m}')\right)_{nn}$ without $m'_{\alpha+\beta}$ also appearing.

Suppose $m'_{\alpha}$ and $m'_{\alpha+\beta}$ both appear in the numerator of $\left(t_R(d,\boldsymbol{m}')\right)_{n-j+1,n-j+1}$ (the proof starting with these terms in the denominator follows similarly). Then by definition $\alpha=\alpha_{lj}$ and $\alpha+\beta=\alpha_{kj}$ for some $k<l<j$. Since $\alpha+\beta$ is a positive root we must have $\beta=\alpha_{kl}$, and so we see that $m'_{\beta}$ cannot appear in this matrix entry. Consequently, by (\ref{eqn prod and quo cood change m' m''}), the form of this matrix entry is unaffected by the coordinate change.

Now suppose $m'_{\alpha}$ appears in the numerator of $\left(t_R(d,\boldsymbol{m}')\right)_{n-j+1,n-j+1}$, but $m'_{\alpha+\beta}$ does not (the proof starting with $m'_{\alpha}$ in the denominator follows similarly). Then by definition $\alpha=\alpha_{lj}$ for some $l$. Since $m'_{\alpha+\beta}$ does not appear in the numerator but $\alpha+\beta$ is a positive root, we must have $\alpha+\beta=\alpha_{lk}$ with $k> j$. Consequently $\beta=\alpha_{jk}$ and so $m'_{\beta}$ must appear in the denominator of this matrix entry. Thus by (\ref{eqn prod and quo cood change m' m''}), the form of this matrix entry is unaffected by the coordinate change.
\end{proof}

Since we wish to define ideal polytopes corresponding to different reduced expressions $\mathbf{i}$ for $w_0$, we use the toric chart $\psi_{\mathbf{i}}$ to generalise two of the maps given in Section \ref{subsec Constructing polytopes}, namely we take
    $$\begin{aligned}
    \phi^{\mathbf{i}}_{t^{\lambda},\boldsymbol{m}'} &: (\mathbf{K}_{>0})^{R_+} \to Z_{t^{\lambda}}(\mathbf{K}_{>0}), \\
    \mathcal{W}^{\mathbf{i}}_{t^{\lambda},\boldsymbol{m}'} &: (\mathbf{K}_{>0})^{R_+} \to \mathbf{K}_{>0}
    \end{aligned}
    $$
such that $\phi_{t^{\lambda},\boldsymbol{m}'} = \phi^{\mathbf{i}_0}_{t^{\lambda},\boldsymbol{m}'}$ and $\mathcal{W}_{t^{\lambda},\boldsymbol{m}'}=\mathcal{W}^{\mathbf{i}_0}_{t^{\lambda},\boldsymbol{m}'}$, where $\lambda$ is a dominant weight. With this notation we are ready to construct our new polytopes; given an arbitrary reduced expression $\mathbf{i}$ for $w_0$, and the associated superpotential $\mathcal{W}^{\mathbf{i}}_{t^{\lambda},\boldsymbol{m}'}$ for $GL_n/B$, we define
    $$\mathcal{P}^{\mathbf{i}}_{\lambda,\boldsymbol{\mu}'} = \left\{ \boldsymbol{\alpha} \in \mathbb{R}^N_{\boldsymbol{\mu}'} \ \big| \ \mathrm{Trop}\left(\mathcal{W}^{\mathbf{i}}_{t^{\lambda},\boldsymbol{m}'}\right)(\boldsymbol{\alpha}) \geq 0 \right\}.
    $$

For the particular reduced expression $\mathbf{i}_0$, we obtain the ideal polytope from Section \ref{subsec Constructing polytopes}, namely $\mathcal{P}^{\mathbf{i}_0}_{\lambda,\boldsymbol{\mu}'} = \mathcal{P}_{\lambda,\boldsymbol{\mu}}$. Moreover we have already seen that this polytope $\mathcal{P}^{\mathbf{i}_0}_{\lambda,\boldsymbol{\mu}'}$ is simply a linear transformation of the string polytope $\mathrm{String}_{\mathbf{i}_0}(\lambda)=\mathcal{P}_{\lambda,\boldsymbol{\zeta}}$. However in general the families of string and ideal polytopes diverge:
\begin{prop} \label{prop pos birat map between toric charts for i}
Given a reduced expression $\mathbf{i}$ for $w_0$, there is a positive birational map of tori transforming the ideal coordinate chart for $\mathbf{i}$ into the string coordinate chart for $\mathbf{i}$:
    $$\begin{tikzcd}[column sep=1.5em]
    T^{\vee}\times\left(\mathbb{K}^*\right)^{R_+} \arrow{dr}[swap]{\psi_{\mathbf{i}}} \arrow[dashed]{rr}{\vartheta} && T^{\vee}\times\left(\mathbb{K}^*\right)^{N} \arrow{dl}{\varphi_{\mathbf{i}}} \\
    & Z
    \end{tikzcd}
    $$
\end{prop}

\begin{proof}
We recall the specific reduced expression
    $$\mathbf{i}_0 := (1,2, \ldots, n-1, 1,2 \ldots, n-2, \ldots, 1, 2, 1)
    $$
for $w_0$ used earlier, and define the map $\vartheta: T^{\vee}\times\left(\mathbb{K}^*\right)^{R_+}_{\mathbf{i}} \dashrightarrow T^{\vee}\times\left(\mathbb{K}^*\right)^{N}_{\mathbf{i}}$ to be the following composition:
    $$\begin{tikzcd}[row sep=0.3em]
    T^{\vee}\times\left(\mathbb{K}^*\right)^{R_+}_{\mathbf{i}} \arrow{r}{}
        & T^{\vee}\times\left(\mathbb{K}^*\right)^{N}_{\mathbf{i}} \arrow[dashed]{r}{}
        & T^{\vee}\times\left(\mathbb{K}^*\right)^{N}_{\mathbf{i}_0} \arrow{r}{}
        & T^{\vee}\times\left(\mathbb{K}^*\right)^{N}_{\mathbf{i}_0} \arrow[dashed]{r}{}
        & T^{\vee}\times\left(\mathbb{K}^*\right)^{N}_{\mathbf{i}} \\
    \left(d,\boldsymbol{m}'_{\alpha^{\mathbf{i}}}\right) \arrow[mapsto]{r}{}
        & \left(d,\boldsymbol{m}'\right) \arrow[mapsto]{r}{}
        & \left(d,\boldsymbol{m}\right) \arrow[mapsto]{r}{}
        & \left(d,\boldsymbol{z}\right) \arrow[mapsto]{r}{}
        & \left(d,\boldsymbol{z}'\right)
    \end{tikzcd}
    $$
The first map simply describes the identification between $T^{\vee}\times\left(\mathbb{K}^*\right)^{R_+}_{\mathbf{i}}$ and $T^{\vee}\times\left(\mathbb{K}^*\right)^{N}_{\mathbf{i}}$, given by the ordering (\ref{eqn ordering on R plus}) on the set of positive roots $R_+$ defined by $\mathbf{i}$, that is $m'_j=m'_{\alpha^{\mathbf{i}}_j}$. The third map is the coordinate change given in Theorem \ref{thm coord change} between the ideal and string coordinates for $\mathbf{i}_0$.

The second and fourth maps in the above composition are the necessary coordinate changes such that $\psi_{\mathbf{i}}(d, \boldsymbol{m}') = \psi(d, \boldsymbol{m})$ and $\varphi_{\mathbf{i}_0}(d, \boldsymbol{z}) = \varphi_{\mathbf{i}}(d, \boldsymbol{z}')$. The second map is given by compositions of coordinate changes similar to (\ref{eqn alpha beta coord change iji jij}) and the fourth map is defined analogously. Both are known to be positive rational maps, but in general, not isomorphisms of tori for arbitrary reduced expressions.
\end{proof}

\begin{ex}
We let $n=4$ and take $\mathbf{i}=(1,2,3,2,1,2)$, recalling that $\mathbf{i}_0=(1,2,3,1,2,1)$. This gives the ordering
    $$\alpha^{\mathbf{i}}_1=\alpha_{12} , \quad
    \alpha^{\mathbf{i}}_2=\alpha_{13} , \quad
    \alpha^{\mathbf{i}}_3=\alpha_{14} , \quad
    \alpha^{\mathbf{i}}_4=\alpha_{34} , \quad
    \alpha^{\mathbf{i}}_5=\alpha_{24} , \quad
    \alpha^{\mathbf{i}}_6=\alpha_{23}.
    $$

The coordinate changes we require, firstly between $\boldsymbol{m}'$ and $\boldsymbol{m}$, secondly between $\boldsymbol{m}$ and $\boldsymbol{z}$ (given by Theorem \ref{thm coord change}), and thirdly between $\boldsymbol{z}$ and $\boldsymbol{z}'$ are respectively as follows:
    $$\begin{aligned}
    m'_1 &= m_1 & m_1 &= z_6 \hspace{1.5cm}& z_1 &= z'_1 \\
    m'_2 &= m_2 & m_2 &= z_4 & z_2 &= z'_2 \\
    m'_3 &= m_3 & m_3 &= z_1 & z_3 &= z'_3 \\
    m'_4 &= \frac{m_5(m_4+m_6)}{m_4} \hspace{1.5cm} & m_4 &= \frac{z_5}{z_4} & z_4 &= \frac{z'_5z'_6}{z'_4z'_6+z'_5} \\
    m'_5 &= \frac{m_4m_5}{m_4+m_6} & m_5 &= \frac{z_2}{z_1} & z_5 &= z'_4z'_6 \\
    m'_6 &= \frac{m_5(m_4+m_6)}{m_6} & m_6 &= \frac{z_3}{z_2} & z_6 &= \frac{z'_4z'_6+z'_5}{z'_6} \\
    \end{aligned}
    $$
Combining the coordinate changes we have
    $$\begin{aligned}
    m'_{\alpha_{12}} & %= m'_{\alpha^{\mathbf{i}}_1} = m'_1
    = \frac{z'_4z'_6+z'_5}{z'_6} \hspace{1.5cm}&
        m'_{\alpha_{34}} & %= m'_{\alpha^{\mathbf{i}}_4} = m'_4
        = \frac{z'_2}{z'_1} + \frac{z'_3z'_5}{z'_1z'_4(z'_4z'_6+z'_5)} \\
    m'_{\alpha_{13}} & %= m'_{\alpha^{\mathbf{i}}_2} = m'_2
    = \frac{z'_5z'_6}{z'_4z'_6+z'_5} &
        m'_{\alpha_{24}} & %= m'_{\alpha^{\mathbf{i}}_5} = m'_5
        = \frac{z'_3z'_4(z'_4z'_6+z'_5)}{z'_2z'_4(z'_4z'_6+z'_5)+z'_3z'_5} \\
    m'_{\alpha_{14}} & %= m'_{\alpha^{\mathbf{i}}_3} = m'_3
    = z'_1 &
        m'_{\alpha_{23}} & %= m'_{\alpha^{\mathbf{i}}_6} = m'_6
        = \frac{z'_2z'_4z'_6}{z'_1z'_3} + \frac{z'_5z'_6}{z'_1(z'_4z'_6+z'_5)} \\
    \end{aligned}
    $$
\end{ex}

It follows from this example that, in general, given some dominant weight $\lambda$ and two reduced expressions $\mathbf{i}, \mathbf{i}'$ for $w_0$, the two polytopes $\mathcal{P}^{\mathbf{i}}_{\lambda,\boldsymbol{\mu}'}$ and $\mathcal{P}^{\mathbf{i}'}_{\lambda,\boldsymbol{\mu}''}$ are related by a piecewise-linear map. However in contrast, the respective tropical critical points $p_{\lambda, \boldsymbol{\mu}'}^{\mathrm{trop}}$ and $p_{\lambda, \boldsymbol{\mu}''}^{\mathrm{trop}}$ coincide:

\begin{prop} \label{prop trop crit pt independent of i}
For a given dominant weight $\lambda$, the tropical critical point $p_{\lambda, \boldsymbol{\mu}'}^{\mathrm{trop}}$ is independent of the choice of reduced expression $\mathbf{i}$ for $w_0$.
\end{prop}

\begin{proof}
It suffices to consider two reduced expressions $\mathbf{i}, \mathbf{i}'$ for $w_0$ that are related by a single transformation (\ref{eqn iji jij swap})
in positions $k-1,k,k+1$. Then the respective sequences of positive roots are given by (\ref{eqn pos root sequences iji jij}) and the coordinate change is given by (\ref{eqn alpha beta coord change iji jij}). We see that the tropical coordinate change is given by
    \begin{equation} \label{eqn trop coord change mu' mu''}
    \begin{aligned}
    \mu''_{\alpha} &= \mu'_{\alpha+\beta} + \min\{ \mu'_{\alpha}, \mu'_{\beta}\} - \mu'_{\beta}, \\
    \mu''_{\alpha+\beta} &= \mu'_{\alpha}+\mu'_{\beta} - \min\{ \mu'_{\alpha}, \mu'_{\beta}\}, \\
    \mu''_{\beta} &= \mu'_{\alpha+\beta} + \min\{ \mu'_{\alpha}, \mu'_{\beta}\} - \mu'_{\alpha}. \\
    \end{aligned}
    \end{equation}

Recall that the positive root $\alpha+\beta$ must appear either to the right of $\alpha$ and above $\beta$, or to the right of $\beta$ and above $\alpha$ in the arrangement (\ref{eqn alpha ij arrangement}). Consequently $\mu'_{\alpha+\beta}$ must appear either to the right of $\mu'_{\alpha}$ and above $\mu'_{\beta}$, or to the right of $\mu'_{\beta}$ and above $\mu'_{\alpha}$ in the filling
    $$\left\{ n_{ij}=\mu'_{\alpha_{ij}}=\mathrm{Val}_{\mathbf{K}}\left(m'_{\alpha_{ij}}\right)\right\}_{1\leq i < j \leq n}
    $$
(c.f. Corollary \ref{cor nu'_i coords ideal filling for lambda}), and similarly for $\mu''_{\alpha+\beta}$ in the respective filling. If we suppose that our filling is an ideal filling (for $\lambda$), then it follows that
    $$\mu'_{\alpha+\beta} = \max\{ \mu'_{\alpha}, \mu'_{\beta} \}.
    $$
Thus, by considering the tropical coordinate change (\ref{eqn trop coord change mu' mu''}), we see that at a critical point
    $$\begin{aligned}
    \mu''_{\alpha} &= \mu'_{\alpha+\beta} + \min\{ \mu'_{\alpha}, \mu'_{\beta}\} - \mu'_{\beta} = \mu'_{\alpha} , \\
    \mu''_{\alpha+\beta} &= \mu'_{\alpha}+\mu'_{\beta} - \min\{ \mu'_{\alpha}, \mu'_{\beta}\} %= \max\{\mu'_{\alpha},\mu'_{\beta}\}
        = \mu'_{\alpha+\beta} , \\
    \mu''_{\beta} &= \mu'_{\alpha+\beta} + \min\{ \mu'_{\alpha}, \mu'_{\beta}\} - \mu'_{\alpha} = \mu'_{\beta}. \\
    \end{aligned}
    $$
It follows that if we index our coordinates by positive roots then the tropical critical point is independent of the choice of reduced expression for $w_0$.
\end{proof}

%%%%%%%%%%%%%%%%%%%%%%%%%%%%%%%%%%%%%%%%%%%%%%%%%%%%%%%%%%%%%%%%%%%%%%%%%%%%%%%%%%%%%%%%%%%%%%%%%%%%%%%%%%%%%%%%%%%%%%%%%%%%%%%%%%%%%%%%%%%%%%%%%%%%%%%%%%%%%%%%%%%%%%%%%%%%%%%%%%%%%%%%%%%%%%%%%%%%%%%%%%%%%%%%%%%%%%%%%%%%%%%%%%%%%%%%%%%%%%%%%%%%%%%%%%%%%%%%%%%%%%%%

\newpage

\chapter*{Partial flag varieties}
\addcontentsline{toc}{chapter}{Partial flag varieties}
\fancyhead[R]{Partial flag varieties}

\section{Notation and definitions} \label{sec G/P Notation and definitions}
\fancyhead[L]{5 \ \ Notation and definitions}
% \fancyhead[L]{6 \ \ Notation and definitions}

In this section we build on Section \ref{subsec Notation and definitions}, introducing notation which will enable us to extend from the setting of full flag varieties to that of partial flag varieties.

We recall from Section \ref{subsec Notation and definitions} that, for $i \in I=\{1, \ldots, n-1\}$, the simple reflections in the Weyl group $s_i \in W$ have representatives in $N_G(T)$ given by $\bar{s}_i$, we denote their inverses by $\dot{s}_i$:
    $$\bar{s}_i = \mathbf{x}_i(-1)\mathbf{y}_i(1)\mathbf{x}_i(-1) = \phi_i\begin{pmatrix} 0 & -1 \\ 1 & 0 \end{pmatrix}, \quad
    \dot{s}_i =\bar{s}_i^{-1}= \mathbf{x}_i(1)\mathbf{y}_i(-1)\mathbf{x}_i(1) = \phi_i\begin{pmatrix} 0 & 1 \\ -1 & 0 \end{pmatrix}.
    $$

Let $P \supseteq B$ be a (fixed) parabolic subgroup of $G=GL_n(\mathbb{K})$ and take $I_P = \{ i \in I \ | \ \dot{s}_i \in P \}$. We denote the complement of $I_P$ in $I$ by
    $$I^P = \{n_1, \ldots, n_l\}, \quad \text{where } 1 \leq n_1 < \ldots < n_l.
    $$
For ease of notation we set $n_{l+1}:=n$ and $n_0:=0$.

Let $W_P$ denote the parabolic subgroup of $W$ associated to $P$, and write $W^P$ for the set of minimal length coset representatives in $W/W_P$:
    $$\begin{aligned}
    W_P &= \langle s_i \ | \ i \in I_P \rangle, \\
    W^P &= \left\{ w \in W \ | \ l(ws_i) > l(w) \text{ for all } i \in I_P \right\}.
    \end{aligned}
    $$
We denote the longest element of $W_P$ by $w_P$, for example $w_B=1$.

We recall from Section \ref{subsec Notation and definitions} that the roots and positive roots of $G$ are
    $$R = \{\alpha_{ij} \ \vert \ i \neq j\} \text{ and } R_+ = \{\alpha_{ij} \ \vert \ i < j\}
    $$
respectively, and the simple roots of $G$ are $\{ \alpha_i \ \vert \ i \in I\}$ where we write $\alpha_i=\alpha_{i, i+1}$ for $i \in I$.

Since from now on we will be interested in partial flag varieties $G/P$, the set of simple roots we need to consider is $\left\{ \alpha_i \ \vert \ i \in I^P \right\}$. The respective set of positive roots is
    $$R^P_+ = \bigcup_{n_r \in I^P} \left\{ \alpha_{ij} \ | \ 1 \leq i \leq n_r, n_{r}+1 \leq j \leq n_{r+1}  \right\}
    $$
and similarly for the set of positive coroots $\alpha_{ij}^{\vee}$, denoted ${R^P_+}^{\vee}$.

Let $\mathbf{i}=(i_1, \ldots, i_M)$ stand for an arbitrary reduced expression $s_{i_1} \cdots s_{i_M}$ for $w_Pw_0$, then we have an ordering on the set of positive roots $R^P_+$ given by
    $$\alpha^{\mathbf{i}}_j=\begin{cases}
    \alpha_{i_1, i_1+1} & \text{for } j=1, \\
    s_{i_1}\cdots s_{i_{j-1}}\alpha_{i_j, i_j+1} & \text{for } j=2, \ldots M.
    \end{cases}
    $$

Writing $T^{W^P}$ for the fixed part of $T$, we denote the set of dominant integral weights by
    $${X^*\left(T^{W^P}\right)}^+ = \left\{\lambda \in X^*\left(T^{W^P}\right) \ \big\vert \ \langle \lambda, \alpha_{ij}^{\vee} \rangle \geq 0 \ \forall \alpha_{ij}^{\vee} \in {R^P_+}^{\vee} \right\}.
    $$

\section{Landau--Ginzburg models} \label{sec G/P Landau-Ginzburg models}
\fancyhead[L]{6 \ \ Landau--Ginzburg models}
% \fancyhead[L]{7 \ \ Landau--Ginzburg models}

In this section we describe the mirror Landau--Ginzburg model to the partial flag variety $G/P$, generalising the one from the $G/B$ setting given in Section \ref{sec Mirror symmetry for G/B applied to representation theory}. We also define the so-called highest weight and weight maps in this more general case.

The Landau--Ginzburg model for $G/P$ \cite{Rietsch2008} consists of a pair $(Z_P,\mathcal{W}_P)$, where $Z_P \subset G^{\vee}$ is an affine variety and $\mathcal{W}_P : Z_P \to \mathbb{K}^*$ is a holomorphic function called the superpotential. The subvariety $Z_P$ is given by
    $$Z_P := B_-^{\vee} \cap U^{\vee} \left(T^{\vee}\right)^{W_P} \dot{w}_P \bar{w}_0 U^{\vee}.
    $$
The superpotential $\mathcal{W}_P$ is defined to be
    $$\mathcal{W}_P : Z_P \to \mathbb{K}^*, \quad u_Ld\dot{w}_P\bar{w}_0u_R \mapsto \chi(u_L)+ \chi(u_R)
    $$
where $u_L, u_R \in U^{\vee}$ and $d \in \left(T^{\vee}\right)^{W_P}$, and we recall from Section \ref{sec Mirror symmetry for G/B applied to representation theory} that $\chi:U^{\vee} \to \mathbb{K}$ is the map giving the sum of above-diagonal elements
    $$\chi(u) := \sum_{i=1}^{n-1} u_{i\, i+1}, \quad u=(u_{ij})\in U^{\vee}.
    $$

Similar to the $G/B$ setting, we equip $Z_P$ with highest weight and weight maps. The highest weight map recovers the original torus factor, $d$, as follows:
    $$\mathrm{hw}_P : Z_P \to %\left(T^{\vee}\right)^{W_P} \subset
    T^{\vee}, \quad u_Ld\dot{w}_P\bar{w}_0u_R \mapsto d.
    $$
For the weight map we first note that each element $b \in Z_P$ may be written as $b=[b]_-[b]_0$ with $[b]_- \in U_-^{\vee}$, $[b]_0\in T^{\vee}$. Then the weight map \cite{BerensteinKazhdan2007} % Rmk 2.21 the map gamma
is given by the projection
    $$
    \mathrm{wt}_P : Z_P \to T^{\vee}, \quad b \mapsto [b]_0.
    $$
% We will often write the above decomposition of $b$ as $b=[b]_-t_R$ to remind us that the torus factor is taken on the right.
Similar to the $G/B$ case, although $\mathrm{hw}_P$ is defined on all of $U^{\vee} \left(T^{\vee}\right)^{W_P} \dot{w}_P \bar{w}_0 U^{\vee}$ and $\mathrm{wt}_P$ is defined on all of $B_-^{\vee}$, these maps will only be of relevance to us as maps on $Z_P$.

% \newpage
\section{Quivers for partial flag varieties} \label{sec Quivers for partial flag varieties}
\fancyhead[L]{7 \ \ Quivers for partial flag varieties}
% \fancyhead[L]{8 \ \ Quivers for partial flag varieties}

% (Quiver construction c.f. Mirror Symmetry and Toric Degenerations of Partial Flag Manifolds - Batyrev, Ciocan-Fontanine, Kim, van Straten)

We wish to construct a coordinate system for the mirror to $G/P$ which is analogous to the ideal toric chart given in the $G/B$ case. Rather than using an analogous `string' toric chart as our starting point, we instead begin with a generalisation, originally due to Batyrev, Ciocan-Fontanine, Kim and van Straten (\cite{BatyrevEtAl2000}), of the Givental-type quivers used in the $G/B$ setting.

In this section we first detail the quiver construction in the $G/P$ case. Then, using this, we generalise the three tori defined in terms of the quiver, as well as the superpotential, highest weight and weight maps. We also extend the definition of the quiver toric chart to this setting and conclude with a further consideration of the weight matrix.

\subsection{Constructing quivers} \label{subsec Constructing quivers}
\fancyhead[L]{7.1 \ \ Constructing quivers}
% \fancyhead[L]{8.1 \ \ Constructing quivers}

% Let $P \supseteq B$ be a (fixed) parabolic subgroup of $G=GL_n(\mathbb{K})$. We recall $I^P = \{n_1, \ldots, n_l\}$, $1 \leq n_1 < \ldots < n_l$, as well as $n_{l+1}:=n$ and $n_0:=0$, from Section \ref{sec G/P Notation and definitions}. We will write $\mathcal{F}_{n_1, \ldots, n_l}(\mathbb{K}^n)$ for the partial flag variety $G/P$.
%
Let $P \supseteq B$ be a (fixed) parabolic subgroup of $G=GL_n(\mathbb{K})$, with $I^P = \{n_1, \ldots, n_l\}$, $1 \leq n_1 < \ldots < n_l$. We will write $\mathcal{F}_{n_1, \ldots, n_l}(\mathbb{K}^n)$ for the partial flag variety $G/P$. This can be thought of as the variety of flags of subspaces of $\mathbb{K}^n$ of dimensions $n_1, \ldots, n_l$.
% We consider the partial flag variety $G/P = \mathcal{F}_{n_1, \ldots, n_l}(\mathbb{K}^n)$.
Let $k_1, k_2, \ldots, k_{l+1}$ be the sequence of positive integers such that $n_i= k_1 + \cdots + k_i$ for $i=1, \dots, l$ and $n= k_1 + \cdots + k_{l+1}$. Additionally, we recall $n_0:=0$ from Section \ref{sec G/P Notation and definitions}.

We wish to draw a quiver corresponding to $G/P$, denoted $Q_P$. To begin, we take an $n\times n$ square and for each $i=1, \ldots, l+1$ place a square of size $k_i \times k_i$, called $L_i$, in order down the leading diagonal. We fill the space below these $l+1$ squares with unit squares and leave the space above empty. For example see Figure \ref{fig squares no quiver for 2,5,6 n8}.

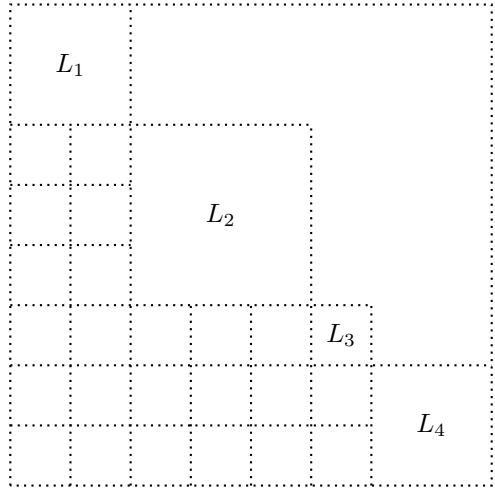
\begin{figure}[ht]
    \centering
\begin{tikzpicture}[scale=0.8]
    %squares
        \draw[dotted, thick] (0.5,0.5) -- (0.5,8.5) -- (8.5,8.5) -- (8.5,0.5) -- cycle;
        \draw[dotted, thick] (1.5,0.5) -- (1.5,6.5);
        \draw[dotted, thick] (2.5,0.5) -- (2.5,8.5);
        \draw[dotted, thick] (3.5,0.5) -- (3.5,3.5);
        \draw[dotted, thick] (4.5,0.5) -- (4.5,3.5);
        \draw[dotted, thick] (5.5,0.5) -- (5.5,6.5);
        \draw[dotted, thick] (6.5,0.5) -- (6.5,3.5);

        \draw[dotted, thick] (0.5,6.5) -- (5.5,6.5);
        \draw[dotted, thick] (0.5,5.5) -- (2.5,5.5);
        \draw[dotted, thick] (0.5,4.5) -- (2.5,4.5);
        \draw[dotted, thick] (0.5,3.5) -- (6.5,3.5);
        \draw[dotted, thick] (0.5,2.5) -- (8.5,2.5);
        \draw[dotted, thick] (0.5,1.5) -- (6.5,1.5);

    %square labels
        \node at (1.5,7.5) {$L_1$};
        \node at (4,5) {$L_2$};
        \node at (6,3) {$L_3$};
        \node at (7.5,1.5) {$L_4$};
    \end{tikzpicture}
\caption{Diagram of squares for $G/P=\mathcal{F}_{2,5,6}(\mathbb{C}^8)$} \label{fig squares no quiver for 2,5,6 n8}
\end{figure}

We form the quiver $Q_P$ from this diagram as follows:
\begin{itemize}
    \item Place a star vertex at the $\left(\frac{1}{2} , \frac{1}{2} \right)$-shift of the lower left corner of each of the squares $L_i$ on the leading diagonal.
    \item Place a dot vertex at the centre of the remaining unit squares, that is, all unit squares strictly below the diagonal.
    \item Label the vertices by $v_{ij}$, similar to the respective entries in an $n\times n$ matrix.
    \item Connect adjacent dot and star vertices with arrows oriented either upwards or to the left.
\end{itemize}

Having constructed the quiver in this way, we then label $n-1$ rows on our diagram by $E_1, \ldots, E_{n-1}$, from top to bottom. Each row $E_i$ will intersect some square $L_j$ on the diagonal. The row $E_i$ contains the vertical arrows with target vertices $v_{ic}$, as well as, in the square $L_j$, exactly $i-n_{j-1}$ copies of $\dot{s}_i$
    % $$\dot{s}_i = \mathbf{x}^{\vee}_i(1)\mathbf{y}^{\vee}_i(-1)\mathbf{x}^{\vee}_i(1) = \phi^{\vee}_i \begin{pmatrix} 0 & 1 \\ -1 & 0 \end{pmatrix} = \bar{s}_i^{-1},
    % $$
each written in a circle. For example see Figure \ref{fig ij quiver in squares for 2,5,6 n8}.

It will be helpful for us to consider an alternate labelling of the dot vertices, defined by their column number and position below the leading diagonal. This labelling is inspired by the $m_{s_k+a}$ indexing of ideal coordinates in the full flag case and is best described visually, so for a continuation of the previous examples see Figure \ref{fig k,a quiver in squares for 2,5,6 n8}. The dot vertices below square $L_i$ are labelled top to bottom and left to right by $v_{(k,a)}$ for $k=n_{i-1}+1, \ldots, n_i$ and $a=n_i-k+1, \ldots, n-k$.

\begin{rem} Although we use the squares $L_i$ to help us define this labelling, the second part of the $v_{(k,a)}$ labels are relative to the leading diagonal itself, not the squares $L_i$. Additionally, we will use parentheses for the second labelling, $v_{(k,a)}$, and not for the first, $v_{ij}$.
\end{rem}

\begin{figure}[ht]
\centering
\makebox[\textwidth][c]{
\begin{minipage}[t]{0.51\textwidth}
\centering
\begin{tikzpicture}[scale=0.8]
    %squares
        \draw[dotted, thick, color=black!50] (0.5,0.5) -- (0.5,8.5) -- (8.5,8.5) -- (8.5,0.5) -- cycle;
        \draw[dotted, thick, color=black!50] (1.5,0.5) -- (1.5,6.5);
        \draw[dotted, thick, color=black!50] (2.5,0.5) -- (2.5,8.5);
        \draw[dotted, thick, color=black!50] (3.5,0.5) -- (3.5,3.5);
        \draw[dotted, thick, color=black!50] (4.5,0.5) -- (4.5,3.5);
        \draw[dotted, thick, color=black!50] (5.5,0.5) -- (5.5,6.5);
        \draw[dotted, thick, color=black!50] (6.5,0.5) -- (6.5,3.5);

        \draw[dotted, thick, color=black!50] (0.5,6.5) -- (5.5,6.5);
        \draw[dotted, thick, color=black!50] (0.5,5.5) -- (2.5,5.5);
        \draw[dotted, thick, color=black!50] (0.5,4.5) -- (2.5,4.5);
        \draw[dotted, thick, color=black!50] (0.5,3.5) -- (6.5,3.5);
        \draw[dotted, thick, color=black!50] (0.5,2.5) -- (8.5,2.5);
        \draw[dotted, thick, color=black!50] (0.5,1.5) -- (6.5,1.5);

    %square labels
        \node[color=black!50] at (2.85,8) {$L_1$};
        \node[color=black!50] at (5.85,6) {$L_2$};
        \node[color=black!50] at (6.85,3) {$L_3$};
        \node[color=black!50] at (8,2.85) {$L_4$};

    %dots and stars
        % \node (11) at (1,8) {$\boldsymbol{*}$};
        \node (21) at (1,7) {$\boldsymbol{*}$};
            \node at (1.36,7.1) {\scriptsize{$v_{21}$}};
        \node (31) at (1,6) {$\bullet$};
            \node at (1.26,6.2) {\scriptsize{$v_{31}$}};
        \node (41) at (1,5) {$\bullet$};
            \node at (1.26,5.2) {\scriptsize{$v_{41}$}};
        \node (51) at (1,4) {$\bullet$};
            \node at (1.26,4.2) {\scriptsize{$v_{51}$}};
        \node (61) at (1,3) {$\bullet$};
            \node at (1.26,3.2) {\scriptsize{$v_{61}$}};
        \node (71) at (1,2) {$\bullet$};
            \node at (1.26,2.2) {\scriptsize{$v_{71}$}};
        \node (81) at (1,1) {$\bullet$};
            \node at (1.26,1.2) {\scriptsize{$v_{81}$}};

        % \node (22) at (2,7) {$\boldsymbol{*}$};
        \node (32) at (2,6) {$\bullet$};
            \node at (2.26,6.2) {\scriptsize{$v_{32}$}};
        \node (42) at (2,5) {$\bullet$};
            \node at (2.26,5.2) {\scriptsize{$v_{42}$}};
        \node (52) at (2,4) {$\bullet$};
            \node at (2.26,4.2) {\scriptsize{$v_{52}$}};
        \node (62) at (2,3) {$\bullet$};
            \node at (2.26,3.2) {\scriptsize{$v_{62}$}};
        \node (72) at (2,2) {$\bullet$};
            \node at (2.26,2.2) {\scriptsize{$v_{72}$}};
        \node (82) at (2,1) {$\bullet$};
            \node at (2.26,1.2) {\scriptsize{$v_{82}$}};

        % \node (33) at (3,6) {$\boldsymbol{*}$};
        % \node (43) at (3,5) {$\bullet$};
        \node (53) at (3,4) {$\boldsymbol{*}$};
            \node at (3.36,4.) {\scriptsize{$v_{53}$}};
        \node (63) at (3,3) {$\bullet$};
            \node at (3.26,3.2) {\scriptsize{$v_{63}$}};
        \node (73) at (3,2) {$\bullet$};
            \node at (3.26,2.2) {\scriptsize{$v_{73}$}};
        \node (83) at (3,1) {$\bullet$};
            \node at (3.26,1.2) {\scriptsize{$v_{83}$}};

        % \node (44) at (4,5) {$\boldsymbol{*}$};
        % \node (54) at (4,4) {$\bullet$};
        \node (64) at (4,3) {$\bullet$};
            \node at (4.26,3.2) {\scriptsize{$v_{64}$}};
        \node (74) at (4,2) {$\bullet$};
            \node at (4.26,2.2) {\scriptsize{$v_{74}$}};
        \node (84) at (4,1) {$\bullet$};
            \node at (4.26,1.2) {\scriptsize{$v_{84}$}};

        % \node (55) at (5,4) {$\boldsymbol{*}$};
        \node (65) at (5,3) {$\bullet$};
            \node at (5.26,3.2) {\scriptsize{$v_{65}$}};
        \node (75) at (5,2) {$\bullet$};
            \node at (5.26,2.2) {\scriptsize{$v_{75}$}};
        \node (85) at (5,1) {$\bullet$};
            \node at (5.26,1.2) {\scriptsize{$v_{85}$}};

        \node (66) at (6,3) {$\boldsymbol{*}$};
            \node at (6.26,3.2) {\scriptsize{$v_{66}$}};
        \node (76) at (6,2) {$\bullet$};
            \node at (6.26,2.2) {\scriptsize{$v_{76}$}};
        \node (86) at (6,1) {$\bullet$};
            \node at (6.26,1.2) {\scriptsize{$v_{86}$}};

        % \node (77) at (7,2) {$\boldsymbol{*}$};
        \node (87) at (7,1) {$\boldsymbol{*}$};
            \node at (7.36,1.1) {\scriptsize{$v_{87}$}};

        % \node (88) at (8,1) {$\boldsymbol{*}$};

    %Row E_i
        \node at (-0.6,1.5) {\small{Row $E_7$}};
        \node at (-0.6,2.5) {\small{Row $E_6$}};
        \node at (-0.6,3.5) {\small{Row $E_5$}};
        \node at (-0.6,4.5) {\small{Row $E_4$}};
        \node at (-0.6,5.5) {\small{Row $E_3$}};
        \node at (-0.6,6.5) {\small{Row $E_2$}};
        \node at (-0.6,7.5) {\small{Row $E_1$}};

    %s_i's
        \node[draw, circle, minimum size=4mm, inner sep=1pt] at (1,7.5) {\scriptsize{$\dot{s}_1$}};
        \node[draw, circle, minimum size=4mm, inner sep=1pt] at (3,5.5) {\scriptsize{$\dot{s}_3$}};
        \node[draw, circle, minimum size=4mm, inner sep=1pt] at (3,4.5) {\scriptsize{$\dot{s}_4$}};
        \node[draw, circle, minimum size=4mm, inner sep=1pt] at (4,4.5) {\scriptsize{$\dot{s}_4$}};
        \node[draw, circle, minimum size=4mm, inner sep=1pt] at (7,1.5) {\scriptsize{$\dot{s}_7$}};

    %vertical arrows
        \draw[->] (81) -- (71);
        \draw[->] (71) -- (61);
        \draw[->] (61) -- (51);
        \draw[->] (51) -- (41);
        \draw[->] (41) -- (31);
        \draw[->] (31) -- (21);
        % \draw[->] (21) -- (11);

        \draw[->] (82) -- (72);
        \draw[->] (72) -- (62);
        \draw[->] (62) -- (52);
        \draw[->] (52) -- (42);
        \draw[->] (42) -- (32);
        % \draw[->] (32) -- (22);

        \draw[->] (83) -- (73);
        \draw[->] (73) -- (63);
        \draw[->] (63) -- (53);
        % \draw[->] (53) -- (43);
        % \draw[->] (43) -- (33);

        \draw[->] (84) -- (74);
        \draw[->] (74) -- (64);
        % \draw[->] (64) -- (54);
        % \draw[->] (54) -- (44);

        \draw[->] (85) -- (75);
        \draw[->] (75) -- (65);
        % \draw[->] (65) -- (55);

        \draw[->] (86) -- (76);
        \draw[->] (76) -- (66);

        % \draw[->] (87) -- (77);

    %horizontal arrows
        % \draw[->] (22) -- (21);
        \draw[->] (32) -- (31);
        \draw[->] (42) -- (41);
        \draw[->] (52) -- (51);
        \draw[->] (62) -- (61);
        \draw[->] (72) -- (71);
        \draw[->] (82) -- (81);

        % \draw[->] (33) -- (32);
        % \draw[->] (43) -- (42);
        \draw[->] (53) -- (52);
        \draw[->] (63) -- (62);
        \draw[->] (73) -- (72);
        \draw[->] (83) -- (82);

        % \draw[->] (44) -- (43);
        % \draw[->] (54) -- (53);
        \draw[->] (64) -- (63);
        \draw[->] (74) -- (73);
        \draw[->] (84) -- (83);

        % \draw[->] (55) -- (54);
        \draw[->] (65) -- (64);
        \draw[->] (75) -- (74);
        \draw[->] (85) -- (84);

        \draw[->] (66) -- (65);
        \draw[->] (76) -- (75);
        \draw[->] (86) -- (85);

        % \draw[->] (77) -- (76);
        \draw[->] (87) -- (86);

        % \draw[->] (88) -- (87);
\end{tikzpicture}
\caption{Quiver $Q_P$ for $\mathcal{F}_{2,5,6}(\mathbb{C}^8)$ with vertices labelled by $v_{ij}$} \label{fig ij quiver in squares for 2,5,6 n8}
\end{minipage} \hspace{0.5cm}
\begin{minipage}[t]{0.51\textwidth}
    \centering
\begin{tikzpicture}[scale=0.8]
    %squares
        \draw[dotted, thick, color=black!50] (0.5,0.5) -- (0.5,8.5) -- (8.5,8.5) -- (8.5,0.5) -- cycle;
        \draw[dotted, thick, color=black!50] (1.5,0.5) -- (1.5,6.5);
        \draw[dotted, thick, color=black!50] (2.5,0.5) -- (2.5,8.5);
        \draw[dotted, thick, color=black!50] (3.5,0.5) -- (3.5,3.5);
        \draw[dotted, thick, color=black!50] (4.5,0.5) -- (4.5,3.5);
        \draw[dotted, thick, color=black!50] (5.5,0.5) -- (5.5,6.5);
        \draw[dotted, thick, color=black!50] (6.5,0.5) -- (6.5,3.5);

        \draw[dotted, thick, color=black!50] (0.5,6.5) -- (5.5,6.5);
        \draw[dotted, thick, color=black!50] (0.5,5.5) -- (2.5,5.5);
        \draw[dotted, thick, color=black!50] (0.5,4.5) -- (2.5,4.5);
        \draw[dotted, thick, color=black!50] (0.5,3.5) -- (6.5,3.5);
        \draw[dotted, thick, color=black!50] (0.5,2.5) -- (8.5,2.5);
        \draw[dotted, thick, color=black!50] (0.5,1.5) -- (6.5,1.5);

    %square labels
        \node[color=black!50] at (2.85,8) {$L_1$};
        \node[color=black!50] at (5.85,6) {$L_2$};
        \node[color=black!50] at (6.85,3) {$L_3$};
        \node[color=black!50] at (8,2.85) {$L_4$};

    %dots and stars
        % \node (11) at (1,8) {$\boldsymbol{*}$};
        \node (21) at (1,7) {$\boldsymbol{*}$};
        \node (31) at (1,6) {$\bullet$};
            \node at (1.42,6.2) {\scriptsize{$v_{(1,2)}$}};
        \node (41) at (1,5) {$\bullet$};
            \node at (1.42,5.2) {\scriptsize{$v_{(1,3)}$}};
        \node (51) at (1,4) {$\bullet$};
            \node at (1.42,4.2) {\scriptsize{$v_{(1,4)}$}};
        \node (61) at (1,3) {$\bullet$};
            \node at (1.42,3.2) {\scriptsize{$v_{(1,5)}$}};
        \node (71) at (1,2) {$\bullet$};
            \node at (1.42,2.2) {\scriptsize{$v_{(1,6)}$}};
        \node (81) at (1,1) {$\bullet$};
            \node at (1.42,1.2) {\scriptsize{$v_{(1,7)}$}};

        % \node (22) at (2,7) {$\boldsymbol{*}$};
        \node (32) at (2,6) {$\bullet$};
            \node at (2.42,6.2) {\scriptsize{$v_{(2,1)}$}};
        \node (42) at (2,5) {$\bullet$};
            \node at (2.42,5.2) {\scriptsize{$v_{(2,2)}$}};
        \node (52) at (2,4) {$\bullet$};
            \node at (2.42,4.2) {\scriptsize{$v_{(2,3)}$}};
        \node (62) at (2,3) {$\bullet$};
            \node at (2.42,3.2) {\scriptsize{$v_{(2,4)}$}};
        \node (72) at (2,2) {$\bullet$};
            \node at (2.42,2.2) {\scriptsize{$v_{(2,5)}$}};
        \node (82) at (2,1) {$\bullet$};
            \node at (2.42,1.2) {\scriptsize{$v_{(2,6)}$}};

        % \node (33) at (3,6) {$\boldsymbol{*}$};
        % \node (43) at (3,5) {$\bullet$};
        \node (53) at (3,4) {$\boldsymbol{*}$};
        \node (63) at (3,3) {$\bullet$};
            \node at (3.42,3.2) {\scriptsize{$v_{(3,3)}$}};
        \node (73) at (3,2) {$\bullet$};
            \node at (3.42,2.2) {\scriptsize{$v_{(3,4)}$}};
        \node (83) at (3,1) {$\bullet$};
            \node at (3.42,1.2) {\scriptsize{$v_{(3,5)}$}};

        % \node (44) at (4,5) {$\boldsymbol{*}$};
        % \node (54) at (4,4) {$\bullet$};
        \node (64) at (4,3) {$\bullet$};
            \node at (4.42,3.2) {\scriptsize{$v_{(4,2)}$}};
        \node (74) at (4,2) {$\bullet$};
            \node at (4.42,2.2) {\scriptsize{$v_{(4,3)}$}};
        \node (84) at (4,1) {$\bullet$};
            \node at (4.42,1.2) {\scriptsize{$v_{(4,4)}$}};

        % \node (55) at (5,4) {$\boldsymbol{*}$};
        \node (65) at (5,3) {$\bullet$};
            \node at (5.42,3.2) {\scriptsize{$v_{(5,1)}$}};
        \node (75) at (5,2) {$\bullet$};
            \node at (5.42,2.2) {\scriptsize{$v_{(5,2)}$}};
        \node (85) at (5,1) {$\bullet$};
            \node at (5.42,1.2) {\scriptsize{$v_{(5,3)}$}};

        \node (66) at (6,3) {$\boldsymbol{*}$};
        \node (76) at (6,2) {$\bullet$};
            \node at (6.42,2.2) {\scriptsize{$v_{(6,1)}$}};
        \node (86) at (6,1) {$\bullet$};
            \node at (6.42,1.2) {\scriptsize{$v_{(6,2)}$}};

        % \node (77) at (7,2) {$\boldsymbol{*}$};
        \node (87) at (7,1) {$\boldsymbol{*}$};

        % \node (88) at (8,1) {$\boldsymbol{*}$};

    %Row E_i
        \node at (-0.6,1.5) {\small{Row $E_7$}};
        \node at (-0.6,2.5) {\small{Row $E_6$}};
        \node at (-0.6,3.5) {\small{Row $E_5$}};
        \node at (-0.6,4.5) {\small{Row $E_4$}};
        \node at (-0.6,5.5) {\small{Row $E_3$}};
        \node at (-0.6,6.5) {\small{Row $E_2$}};
        \node at (-0.6,7.5) {\small{Row $E_1$}};

    %s_i's
        \node[draw, circle, minimum size=4mm, inner sep=1pt] at (1,7.5) {\scriptsize{$\dot{s}_1$}};
        \node[draw, circle, minimum size=4mm, inner sep=1pt] at (3,5.5) {\scriptsize{$\dot{s}_3$}};
        \node[draw, circle, minimum size=4mm, inner sep=1pt] at (3,4.5) {\scriptsize{$\dot{s}_4$}};
        \node[draw, circle, minimum size=4mm, inner sep=1pt] at (4,4.5) {\scriptsize{$\dot{s}_4$}};
        \node[draw, circle, minimum size=4mm, inner sep=1pt] at (7,1.5) {\scriptsize{$\dot{s}_7$}};

    %vertical arrows
        \draw[->] (81) -- (71);
        \draw[->] (71) -- (61);
        \draw[->] (61) -- (51);
        \draw[->] (51) -- (41);
        \draw[->] (41) -- (31);
        \draw[->] (31) -- (21);
        % \draw[->] (21) -- (11);

        \draw[->] (82) -- (72);
        \draw[->] (72) -- (62);
        \draw[->] (62) -- (52);
        \draw[->] (52) -- (42);
        \draw[->] (42) -- (32);
        % \draw[->] (32) -- (22);

        \draw[->] (83) -- (73);
        \draw[->] (73) -- (63);
        \draw[->] (63) -- (53);
        % \draw[->] (53) -- (43);
        % \draw[->] (43) -- (33);

        \draw[->] (84) -- (74);
        \draw[->] (74) -- (64);
        % \draw[->] (64) -- (54);
        % \draw[->] (54) -- (44);

        \draw[->] (85) -- (75);
        \draw[->] (75) -- (65);
        % \draw[->] (65) -- (55);

        \draw[->] (86) -- (76);
        \draw[->] (76) -- (66);

        % \draw[->] (87) -- (77);

    %horizontal arrows
        % \draw[->] (22) -- (21);
        \draw[->] (32) -- (31);
        \draw[->] (42) -- (41);
        \draw[->] (52) -- (51);
        \draw[->] (62) -- (61);
        \draw[->] (72) -- (71);
        \draw[->] (82) -- (81);

        % \draw[->] (33) -- (32);
        % \draw[->] (43) -- (42);
        \draw[->] (53) -- (52);
        \draw[->] (63) -- (62);
        \draw[->] (73) -- (72);
        \draw[->] (83) -- (82);

        % \draw[->] (44) -- (43);
        % \draw[->] (54) -- (53);
        \draw[->] (64) -- (63);
        \draw[->] (74) -- (73);
        \draw[->] (84) -- (83);

        % \draw[->] (55) -- (54);
        \draw[->] (65) -- (64);
        \draw[->] (75) -- (74);
        \draw[->] (85) -- (84);

        \draw[->] (66) -- (65);
        \draw[->] (76) -- (75);
        \draw[->] (86) -- (85);

        % \draw[->] (77) -- (76);
        \draw[->] (87) -- (86);

        % \draw[->] (88) -- (87);
\end{tikzpicture}
\caption{Quiver $Q_P$ for $\mathcal{F}_{2,5,6}(\mathbb{C}^8)$ with dot vertices labelled by $v_{(k,a)}$} \label{fig k,a quiver in squares for 2,5,6 n8}
\end{minipage}}
\end{figure}

% We will use both of these labelling systems.
For the dot vertices, we can change between these two sets of labellings by sending $v_{ij}$ to $v_{(j,i-j)}$. % and similarly by sending $v_{(k,a)}$ to $v_{k+a \ k}$.
Additionally we may extend the $v_{(k,a)}$ labelling to star vertices in the natural way; if a star vertex $v_{ii}$ appears on the leading diagonal we label it $v_{(i,0)}$.

We let $\mathcal{V}_P=\mathcal{V}^*_P \cup \mathcal{V}^{\bullet}_P$ denote the vertices of the quiver $Q_P$ and denote the set of arrows of $Q_P$ by $\mathcal{A}_P=\mathcal{A}_{P,\mathrm{v}} \cup \mathcal{A}_{P,\mathrm{h}}$. Similar to the $G/B$ case, the vertical arrows $a_{ij}\in \mathcal{A}_{P,\mathrm{v}}$ are labelled such that $h(a_{ij})=v_{ij}$ where $h(a) \in \mathcal{V}_P$ denotes the head of the arrow $a$. The horizontal arrows $b_{ij} \in \mathcal{A}_{P,\mathrm{h}}$ are labelled such that $t(b_{ij})=v_{ij}$ where $t(a)\in \mathcal{V}_P$ denotes the tail of the arrow $a$. Similar to the vertices, we will also label arrows by $(k,a)$ pairs, as required.

% \subsection[Toric charts on \texorpdfstring{$Z_P$}{Z}]{Toric charts on {$Z_P$}} \label{subsec Toric charts on ZP}

\subsection{The superpotential, highest weight and weight maps} \label{subsec The superpotential, highest weight and weight maps}
\fancyhead[L]{7.2 \ \ The superpotential, highest weight and weight maps}
% \fancyhead[L]{8.2 \ \ The superpotential, highest weight and weight maps}

We begin this section by generalising two tori defined in terms of the quiver, the vertex and arrow tori, from the $G/B$ setting to the $G/P$ setting. When we worked with $G/B$ we also described a third torus, the quiver torus. We will extend this to the $G/P$ case in Section \ref{subsec G/P The quiver torus and toric chart}.
In the rest of this section we generalise the superpotential, highest weight and weight maps on the vertex torus, from the $G/B$ case to the $G/P$ case.

The first torus is the vertex torus $(\mathbb{K}^*)^{\mathcal{V}_P}$, with coordinates $x_v$ for $v \in \mathcal{V}_P$. The second torus corresponds to the arrows of the quiver so we call it the arrow torus, denoted $\bar{\mathcal{M}}_P \subset (\mathbb{K}^*)^{\mathcal{A}_P}$. It is given by
    $$ \bar{\mathcal{M}}_P:= \left\{ (r_a)_{a \in \mathcal{A}_P} \in (\mathbb{K}^*)^{\mathcal{A}_P} \ | \ r_{a_1}r_{a_2}=r_{a_3}r_{a_4} \text{ when } a_1,a_2,a_3,a_4 \text{ form a square as in Figure \ref{fig G/P quiver box relations}} \right\}.
    $$
\begin{figure}[ht]
    \centering
\begin{tikzpicture}
    %dots and stars
    \node (41) at (0,0) {};
    \node (31) at (0,1.5) {};
    \node (42) at (1.5,0) {};
    \node (32) at (1.5,1.5) {};

    %arrows
    \draw[->] (41) -- (31);
    \draw[->] (42) -- (32);

    \draw[->] (32) -- (31);
    \draw[->] (42) -- (41);

    %arrow labels
    \node at (-0.4,0.75) {$a_2$};
    \node at (1.9,0.75) {$a_3$};

    \node at (0.75,1.75) {$a_4$};
    \node at (0.75,-0.3) {$a_1$};
\end{tikzpicture}
\caption{Subquiver for box relations} \label{fig G/P quiver box relations}
\end{figure}
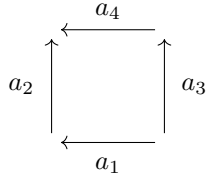

These two tori are related by the following surjection, given coordinate-wise: % c.f. Jamie's thesis p15?
    \begin{equation*} %\label{eqn G/P quiver tori relation map}
    (\mathbb{K}^*)^{\mathcal{V}_P} \to \bar{\mathcal{M}}_P, \quad r_a = \frac{x_{h(a)}}{x_{t(a)}}.
    \end{equation*}

%Let $\mathcal{V}_P=\mathcal{V}^*_P \cup \mathcal{V}^{\bullet}_P$ denote the vertices of the quiver $Q_P$.

Having defined the vertex and arrow tori we are able to describe the three maps; firstly, the superpotential \cite{BatyrevEtAl2000} is given by
    $$\mathcal{F}_P : \left(\mathbb{K}^* \right)^{\mathcal{V}_P} \to \mathbb{K}^* , \quad \boldsymbol{x}_{\mathcal{V}_P} \mapsto \sum_{a \in \mathcal{A}} \frac{x_{h(a)}}{x_{t(a)}}.
    $$

Secondly, the highest weight map is
    \begin{equation} \label{eqn defn kappa hw map}
    \kappa_P : \left(\mathbb{K}^* \right)^{\mathcal{V}_P} \to \left(T^{\vee}\right)^{W_P} , \quad \boldsymbol{x}_{\mathcal{V}_P} \mapsto \left(\kappa_P\left(\boldsymbol{x}_{\mathcal{V}_P}\right)\right)_{jj}:=x_{v_{n_i, n_{i-1}+1}},
    % \quad (d)_{jj}:=d_i
    \quad i=1, \ldots, l+1, \quad j=n_{i-1}+1, \ldots, n_{i}
    \end{equation}
where we note that $v_{n_i, n_{i-1}+1}$ is exactly the star vertex in the square $L_i$.

Finally we have the weight map which, similar to $G/B$ case, is defined in two steps; firstly, as in the $G/B$ case, for $i=1, \ldots,n$ we let $\mathcal{D}_i := \{ v_{i,1}, v_{i+1,2}, \ldots, v_{n, n-i+1} \}$ be the $i$-th diagonal. We note that in some diagonals $\mathcal{D}_i$, a number of the vertices from this set may not appear in the quiver $Q_P$. This will not pose a problem for us, however, as we simply use this an an indexing set: we define
    \begin{equation} \label{eqn G/P xi for wt map defn}
    \Xi_{P,i} := \left( \prod_{v\in \mathcal{D}_i \cap \mathcal{V}_P} x_v \right) \left( \prod_{\substack{v\in \mathcal{D}_i \cap L_j \\ v \notin \mathcal{V}_P}} x_{v_{n_j, n_{j-1}+1}} \right)
            % x_{v_{n_j, n_{j-1}+1}} = d_j in quiver labelling
    % \Xi_i := \left( \prod_{\substack{v\in \mathcal{D}_i \\ v \in \mathcal{V}_P}} x_v \right) \left( \prod_{\substack{v_{i+k-1,k}\in \mathcal{D}_i \\ n_{j-1}+1 \leq k \leq n_j \\ v_{i+k-1,k} \notin \mathcal{V}}} d_j \right)
    \quad \text{with} \quad \Xi_{P,n+1}:=1,
    \end{equation}
where by the notation $\{ v\in \mathcal{D}_i \cap L_j \} \cap \{ v \notin \mathcal{V}_P \}$ we mean that the vertex $v$ does not appear in our quiver, but if it did it would lie in the intersection $\mathcal{D}_i \cap L_j$, that is, it would be of the form $v_{i+l-1,l} \in \mathcal{D}_i$ with $n_{j-1}+1 \leq l \leq n_j$. Again, $v_{n_j, n_{j-1}+1}$ is exactly the star vertex in the square $L_j$. %star vertices included in left hand product, not in right hand product **?? dot vertices vs star vertices??**.
The weight map is then given by
    \begin{equation} \label{eqn G/P defn gamma and t}
    \gamma_P :(\mathbb{K}^*)^{\mathcal{V}_P}\to T^{\vee}, \quad \boldsymbol{x}_{\mathcal{V}_P} \mapsto (t_{P,i})_{i=1, \ldots, n} \quad \text{where} \quad t_{P,i} := \frac{\Xi_{P,i}}{\Xi_{P,i+1}}.
    \end{equation}

\begin{rem}
All three of these maps, $\mathcal{F}_P$, $\kappa_P$ and $\gamma_P$, descend to the respective definitions in the setting of full flag varieties $G/B$, noting that in this case we have $\mathcal{V}_P=\mathcal{V}$. Moreover, we see that the critical point conditions in the quiver are analogous to those in the full flag case (see (\ref{eqn crit pt conditions}) in Section \ref{subsec Critical points on the superpotential}), namely
        \begin{equation}\label{eqn G/P crit pt conditions}
        \sum_{a:h(a)=v} r_a = \sum_{a:t(a)=v} r_a \quad \text{for} \ v \in \mathcal{V}_P^{\bullet}.
    \end{equation}
\end{rem}

\subsection{Constructing matrices from a given quiver decoration} \label{subsec G/P Constructing matrices from a given quiver decoration}
\fancyhead[L]{7.3 \ \ Constructing matrices from a given quiver decoration}
% \fancyhead[L]{8.3 \ \ Constructing matrices from a given quiver decoration}

In this section, inspired by Marsh and Rietsch in \cite{MarshRietsch2020}, we construct two matrices from the decoration of the quiver $Q_P$.
These matrices will later form a fundamental part of the generalisation of the quiver toric chart, now on $Z_P$. In particular, they relate to the first term of the quiver chart, which was called $u_1$ in the $G/B$ setting (see Section \ref{subsec The quiver torus as another toric chart on Z}).
% Moreover, taken together with the quiver decoration defined in terms of $m_i$ coordinates, given in Section \ref{subsec Quiver decoration}, it will enable us to generalise the ideal toric chart to the setting of partial flag varieties.

The two matrices in question may be thought of as maps on the vertex (or equivalently arrow) torus, we denote them by
    $$\begin{aligned}
    g_L &: \left(\mathbb{K}^*\right)^{\mathcal{V}_P} \to G^{\vee}, %\quad \boldsymbol{x}_{\mathcal{V}_P} \mapsto g_L\left(\boldsymbol{x}_{\mathcal{V}_P}\right)
    \\
    u_L &: \left(\mathbb{K}^*\right)^{\mathcal{V}_P} \to U^{\vee}.
    \end{aligned}
    $$
To simplify notation, unless otherwise stated we will write $g_L$ for the matrix $g_L(\boldsymbol{x}_{\mathcal{V}_P}) \in G^{\vee}$, and similarly for $u_L$.

In order to construct these two maps we require the following definition:
\begin{defn} \label{defn 1 paths leaving}
A $1$-path is a path in $Q_P$ which contain exactly one vertical arrow. We say a $1$-path beginning (equiv. ending) at a given vertex is minimal if it is the shortest such path.

For each dot vertex $v \in \mathcal{V}_P^{\bullet}$ there is a unique minimal $1$-path which begins at this vertex, which we denote by $p(v)$.
\end{defn}

We are now ready to construct the map $g_L$ following Marsh and Rietsch, by considering minimal $1$-paths in the quiver $Q_P$:
\begin{itemize}
    \item Starting at the bottom-left vertex of $Q_P$ and working upwards, we list the minimal $1$-paths leaving each dot vertex in succession, followed by the circled $\dot{s}_i$'s which appear at the top of the column.
    \item We repeat this, working column by column to the right, until we have treated column $n_l$, in other words we exclude any contribution from within the square $L_{l+1}$. %i.e. no $\dot{s}_i$ from final square $L_{l+1}$
    \item Then, to the circled $\dot{s}_i$'s we associate the obvious factors $\dot{s}_i$. To any $1$-path which crosses row $E_i$, we associate the element $\mathbf{x}^{\vee}_i(x_{v_e}/x_{v_s}) \in U^{\vee}$, where $x_{v_s}$ is the vertex coordinate at the start of the path $x_{v_e}$ is the vertex coordinate at the end of the path. Note that the quotient $x_{v_e}/x_{v_s}$ is equal to the product of arrow coordinates of all the arrows in the respective $1$-path.
\end{itemize}
The matrix we obtain from taking the product of $\mathbf{x}^{\vee}_i$ and $\dot{s}_i$ factors in the order defined by the above list, is $g_L$.

\begin{rem} \label{rem G/P gL in G/B case gives u1}
Constructing $g_L$ from the quiver corresponding to $G/B$ gives the matrix $u_1$.
\end{rem}

Since we wish to generalise the quiver toric chart, we would like the first term to be a matrix in $U^{\vee}$. We have seen in Remark \ref{rem G/P gL in G/B case gives u1} that, using $g_L$, this would be the case in the $G/B$ setting, however in general $g_L \notin U^{\vee}$. Consequently we will use the quiver decoration to construct our second matrix, $u_L \in U^{\vee}$, which we will show is a factor of $g_L$. Both $u_L$ and $g_L$ will be helpful later.

To define $u_L$ we proceed in a similar way to the construction of $g_L$. However there are two differences when we construct $u_L$, the first being that when we list the minimal $1$-paths we do not include any $\dot{s}_i$'s. The second difference comes in the association of elements $\mathbf{x}^{\vee}_i(x_{v_e}/x_{v_s}) \in U^{\vee}$ to $1$-paths crossing row $E_i$. If a $1$-path in our list has length $1$ then we treat it as before, namely we assign the element $\mathbf{x}^{\vee}_i(x_{v_e}/x_{v_s}) \in U^{\vee}$. To each $1$-path of length $\alpha\geq2$ in our list, we associate the element
    $$X^{\vee}_{i,\alpha}(x_{v_e}/x_{v_s}):= I + \frac{x_{v_t}}{x_{v_s}}E_{i-\alpha+1,i+1}
    $$
where $I$ is the identity matrix and $E_{jk}$ is the elementary matrix with $1$ in position $(j,k)$ and $0$'s elsewhere. The resulting product of $\mathbf{x}^{\vee}_i$ and $X^{\vee}_{i,\alpha}$ factors defines the matrix $u_L$.

In order to succinctly describe the relation between $u_L$ and $g_L$ we introduce some notation: let $\dot{w}_{L_{[i,j]}}$ denote the product of $\dot{s}_k$'s given by the contributions in $g_L$ from within squares $L_i, \ldots, L_j$, with respect to the ordering in $g_L$. If we consider a single square $L_j$, we will write $\dot{w}_{L_j}$ instead of $\dot{w}_{L_{\{j\}}}$ to simplify notation.

More explicitly, for $j=1, \ldots, l+1$, $\dot{w}_{L_j}$ is a representative of the longest element of
    $$\left\langle s_i \ | \ i \in \{ n_{j-1}+1, \ldots, n_j-1 \} \right\rangle \subset W
    $$
given by the reduced expression
    $$(n_j-1, n_j-2, \ldots, n_{j-1}+1, \ldots, n_j-1, n_j-2, n_j-1)
    $$
recalling $n_0=0$ and $n_{l+1}=n$.
This follows from the row set in the quiver $Q_P$ which the interior of $L_j$ intersects, namely the set $\{ E_{n_{j-1}+1}, \ldots, E_{n_j-1} \}$, together with the following two facts. Firstly, $\dot{w}_{L_j}$ is given by a product of $\dot{s}_i$ factors, read from the square $L_j$ by starting at the bottom of the left-most column and working upwards, then proceeding column by column to the right. Secondly, by definition, the intersection of the square $L_j$ and a row $E_i$, contains $i-n_{j-1}$ copies of $\dot{s}_i$.

We see that $\dot{w}_{L_{[i,j]}} = \dot{w}_{L_i} \dot{w}_{L_{i+1}} \cdots \dot{w}_{L_j}$, where some of the $\dot{w}_{L_r}$ may be trivial. For ease of notation we also define $\dot{w}_L := \dot{w}_{L_{[1,l]}}$.

Now by definition, $\dot{w}_{P}$ is a representative of the longest element of
    $$\left\langle s_i \ | \ i \in \{1, \ldots, n-1\} \setminus \{ n_1, \ldots, n_l \} \right\rangle \subseteq W
    $$
and may be written as a product of representatives of the longest elements of
    $$\left\langle s_i \ | \ i \in \{ n_{j-1}+1, \ldots, n_j-1 \} \right\rangle, \quad j=1, \ldots, l+1,
    $$
recalling $n_0=0$ and $n_{l+1}=n$. Consequently we see that
    $$\dot{w}_{P} = \dot{w}_{L_1} \cdots \dot{w}_{L_{l+1}} %= \dot{w}_{L}\dot{w}_{L_{l+1}}
    = \dot{w}_{L_{[1,l+1]}}
    $$
and we have the relation $\dot{w}_L = \dot{w}_P \dot{w}_{L_{l+1}}^{-1}$.

\begin{ex} \label{ex F2,5,6,C8 wLi}
In our running example of $\mathcal{F}_{2,5,6}(\mathbb{C}^8)$ we have
    $$\dot{w}_{L_1}=\dot{s}_1, \quad \dot{w}_{L_2}=\dot{s}_4 \dot{s}_3 \dot{s}_4, \quad \dot{w}_{L_3}=1, \quad \dot{w}_{L_4}=\dot{s}_7 \quad \Rightarrow \quad \dot{w}_P = \dot{s}_1 \dot{s}_4 \dot{s}_3 \dot{s}_4 \dot{s}_7
    $$
\end{ex}

% It will be helpful later to have an alternate factorisation of $g_L$:
% *********Why start with $g_L$ why not simply just define $u_L$ since that's all I need for the definition of $b \in Z_P$..?**** NEED $g_L$ for weight matrix proof
\begin{lem} \label{lem gL in terms of uL and wL}
With the above $\dot{w}_{L_i}$ notation, the matrices $u_L$ and $g_L$ are related by $g_L \dot{w}_{L_{l+1}} = u_L \dot{w}_P$, or equivalently
    \begin{equation} \label{eqn gL in terms of uL and si}
    \quad g_L = u_L \dot{w}_L
    \end{equation}
since by definition $\dot{w}_P = \dot{w}_{L_{[1,l+1]}} = \dot{w}_{L_1} \dot{w}_{L_{2}} \cdots \dot{w}_{L_{l+1}}$ and $\dot{w}_L = \dot{w}_{L_{[1,l]}}= \dot{w}_P \dot{w}_{L_{l+1}}^{-1}$.
\end{lem}

To aid our familiarity with $u_L$ and $g_L$ and motivate the proof of this lemma, we give an example.

\begin{ex} \label{ex gL temp quiver decoration} We will continue our running example of $\mathcal{F}_{2,5,6}(\mathbb{C}^8)$. For ease of notation, in this example we label a number of arrows in the $Q_P$ quiver as in Figure \ref{fig gL temp quiver decoration}, and write $a$ in place of the arrow coordinate $r_a$.

\begin{figure}[ht]
\centering
\begin{tikzpicture}[scale=0.8]
    %squares
        \draw[dotted, thick, color=black!50] (0.5,0.5) -- (0.5,8.5) -- (8.5,8.5) -- (8.5,0.5) -- cycle;
        \draw[dotted, thick, color=black!50] (1.5,0.5) -- (1.5,6.5);
        \draw[dotted, thick, color=black!50] (2.5,0.5) -- (2.5,8.5);
        \draw[dotted, thick, color=black!50] (3.5,0.5) -- (3.5,3.5);
        \draw[dotted, thick, color=black!50] (4.5,0.5) -- (4.5,3.5);
        \draw[dotted, thick, color=black!50] (5.5,0.5) -- (5.5,6.5);
        \draw[dotted, thick, color=black!50] (6.5,0.5) -- (6.5,3.5);

        \draw[dotted, thick, color=black!50] (0.5,6.5) -- (5.5,6.5);
        \draw[dotted, thick, color=black!50] (0.5,5.5) -- (2.5,5.5);
        \draw[dotted, thick, color=black!50] (0.5,4.5) -- (2.5,4.5);
        \draw[dotted, thick, color=black!50] (0.5,3.5) -- (6.5,3.5);
        \draw[dotted, thick, color=black!50] (0.5,2.5) -- (8.5,2.5);
        \draw[dotted, thick, color=black!50] (0.5,1.5) -- (6.5,1.5);

    %square labels
        \node[color=black!50] at (2.85,8) {$L_1$};
        \node[color=black!50] at (5.85,6) {$L_2$};
        \node[color=black!50] at (6.85,3) {$L_3$};
        \node[color=black!50] at (8,2.85) {$L_4$};

    %dots and stars + arrow labels
        \node (21) at (1,7) {$\boldsymbol{*}$};
        \node (31) at (1,6) {$\bullet$};
            \node at (0.82,6.34) {\scriptsize{$f$}};
        \node (41) at (1,5) {$\bullet$};
            \node at (0.82,5.34) {\scriptsize{$e$}};
        \node (51) at (1,4) {$\bullet$};
            \node at (0.82,4.34) {\scriptsize{$d$}};
        \node (61) at (1,3) {$\bullet$};
            \node at (0.82,3.34) {\scriptsize{$c$}};
        \node (71) at (1,2) {$\bullet$};
            \node at (0.82,2.34) {\scriptsize{$b$}};
        \node (81) at (1,1) {$\bullet$};
            \node at (0.82,1.34) {\scriptsize{$a$}};

        \node (32) at (2,6) {$\bullet$};
            \node at (1.62,6.2) {\scriptsize{$l$}};
        \node (42) at (2,5) {$\bullet$};
            \node at (1.82,5.34) {\scriptsize{$k$}};
        \node (52) at (2,4) {$\bullet$};
            \node at (1.82,4.34) {\scriptsize{$j$}};
        \node (62) at (2,3) {$\bullet$};
            \node at (1.82,3.34) {\scriptsize{$i$}};
        \node (72) at (2,2) {$\bullet$};
            \node at (1.82,2.34) {\scriptsize{$h$}};
        \node (82) at (2,1) {$\bullet$};
            \node at (1.82,1.34) {\scriptsize{$g$}};

        % \node (33) at (3,6) {$\boldsymbol{*}$};
        % \node (43) at (3,5) {$\bullet$};
        \node (53) at (3,4) {$\boldsymbol{*}$};
        \node (63) at (3,3) {$\bullet$};
            \node at (2.82,3.34) {\scriptsize{$p$}};
        \node (73) at (3,2) {$\bullet$};
            \node at (2.82,2.34) {\scriptsize{$n$}};
        \node (83) at (3,1) {$\bullet$};
            \node at (2.82,1.34) {\scriptsize{$m$}};

        % \node (44) at (4,5) {$\boldsymbol{*}$};
        % \node (54) at (4,4) {$\bullet$};
        \node (64) at (4,3) {$\bullet$};
            \node at (3.62,3.2) {\scriptsize{$s$}};
        \node (74) at (4,2) {$\bullet$};
            \node at (3.82,2.34) {\scriptsize{$r$}};
        \node (84) at (4,1) {$\bullet$};
            \node at (3.82,1.34) {\scriptsize{$q$}};

        % \node (55) at (5,4) {$\boldsymbol{*}$};
        \node (65) at (5,3) {$\bullet$};
            \node at (4.62,3.2) {\scriptsize{$v$}};
        \node (75) at (5,2) {$\bullet$};
            \node at (4.82,2.34) {\scriptsize{$u$}};
        \node (85) at (5,1) {$\bullet$};
            \node at (4.82,1.34) {\scriptsize{$t$}};

        \node (66) at (6,3) {$\boldsymbol{*}$};
        \node (76) at (6,2) {$\bullet$};
            \node at (5.82,2.34) {\scriptsize{$x$}};
        \node (86) at (6,1) {$\bullet$};
            \node at (5.82,1.34) {\scriptsize{$w$}};

        % \node (77) at (7,2) {$\boldsymbol{*}$};
        \node (87) at (7,1) {$\boldsymbol{*}$};

        % \node (88) at (8,1) {$\boldsymbol{*}$};

    %Row E_i
        \node at (-0.6,1.5) {\small{Row $E_7$}};
        \node at (-0.6,2.5) {\small{Row $E_6$}};
        \node at (-0.6,3.5) {\small{Row $E_5$}};
        \node at (-0.6,4.5) {\small{Row $E_4$}};
        \node at (-0.6,5.5) {\small{Row $E_3$}};
        \node at (-0.6,6.5) {\small{Row $E_2$}};
        \node at (-0.6,7.5) {\small{Row $E_1$}};

    %s_i's
        \node[draw, circle, minimum size=4mm, inner sep=1pt] at (1,7.5) {\scriptsize{$\dot{s}_1$}};
        \node[draw, circle, minimum size=4mm, inner sep=1pt] at (3,5.5) {\scriptsize{$\dot{s}_3$}};
        \node[draw, circle, minimum size=4mm, inner sep=1pt] at (3,4.5) {\scriptsize{$\dot{s}_4$}};
        \node[draw, circle, minimum size=4mm, inner sep=1pt] at (4,4.5) {\scriptsize{$\dot{s}_4$}};
        \node[draw, circle, minimum size=4mm, inner sep=1pt] at (7,1.5) {\scriptsize{$\dot{s}_7$}};

    %vertical arrows
        \draw[->] (81) -- (71);
        \draw[->] (71) -- (61);
        \draw[->] (61) -- (51);
        \draw[->] (51) -- (41);
        \draw[->] (41) -- (31);
        \draw[->] (31) -- (21);
        % \draw[->] (21) -- (11);

        \draw[->] (82) -- (72);
        \draw[->] (72) -- (62);
        \draw[->] (62) -- (52);
        \draw[->] (52) -- (42);
        \draw[->] (42) -- (32);
        % \draw[->] (32) -- (22);

        \draw[->] (83) -- (73);
        \draw[->] (73) -- (63);
        \draw[->] (63) -- (53);
        % \draw[->] (53) -- (43);
        % \draw[->] (43) -- (33);

        \draw[->] (84) -- (74);
        \draw[->] (74) -- (64);
        % \draw[->] (64) -- (54);
        % \draw[->] (54) -- (44);

        \draw[->] (85) -- (75);
        \draw[->] (75) -- (65);
        % \draw[->] (65) -- (55);

        \draw[->] (86) -- (76);
        \draw[->] (76) -- (66);

        % \draw[->] (87) -- (77);

    %horizontal arrows
        % \draw[->] (22) -- (21);
        \draw[->] (32) -- (31);
        \draw[->] (42) -- (41);
        \draw[->] (52) -- (51);
        \draw[->] (62) -- (61);
        \draw[->] (72) -- (71);
        \draw[->] (82) -- (81);

        % \draw[->] (33) -- (32);
        % \draw[->] (43) -- (42);
        \draw[->] (53) -- (52);
        \draw[->] (63) -- (62);
        \draw[->] (73) -- (72);
        \draw[->] (83) -- (82);

        % \draw[->] (44) -- (43);
        % \draw[->] (54) -- (53);
        \draw[->] (64) -- (63);
        \draw[->] (74) -- (73);
        \draw[->] (84) -- (83);

        % \draw[->] (55) -- (54);
        \draw[->] (65) -- (64);
        \draw[->] (75) -- (74);
        \draw[->] (85) -- (84);

        \draw[->] (66) -- (65);
        \draw[->] (76) -- (75);
        \draw[->] (86) -- (85);

        % \draw[->] (77) -- (76);
        \draw[->] (87) -- (86);

        % \draw[->] (88) -- (87);
    \end{tikzpicture}
\caption{Temporary $Q_P$ quiver labelling for Example {\ref{ex gL temp quiver decoration}}} \label{fig gL temp quiver decoration}
\end{figure}
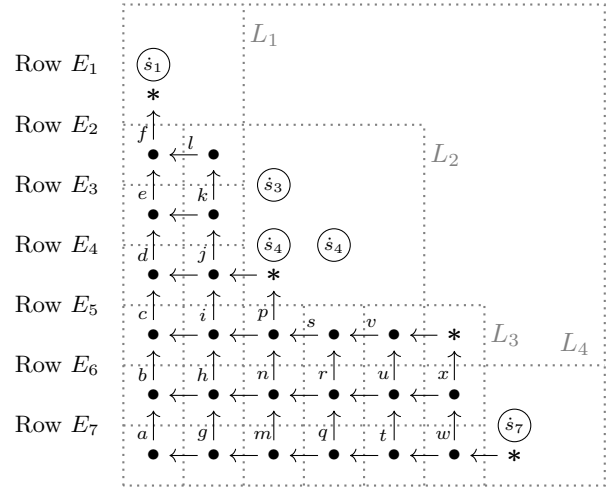

We will use the above construction to obtain $g_L$. We then give an alternate factorisation of this matrix with all $\dot{s}_i$ factors on the right, that is $u_L \dot{w}_L$. Most $\mathbf{x}^{\vee}_i$ factors appear in both factorisations with a similar pattern. For ease of comparison we have underlined the $\mathbf{x}^{\vee}_i$ and $X^{\vee}_{i,\alpha}$ factors which don't appear in both factorisations.

    $$\begin{aligned}
    g_L &=  \mathbf{x}^{\vee}_7(a) \mathbf{x}^{\vee}_6(b) \mathbf{x}^{\vee}_5(c) \mathbf{x}^{\vee}_4(d) \mathbf{x}^{\vee}_3(e) \mathbf{x}^{\vee}_2(f) \dot{s}_1 \\
        &\hspace{0.47cm} \mathbf{x}^{\vee}_7(g) \mathbf{x}^{\vee}_6(h) \mathbf{x}^{\vee}_5(i) \mathbf{x}^{\vee}_4(j) \mathbf{x}^{\vee}_3(k) \underline{\mathbf{x}^{\vee}_2(fl)} \\
        &\hspace{0.47cm} \mathbf{x}^{\vee}_7(m) \mathbf{x}^{\vee}_6(n) \mathbf{x}^{\vee}_5(p) \dot{s}_4 \dot{s}_3 \\
        &\hspace{0.47cm} \mathbf{x}^{\vee}_7(q) \mathbf{x}^{\vee}_6(r) \underline{\mathbf{x}^{\vee}_5(ps)} \dot{s}_4 \\
        &\hspace{0.47cm} \mathbf{x}^{\vee}_7(t) \mathbf{x}^{\vee}_6(u) \underline{\mathbf{x}^{\vee}_5(psv)} \\
        &\hspace{0.47cm} \mathbf{x}^{\vee}_7(w) \mathbf{x}^{\vee}_6(x) \\
    \end{aligned}
    $$
    $$\begin{aligned}
    g_L
    &= \begin{pmatrix}
        1 & 0 & & & & & & \\ & 1 & f & & & & & \\ & & 1 & e & & & & \\ & & & 1 & d & & & \\ & & & & 1 & c & & \\ & & & & & 1 & b & \\ & & & & & & 1 & a \\ & & & & & & & 1
    \end{pmatrix} \dot{s}_1
        \begin{pmatrix}
        1 & 0 & & & & & & \\ & 1 & fl & & & & & \\ & & 1 & k & & & & \\ & & & 1 & j & & & \\ & & & & 1 & i & & \\ & & & & & 1 & h & \\ & & & & & & 1 & g \\ & & & & & & & 1
    \end{pmatrix}
        \begin{pmatrix}
        1 & 0 & & & & & & \\ & 1 & 0 & & & & & \\ & & 1 & 0 & & & & \\ & & & 1 & 0 & & & \\ & & & & 1 & p & & \\ & & & & & 1 & n & \\ & & & & & & 1 & m \\ & & & & & & & 1
    \end{pmatrix} \dot{s}_4 \dot{s}_3 \\
    &\hspace{0.45cm} \times \begin{pmatrix}
        1 & 0 & & & & & & \\ & 1 & 0 & & & & & \\ & & 1 & 0 & & & & \\ & & & 1 & 0 & & & \\ & & & & 1 & ps & & \\ & & & & & 1 & r & \\ & & & & & & 1 & q \\ & & & & & & & 1
    \end{pmatrix} \dot{s}_4
        \begin{pmatrix}
        1 & 0 & & & & & & \\ & 1 & 0 & & & & & \\ & & 1 & 0 & & & & \\ & & & 1 & 0 & & & \\ & & & & 1 & psv & & \\ & & & & & 1 & u & \\ & & & & & & 1 & t \\ & & & & & & & 1
    \end{pmatrix}
        \begin{pmatrix}
        1 & 0 & & & & & & \\ & 1 & 0 & & & & & \\ & & 1 & 0 & & & & \\ & & & 1 & 0 & & & \\ & & & & 1 & 0 & & \\ & & & & & 1 & x & \\ & & & & & & 1 & w \\ & & & & & & & 1
    \end{pmatrix} \\
    \end{aligned}
    $$
Permuting all $\dot{s}_i$ factors to the right, we obtain:
    $$\begin{aligned}
    g_L
    &= \begin{pmatrix}
        1 & 0 & & & & & & \\ & 1 & f & & & & & \\ & & 1 & e & & & & \\ & & & 1 & d & & & \\ & & & & 1 & c & & \\ & & & & & 1 & b & \\ & & & & & & 1 & a \\ & & & & & & & 1
    \end{pmatrix}
        \begin{pmatrix}
        1 & 0 & fl & & & & & \\ & 1 & 0 & & & & & \\ & & 1 & k & & & & \\ & & & 1 & j & & & \\ & & & & 1 & i & & \\ & & & & & 1 & h & \\ & & & & & & 1 & g \\ & & & & & & & 1
    \end{pmatrix}
        \begin{pmatrix}
        1 & 0 & & & & & & \\ & 1 & 0 & & & & & \\ & & 1 & 0 & & & & \\ & & & 1 & 0 & & & \\ & & & & 1 & p & & \\ & & & & & 1 & n & \\ & & & & & & 1 & m \\ & & & & & & & 1
    \end{pmatrix} \\
    &\hspace{0.45cm} \times \begin{pmatrix}
        1 & 0 & & & & & & \\ & 1 & 0 & & & & & \\ & & 1 & 0 & & & & \\ & & & 1 & 0 & ps & & \\ & & & & 1 & 0 & & \\ & & & & & 1 & r & \\ & & & & & & 1 & q \\ & & & & & & & 1
    \end{pmatrix}
        \begin{pmatrix}
        1 & 0 & & & & & & \\ & 1 & 0 & & & & & \\ & & 1 & 0 & 0 & psv & & \\ & & & 1 & 0 & 0 & & \\ & & & & 1 & 0 & & \\ & & & & & 1 & u & \\ & & & & & & 1 & t \\ & & & & & & & 1
    \end{pmatrix}
        \begin{pmatrix}
        1 & 0 & & & & & & \\ & 1 & 0 & & & & & \\ & & 1 & 0 & & & & \\ & & & 1 & 0 & & & \\ & & & & 1 & 0 & & \\ & & & & & 1 & x & \\ & & & & & & 1 & w \\ & & & & & & & 1
    \end{pmatrix} \dot{s}_1 \dot{s}_4 \dot{s}_3 \dot{s}_4 \\
    \end{aligned}
    $$
    $$\begin{aligned}
    g_L
    &= \mathbf{x}^{\vee}_7(a) \mathbf{x}^{\vee}_6(b) \mathbf{x}^{\vee}_5(c) \mathbf{x}^{\vee}_4(d) \mathbf{x}^{\vee}_3(e) \mathbf{x}^{\vee}_2(f) \\
        &\hspace{0.47cm} \mathbf{x}^{\vee}_7(g) \mathbf{x}^{\vee}_6(h) \mathbf{x}^{\vee}_5(i) \mathbf{x}^{\vee}_4(j) \mathbf{x}^{\vee}_3(k) \underline{X^{\vee}_{2,2}(fl)} \\
        &\hspace{0.47cm} \mathbf{x}^{\vee}_7(m) \mathbf{x}^{\vee}_6(n) \mathbf{x}^{\vee}_5(p) \\
        &\hspace{0.47cm} \mathbf{x}^{\vee}_7(q) \mathbf{x}^{\vee}_6(r) \underline{X^{\vee}_{5,2}(ps)}  \\
        &\hspace{0.47cm} \mathbf{x}^{\vee}_7(t) \mathbf{x}^{\vee}_6(u) \underline{X^{\vee}_{5,3}(psv)} \\
        &\hspace{0.47cm} \mathbf{x}^{\vee}_7(w) \mathbf{x}^{\vee}_6(x) \dot{s}_1 \dot{s}_4 \dot{s}_3 \dot{s}_4. \\
    % &= \mathbf{x}^{\vee}_7(a) \mathbf{x}^{\vee}_6(b) \mathbf{x}^{\vee}_5(c) \mathbf{x}^{\vee}_4(d) \mathbf{x}^{\vee}_3(e) \mathbf{x}^{\vee}_2(f) \\
    %     &\hspace{0.47cm} \mathbf{x}^{\vee}_7(g) \mathbf{x}^{\vee}_6(h) \mathbf{x}^{\vee}_5(i) \mathbf{x}^{\vee}_4(j) \mathbf{x}^{\vee}_3(k) \underline{\big[ \mathbf{x}^{\vee}_1(f) \mathbf{x}^{\vee}_2(l) \mathbf{x}^{\vee}_1(-f) \mathbf{x}^{\vee}_2(-l) \big]} \\
    %     &\hspace{0.47cm} \mathbf{x}^{\vee}_7(m) \mathbf{x}^{\vee}_6(n) \mathbf{x}^{\vee}_5(p) \\
    %     &\hspace{0.47cm} \mathbf{x}^{\vee}_7(q) \mathbf{x}^{\vee}_6(r) \underline{\big[ \mathbf{x}^{\vee}_4(p) \mathbf{x}^{\vee}_5(s) \mathbf{x}^{\vee}_4(-p) \mathbf{x}^{\vee}_5(-s) \big]}  \\
    %     &\hspace{0.47cm} \mathbf{x}^{\vee}_7(t) \mathbf{x}^{\vee}_6(u) \underline{\left[ \big[ \mathbf{x}^{\vee}_3(p) \mathbf{x}^{\vee}_4(s) \mathbf{x}^{\vee}_3(-p) \mathbf{x}^{\vee}_4(-s) \big] \mathbf{x}^{\vee}_5(v) \big[ \mathbf{x}^{\vee}_3(p) \mathbf{x}^{\vee}_4(-s) \mathbf{x}^{\vee}_3(-p) \mathbf{x}^{\vee}_4(s) \big] \mathbf{x}^{\vee}_5(-v) \right]} \\
    %     &\hspace{0.47cm} \mathbf{x}^{\vee}_7(w) \mathbf{x}^{\vee}_6(x) \dot{s}_1 \dot{s}_4 \dot{s}_3 \dot{s}_4
    \end{aligned}
    $$
\end{ex}

\begin{proof}[Proof of Lemma {\ref{lem gL in terms of uL and wL}}]
To obtain the relation (\ref{eqn gL in terms of uL and si}), we start with the description of $g_L$ given above and permute all the $\dot{s}_i$ factors to the right. It follows that the resulting sequence of $\dot{s}_i$ factors may be read directly from the quiver $Q_P$ by starting in the left-most column and working upwards, only noting any $\dot{s}_i$'s we come across. Then proceeding column by column to the right in this way, and stopping just before the last square $L_{l+1}$. This product is exactly $\dot{w}_L := \dot{w}_{L_{[1,l]}} = \dot{w}_P \dot{w}_{L_{l+1}}^{-1}$.

It remains to show that, after this sequence of permutations, the remaining factor on the left of $\dot{w}_L$ is $u_L$.
To do this we begin by making some observations on the above permutations in $g_L$. Firstly, due to the constructions of the quiver $Q_P$ and the matrix $g_L$, we only have to permute the $\dot{s}_j$'s past $\mathbf{x}^{\vee}_i(x_{v_e}/x_{v_s})$'s for $j < i$. These matrices commute if $i-j \geq 2$, but if $j=i-1$ then
    $$\begin{gathered}
    \dot{s}_{i-1} \mathbf{x}^{\vee}_i\left(\frac{x_{v_e}}{x_{v_s}}\right)
        =\begin{pmatrix}
        1 & 0 & & \cdots & & & 0 \\
        & \ddots & \ddots & & & & \\
        & & 0 & 1 & \frac{x_{v_e}}{x_{v_s}} & & \\
        & & -1 & 0 & 0 & & \vdots \\
        & & 0 & 0 & 1 & \ddots & \\
        & & & & & \ddots & 0 \\
        & & & & & & 1
            \end{pmatrix}
        =\begin{pmatrix}
        1 & 0 & & \cdots & & & 0 \\
        & \ddots & \ddots & & & & \\
        & & 1 & 0 & \frac{x_{v_e}}{x_{v_s}} & & \\
        & & & 1 & 0 & & \vdots \\
        & & & & 1 & \ddots & \\
        & & & & & \ddots & 0 \\
        & & & & & & 1
            \end{pmatrix} \dot{s}_{i-1} \\
    \Rightarrow \quad \dot{s}_{i-1} \mathbf{x}^{\vee}_i\left(\frac{x_{v_e}}{x_{v_s}}\right)
        = \dot{s}_{i-1} \left(I + \frac{x_{v_e}}{x_{v_s}} E_{i,i+1}\right)
        = \left(I + \frac{x_{v_e}}{x_{v_s}} E_{i-1,i+1}\right) \dot{s}_{i-1}
        = X^{\vee}_{i,2}\left(\frac{x_{v_e}}{x_{v_s}}\right) \dot{s}_{i-1}
    \end{gathered}
    $$
where in both of the matrices written in full, the entry $x_{v_e}/x_{v_s}$ appears in position $(i-1,i+1)$.
% Roughly speaking, when we permute $\dot{s}_{i-1}$ past $\mathbf{x}^{\vee}_i(x_{v_e}/x_{v_s})$ it has the effect of `moving' the entry $x_{v_e}/x_{v_s}$ from position $(i,i+1)$ to position $(i-1, i+1)$, that is the entry `moves upwards' by one row.

This new matrix, $X^{\vee}_{i,2}\left(\frac{x_{v_e}}{x_{v_s}}\right) = \left(I + \frac{x_{v_e}}{x_{v_s}} E_{i-1,i+1}\right)$, will commute with all $\dot{s}_j$'s apart from $\dot{s}_{i-2}$. Permuting $\dot{s}_{i-2}$ past it has a similar effect, the entry $x_{v_e}/x_{v_s}$ will now be found in position $(i-2, i+1)$. This pattern repeats for all permutations.

It follows that we need to keep track of which factors $\mathbf{x}^{\vee}_i(x_{v_e}/x_{v_s})$ will be affected by our sequence of permutations. Thanks to the constructions of the quiver $Q_P$ and matrix $g_L$, we see that the only affected factors are those arising from $1$-paths of length at least $2$.

In particular, suppose we have a $1$-path of length $\alpha\geq 2$ crossing row $E_i$, with corresponding factor $\mathbf{x}^{\vee}_i(x_{v_e}/x_{v_s})$ of $g_L$. We will have to permute a product of $\dot{s}_j$'s past $\mathbf{x}^{\vee}_i(x_{v_e}/x_{v_s})$, however some of these $\dot{s}_j$'s will have no effect on the factor in question. Indeed by the above argument on permutations, together with the locations of the $\dot{s}_j$'s in the quiver $Q_L$, we see that $\mathbf{x}^{\vee}_i(x_{v_e}/x_{v_s})$ will only be affected by permuting the sub-product $\dot{s}_{i-\alpha+1} \cdots \dot{s}_{i-1}$ past it, giving
    $$\dot{s}_{i-\alpha+1} \cdots \dot{s}_{i-1} \mathbf{x}^{\vee}_i\left(\frac{x_{v_e}}{x_{v_s}}\right) = X^{\vee}_{i,\alpha}\left(\frac{x_{v_e}}{x_{v_s}}\right) \dot{s}_{i-\alpha+1} \cdots \dot{s}_{i-1}.
    $$
This sub-product is exactly the contribution to $g_L$ from those $\dot{s}_i$'s which are found both in the same square $L_j$ as the end vertex $v_e$ and also in the same diagonal $\mathcal{D}_k$ as the starting vertex $v_s$.

It follows that after all permutations, we will have expressed $g_L$ as $u_L$ multiplied on the right by $\dot{w}_L$, as desired.
\end{proof}

We will later need to compute minors of $u_L$ in an application of the Chamber Ansatz. As in the $G/B$ case, we do this via graphs which are defined by $\mathbf{x}^{\vee}_j$ factors. Thus it will be useful to be able to express $X^{\vee}_{i,\alpha}(x_{v_e}/x_{v_s})$ as a product of these $\mathbf{x}^{\vee}_j$'s. Using the quiver $Q_P$ we do this recursively as follows:

Consider the $1$-path of length $\alpha$ from $v_s$ to $v_t$ which crosses row $E_i$, as in Figure \ref{fig Subquiver showing 1-path of length alpha}.
% from $v_s = v_{n_i+1 , \, n_{i-1}+1+\alpha}$? to $v_t = v_{n_i , \, n_{i-1}+1}$? which crosses row $E_{n_i}$? $E_i$:
\begin{figure}[ht]
\centering
\begin{tikzpicture}
    %squares
        \draw[dotted, thick, color=black!50] (0.3,2.5) -- (0.5,2.5);
        \draw[dotted, thick, color=black!50] (0.3,1.5) -- (7.1,1.5);
        \draw[dotted, thick, color=black!50] (0.3,0.5) -- (7.1,0.5);

        \draw[dotted, thick, color=black!50] (0.5,0.3) -- (0.5,2.7);
        \draw[dotted, thick, color=black!50] (1.5,0.3) -- (1.5,1.5);
        \draw[dotted, thick, color=black!50] (2.5,0.3) -- (2.5,1.5);
        \draw[dotted, thick, color=black!50] (3.5,0.3) -- (3.5,1.5);
        \draw[dotted, thick, color=black!50] (4.88,0.3) -- (4.88,1.5);
        \draw[dotted, thick, color=black!50] (5.88,0.3) -- (5.88,1.5);
        \draw[dotted, thick, color=black!50] (6.88,0.3) -- (6.88,1.5);

    %dots and stars + arrow labels
        \node (11) at (1,2) {$\boldsymbol{*}$};
            \node at (1.2,2.2) {\scriptsize{$v_t$}};
        \node (21) at (1,1) {$\bullet$};
            \node at (0.82,1.38) {\scriptsize{$a_1$}};
        % \node (31) at (1,0) {\phantom{$\bullet$}};

        \node (22) at (2,1) {$\bullet$};
            \node at (1.5,1.18) {\scriptsize{$a_2$}};
        % \node (32) at (2,0) {\phantom{$\bullet$}};

        \node (23) at (3,1) {$\bullet$};
            \node at (2.5,1.18) {\scriptsize{$a_3$}};
        % \node (33) at (3,0) {\phantom{$\bullet$}};

        \node (24) at (4.19,1) {$\cdots$};
            \node at (3.5,1.18) {\scriptsize{$a_4$}};
        % \node (34) at (4,0) {\phantom{$\bullet$}};

        \node (25) at (5.38,1) {$\bullet$};
            \node at (4.88,1.16) {\scriptsize{$a_{\alpha-1}$}};
        % \node (35) at (5,0) {\phantom{$\bullet$}};

        \node (26) at (6.38,1) {$\bullet$};
            \node at (6.58,0.78) {\scriptsize{$v_s$}};
            \node at (5.88,1.16) {\scriptsize{$a_{\alpha}$}};
        % \node (36) at (6,0) {\phantom{$\bullet$}};

        \node (27) at (7.38,1) {\phantom{$\bullet$}};

        \node at (-0.8,1.5) {\small{Row $E_i$}};

    %vertical arrows
        \draw[->] (21) -- (11);
        % \draw[->] (31) -- (21);

        % \draw[->] (32) -- (22);
        % \draw[->] (33) -- (23);
        % \draw[->] (34) -- (24);
        % \draw[->] (35) -- (25);
        % \draw[->] (36) -- (26);

    %horizontal arrows
        \draw[->] (22) -- (21);
        \draw[->] (23) -- (22);
        \draw[->] (24) -- (23);
        \draw[->] (25) -- (24);
        \draw[->] (26) -- (25);
        \draw[->] (27) -- (26);
    \end{tikzpicture}
\caption{Subquiver showing $1$-path of length $\alpha$} \label{fig Subquiver showing 1-path of length alpha}
\end{figure}
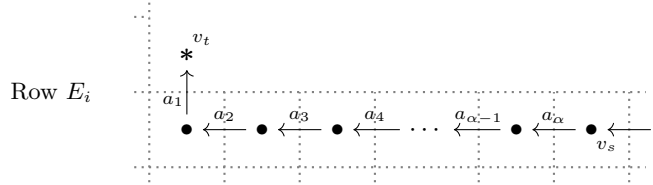
% Let $\beta = i-\alpha$, then
The matrix associated to this $1$-path $x_{v_e}/x_{v_s} = r_{a_1}r_{a_2}\cdots r_{a_{\alpha}}$ (where we have written the product of arrow coordinates in the same order as the composition of functions for ease of notation) is $\tilde{X}^{\vee}_{i,\alpha}(r_{a_1}, \ldots, r_{a_{\alpha}}) \in U^{\vee}$, defined recursively by
    % $$\begin{aligned}
    % X^{\vee}_{i,1}(r_{a_1}) &:= \mathbf{x}^{\vee}_{i-\alpha +1}(r_{a_1}) \\
    % % X^{\vee}_{i,2}(r_{a_1}, r_{a_2}) &:= \mathbf{x}^{\vee}_{\beta +1}(r_{a_1}) \mathbf{x}^{\vee}_{\beta +2}(r_{a_2}) \mathbf{x}^{\vee}_{\beta +1}(-r_{a_1}) \mathbf{x}^{\vee}_{\beta +2}(-r_{a_2}) \\
    % % X^{\vee}_{i,3}(r_{a_1}, r_{a_2}, r_{a_3}) &:= X^{\vee}_2(r_{a_1}, r_{a_2}) \mathbf{x}^{\vee}_{\beta +3}(r_{a_3}) X_2(r_{a_1}, -r_{a_2}) \mathbf{x}^{\vee}_{\beta +3}(-r_{a_3}) \\
    % X^{\vee}_{i,j}(r_{a_1}, \ldots, r_{a_{j -1}}, r_{a_{j}})  &:= X^{\vee}_{i,j -1}(r_{a_1}, \ldots, r_{a_{j -1}}) \mathbf{x}^{\vee}_{i-\alpha +j}(r_{a_{j}}) X^{\vee}_{i,j -1}(r_{a_1}, \ldots, r_{a_{j -2}}, -r_{a_{j -1}}) \mathbf{x}^{\vee}_{i-\alpha +j}(-r_{a_{j}}) & \\
    %     &  \hspace{9.74cm} \text{for} \ j=2, \ldots, \alpha.
    % \end{aligned}
    % $$
    \begin{equation} \label{eqn factor Xi into xi}
    \tilde{X}^{\vee}_{i,j}(r_{a_1}, \ldots, r_{a_{j}})  := \begin{cases}
        \mathbf{x}^{\vee}_{i-\alpha +1}(r_{a_1}) & \text{for} \ j=1, \\
        \begin{aligned}& \tilde{X}^{\vee}_{i,j -1}(r_{a_1}, \ldots, r_{a_{j -1}}) \mathbf{x}^{\vee}_{i-\alpha +j}(r_{a_{j}}) \\ & \qquad \times \tilde{X}^{\vee}_{i,j -1}(r_{a_1}, \ldots, r_{a_{j -2}}, -r_{a_{j -1}}) \mathbf{x}^{\vee}_{i-\alpha +j}(-r_{a_{j}}) \end{aligned} \quad& \text{for} \ j=2, \ldots, \alpha.
    \end{cases}
    \end{equation}
It follows by induction that we have the desired equality:
    % $$ \tilde{X}^{\vee}_{i,\alpha}(r_{a_1}, \ldots, r_{a_{\alpha}})
    % = \begin{pmatrix}
    %     1 & 0 & & & \cdots & & & 0 & \cdots & 0 \\
    %     & \ddots & \ddots & & & & & \vdots & & \vdots \\
    %     & & & & & & & 0 & & \\
    %     & & & 1 & 0 & \cdots & 0 & \prod_{j=1}^{\alpha} r_{a_j} & 0 & \\
    %     & & & & & & & 0 & & \\
    %     & & & & & \ddots & \ddots & \vdots & & \\
    %     & & & & & & & 0 & & \\
    %     & & & & & & & 1 & \ddots & \\
    %     & & & & & & & & \ddots & 0 \\
    %     & & & & & & & & & 1
    % \end{pmatrix} \in U^{\vee}.
    % $$
    % $$ \tilde{X}^{\vee}_{i,\alpha}(r_{a_1}, \ldots, r_{a_{\alpha}})
    % = \begin{pmatrix}
    %     1 & 0 & & & & \cdots & & & 0 \\
    %     & \ddots & \ddots & & & & & & \\
    %     & & 1 & 0 & \cdots & 0 & \prod_{j=1}^{\alpha} r_{a_j} & & \\
    %     & & & & & & 0 & & \vdots \\
    %     & & & & \ddots & & \vdots & & \\
    %     & & & & & & 0 & & \\
    %     & & & & & & 1 & \ddots & \\
    %     & & & & & & & \ddots & 0 \\
    %     & & & & & & & & 1
    % \end{pmatrix} \in U^{\vee}.
    % $$
    $$ \tilde{X}^{\vee}_{i,\alpha}(r_{a_1}, \ldots, r_{a_{\alpha}}) = X^{\vee}_{i,\alpha}\left(\prod_{j=1}^{\alpha}r_{a_j}\right)
    := \begin{pmatrix}
        1 & 0 & & & \cdots & & & 0 \\
        & \ddots & \ddots & & & & & \\
        & & 1 & 0 & & \prod_{j=1}^{\alpha} r_{a_j} & & \\
        & & & & \ddots & & & \vdots \\
        & & & & \ddots & 0 & & \\
        & & & & & 1 & \ddots & \\
        & & & & & & \ddots & 0 \\
        & & & & & & & 1
    \end{pmatrix} \in U^{\vee}.
    $$
    % $$ \tilde{X}^{\vee}_{i,\alpha}(r_{a_1}, \ldots, r_{a_{\alpha}})
    % = \begin{pmatrix}
    %     1 & 0 & & \cdots & & & 0 \\
    %     & \ddots & & & & & \\
    %     & & 1 & & \prod_{j=1}^{\alpha} r_{a_j} & & \vdots \\
    %     & & & \ddots & & & \\
    %     & & & & 1 & & \\
    %     & & & & & \ddots & 0 \\
    %     & & & & & & 1
    % \end{pmatrix} \in U^{\vee}.
    % $$
In particular, this allows us to write $X^{\vee}_{i,\alpha}(x_{v_e}/x_{v_s})=\tilde{X}^{\vee}_{i,\alpha}(r_{a_1}, \ldots, r_{a_{\alpha}})$ as a product of $\mathbf{x}^{\vee}_j$'s.

\begin{ex}
One example of a subquiver in the form of Figure \ref{fig Subquiver showing 1-path of length alpha} with $i=5$, $\alpha=3$, is given in Figure \ref{fig Subquiver showing 1-path of length 3}.
\begin{figure}[ht]
\centering
\begin{tikzpicture}
    %squares
        \draw[dotted, thick, color=black!50] (0.3,2.5) -- (0.5,2.5);
        \draw[dotted, thick, color=black!50] (0.3,1.5) -- (3.7,1.5);
        \draw[dotted, thick, color=black!50] (0.3,0.5) -- (3.7,0.5);

        \draw[dotted, thick, color=black!50] (0.5,0.3) -- (0.5,2.7);
        \draw[dotted, thick, color=black!50] (1.5,0.3) -- (1.5,1.5);
        \draw[dotted, thick, color=black!50] (2.5,0.3) -- (2.5,1.5);
        \draw[dotted, thick, color=black!50] (3.5,0.3) -- (3.5,1.5);

    %dots and stars + arrow labels
        \node (11) at (1,2) {$\boldsymbol{*}$};
            \node at (1.2,2.2) {\scriptsize{$v_t$}};
        \node (21) at (1,1) {$\bullet$};
            \node at (0.82,1.38) {\scriptsize{$a_1$}};
        % \node (31) at (1,0) {\phantom{$\bullet$}};

        \node (22) at (2,1) {$\bullet$};
            \node at (1.5,1.18) {\scriptsize{$a_2$}};
        % \node (32) at (2,0) {\phantom{$\bullet$}};

        \node (23) at (3,1) {$\bullet$};
            \node at (2.5,1.18) {\scriptsize{$a_3$}};
            \node at (3.2,0.78) {\scriptsize{$v_s$}};
        % \node (33) at (3,0) {\phantom{$\bullet$}};

        \node (24) at (4,1) {\phantom{$\bullet$}};

        \node at (-0.8,1.5) {\small{Row $E_5$}};

    %vertical arrows
        \draw[->] (21) -- (11);
        % \draw[->] (31) -- (21);

        % \draw[->] (32) -- (22);
        % \draw[->] (33) -- (23);
        % \draw[->] (34) -- (24);
        % \draw[->] (35) -- (25);
        % \draw[->] (36) -- (26);

    %horizontal arrows
        \draw[->] (22) -- (21);
        \draw[->] (23) -- (22);
        \draw[->] (24) -- (23);
    \end{tikzpicture}
\caption{Subquiver showing $1$-path of length $3$} \label{fig Subquiver showing 1-path of length 3}
\end{figure}
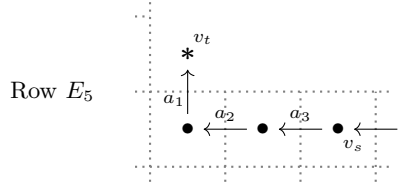
The matrix associated to the $1$-path $x_{v_e}/x_{v_s} = r_{a_1}r_{a_2}r_{a_3}$ is $\tilde{X}^{\vee}_{5,3}(r_{a_1}, r_{a_2}, r_{a_3}) \in U^{\vee}$, which is given in terms of $\mathbf{x}^{\vee}_j$'s by
    $$\begin{aligned}\tilde{X}^{\vee}_{5,3}(r_{a_1}, r_{a_2}, r_{a_3})
        &= \left[\mathbf{x}^{\vee}_{3}(r_{a_1})\mathbf{x}^{\vee}_{4}(r_{a_2})\mathbf{x}^{\vee}_{3}(-r_{a_1})\mathbf{x}^{\vee}_{4}(-r_{a_2})\right] \mathbf{x}^{\vee}_{5}(r_{a_3}) \\
        &  \qquad \times \left[\mathbf{x}^{\vee}_{3}(r_{a_1})\mathbf{x}^{\vee}_{4}(-r_{a_2})\mathbf{x}^{\vee}_{3}(-r_{a_1})\mathbf{x}^{\vee}_{4}(r_{a_2})\right] \mathbf{x}^{\vee}_{5}(-r_{a_3}) \end{aligned}
    $$
where we have used square braces to denote the products from the penultimate step, i.e. $j=\alpha-1=2$.
\end{ex}

\begin{ex} \label{ex uL temp quiver decoration cont}
Continuing Example \ref{ex gL temp quiver decoration}, we have
    $$\begin{aligned}
    u_L &= \mathbf{x}^{\vee}_7(a) \mathbf{x}^{\vee}_6(b) \mathbf{x}^{\vee}_5(c) \mathbf{x}^{\vee}_4(d) \mathbf{x}^{\vee}_3(e) \mathbf{x}^{\vee}_2(f) \\
        &\hspace{0.47cm} \mathbf{x}^{\vee}_7(g) \mathbf{x}^{\vee}_6(h) \mathbf{x}^{\vee}_5(i) \mathbf{x}^{\vee}_4(j) \mathbf{x}^{\vee}_3(k) \underline{\big[ \mathbf{x}^{\vee}_1(f) \mathbf{x}^{\vee}_2(l) \mathbf{x}^{\vee}_1(-f) \mathbf{x}^{\vee}_2(-l) \big]} \\
        &\hspace{0.47cm} \mathbf{x}^{\vee}_7(m) \mathbf{x}^{\vee}_6(n) \mathbf{x}^{\vee}_5(p) \\
        &\hspace{0.47cm} \mathbf{x}^{\vee}_7(q) \mathbf{x}^{\vee}_6(r) \underline{\big[ \mathbf{x}^{\vee}_4(p) \mathbf{x}^{\vee}_5(s) \mathbf{x}^{\vee}_4(-p) \mathbf{x}^{\vee}_5(-s) \big]}  \\
        &\hspace{0.47cm} \mathbf{x}^{\vee}_7(t) \mathbf{x}^{\vee}_6(u) \underline{\left[ \big[ \mathbf{x}^{\vee}_3(p) \mathbf{x}^{\vee}_4(s) \mathbf{x}^{\vee}_3(-p) \mathbf{x}^{\vee}_4(-s) \big] \mathbf{x}^{\vee}_5(v) \big[ \mathbf{x}^{\vee}_3(p) \mathbf{x}^{\vee}_4(-s) \mathbf{x}^{\vee}_3(-p) \mathbf{x}^{\vee}_4(s) \big] \mathbf{x}^{\vee}_5(-v) \right]} \\
        &\hspace{0.47cm} \mathbf{x}^{\vee}_7(w) \mathbf{x}^{\vee}_6(x). %\dot{s}_1 \dot{s}_4 \dot{s}_3 \dot{s}_4
    \end{aligned}
    $$

\end{ex}

\newpage
\subsection{The quiver torus and toric chart} \label{subsec G/P The quiver torus and toric chart}
\fancyhead[L]{7.4 \ \ The quiver torus and toric chart}
% \fancyhead[L]{8.4 \ \ The quiver torus and toric chart}

We recall from Section \ref{subsec The superpotential, highest weight and weight maps} that the vertex and arrow tori are defined similarly in the $G/B$ and $G/P$ cases. To define our third torus however, the quiver torus, we will need to extend our previous definition (given in Section \ref{subsec The Givental superpotential}). In the $G/B$ case it was sufficient to work with the vertical arrow and star vertex coordinates. These, together with the relations in the quiver, uniquely determined the rest of the quiver decoration. Working now in the $G/P$ case, this is no longer sufficient as there are not enough star vertex coordinates to describe the full quiver decoration using only the coordinates of the vertical arrows and star vertices. To solve this issue we will also use some horizontal arrow coordinates. In particular, we will use the coordinates of exactly those horizontal arrows used in the definitions of $u_L$ and $g_L$ in Section \ref{subsec G/P Constructing matrices from a given quiver decoration}, i.e. for each $i=1,\ldots, l+1$ we use the coordinates of the horizontal arrows connecting two dot vertices directly below squares $L_i$.
 % for example see the horizontal arrows in Figure \ref{fig to help arrow coord description}.

We define an ordering on this set of arrows as follows: starting in the lower left corner of the quiver $Q_P$ we move up each column of arrows in succession, taking note of the arrow coordinates we pass. If we come to a star vertex at the end of a column of arrows, then we move on to the next column, from left to right. If the column of arrows ends with a dot vertex then, after reaching this top-most dot vertex, we also include coordinate of the single horizontal arrow leaving this vertex. We call the resulting sequence $\boldsymbol{r}_{\mathcal{A}_{P,\hat{\mathrm{v}}}}$.

\begin{rem}
In the $G/B$ case, the above procedure simply returns the vertical arrow coordinates, $\boldsymbol{r}_{\mathcal{A}_{\mathrm{v}}}$ (defined in Section \ref{subsec The quiver torus as another toric chart on Z}). We also note that since $\boldsymbol{r}_{\mathcal{A}_{P,\hat{\mathrm{v}}}}$ is exactly comprised of the arrow coordinates which were used in the definitions of $u_L$ and $g_L$ (see Section \ref{subsec G/P Constructing matrices from a given quiver decoration}), we may think of these maps as maps on the quiver torus.
\end{rem}

\begin{ex}
We recall the temporary labelling of the arrows given in Figure \ref{fig gL temp quiver decoration}. In order to write down $\boldsymbol{r}_{\mathcal{A}_{P,\hat{\mathrm{v}}}}$ in this example, we need the arrow coordinates which have been explicitly labelled. Moreover they should be taken in alphabetical order.
\end{ex}

In addition to $\boldsymbol{r}_{\mathcal{A}_{P,\hat{\mathrm{v}}}}$, we also require an ordering of the star vertices. We take
    % $$ %OLD VERSION
    % \boldsymbol{x}_{\mathcal{V}_P^{\hat{*}}}:=\left( x_{v_{n_1,1}}, \ \ldots, \ x_{v_{n_1,1}}, \ x_{v_{n_2,n_1+1}}, \ \ldots, \ x_{v_{n_2,n_1+1}}, \ \ldots, \ x_{v_{n_{l+1},n_l+1}}, \ \ldots, \ x_{v_{n_{l+1},n_l+1}} \right)
    % $$
    $$ %NEW VERSION
    \boldsymbol{x}_{\mathcal{V}_P^*}:=\left( x_{v_{n_1,1}}, \ x_{v_{n_2,n_1+1}}, \ \ldots, \ x_{v_{n_{l+1},n_l+1}} \right).
    $$
% where $x_{v_{n_r,n_{r-1}+1}}$ appears $k_r$ times. %OLD VERSION
Again this descends to the star vertex coordinates, $\boldsymbol{x}_{\mathcal{V}^*}$, in the $G/B$ case.

With this notation the quiver torus is
    $$\mathcal{M}_P := \left\{ \left( \boldsymbol{x}_{\mathcal{V}_P^*} , \boldsymbol{r}_{\mathcal{A}_{P,\hat{\mathrm{v}}}}\right)
    \in (\mathbb{K}^*)^{\mathcal{V}_P^*} \times (\mathbb{K}^*)^{\mathcal{A}_{P,\mathrm{v}}} \right\}
    \hookrightarrow (\mathbb{K}^*)^{\mathcal{V}_P} \times \bar{\mathcal{M}}_P.
    $$

We are now ready to define the quiver toric chart on $Z_P$, which we recall is defined to be
    $$Z_P := B_-^{\vee}\cap U^{\vee}\left(T^{\vee}\right)^{W_P}\dot{w}_P\bar{w}_0U^{\vee}.
    $$

We define the quiver toric chart $\theta_P : \mathcal{M}_P \to \mathcal{Z}_P$ by
    $$\left( \boldsymbol{x}_{\mathcal{V}_P^*} , \boldsymbol{r}_{\mathcal{A}_{P,\hat{\mathrm{v}}}}\right)
    \mapsto
    u_L\left(\boldsymbol{r}_{\mathcal{A}_{P,\hat{\mathrm{v}}}}\right)
    \kappa_P\left(\boldsymbol{x}_{\mathcal{V}_P^*}\right)
    \dot{w}_P \bar{w}_0 u_R =:b_P
    $$
where $u_L\left(\boldsymbol{r}_{\mathcal{A}_{P,\hat{\mathrm{v}}}}\right)$ is the matrix defined in Section \ref{subsec G/P Constructing matrices from a given quiver decoration} and $u_R \in U^{\vee}$ is the unique element such that $b_P$ lies in $B_-^{\vee}$. The proof of the existence and uniqueness of $u_R$ is similar to that of $u_2$ in the $G/B$ case (see Section \ref{subsec The string coordinates}).

% \subsection[The form of the weight matrix (\texorpdfstring{$G/P$}{G/P} setting)]{The form of the weight matrix ($G/P$ setting)}

\subsection{The form of the weight matrix}
\fancyhead[L]{7.5 \ \ The form of the weight matrix}
% \fancyhead[L]{8.5 \ \ The form of the weight matrix}

We begin this section by making an observation on the definition of the weight map, namely that we can make an equivalent definition of the $t_{P,i}$ defining the weight map (see (\ref{eqn G/P defn gamma and t}) for the original definition), in terms of minimal $1$-paths. Then we prove that the weight map returns exactly the diagonal component of $b_P \in Z_P$, as desired.

Recall that each arrow coordinate $r_a$ is given in terms of vertex coordinates as $r_a = x_{v_{h(a)}}/x_{v_{t(a)}}$, where $h(a)=v_{h(a)}$ and $t(a)=v_{t(a)}$ denote the vertices at the head and tail of the arrow $a$ respectively. Additionally, writing $p(v_{jk})$ for the minimal $1$-path which begins at the vertex $v_{jk}$ (see Definition \ref{defn 1 paths leaving}) we have
    \begin{equation} \label{eqn 2nd G/P defn gamma and t}
    % \gamma :(\mathbb{K}^*)^{\mathcal{V}_L}\to T^{\vee}, \quad \boldsymbol{x}_{\mathcal{V}_P} \mapsto (t_i)_{i=1, \ldots, n} \quad \text{where} \quad
    t_{P,i} = x_{v_{n,n-i+1}} \prod\limits_{v \in \mathcal{D}_{i+1} \cap \mathcal{V}_P^{\bullet}} \prod\limits_{a \in p(v)} \frac{x_{v_{h(a)}}}{x_{v_{t(a)}}}
        = x_{v_{n,n-i+1}} \prod\limits_{v \in \mathcal{D}_{i+1} \cap \mathcal{V}_P^{\bullet}} \prod\limits_{a \in p(v)} r_a
    \end{equation}
where we set $x_{v_{n,n-i+1}}=x_{v_{n,n_l+1}}$, the star vertex coordinate in the last square $L_{l+1}$, if $v_{n,n-i+1}$ is not present in the quiver (similar to the terms in second product in (\ref{eqn G/P xi for wt map defn})).

For an example computation of $\gamma_P$ in our running example of $\mathcal{F}_{2,5,6}(\mathbb{C}^8)$, see %the example %\ref{ex wt matrix in full F2,5,6,C8} given in
Appendix \ref{append Example of complete quiver labelling}.

To see that the two definitions of $t_{P,i}$ given in (\ref{eqn G/P defn gamma and t}) and (\ref{eqn 2nd G/P defn gamma and t}) are equivalent, we notice that
    \begin{equation} \label{eqn prod of arrows in 1-path}
    \prod\limits_{a \in p(v)} \frac{x_{v_{h(a)}}}{x_{v_{t(a)}}} = \frac{x_{v_e(p(v))}}{x_{v_s(p(v))}}
    \end{equation}
where $v_s(p(v))=v$ is the vertex at the start of the path $p(v)$ and $v_e(p(v))$ is the vertex at the end of the same path.
If the $1$-path $p(v)$ has length one then the end vertex $v_e(p(v))$ is directly above the starting vertex $v_s(p(v))=v$, that is, the end vertex lies in the diagonal $\Xi_{P,i}$ and the starting vertex lies directly beneath it in the diagonal $\Xi_{P,i+1}$.

Suppose, however, that the $1$-path $p(v)$ has length greater than one. Necessarily this path must end at a star vertex and so $v_e(p(v))=x_{v_{n_j, n_{j-1}+1}}$ for some $j$, since $v$ lies below some square $L_j$ and $v_{n_j, n_{j-1}+1}$ is exactly the star vertex in $L_j$. Moreover, we see the contribution to $t_{P,i}$ from the path $p(v)$, as described in (\ref{eqn 2nd G/P defn gamma and t}), is exactly equal to the contributions in (\ref{eqn G/P defn gamma and t}) from $\mathcal{D}_i\cap L_j$, $\mathcal{D}_{i+1} \cap L_j$ and the vertex $v$:
    $$\prod\limits_{a \in p(v)} \frac{x_{v_{h(a)}}}{x_{v_{t(a)}}}
        = \frac{x_{v_e(p(v))}}{x_{v_s(p(v))}}
        = \frac{x_{v_{n_j, n_{j-1}+1}}}{x_v}
        =\frac{\left( \prod\limits_{\substack{v\in \mathcal{D}_i \cap L_j \\ v \in \mathcal{V}_P}} x_v \right) \left( \prod\limits_{\substack{v\in \mathcal{D}_i \cap L_j \\ v \notin \mathcal{V}_P}} x_{v_{n_j, n_{j-1}+1}} \right)}{\left( \prod\limits_{\substack{v\in \mathcal{D}_{i+1} \cap L_j \\ v \in \mathcal{V}_P}} x_v \right) \left( \prod\limits_{\substack{v\in \mathcal{D}_{i+1} \cap L_j \\ v \notin \mathcal{V}_P}} x_{v_{n_j, n_{j-1}+1}} \right) x_v }.
    $$
% since the star vertex coordinate in the square $L_j$ is exactly $x_{v_{n_j, n_{j-1}+1}}$.
Finally we note that any subsequent vertices from the indexing set $\mathcal{D}_{i+1} \cap \mathcal{V}_P^{\bullet}$ (taken in the order described in Section \ref{subsec The superpotential, highest weight and weight maps}) which lie below the same square $L_j$ as the vertex $v$, must give rise a minimal $1$-paths of length exactly one. Consequently we can be sure that all terms are considered but that there is no double counting.

% Extending Lemma 9.3 (pp.2432-2433) in Konni's `A mirror symmetric construction of {$qH_T(G/P)_{(q)}$}' / see Lemma {\ref{lem factor three maps through quiver torus}} in paper from upgrade report
\begin{lem}[Generalisation of {\cite[Lemma 9.3]{Rietsch2008}}, compare also Lemma {\ref{lem factor three maps through quiver torus}}] \label{lem b0 is gamma wt matrix}
If we factorise a general element of the quiver toric chart $\theta_P \left( \boldsymbol{x}_{\mathcal{V}_P^*} , \boldsymbol{r}_{\mathcal{A}_{P,\hat{\mathrm{v}}}}\right) =: b_P \in Z_P$ as $b_P=\left[b_P\right]_-\left[b_P\right]_0$, where $\left[b_P\right]_- \in U^{\vee}_-$ and $\left[b_P\right]_0 \in T^{\vee}$, then the diagonal component is exactly the weight matrix (defined in Section \ref{subsec The superpotential, highest weight and weight maps}):
    $$
    % \left[\theta_P \left( \boldsymbol{x}_{\mathcal{V}_P^*} , \boldsymbol{r}_{\mathcal{A}_{P,\hat{\mathrm{v}}}}\right) \right]_0 =
    \left[b_P\right]_0 = \gamma_P\left(\boldsymbol{x}_{\mathcal{V}_P} \right).
    % \quad \text{or equivalently} \quad
    % \left(\left[b_P\right]_0 \right)_{ii} = t_{P,i}.
    $$
\end{lem}

\begin{proof}
Let $\{ v_1, \ldots, v_n \}$ be the standard basis of $\mathbb{C}^n$ and choose the standard highest weight vector $v_{\omega_k}^+:= v_1 \wedge \cdots \wedge v_k$ in $V(\omega_k)=\bigwedge^k \mathbb{C}^n$. Then we have the lowest weight vector
    $$v_{\omega_k}^-:= \bar{w}_0 \cdot (v_1 \wedge \cdots \wedge v_k) = v_{n-k+1} \wedge \cdots \wedge v_n.
    $$

We recall the element $b_P \in Z_P$ defined by
    $$b_P:= u_L \kappa_P \dot{w}_P \bar{w}_0 u_R
    \in Z_P = B_-^{\vee}\cap U^{\vee}T^{W_P}\dot{w}_P\bar{w}_0U^{\vee}
    $$
where $u_L$ is short for the matrix $u_L\left(\boldsymbol{r}_{\mathcal{A}_{P,\hat{\mathrm{v}}}}\right)$ defined in Section \ref{subsec G/P Constructing matrices from a given quiver decoration}, $\kappa_P$ is short for the matrix $\kappa_P\left(\boldsymbol{x}_{\mathcal{V}_P^*}\right)$ defined in Section \ref{subsec The superpotential, highest weight and weight maps}, and $u_R \in U^{\vee}$ is the unique element such that $b_P$ lies in $B_-^{\vee}$. In particular, recalling the matrix $g_L$ defined in Section \ref{subsec G/P Constructing matrices from a given quiver decoration}, we have
    $$u_L \kappa_P \dot{w}_P \bar{w}_0 u_R = g_L \dot{w}_{L_{l+1}} \kappa_P \bar{w}_0 u_R
    $$
by Lemma \ref{lem gL in terms of uL and wL}.
We will show that
    $$\langle b_P \cdot v_{\omega_k}^+ , v_{\omega_k}^+ \rangle = \prod_{i=1}^k t_{P,i}.
    $$
To begin with we will work with arrow coordinates for ease of notation and then change to vertex coordinates when necessary.

First we observe the following:
    \begin{equation}
    \begin{aligned} \label{eqn b dot v plus}
    b_P \cdot v_{\omega_k}^+ &= g_L \dot{w}_{L_{l+1}} \kappa_P
        \cdot (v_{n-k+1} \wedge \cdots \wedge v_n) \\
    &= \left( \prod_{j=n-k+1}^n (\kappa_P)_{jj} \right) g_L \dot{w}_{L_{l+1}} \cdot (v_{n-k+1} \wedge \cdots \wedge v_n).
    \end{aligned}
    \end{equation}

Now we note that, written out,
    % $$\begin{aligned}
    % g_L \dot{w}_{L_{l+1}} &= \mathbf{x}^{\vee}_{n-1}\left(\prod\limits_{a \in p(v_{n1})} r_a \right) \cdots \mathbf{x}^{\vee}_{n_1}\left(\prod\limits_{a \in p(v_{n_1+1,1})} r_a \right) \dot{s}_{n_1-1} \cdots \dot{s}_1 \\
    %     &\hspace{0.47cm} \vdots \\
    %     &\hspace{0.47cm} \mathbf{x}^{\vee}_{n-1}\left(\prod\limits_{a \in p(v_{n, n_1})} r_a \right) \cdots \mathbf{x}^{\vee}_{n_1}\left(\prod\limits_{a \in p(v_{n_1+1, n_1})} r_a \right) \dot{s}_{n_1-1} \cdots \dot{s}_{n_1} \\
    %     %
    %     &\hspace{0.47cm} \mathbf{x}^{\vee}_{n-1}\left(\prod\limits_{a \in p(v_{n, n_1+1})} r_a \right) \cdots \mathbf{x}^{\vee}_{n_2}\left(\prod\limits_{a \in p(v_{n_2+1,n_1+1})} r_a \right) \dot{s}_{n_2-1} \cdots \dot{s}_{n_1+1} \\
    %     &\hspace{0.47cm} \vdots \\
    %     &\hspace{0.47cm} \mathbf{x}^{\vee}_{n-1}\left(\prod\limits_{a \in p(v_{n, n_2})} r_a \right) \cdots \mathbf{x}^{\vee}_{n_2}\left(\prod\limits_{a \in p(v_{n_2+1, n_2})} r_a \right) \dot{s}_{n_2-1} \cdots \dot{s}_{n_2} \\
    %     &\hspace{0.47cm} \vdots \\
    %     %
    %     &= \prod_{r=1}^{l+1} \prod_{i=n_{r-1}+1}^{n_r} \mathbf{x}^{\vee}_{n-1}\left(\prod\limits_{a \in p(v_{n, i})} r_a \right) \cdots \mathbf{x}^{\vee}_{n_r}\left(\prod\limits_{a \in p(v_{n_r+1, i})} r_a \right) \dot{s}_{n_r-1} \cdots \dot{s}_{i}  \\
    % \end{aligned}
    % $$
    \begin{equation} \label{eqn gL wLl+1 written in terms of 1-paths}
    g_L \dot{w}_{L_{l+1}} = \prod_{r=1}^{l+1} \prod_{i=n_{r-1}+1}^{n_r}
        \left( \mathbf{x}^{\vee}_{n-1}\left(\prod\limits_{a \in p(v_{n, i})} r_a \right) \mathbf{x}^{\vee}_{n-2}\left(\prod\limits_{a \in p(v_{n-1, i})} r_a \right)
        \cdots \mathbf{x}^{\vee}_{n_r}\left(\prod\limits_{a \in p(v_{n_r+1, i})} r_a \right) \dot{s}_{n_r-1} \cdots \dot{s}_{i} \right).
    \end{equation}
We also recall
    $$ \mathbf{x}^{\vee}_j (z) = \phi^{\vee}_i \begin{pmatrix} 1 & z \\ 0 & 1 \end{pmatrix}, \quad
    \dot{s}_i %= \mathbf{x}^{\vee}_i(1)\mathbf{y}^{\vee}_i(-1)\mathbf{x}^{\vee}_i(1)
    = \phi^{\vee}_i \begin{pmatrix} 0 & 1 \\ -1 & 0 \end{pmatrix}.
    $$
% Since each $\mathbf{x}^{\vee}_j (a) = \exp(ae_j)$ acts by $1+ae_j$ on $\bigwedge^k \mathbb{C}^n$,
Thus in order to get from the lowest to the highest weight space %**
via $g_L \dot{w}_{L_{l+1}}$ we need to take the $e_j$-summand from the last $k$ factors of each term (in the product (\ref{eqn gL wLl+1 written in terms of 1-paths})) which is indexed by a pair $(r,i)$ such that $i\leq n-k$, that is, of each term with at least $k$ factors. This is non-trivial if $n_r-1<k$ (or equivalently if $v_{i+k,i}\in \mathcal{V}_P^{\bullet}$, in addition to the constraints we have already placed on $i$) in which case we obtain
    % $$ \prod_{r=1}^{l+1} \prod_{\substack{i=n_{r-1}+1 \\ i\leq n-k}}^{n_r} \left( \mathbf{x}^{\vee}_{i+k-1}\left(\prod\limits_{a \in p(v_{i+k, i})} r_a \right) \cdots \mathbf{x}^{\vee}_{n_r}\left(\prod\limits_{a \in p(v_{n_r+1, i})} r_a \right) \dot{s}_{n_r-1} \cdots \dot{s}_{i} \right)
    % $$ % this is product of factors contributing, not the $e_j$-summands of those factors!
    $$ \begin{aligned}
    \prod_{r=1}^{l+1} \prod_{\substack{i=n_{r-1}+1 \\ i\leq n-k \\ v_{i+k,i}\in \mathcal{V}_P^{\bullet}}}^{n_r} \left( \left(\prod\limits_{a \in p(v_{i+k, i})} r_a \right) \cdots \left(\prod\limits_{a \in p(v_{n_r+1, i})} r_a \right) \right)
        % &= \prod\limits_{v \in \left(\bigcup_{i=2}^{k+1}\mathcal{D}_{i}\right) \cap \mathcal{V}_P^{\bullet}} \prod\limits_{a \in p(v)} r_a \\
        &= \prod_{r=1}^{l+1} \prod_{\substack{i=n_{r-1}+1 \\ i\leq n-k \\ v_{i+k,i} \in \mathcal{V}_P^{\bullet}}}^{n_r} \frac{x_{v_{n_r, n_{r-1}+1}}}{x_{v_{i+k,i}}} & \text{by (\ref{eqn prod of arrows in 1-path})\footnotemark} \\
        &= \prod_{r=1}^{l+1} \prod_{\substack{i=n_{r-1}+1 \\ i\leq n-k \\ v_{i+k,i} \in \mathcal{V}_P^{\bullet}}}^{n_r} \frac{(\kappa_P)_{ii}}{x_{v_{i+k,i}}}
            & \begin{aligned}[t] \text{using (\ref{eqn defn kappa hw map}) since} & \\
            v_{n_r,n_{r-1}+1} \text{ is the} & \\
            \text{star vertex in } %& \\
            % \text{the square }
            L_r & \end{aligned} \\
        &= \prod_{v_{i+k,i} \in \mathcal{D}_{k+1}\cap \mathcal{V}_P^{\bullet}} \frac{(\kappa_P)_{ii}}{x_{v_{i+k,i}}}
    \end{aligned}
    $$
\footnotetext{On the left hand side, within the outer parentheses, the $k$ minimal $1$-paths can be concatenated to form a path that starts at the vertex $v_{i+k,i}$. The first $k-1$ of these $1$-paths all have length $1$, the last $1$-path might be longer. Subsequently, the new path travels vertically upwards until it reaches the topmost dot vertex in the column $i$, and then completes the minimal $1$-path starting at this vertex; it will potentially travel left until it reaches the dot vertex directly below the star vertex in the square $L_r$, and finally travel upwards to end at this star vertex, $x_{v_{n_r, n_{r-1}+1}}$.

Here we have chosen to order our products of arrow coordinates to facilitate comparison with (\ref{eqn gL wLl+1 written in terms of 1-paths}) (rather than the ordering used on p.~\pageref{eqn factor Xi into xi}).}

Combining this with (\ref{eqn b dot v plus}) we have
    $$\begin{aligned}
    \langle b_P \cdot v_{\omega_k}^+ , v_{\omega_k}^+ \rangle
        % &= \left( \prod_{j=n-k+1}^n (t_q)_{jj} \right) \prod\limits_{v \in \left(\bigcup_{i=2}^{k+1}\mathcal{D}_{i}\right) \cap \mathcal{V}^{\bullet}} \prod\limits_{a \in p(v)} r_a \\
        &= \left( \prod_{j=n-k+1}^n (\kappa_P)_{jj} \right) \left( \prod_{v_{i+k,i} \in \mathcal{D}_{k+1}\cap \mathcal{V}_P^{\bullet}} \frac{(\kappa_P)_{ii}}{x_{v_{i+k,i}}} \right) \\
        &= \left( \prod_{j=n-k+1}^n (\kappa_P)_{jj} \right) \left( \prod_{v_{i+k,i} \in \mathcal{D}_{k+1}\cap \mathcal{V}_P^{\bullet}} \frac{(\kappa_P)_{ii}}{x_{v_{i+k,i}}} \right) \left( \prod_{\substack{v_{i+k,i} \in \mathcal{D}_{k+1} \\ v_{i+k,i} \notin \mathcal{V}_P^{\bullet}}} \frac{(\kappa_P)_{ii}}{(\kappa_P)_{ii}} \right) \\
        &= \left( \prod_{j=1}^n (\kappa_P)_{jj} \right) \left( \prod_{v_{i+k,i} \in \mathcal{D}_{k+1}\cap \mathcal{V}_P^{\bullet}} \frac{1}{x_{v_{i+k,i}}} \right) \left( \prod_{\substack{v_{i+k,i} \in \mathcal{D}_{k+1} \\ v_{i+k,i} \notin \mathcal{V}_P^{\bullet}}} \frac{1}{(\kappa_P)_{ii}} \right) \\
        &= \left( \prod_{j=1}^n (\kappa_P)_{jj} \right) \left( \prod_{v_{i+k,i} \in \mathcal{D}_{k+1}\cap \mathcal{V}_P} \frac{1}{x_{v_{i+k,i}}} \right) \left( \prod_{\substack{v_{i+k,i} \in \mathcal{D}_{k+1} \\ v_{i+k,i} \notin \mathcal{V}_P}} \frac{1}{(\kappa_P)_{ii}} \right) \\
        &= \frac{\Xi_{P,1}}{\Xi_{P,k+1}} = \prod_{i=1}^k \frac{\Xi_{P,i}}{\Xi_{P,i+1}} = \prod_{i=1}^k t_{P,i}
    \end{aligned}
    $$
\end{proof}

\section[A conjecture on the form of elements of \texorpdfstring{$Z_P$}{Z}]{A conjecture on the form of elements of {$Z_P$}} \label{sec A conjecture on the form of elements of ZP}
\fancyhead[L]{8 \ \ A conjecture on the form of elements of {$Z_P$}}
% \fancyhead[L]{9 \ \ A conjecture on the form of elements of {$Z_P$}}

In this section we present a second quiver construction for $G/P$ and use it to construct matrices similar to $g_L$ and $u_L$ in Section \ref{subsec G/P Constructing matrices from a given quiver decoration}. Our goal in doing so is to describe the last factor in the quiver toric chart explicitly, without reference to the other four factors or $B_-^{\vee}$. We present this as a conjecture, together with supporting evidence.

\subsection[The decorated quiver \texorpdfstring{$Q_{P,R}$}{Q\_{P,R}}]{The decorated quiver $Q_{P,R}$} \label{subsec G/P The right-quiver}
\fancyhead[L]{8.1 \ \ The decorated quiver $Q_{P,R}$}
% \fancyhead[L]{9.1 \ \ The decorated quiver $Q_{P,R}$}

In this section we define a decorated quiver $Q_{P,R}$ from the quiver $Q_P$, with the intention of expressing the factor $u_R$ on the right hand side of $b_P\in Z_P$ in terms of arrow coordinates. To define $Q_{P,R}$ we begin with the vertices and arrows of the quiver $Q_P$ and reflect this through the anti-diagonal, that is the line `$x=y$'. The quiver we obtain looks the same as if we had started with squares $R_i$ of size $k_{l+2-i} \times k_{l+2-i}$ instead the squares $L_i$ (that is, $L_i$ is sent to $R_{l+2-i}$ under the reflection), and then oriented our arrows down and to the right.

Taking our new quiver we again label the $n-1$ rows from top to bottom by $E_1, \ldots, E_{n-1}$. Each row $E_i$ intersects some square $R_{l+2-j}$ on the diagonal. Similar to the quiver $Q_P$, the intersection of this row and square contains $i- \sum_{r=j+1}^{l+1} k_r =i-(n-n_{j})$ copies of $\dot{s}_i$, each written in a circle. For example see Figure \ref{fig right quiver in squares for 2,5,6 n8}.

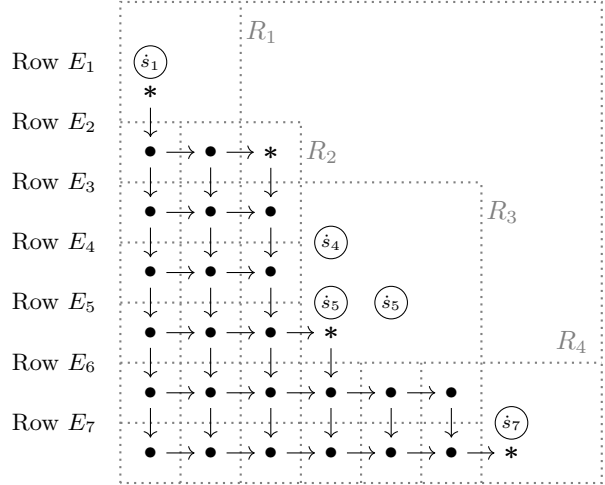
\begin{figure}[ht]
\centering
\begin{tikzpicture}[scale=0.8]
    %squares
        \draw[dotted, thick, color=black!50] (0.5,0.5) -- (0.5,8.5) -- (8.5,8.5) -- (8.5,0.5) -- cycle;
        \draw[dotted, thick, color=black!50] (1.5,0.5) -- (1.5,6.5);
        \draw[dotted, thick, color=black!50] (2.5,0.5) -- (2.5,8.5);
        \draw[dotted, thick, color=black!50] (3.5,0.5) -- (3.5,6.5);
        \draw[dotted, thick, color=black!50] (4.5,0.5) -- (4.5,2.5);
        \draw[dotted, thick, color=black!50] (5.5,0.5) -- (5.5,2.5);
        \draw[dotted, thick, color=black!50] (6.5,0.5) -- (6.5,5.5);

        \draw[dotted, thick, color=black!50] (0.5,6.5) -- (3.5,6.5);
        \draw[dotted, thick, color=black!50] (0.5,5.5) -- (6.5,5.5);
        \draw[dotted, thick, color=black!50] (0.5,4.5) -- (3.5,4.5);
        \draw[dotted, thick, color=black!50] (0.5,3.5) -- (3.5,3.5);
        \draw[dotted, thick, color=black!50] (0.5,2.5) -- (8.5,2.5);
        \draw[dotted, thick, color=black!50] (0.5,1.5) -- (6.5,1.5);

    %square labels
        \node[color=black!50] at (2.85,8) {$R_1$};
        \node[color=black!50] at (3.85,6) {$R_2$};
        \node[color=black!50] at (6.85,5) {$R_3$};
        \node[color=black!50] at (8,2.85) {$R_4$};

    %dots and stars
        % \node (11) at (1,8) {$\boldsymbol{*}$};
        \node (21) at (1,7) {$\boldsymbol{*}$};
        \node (31) at (1,6) {$\bullet$};
        \node (41) at (1,5) {$\bullet$};
        \node (51) at (1,4) {$\bullet$};
        \node (61) at (1,3) {$\bullet$};
        \node (71) at (1,2) {$\bullet$};
        \node (81) at (1,1) {$\bullet$};

        % \node (22) at (2,7) {$\boldsymbol{*}$};
        \node (32) at (2,6) {$\bullet$};
        \node (42) at (2,5) {$\bullet$};
        \node (52) at (2,4) {$\bullet$};
        \node (62) at (2,3) {$\bullet$};
        \node (72) at (2,2) {$\bullet$};
        \node (82) at (2,1) {$\bullet$};

        \node (33) at (3,6) {$\boldsymbol{*}$};
        \node (43) at (3,5) {$\bullet$};
        \node (53) at (3,4) {$\bullet$};
        \node (63) at (3,3) {$\bullet$};
        \node (73) at (3,2) {$\bullet$};
        \node (83) at (3,1) {$\bullet$};

        % \node (44) at (4,5) {$\boldsymbol{*}$};
        % \node (54) at (4,4) {$\bullet$};
        \node (64) at (4,3) {$\boldsymbol{*}$};
        \node (74) at (4,2) {$\bullet$};
        \node (84) at (4,1) {$\bullet$};

        % \node (55) at (5,4) {$\boldsymbol{*}$};
        % \node (65) at (5,3) {$\bullet$};
        \node (75) at (5,2) {$\bullet$};
        \node (85) at (5,1) {$\bullet$};

        % \node (66) at (6,3) {$\boldsymbol{*}$};
        \node (76) at (6,2) {$\bullet$};
        \node (86) at (6,1) {$\bullet$};

        % \node (77) at (7,2) {$\boldsymbol{*}$};
        \node (87) at (7,1) {$\boldsymbol{*}$};

        % \node (88) at (8,1) {$\boldsymbol{*}$};

    %Row E_i
        \node at (-0.6,1.5) {\small{Row $E_7$}};
        \node at (-0.6,2.5) {\small{Row $E_6$}};
        \node at (-0.6,3.5) {\small{Row $E_5$}};
        \node at (-0.6,4.5) {\small{Row $E_4$}};
        \node at (-0.6,5.5) {\small{Row $E_3$}};
        \node at (-0.6,6.5) {\small{Row $E_2$}};
        \node at (-0.6,7.5) {\small{Row $E_1$}};

    %s_i's
        \node[draw, circle, minimum size=4mm, inner sep=1pt] at (1,7.5) {\scriptsize{$\dot{s}_1$}};
        \node[draw, circle, minimum size=4mm, inner sep=1pt] at (4,4.5) {\scriptsize{$\dot{s}_4$}};
        \node[draw, circle, minimum size=4mm, inner sep=1pt] at (4,3.5) {\scriptsize{$\dot{s}_5$}};
        \node[draw, circle, minimum size=4mm, inner sep=1pt] at (5,3.5) {\scriptsize{$\dot{s}_5$}};
        \node[draw, circle, minimum size=4mm, inner sep=1pt] at (7,1.5) {\scriptsize{$\dot{s}_7$}};

    %vertical arrows
        \draw[<-] (81) -- (71);
        \draw[<-] (71) -- (61);
        \draw[<-] (61) -- (51);
        \draw[<-] (51) -- (41);
        \draw[<-] (41) -- (31);
        \draw[<-] (31) -- (21);
        % \draw[<-] (21) -- (11);

        \draw[<-] (82) -- (72);
        \draw[<-] (72) -- (62);
        \draw[<-] (62) -- (52);
        \draw[<-] (52) -- (42);
        \draw[<-] (42) -- (32);
        % \draw[<-] (32) -- (22);

        \draw[<-] (83) -- (73);
        \draw[<-] (73) -- (63);
        \draw[<-] (63) -- (53);
        \draw[<-] (53) -- (43);
        \draw[<-] (43) -- (33);

        \draw[<-] (84) -- (74);
        \draw[<-] (74) -- (64);
        % \draw[<-] (64) -- (54);
        % \draw[<-] (54) -- (44);

        \draw[<-] (85) -- (75);
        % \draw[<-] (75) -- (65);
        % \draw[<-] (65) -- (55);

        \draw[<-] (86) -- (76);
        % \draw[<-] (76) -- (66);

        % \draw[<-] (87) -- (77);

    %horizontal arrows
        % \draw[<-] (22) -- (21);
        \draw[<-] (32) -- (31);
        \draw[<-] (42) -- (41);
        \draw[<-] (52) -- (51);
        \draw[<-] (62) -- (61);
        \draw[<-] (72) -- (71);
        \draw[<-] (82) -- (81);

        \draw[<-] (33) -- (32);
        \draw[<-] (43) -- (42);
        \draw[<-] (53) -- (52);
        \draw[<-] (63) -- (62);
        \draw[<-] (73) -- (72);
        \draw[<-] (83) -- (82);

        % \draw[<-] (44) -- (43);
        % \draw[<-] (54) -- (53);
        \draw[<-] (64) -- (63);
        \draw[<-] (74) -- (73);
        \draw[<-] (84) -- (83);

        % \draw[<-] (55) -- (54);
        % \draw[<-] (65) -- (64);
        \draw[<-] (75) -- (74);
        \draw[<-] (85) -- (84);

        % \draw[<-] (66) -- (65);
        \draw[<-] (76) -- (75);
        \draw[<-] (86) -- (85);

        % \draw[<-] (77) -- (76);
        \draw[<-] (87) -- (86);

        % \draw[<-] (88) -- (87);
    \end{tikzpicture}
\caption{Quiver $Q_{P,R}$ for $\mathcal{F}_{2,5,6}(\mathbb{C}^8)$} \label{fig right quiver in squares for 2,5,6 n8}
\end{figure}

\begin{rem}
During the above reflection, the arrows and vertices retain any associated decoration. For example, if the topmost upwards arrow in the first column of $Q_P$ has coordinate $r$, then the rightmost arrow directed to the right in the bottom row of $Q_{P,R}$ also has coordinate $r$.
\end{rem}

% \newpage
\subsection[Constructing matrices from the \texorpdfstring{$Q_{P,R}$}{Q\_{P,R}} quiver decoration]{Constructing matrices from the $Q_{P,R}$ quiver decoration} \label{subsec G/P Constructing matrices from the right-quiver decoration}
\fancyhead[L]{8.2 \ \ Constructing matrices from the $Q_{P,R}$ quiver decoration}
% \fancyhead[L]{9.2 \ \ Constructing matrices from the $Q_{P,R}$ quiver decoration}

Similar to the construction of the matrices $g_L$ and $u_L$ in Section \ref{subsec G/P Constructing matrices from a given quiver decoration}, in this section we construct two matrices, $g_R$ and $\tilde{u}_R$, from the decoration of the quiver $Q_{P,R}$, again inspired by Marsh and Rietsch in \cite{MarshRietsch2020}.
We recall that $g_L$ and $u_L$ relate to the first term of the quiver toric chart on $Z_P$ (defined in Section \ref{subsec G/P The quiver torus and toric chart}), which was called $u_1$ in the $G/B$ setting. The last term in this chart, called $u_R$ (or $u_2$ in the $G/B$ setting) is the unique matrix in $U^{\vee}$ such that the whole product, $b_P$, lies in $Z_P$ (resp. $b \in Z$). The new matrices $g_R$ and $\tilde{u}_R$ will conjecturally allow us to describe this last term explicitly.

As with $g_L$ and $u_L$, the two new matrices may be thought of as maps on the vertex (or equivalently arrow) torus, we denote them by
    $$\begin{aligned}
    g_R &: \left(\mathbb{K}^*\right)^{\mathcal{V}_P} \to G^{\vee}, %\quad \boldsymbol{x}_{\mathcal{V}_P} \mapsto g_R\left(\boldsymbol{x}_{\mathcal{V}_P}\right)
    \\
    \tilde{u}_R &: \left(\mathbb{K}^*\right)^{\mathcal{V}_P} \to U^{\vee}.
    \end{aligned}
    $$
Unless otherwise stated we will write $g_R$ for the matrix $g_R(\boldsymbol{x}_{\mathcal{V}_P}) \in G^{\vee}$ in order to simplify notation, and similarly for $\tilde{u}_R$.

We begin with $g_R$ and, similar to the construction of $g_L$, we consider $1$-paths, however now in the quiver $Q_{P,R}$. In particular we are interested in the minimal $1$-paths which end (rather than start, like in Definition \ref{defn 1 paths leaving}) at dot vertices. We will denote the unique minimal $1$-path which \emph{ends} at a dot vertex $v\in \mathcal{V}_P^{\bullet}$ by $\hat{p}(v)$. With this notation we are ready to construct $g_R$ as follows:
\begin{itemize}
    \item We start at the top-most dot vertex in the last column of dot vertices, namely the $(n-n_1)$-th column. Working downwards, we list the minimal $1$-paths ending at each dot vertex in succession.

    \item We repeat this column by column to the left, making note of any circled $\dot{s}_i$'s which appear at the top of a given column and then listing the minimal $1$-paths as before, until we have considered all columns containing dot vertices.
    Of note, there are no contributions from within the last square, $R_{l+1}$, just like there were no contributions from within the last square $L_{l+1}$ in the construction of $g_L$.

    \item As before, to the circled $\dot{s}_i$'s we associate the obvious factors $\dot{s}_i$. Similarly to any $1$-path which crosses row $E_i$, we associate the element $\mathbf{x}^{\vee}_i(x_{v_e}/x_{v_s}) \in U^{\vee}$, where $x_{v_s}$ is the vertex coordinate at the start of the path and $x_{v_e}$ is the vertex coordinate at the end of the path. Again we note that the quotient $x_{v_e}/x_{v_s}$ is equal to the product of arrow coordinates of all the arrows in the $1$-path.
\end{itemize}
The matrix we obtain from taking the product of $\mathbf{x}^{\vee}_i$ and $\dot{s}_i$ factors in the order defined by the above list, is $g_R$.

Now we note that the last term in our quiver toric chart, $u_R$ is an element of $U^{\vee}$, however in general $g_R \notin U^{\vee}$. Consequently we will use the quiver decoration to construct our second matrix, $\tilde{u}_R \in U_-^{\vee}$, which we will show is a factor of $g_R$. Both $\tilde{u}_R$ and $g_R$ will be helpful later.

Similar to the definition of $u_L$ in Section \ref{subsec G/P Constructing matrices from a given quiver decoration}, to define $\tilde{u}_R$ we proceed in a similar way to the construction of $g_R$. However there are two differences when we construct $\tilde{u}_R$, the first being that when we list the minimal $1$-paths we do not include any $\dot{s}_i$'s. The second difference comes in the association of elements $\mathbf{x}^{\vee}_i(x_{v_e}/x_{v_s}) \in U^{\vee}$ to $1$-paths crossing row $E_i$. If a $1$-path in our list has length $1$ then we treat it as before, namely we assign the element $\mathbf{x}^{\vee}_i(x_{v_e}/x_{v_s}) \in U^{\vee}$. To each $1$-path of length $\alpha\geq2$ in our list, we associate the element
    $$X^{\vee}_{i,\alpha}((-1)^{\alpha-1}x_{v_e}/x_{v_s}):= I + (-1)^{\alpha-1}\frac{x_{v_e}}{x_{v_s}}E_{i+\alpha+1,i+1}
    $$
where $I$ is the identity matrix and $E_{jk}$ is the elementary matrix with $1$ in position $(j,k)$ and $0$'s elsewhere. The resulting product of $\mathbf{x}^{\vee}_i$ and $X^{\vee}_{i,\alpha}$ factors defines the matrix $\tilde{u}_R$.

In order to succinctly describe the relation between $\tilde{u}_R$ and $g_R$ we introduce some notation. Recalling the definition of $\dot{w}_{L_{[i,j]}}$ from Section \ref{subsec G/P Constructing matrices from a given quiver decoration}, we analogously let $\dot{w}_{R_{[i,j]}}$ denote the sub-product of $\dot{s}_i$'s given by the contributions in $g_R$ from within squares $R_i, \ldots, R_j$, with respect to the ordering in $g_R$. If we consider a single square $R_i$, we will write $\dot{w}_{R_i}$ instead of $\dot{w}_{R_{\{i\}}}$ to simplify notation.
% we have
%     $$\dot{w}_R = \dot{w}_{R_{[1,l]}}, \quad  \dot{w}_{R_{[1,i]}} = \left(\dot{w}_{R_{[i+1,l]}}\right)^{-1} \dot{w}_R.
%     $$

More explicitly, for $j=1, \ldots, l+1$, $\dot{w}_{R_{l+2-j}}$ (that is, the image of $L_j$ under the reflection described at the start of Section \ref{subsec G/P The right-quiver}) is a representative of the longest element of
    $$\left\langle s_i \ \bigg| \ i \in \left\{ \left(\sum_{r=j+1}^{l+1}k_r \right)+1, \ldots, \left(\sum_{r=j}^{l+1}k_r \right) -1 \right\} \right\rangle
        =\left\langle s_i \ | \ i \in \{ n-(n_j-1), \ldots, n-(n_{j-1}+1) \} \right\rangle \subset W
    $$
given by the reduced expression
    $$(n-(n_{j-1}+1), n-(n_{j-1}+2), n-(n_{j-1}+1), \ldots, n-(n_j-1) \ldots, n-(n_{j-1}+1))
    $$
recalling $n_0=0$ and $n_{l+1}=n$. As as the case of $\dot{w}_{L_j}$, this follows from the row set in the quiver $Q_{P,R}$ which the interior of $R_{l+2-j}$ intersects, namely the set $\{ E_{n-(n_j-1)}, \ldots, E_{n-(n_{j-1}+1)} \}$, together with two facts. Firstly, $\dot{w}_{R_{l+2-j}}$ is given by a product of $\dot{s}_i$ factors, read from the square $R_{l+2-j}$ by starting at the top of the right-most column and working downwards, then proceeding column by column to the left. Secondly, by definition, the intersection of the square $R_{l+2-j}$ and a row $E_i$, contains $i-\sum_{r=j+1}^{l+1}k_r =i-(n-n_j)$ copies of $\dot{s}_i$.

We see that $\dot{w}_{R_{[i,j]}} = \dot{w}_{R_j} \dot{w}_{R_{j-1}} \cdots \dot{w}_{R_i}$, where some of the $\dot{w}_{R_r}$ may be trivial.
For ease of notation we also define $\dot{w}_R:= \dot{w}_{R_{[1,l]}} = \dot{w}_{R_{l+1}}^{-1} \dot{w}_P$

\begin{lem} \label{lem gR in terms of uR and wR}
With the above $\dot{w}_{R_i}$ notation, the matrices $\tilde{u}_R$ and $g_R$ are related by $\dot{w}_{R_{l+1}} g_R =  \dot{w}_P \tilde{u}_R$, or equivalently
    \begin{equation} \label{eqn gR in terms of uR and si}
    g_R = \dot{w}_R \tilde{u}_R
    \end{equation}
since by definition $\dot{w}_P = \dot{w}_{R_{[1,l+1]}} = \dot{w}_{R_{l+1}} \dot{w}_{R_l} \cdots \dot{w}_{R_1}$ and $\dot{w}_R = \dot{w}_{R_{l+1}}^{-1} \dot{w}_P$.
\end{lem}

To aid our familiarity with $\tilde{u}_R$ and $g_R$ and motivate the proof of this lemma, we give an example.

\begin{ex} \label{ex gR temp quiver decoration}
We will continue our running example of $\mathcal{F}_{2,5,6}(\mathbb{C}^8)$. For ease of notation, in this example we label a number of arrows in the $Q_{P,R}$ quiver as in Figure \ref{fig gR temp quiver decoration}, and write $a$ in place of the arrow coordinate $r_a$.

\begin{figure}[ht]
\centering
\begin{tikzpicture}[scale=0.8]
    %squares
        \draw[dotted, thick, color=black!50] (0.5,0.5) -- (0.5,8.5) -- (8.5,8.5) -- (8.5,0.5) -- cycle;
        \draw[dotted, thick, color=black!50] (1.5,0.5) -- (1.5,6.5);
        \draw[dotted, thick, color=black!50] (2.5,0.5) -- (2.5,8.5);
        \draw[dotted, thick, color=black!50] (3.5,0.5) -- (3.5,6.5);
        \draw[dotted, thick, color=black!50] (4.5,0.5) -- (4.5,2.5);
        \draw[dotted, thick, color=black!50] (5.5,0.5) -- (5.5,2.5);
        \draw[dotted, thick, color=black!50] (6.5,0.5) -- (6.5,5.5);

        \draw[dotted, thick, color=black!50] (0.5,6.5) -- (3.5,6.5);
        \draw[dotted, thick, color=black!50] (0.5,5.5) -- (6.5,5.5);
        \draw[dotted, thick, color=black!50] (0.5,4.5) -- (3.5,4.5);
        \draw[dotted, thick, color=black!50] (0.5,3.5) -- (3.5,3.5);
        \draw[dotted, thick, color=black!50] (0.5,2.5) -- (8.5,2.5);
        \draw[dotted, thick, color=black!50] (0.5,1.5) -- (6.5,1.5);

    %square labels
        \node[color=black!50] at (2.85,8) {$R_1$};
        \node[color=black!50] at (3.85,6) {$R_2$};
        \node[color=black!50] at (6.85,5) {$R_3$};
        \node[color=black!50] at (8,2.85) {$R_4$};

    %dots and stars
        % \node (11) at (1,8) {$\boldsymbol{*}$};
        \node (21) at (1,7) {$\boldsymbol{*}$};
        \node (31) at (1,6) {$\bullet$};
        \node (41) at (1,5) {$\bullet$};
        \node (51) at (1,4) {$\bullet$};
        \node (61) at (1,3) {$\bullet$};
        \node (71) at (1,2) {$\bullet$};
        \node (81) at (1,1) {$\bullet$};

        % \node (22) at (2,7) {$\boldsymbol{*}$};
        \node (32) at (2,6) {$\bullet$};
        \node (42) at (2,5) {$\bullet$};
        \node (52) at (2,4) {$\bullet$};
        \node (62) at (2,3) {$\bullet$};
        \node (72) at (2,2) {$\bullet$};
        \node (82) at (2,1) {$\bullet$};

        \node (33) at (3,6) {$\boldsymbol{*}$};
        \node (43) at (3,5) {$\bullet$};
        \node (53) at (3,4) {$\bullet$};
        \node (63) at (3,3) {$\bullet$};
        \node (73) at (3,2) {$\bullet$};
        \node (83) at (3,1) {$\bullet$};

        % \node (44) at (4,5) {$\boldsymbol{*}$};
        % \node (54) at (4,4) {$\bullet$};
        \node (64) at (4,3) {$\boldsymbol{*}$};
        \node (74) at (4,2) {$\bullet$};
        \node (84) at (4,1) {$\bullet$};

        % \node (55) at (5,4) {$\boldsymbol{*}$};
        % \node (65) at (5,3) {$\bullet$};
        \node (75) at (5,2) {$\bullet$};
        \node (85) at (5,1) {$\bullet$};

        % \node (66) at (6,3) {$\boldsymbol{*}$};
        \node (76) at (6,2) {$\bullet$};
        \node (86) at (6,1) {$\bullet$};

        % \node (77) at (7,2) {$\boldsymbol{*}$};
        \node (87) at (7,1) {$\boldsymbol{*}$};

        % \node (88) at (8,1) {$\boldsymbol{*}$};

    % arrow labels
            \node at (0.82,6.64) {\scriptsize{$s$}};
            \node at (0.82,5.64) {\scriptsize{$t$}};
            \node at (0.82,4.64) {\scriptsize{$u$}};
            \node at (0.82,3.64) {\scriptsize{$v$}};
            \node at (0.82,2.64) {\scriptsize{$w$}};
            \node at (0.82,1.64) {\scriptsize{$x$}};

            \node at (1.38,6.2) {\scriptsize{$l$}};
            \node at (1.82,5.64) {\scriptsize{$m$}};
            \node at (1.82,4.64) {\scriptsize{$n$}};
            \node at (1.82,3.64) {\scriptsize{$p$}};
            \node at (1.82,2.64) {\scriptsize{$q$}};
            \node at (1.82,1.64) {\scriptsize{$r$}};

            \node at (2.82,5.64) {\scriptsize{$g$}};
            \node at (2.82,4.64) {\scriptsize{$h$}};
            \node at (2.82,3.64) {\scriptsize{$i$}};
            \node at (2.82,2.64) {\scriptsize{$j$}};
            \node at (2.82,1.64) {\scriptsize{$k$}};

            \node at (3.82,2.64) {\scriptsize{$e$}};
            \node at (3.82,1.64) {\scriptsize{$f$}};

            \node at (4.38,2.2) {\scriptsize{$c$}};
            \node at (4.82,1.64) {\scriptsize{$d$}};

            \node at (5.38,2.2) {\scriptsize{$a$}};
            \node at (5.82,1.64) {\scriptsize{$b$}};

    %Row E_i
        \node at (-0.6,1.5) {\small{Row $E_7$}};
        \node at (-0.6,2.5) {\small{Row $E_6$}};
        \node at (-0.6,3.5) {\small{Row $E_5$}};
        \node at (-0.6,4.5) {\small{Row $E_4$}};
        \node at (-0.6,5.5) {\small{Row $E_3$}};
        \node at (-0.6,6.5) {\small{Row $E_2$}};
        \node at (-0.6,7.5) {\small{Row $E_1$}};

    %s_i's
        \node[draw, circle, minimum size=4mm, inner sep=1pt] at (1,7.5) {\scriptsize{$\dot{s}_1$}};
        \node[draw, circle, minimum size=4mm, inner sep=1pt] at (4,4.5) {\scriptsize{$\dot{s}_4$}};
        \node[draw, circle, minimum size=4mm, inner sep=1pt] at (4,3.5) {\scriptsize{$\dot{s}_5$}};
        \node[draw, circle, minimum size=4mm, inner sep=1pt] at (5,3.5) {\scriptsize{$\dot{s}_5$}};
        \node[draw, circle, minimum size=4mm, inner sep=1pt] at (7,1.5) {\scriptsize{$\dot{s}_7$}};

    %vertical arrows
        \draw[<-] (81) -- (71);
        \draw[<-] (71) -- (61);
        \draw[<-] (61) -- (51);
        \draw[<-] (51) -- (41);
        \draw[<-] (41) -- (31);
        \draw[<-] (31) -- (21);
        % \draw[<-] (21) -- (11);

        \draw[<-] (82) -- (72);
        \draw[<-] (72) -- (62);
        \draw[<-] (62) -- (52);
        \draw[<-] (52) -- (42);
        \draw[<-] (42) -- (32);
        % \draw[<-] (32) -- (22);

        \draw[<-] (83) -- (73);
        \draw[<-] (73) -- (63);
        \draw[<-] (63) -- (53);
        \draw[<-] (53) -- (43);
        \draw[<-] (43) -- (33);

        \draw[<-] (84) -- (74);
        \draw[<-] (74) -- (64);
        % \draw[<-] (64) -- (54);
        % \draw[<-] (54) -- (44);

        \draw[<-] (85) -- (75);
        % \draw[<-] (75) -- (65);
        % \draw[<-] (65) -- (55);

        \draw[<-] (86) -- (76);
        % \draw[<-] (76) -- (66);

        % \draw[<-] (87) -- (77);

    %horizontal arrows
        % \draw[<-] (22) -- (21);
        \draw[<-] (32) -- (31);
        \draw[<-] (42) -- (41);
        \draw[<-] (52) -- (51);
        \draw[<-] (62) -- (61);
        \draw[<-] (72) -- (71);
        \draw[<-] (82) -- (81);

        \draw[<-] (33) -- (32);
        \draw[<-] (43) -- (42);
        \draw[<-] (53) -- (52);
        \draw[<-] (63) -- (62);
        \draw[<-] (73) -- (72);
        \draw[<-] (83) -- (82);

        % \draw[<-] (44) -- (43);
        % \draw[<-] (54) -- (53);
        \draw[<-] (64) -- (63);
        \draw[<-] (74) -- (73);
        \draw[<-] (84) -- (83);

        % \draw[<-] (55) -- (54);
        % \draw[<-] (65) -- (64);
        \draw[<-] (75) -- (74);
        \draw[<-] (85) -- (84);

        % \draw[<-] (66) -- (65);
        \draw[<-] (76) -- (75);
        \draw[<-] (86) -- (85);

        % \draw[<-] (77) -- (76);
        \draw[<-] (87) -- (86);

        % \draw[<-] (88) -- (87);
    \end{tikzpicture}
\caption{Temporary $Q_{P,R}$ quiver labelling for Example {\ref{ex gR temp quiver decoration}}} \label{fig gR temp quiver decoration}
\end{figure}
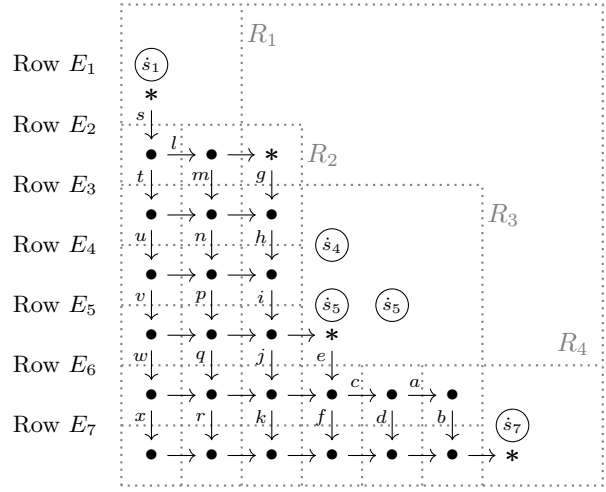

We will use the above construction to obtain $g_R$. We then give an alternate factorisation of this matrix with all $\dot{s}_i$ factors on the left, that is $\dot{w}_R\tilde{u}_R$. Most $\mathbf{x}^{\vee}_i$ factors appear in both factorisations with a similar pattern. For ease of comparison we have underlined the $\mathbf{x}^{\vee}_i$ and $X^{\vee}_{i,\alpha}$ factors that don't appear in both factorisations.

    % $$\begin{aligned}
    % g_R &=  \underline{\mathbf{x}^{\vee}_6(ace)} \mathbf{x}^{\vee}_7(b) \\
    %     &\hspace{0.47cm} \dot{s}_5 \underline{\mathbf{x}^{\vee}_6(ce)} \mathbf{x}^{\vee}_7(d) \\
    %     &\hspace{0.47cm} \dot{s}_4 \dot{s}_5 \mathbf{x}^{\vee}_6(e) \mathbf{x}^{\vee}_7(f) \\
    %     &\hspace{0.47cm} \mathbf{x}^{\vee}_3(g) \mathbf{x}^{\vee}_4(h) \mathbf{x}^{\vee}_5(i) \mathbf{x}^{\vee}_6(j) \mathbf{x}^{\vee}_7(k) \\
    %     &\hspace{0.47cm} \underline{\mathbf{x}^{\vee}_2(ls)} \mathbf{x}^{\vee}_3(m) \mathbf{x}^{\vee}_4(n) \mathbf{x}^{\vee}_5(p) \mathbf{x}^{\vee}_6(q) \mathbf{x}^{\vee}_7(r) \\
    %     &\hspace{0.47cm} \dot{s}_1 \mathbf{x}^{\vee}_2(s) \mathbf{x}^{\vee}_3(t) \mathbf{x}^{\vee}_4(u) \mathbf{x}^{\vee}_5(v) \mathbf{x}^{\vee}_6(w) \mathbf{x}^{\vee}_7(x) \\
    % \end{aligned}
    % $$
    $$ \begin{aligned}
    g_R &= \begin{aligned}[t] && \hspace{-0.32cm} \underline{\mathbf{x}^{\vee}_6(ace)} \mathbf{x}^{\vee}_7(b) \\
        && \dot{s}_5 \underline{\mathbf{x}^{\vee}_6(ce)} \mathbf{x}^{\vee}_7(d) \\
        && \dot{s}_4 \dot{s}_5 \mathbf{x}^{\vee}_6(e) \mathbf{x}^{\vee}_7(f) \\
        && \mathbf{x}^{\vee}_3(g) \mathbf{x}^{\vee}_4(h) \mathbf{x}^{\vee}_5(i) \mathbf{x}^{\vee}_6(j) \mathbf{x}^{\vee}_7(k) \\
        && \underline{\mathbf{x}^{\vee}_2(ls)} \mathbf{x}^{\vee}_3(m) \mathbf{x}^{\vee}_4(n) \mathbf{x}^{\vee}_5(p) \mathbf{x}^{\vee}_6(q) \mathbf{x}^{\vee}_7(r) \\
        && \dot{s}_1 \mathbf{x}^{\vee}_2(s) \mathbf{x}^{\vee}_3(t) \mathbf{x}^{\vee}_4(u) \mathbf{x}^{\vee}_5(v) \mathbf{x}^{\vee}_6(w) \mathbf{x}^{\vee}_7(x) \\
        \end{aligned} \\
    \end{aligned}
    $$
    $$
    \begin{aligned}
    g_R
    &= \begin{aligned}[t] && \dot{s}_5 \dot{s}_4 \dot{s}_5 \dot{s}_1 \underline{X_{6,3}(ace)} \mathbf{x}_7(b) \\
        && \underline{X_{6,2}(-ce)} \mathbf{x}_7(d) \\
        && \mathbf{x}_6(e) \mathbf{x}_7(f) \\
        && \mathbf{x}_3(g) \mathbf{x}_4(h) \mathbf{x}_5(i) \mathbf{x}_6(j) \mathbf{x}_7(k) \\
        && \underline{X_{2,2}(-ls)} \mathbf{x}_3(m) \mathbf{x}_4(n) \mathbf{x}_5(p) \mathbf{x}_6(q) \mathbf{x}_7(r) \\
        && \mathbf{x}_2(s) \mathbf{x}_3(t) \mathbf{x}_4(u) \mathbf{x}_5(v) \mathbf{x}_6(w) \mathbf{x}_7(x).
    \end{aligned}
    \end{aligned}
    $$
\end{ex}

\begin{proof}[Proof of Lemma {\ref{lem gR in terms of uR and wR}}]
To obtain the relation (\ref{eqn gR in terms of uR and si}), we start with the description of $g_R$ given above and permute all the $\dot{s}_i$ factors to the left. It follows that the resulting sequence of $\dot{s}_i$ factors may be read directly from the quiver $Q_{P,R}$ by starting at the top of the right-most column containing dot vertices and working downwards, only noting any $\dot{s}_i$'s we come across, and then proceeding column by column to the left. This product is exactly $\dot{w}_R:= \dot{w}_{R_{[1,l]}} = \dot{w}_{R_{l+1}}^{-1} \dot{w}_P$.

It remains to show that, after this sequence of permutations, the remaining factor on the right of $\dot{w}_R$ is $\tilde{u}_R$. To do this we again begin by making some observations on the above permutations in $g_R$. Firstly, due to the constructions of the quiver $Q_{P,R}$ and the matrix $g_R$, we only have to permute the $\dot{s}_j$'s past $\mathbf{x}^{\vee}_i(x_{v_e}/x_{v_s})$'s for $j < i$. These matrices commute if $i-j \geq 2$, but if $j=i-1$ then
    $$\begin{gathered}
    \begin{aligned}\mathbf{x}^{\vee}_i\left(\frac{x_{v_e}}{x_{v_s}}\right) \dot{s}_{i-1}
            =\begin{pmatrix}
            1 & 0 & & \cdots & & & 0 \\
            & \ddots & \ddots & & & & \\
            & & 0 & 1 & \frac{x_{v_e}}{x_{v_s}} & & \\
            & & -1 & 0 & 0 & & \vdots \\
            & & 0 & 0 & 1 & \ddots & \\
            & & & & & \ddots & 0 \\
            & & & & & & 1
                \end{pmatrix}
        &=\dot{s}_{i-1} \begin{pmatrix}
        1 & 0 & & \cdots & & & 0 \\
        & \ddots & \ddots & & & & \\
        & & 1 & 0 & -\frac{x_{v_e}}{x_{v_s}} & & \\
        & & & 1 & 0 & & \vdots \\
        & & & & 1 & \ddots & \\
        & & & & & \ddots & 0 \\
        & & & & & & 1
            \end{pmatrix} \\
        % &=\dot{s}_{i-1} X^{\vee}_{i,2}\left(-\frac{x_{v_e}}{x_{v_s}}\right)
    \end{aligned} \\
    \Rightarrow \quad \mathbf{x}^{\vee}_i\left(\frac{x_{v_e}}{x_{v_s}}\right) \dot{s}_{i-1}
        = \left(I + \frac{x_{v_e}}{x_{v_s}} E_{i,i+1}\right) \dot{s}_{i-1}
        = \dot{s}_{i-1} \left(I - \frac{x_{v_e}}{x_{v_s}} E_{i-1,i+1}\right)
        = \dot{s}_{i-1} X^{\vee}_{i,2}\left(-\frac{x_{v_e}}{x_{v_s}}\right)
    \end{gathered}
    $$
where in both of the matrices written in full, the entry $x_{v_e}/x_{v_s}$ appears in position $(i-1,i+1)$. This new matrix $X^{\vee}_{i,2}\left(-x_{v_e}/x_{v_s}\right)$ will commute with all $\dot{s}_j$'s apart from $\dot{s}_{i-2}$.
Permuting $\dot{s}_{i-2}$ past it has a similar effect, the entry $x_{v_e}/x_{v_s}$ will now be found in position $(i-2, i+1)$ with sign $(-1)^2$. This pattern repeats for all permutations.

It follows that we need to keep track of which factors $\mathbf{x}^{\vee}_i(x_{v_e}/x_{v_s})$ will be affected by our sequence of permutations. Similar to $u_L$, and thanks to the constructions of the quiver $Q_{P,R}$ and matrix $g_R$, we see that the only affected factors are those arising from $1$-paths of length at least $2$.

In particular, suppose we have a $1$-path of length $\alpha\geq 2$ crossing row $E_i$, with corresponding factor $\mathbf{x}^{\vee}_i(x_{v_e}/x_{v_s})$ of $g_R$. We will have to permute a product of $\dot{s}_j$'s past $\mathbf{x}_i(x_{v_e}/x_{v_s})$, however some of these $\dot{s}_j$'s will have no effect on the factor in question. Indeed by the above argument on permutations and the location of $\dot{s}_j$'s in the quiver $Q_{P,R}$, we see that $\mathbf{x}^{\vee}_i(x_{v_e}/x_{v_s})$ will only be affected by permuting the sub-product $\dot{s}_{i-1} \cdots \dot{s}_{i-\alpha+1}$ past it, giving
    $$\mathbf{x}^{\vee}_i\left(\frac{x_{v_e}}{x_{v_s}}\right) \dot{s}_{i-1} \cdots \dot{s}_{i-\alpha+1} = \dot{s}_{i-1} \cdots \dot{s}_{i-\alpha+1} X^{\vee}_{i,\alpha}\left((-1)^{\alpha-1} \frac{x_{v_e}}{x_{v_s}}\right).
    $$
% where $X^{\vee}_{i,\alpha}(x_{v_t}/x_{v_s})$ is defined to be the matrix with $1$'s on the diagonal, the entry $x_{v_t}/x_{v_s}$ in position $(i-\alpha+1,i+1)$, and $0$'s elsewhere. % example matrix here?
This sub-product is exactly the contribution to $g_R$ from those $\dot{s}_i$'s which are found both in the same square $R_j$ as the starting vertex $v_s$ and also in the same diagonal $\mathcal{D}_k$ as the end vertex $v_e$.

It follows that after all permutations, we will have expressed $g_R$ as $\tilde{u}_R$ multiplied on the left by $\dot{w}_R$, as desired.
\end{proof}

\newpage
\subsection{The conjecture}
\fancyhead[L]{8.3 \ \ The conjecture}
% \fancyhead[L]{9.3 \ \ The conjecture}

\begin{conj} \label{conj b lies in Z}

The matrix $\tilde{u}_R \in U^{\vee}$ defined in Section \ref{subsec G/P Constructing matrices from the right-quiver decoration} is exactly the matrix $u_R \in U^{\vee}$ defined in Section \ref{subsec G/P The quiver torus and toric chart} as the unique matrix such that
    $$b_P:= u_L \kappa_P \dot{w}_P \bar{w}_0 u_R %\in Z_P := B_-^{\vee}\cap U^{\vee}\left(T^{\vee}\right)^{W_P}\dot{w}_P\bar{w}_0U^{\vee}
    $$
lies in $B_-^{\vee}$.
% **** Moreover, if we define
%     \begin{equation*} %\label{eqn g tilde R}
%     \tilde{g}_R := \left(\dot{w}_{R_{[2,l]}}\right)^{-1} g_R = \left(\dot{w}_{R_{[2,l]}}\right)^{-1} \dot{w}_R \tilde{u}_R = \dot{w}_{R_1}\tilde{u}_R,
%     \end{equation*}
% then
%     \begin{equation} \label{eqn b is gL kP w0 gtildeR}
%     b_P = g_L \kappa_P \bar{w}_0 \tilde{g}_R.
%     \end{equation}
% % where $\tilde{g}_R$ is defined in (\ref{eqn g tilde R}).

% We recall the definition
%     $$b_P:= u_L \kappa_P \dot{w}_P \bar{w}_0 u_R \in Z_P := B_-^{\vee}\cap U^{\vee}\left(T^{\vee}\right)^{W_P}\dot{w}_P\bar{w}_0U^{\vee}
%     $$
% and define
%     \begin{equation*} %\label{eqn g tilde R}
%     \tilde{g}_R := \left(\dot{w}_{R_{[2,l]}}\right)^{-1} g_R = \left(\dot{w}_{R_{[2,l]}}\right)^{-1} \dot{w}_R \tilde{u}_R = \dot{w}_{R_1}\tilde{u}_R.
%     \end{equation*}
% Then $b_P = g_L \kappa_P \bar{w}_0 \tilde{g}_R$.
\end{conj}

% \begin{rem}
% ***?***
% By stating the conjecture in the form (\ref{eqn b is gL kP w0 gtildeR}), we see that, if it is true, then this gives a more succinct*?* generalisation of the element $b\in Z$ from the $G/B$ case.
% \end{rem}

% At present we are able to prove the following lemma:
% \begin{lem} \label{lem b lies in UTwwU}
% With the above notation
%     $$g_L \kappa_P \bar{w}_0 \tilde{g}_R \in U^{\vee}\left(T^{\vee}\right)^{W_P}\dot{w}_P\bar{w}_0U^{\vee}.
%     $$
% \end{lem}

% \begin{proof}[Proof of Lemma {\ref{lem b lies in UTwwU}}]
% We see that
%     $$\begin{aligned}
%     g_L \kappa_P \bar{w}_0 \tilde{g}_R &= u_L \dot{w}_{L} \kappa_P \bar{w}_0 \dot{w}_{R_1}\tilde{u}_R & \text{by Lemma \ref{lem gL in terms of uL and wL}} \\
%         &= u_L \kappa_P \dot{w}_{L} \dot{w}_{L_{l+1}} \bar{w}_0 \tilde{u}_R & \text{by Lemma \ref{lem pass weyl group elt from square past w0} and since } \kappa_P \text{, defined in (\ref{eqn defn kappa hw map}),} \\
%         & & \text{permutes with } \dot{w}_{L} \text{ by construction} \\
%         &= u_L \kappa_P \dot{w}_{P} \bar{w}_0 \tilde{u}_R  &
%     \end{aligned}
%     $$
% where the last equality holds by definition of $\dot{w}_{P}$ (see Section \ref{subsec G/P Constructing matrices from a given quiver decoration});
%     $$\dot{w}_{P} = \dot{w}_{L_{[1,l+1]}} = \dot{w}_{L_1} \cdots \dot{w}_{L_{l+1}} = \dot{w}_{L}\dot{w}_{L_{l+1}}.
%     $$
% In particular $g_L \kappa_P \bar{w}_0 \tilde{g}_R \in U^{\vee}\left(T^{\vee}\right)^{W_P}\dot{w}_P\bar{w}_0U^{\vee}$ as desired.
% \end{proof}

The conjecture holds in the case of Grassmannians due to Marsh and Rietsch in \cite{MarshRietsch2020}. In order to see this we require a lemma:
\begin{lem} \label{lem pass weyl group elt from square past w0}
With the above notation (see Sections \ref{subsec G/P Constructing matrices from a given quiver decoration} and \ref{subsec G/P Constructing matrices from the right-quiver decoration}), $\dot{w}_{L_j} \bar{w}_0$ and $\bar{w}_0 \dot{w}_{R_{l+2-j}}$ are different representatives of the same element of the Weyl group.
    % $$\dot{w}_{L_j} \bar{w}_0 = \bar{w}_0 \dot{w}_{R_{l+2-j}}.
    % $$
\end{lem}

\begin{proof}
In the quiver $Q_P$, suppose $L_j$ is a square of size $k_j \times k_j$ such that $k_j \geq 2$, since when $k_j=1$ the statement is trivial. We recall from Section \ref{subsec G/P Constructing matrices from a given quiver decoration} that $\dot{w}_{L_j}$ is a representative of the longest element of
    $$\left\langle s_i \ | \ i \in \{ n_{j-1}+1, \ldots, n_j-1 \} \right\rangle \subset W
    $$
given by the reduced expression
    $$(n_j-1, n_j-2, \ldots, n_{j-1}+1, \ldots, n_j-1, n_j-2, n_j-1).
    $$
% This follows from the row set in the quiver $Q_P$ which the interior of $L_j$ intersects, namely the set $\{ E_{n_{j-1}+1}, \ldots, E_{n_j-1} \}$, together with the following two facts. Firstly, $\dot{w}_{L_j}$ is given by a product of $\dot{s}_i$ factors, read from the square $L_j$ by starting at the bottom of the left-most column and working upwards, then proceeding column by column to the right. Secondly, by definition, the intersection of the square $L_j$ and a row $E_i$, contains $i-n_{j-1}$ copies of $\dot{s}_i$. ********************

Now we consider the quiver $Q_{P,R}$ and the square $R_{l+2-j}$, which is also of size $k_j \times k_j$ by construction. We recall from Section \ref{subsec G/P Constructing matrices from the right-quiver decoration} that $\dot{w}_{R_{l+2-j}}$ is a representative of the longest element of
    $$\left\langle s_i \ \bigg| \ i \in \left\{ \left(\sum_{r=j+1}^{l+1}k_r \right)+1, \ldots, \left(\sum_{r=j}^{l+1}k_r \right) -1 \right\} \right\rangle
        =\left\langle s_i \ | \ i \in \{ n-(n_j-1), \ldots, n-(n_{j-1}+1) \} \right\rangle \subset W
    $$
given by the reduced expression
    $$(n-(n_{j-1}+1), n-(n_{j-1}+2), n-(n_{j-1}+1), \ldots, n-(n_j-1) \ldots, n-(n_{j-1}+1)).
    $$

To complete the proof it remains to note that $\dot{s}_i\bar{w}_0 = \bar{w}_0\dot{s}_{n-i}$. Thus we see that $\dot{w}_{L_j} \bar{w}_0$ and $\bar{w}_0 \dot{w}_{R_{l+2-j}}$ are given by different products of $\dot{s}_i$'s and $\bar{s}_i$'s, but that they represent the same Weyl group element.
\end{proof}

Now if we define
    \begin{equation*} %\label{eqn g tilde R}
    \tilde{g}_R := \left(\dot{w}_{R_{[2,l]}}\right)^{-1} g_R = \left(\dot{w}_{R_{[2,l]}}\right)^{-1} \dot{w}_R \tilde{u}_R = \dot{w}_{R_1}\tilde{u}_R,
    \end{equation*}
then Conjecture \ref{conj b lies in Z} is equivalent to the statement
    \begin{equation} \label{eqn b is gL kP w0 gtildeR}
    b_P = g_L \kappa_P \bar{w}_0 \tilde{g}_R.
    \end{equation}
We see this as follows:
    $$\begin{aligned}
    g_L \kappa_P \bar{w}_0 \tilde{g}_R &= u_L \dot{w}_{L} \kappa_P \bar{w}_0 \dot{w}_{R_1}\tilde{u}_R & \text{by Lemma \ref{lem gL in terms of uL and wL}} \\
        &= u_L \kappa_P \dot{w}_{L} \dot{w}_{L_{l+1}} \bar{w}_0 \tilde{u}_R & \text{by Lemma \ref{lem pass weyl group elt from square past w0} and since } \kappa_P \text{, defined in (\ref{eqn defn kappa hw map}),} \\
        & & \text{permutes with } \dot{w}_{L} \text{ by construction} \\
        &= u_L \kappa_P \dot{w}_{P} \bar{w}_0 \tilde{u}_R  &
    \end{aligned}
    $$
where the last equality holds by definition of $\dot{w}_{P}$ (see Section \ref{subsec G/P Constructing matrices from a given quiver decoration});
    $$\dot{w}_{P} = \dot{w}_{L_{[1,l+1]}} = \dot{w}_{L_1} \cdots \dot{w}_{L_{l+1}} = \dot{w}_{L}\dot{w}_{L_{l+1}}.
    $$

In the Grassmannian case $\tilde{g}_R=g_R$, and so by stating Conjecture \ref{conj b lies in Z} in the form (\ref{eqn b is gL kP w0 gtildeR}) we are exactly in the case considered by Marsh and Rietsch (\cite[see the proof of Proposition 8.6]{MarshRietsch2020}). Namely, the conjecture holds in the Grassmannian case since our constructions of $g_L$ and $g_R$ descend to the respective constructions given by Marsh and Rietsch in this case. Their proof follows from a careful study of the matrix $b_P$ in terms of a concatenation of `chips' (corresponding to $\mathbf{x}^{\vee}_i$ factors) and wiring diagrams (corresponding to products of $\dot{s}_i$'s).
An example of the chips are found in the graphs used to compute Chamber Ansatz minors in Section \ref{subsec Chamber Ansatz minors} (for example Figures \ref{The graph for u_1 when n=4} and \ref{The graph for u^T when n=4})
% also used in Section \ref{subsec G/P Chamber Ansatz minors} (for example Figure \ref{fig graph for computing minors of uL for F2,5,6,C8})
with each chip describing a diagonal step upwards in one of these graphs.
The wiring diagrams are non-singular braid diagrams, for example non-singular versions of the ansatz arrangements introduced in Section \ref{subsec The Chamber Ansatz} (for example Figures \ref{fig ansatz arrangement i_0 dim 4} and \ref{fig ansatz arrangement i'_0^op dim 4}).
% used in Section \ref{subsec G/P Chamber Ansatz minors} (e.g. Figure \ref{fig ansatz arrangement wPw0 F2,5,6,C8}). ***refer to $G/B$ case to avoid referring to later sections??***

Due to the similarity between the constructions of $g_L$ and $g_R$ and those used by Marsh and Rietsch, we expect a similar method of proof to work for Conjecture \ref{conj b lies in Z}. However, in our attempts this appears to be much more complicated for partial flag varieties than for Grassmannians, and so at present the conjecture remains open.

% \section[The ideal coordinates (\texorpdfstring{$G/P$}{G/P} setting)]{The ideal coordinates ($G/P$ setting)}

\section{The ideal coordinates} \label{sec G/P The ideal coordinates}
\fancyhead[L]{9 \ \ The ideal coordinates}
% \fancyhead[L]{10 \ \ The ideal coordinates}

In this section we define an analogue of the ideal coordinate system in the setting of partial flag varieties, and compare it to the quiver toric chart. In the $G/B$ case we were able to define the ideal coordinate chart directly, however to do this for $G/P$ we would need an explicit description of the weight matrix in terms of $m_i$ and $d_i$ coordinates, as we had in the $G/B$ setting (Corollary \ref{cor wt matrix in m ideal coords}). It is possible to compute the weight matrix in examples using the weight map $\gamma_P$ defined in Section \ref{subsec The superpotential, highest weight and weight maps}, however in general it is not straightforward to write down an explicit formula. Consequently, we define our ideal coordinate system via a particular choice of quiver decoration.

\subsection{Quiver decoration} \label{subsec G/P Quiver decoration}
\fancyhead[L]{9.1 \ \ Quiver decoration}
% \fancyhead[L]{10.1 \ \ Quiver decoration}

In this section we define a decoration of the quiver $Q_P$ which generalises the labelling of the $G/B$ quiver in terms of the original ideal coordinates $(d,\boldsymbol{m})$. In particular, since the critical point conditions in the quiver are purely combinatorial, we show that this labelling of $Q_P$ satisfies the same relation at critical points as the original $(d,\boldsymbol{m})$ labelling in the $G/B$ case (Proposition \ref{prop crit points, sum at vertex is nu_i}). This is the first of two results in this section which combine our quiver decoration and the critical point conditions. The second is an observation on the form of the weight matrix at critical points, extending Proposition \ref{prop weight at crit point non trop}. We complete the section by defining the ideal coordinate chart in the $G/P$ setting and stating a result comparing it to the quiver toric chart.

Recall that the vertical arrows of the quiver are labelled by $\{a_{(k,a)}\}$ such that $h(a_{(k,a)})=v_{(k,a)}$ and the horizontal arrows of the quiver are labelled by $\{b_{(k,a)}\}$ such that $t(b_{(k,a)})=v_{(k,a)}$. Additionally recall the definition
    $$s_k:= \sum_{j=1}^{k-1}(n-j)
    $$
and denote the numerator of $r_a$ by $n(r_a)$.

\begin{defn}[Decoration of vertical arrows]
% The vertical arrow (ideal) coordinates of the quiver $Q_L$ for $G/P = \mathcal{F}_{n_1, \ldots, n_l}(\mathbb{C}^n)$ are the same as the respective arrow (ideal) coordinates of the quiver for $GL_n(\mathbb{C})/B$.

% More explicitly,
We take the coordinate of the vertical arrow leaving $v_{(k,a)}$ to be
    \begin{equation} \label{eqn coord of vert arrow leaving vertex k,a}
    r_{a_{(k,a-1)}}:=\begin{cases}
    m_a & \text{if } k=1 \\
    r_{a_{(k-1,a)}} \frac{m_{s_{k}+a}}{m_{s_{k-1}+a}} & \text{if } k\geq 2.
    \end{cases}
    \end{equation}
\end{defn}

We note that this definitions is equivalent to the iterative description in the full flag case, wherever it makes sense (see (\ref{eqn rai,j+1 in terms of m's}) at the end of Section \ref{subsec The quiver torus as another toric chart on Z}):
    \begin{equation} \label{eqn coord of arrow leaving vertex k,a}
    r_{a_{(k,a-1)}}=\begin{cases}
    m_a & \text{if } k=1 \\
    m_{s_{k}+a} \frac{n(r_{a_{(k-1,a)}})}{n(r_{a_{(k-1,a-1)}})} & \text{if } k\geq 2.
    \end{cases}
    \end{equation}
where we recall $v_{ij}=v_{(j,i-j)}$ (and subsequently $r_{a_{ij}}=r_{a_{(j,i-j)}}$) from Section \ref{subsec Constructing quivers}.

To see the equivalence of these definitions, we first note that the numerators are clearly the same, so it remains to consider the denominators for $k\geq2$.
Writing (\ref{eqn coord of vert arrow leaving vertex k,a}) in terms of $m_i$'s we have:
    $$r_{a_{(k,a-1)}}= r_{a_{(k-1,a)}} \frac{m_{s_{k}+a}}{m_{s_{k-1}+a}}
        = r_{a_{(k-2,a+1)}} \frac{m_{s_{k-1}+a+1}}{m_{s_{k-2}+a+1}} \frac{m_{s_{k}+a}}{m_{s_{k-1}+a}}
        = \cdots
        = \frac{\prod_{i=1}^{k} m_{s_{k-i+1}+a+i-1}}{\prod_{i=1}^{k-1} m_{s_{k-i}+a+i-1}}.
    $$
Then we see the desired equality holds by recalling from (\ref{eqn rai,j+1 in terms of m's}), that
    $$n(r_{a_{(k-1,a-1)}}) = n(r_{a_{k+a-2,k-1}}) =\prod_{i=1}^{k-1} m_{s_{k-i}+a+i-1}.
    $$

\begin{rem} \label{rem need horizontal arrows in description}
The description (\ref{eqn coord of arrow leaving vertex k,a}) is enough to define the labelling of most vertical arrows, but not all. In particular it is not sufficient for arrows entering dot vertices directly below squares $L_i$, which are not below star vertices, for example the vertical arrows entering vertices $v_{\left(n_{i-1}+2,k_i-1\right)}, \ldots, v_{\left(n_{i},1\right)}$ in Figure \ref{fig to help arrow coord description}.

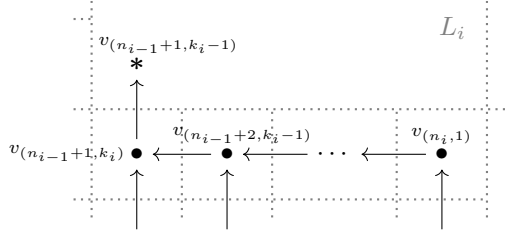
\begin{figure}[ht]
\centering
\begin{tikzpicture}[scale=1.2]
    %squares
        \draw[dotted, thick, color=black!50] (0.3,2.5) -- (0.5,2.5);
        \draw[dotted, thick, color=black!50] (0.3,1.5) -- (5.1,1.5);
        \draw[dotted, thick, color=black!50] (0.3,0.5) -- (4.88,0.5);

        \draw[dotted, thick, color=black!50] (0.5,0.3) -- (0.5,2.7);
        \draw[dotted, thick, color=black!50] (1.5,0.3) -- (1.5,1.5);
        \draw[dotted, thick, color=black!50] (2.5,0.3) -- (2.5,1.5);
        \draw[dotted, thick, color=black!50] (3.88,0.3) -- (3.88,1.5);
        \draw[dotted, thick, color=black!50] (4.88,0.3) -- (4.88,2.7);

    %dots and stars + arrow labels
        \node (11) at (1,2) {$\boldsymbol{*}$};
            \node at (1.35,2.2) {\scriptsize{$v_{(n_{i-1}+1,k_i-1)}$}};
        \node (21) at (1,1) {$\bullet$};
            \node at (0.23,1) {\scriptsize{$v_{(n_{i-1}+1,k_i)}$}};
        \node (31) at (1,0) {\phantom{$\bullet$}};

        \node (22) at (2,1) {$\bullet$};
            \node at (2.17,1.2) {\scriptsize{$v_{(n_{i-1}+2,k_i-1)}$}};
        \node (32) at (2,0) {\phantom{$\bullet$}};

        \node (23) at (3.19,1) {$\cdots$};
        % \node (33) at (3,0) {\phantom{$\bullet$}};

        \node (24) at (4.38,1) {$\bullet$};
            \node at (4.38,1.2) {\scriptsize{$v_{(n_{i},1)}$}};
        \node (34) at (4.38,0) {\phantom{$\bullet$}};

        % \node (25) at (5.38,1) {$\bullet$};
        %     \node at (5.38,1.2) {\scriptsize{$v_{(n_{i-1}+k_i-1,2)}$}};

        % \node (35) at (5.38,0) {\phantom{$\bullet$}};

        % \node (26) at (6.38,1) {$\bullet$};
        % \node (36) at (6.38,0) {\phantom{$\bullet$}};

        % % \node (27) at (7.38,1) {\phantom{$\bullet$}};

        \node[color=black!50] at (4.5,2.4) {{$L_i$}};

    %vertical arrows
        \draw[->] (21) -- (11);
        \draw[->] (31) -- (21);

        \draw[->] (32) -- (22);
        % \draw[->] (33) -- (23);
        \draw[->] (34) -- (24);
        % \draw[->] (35) -- (25);
        % \draw[->] (36) -- (26);

    %horizontal arrows
        \draw[->] (22) -- (21);
        \draw[->] (23) -- (22);
        \draw[->] (24) -- (23);
        % \draw[->] (25) -- (24);
        % \draw[->] (26) -- (25);
        % \draw[->] (27) -- (26);
\end{tikzpicture}
\caption{Subquiver to aid description of arrow coordinates} \label{fig to help arrow coord description}
\end{figure}
\end{rem}

% Since we wish to generalise the $G/B$ quiver decoration ....*
% The horizontal arrow (ideal) coordinates of the quiver $Q_P$ for $G/P$ are similar to the respective arrow (ideal) coordinates of the quiver for $G/B$ in the following two ways:
% \begin{enumerate}
%     \item The coordinate of every horizontal arrow in $Q_P$ for $G/P$ contains a quotient of products of $m_i$'s. % NO: For a given arrow this quotient comes from the quotient of products of $m_i$'s of the respective arrow coordinate in the quiver for $G/B$ as follows: we set $m_i=1$ if $m_i$ does not appear in any of the vertical arrow coordinates of $Q_P$.

%     \item If a horizontal arrow in $Q_P$ for $G/P$ leaves a star vertex with coordinate $d_{i+1}$, then its arrow coordinate (and the coordinate of every horizontal arrow directly below it) will be multiplied by the quotient $d_i/d_{i+1}$.
% \end{enumerate}

We recall from Section \ref{subsec G/P The quiver torus and toric chart} that in the $G/P$ case it is no longer sufficient to only work with the vertical arrow and star vertex coordinates. Consequently we required the coordinates of exactly those horizontal arrows used in the definitions of $u_L$ and $g_L$ in Section \ref{subsec G/P Constructing matrices from a given quiver decoration}. It follows that we will need to specify the decoration of these horizontal arrow coordinates. Namely for each $i=1,\ldots, l+1$, we define the labelling of the horizontal arrows connecting dot vertices directly below squares $L_i$, for example the horizontal arrows in Figure \ref{fig to help arrow coord description}.

\begin{defn}[Decoration of star vertices and horizontal arrows] \label{defn decoration star vert and horiz arr}
We take the coordinate of the star vertex in the square $L_i$, for $i=1, \ldots, l+1$, to be $d_i$, similar to the $G/B$ case. Explicitly this is
    $$x_{v_{n_i, n_{i-1}+1}}= x_{v_{\left(n_{i-1}+1,k_i-1\right)}}:=d_i.
    $$

For $i=1, \ldots, l$, we take the coordinate of the horizontal arrow leaving the vertex
% $v_{(k,a)}$ where $k=n_i-a+1$, $a=1, \ldots, k_i$, to be
% $v_{(n_i-a+1,a)}$, for $a=1, \ldots, k_i$, to be
$v_{(k,a)}$ where $k\in \{ n_{i-1}+1, \ldots, n_i \}$, $a=n_i-k+1$, to be
    \begin{equation} \label{eqn coord of horiz arrow leaving vertex k,a}
    r_{b_{(k,a)}} :=\begin{cases}
    m_{s_k+a} & \text{if } i=1 \\
    r_{a_{\left(n_{i-1},a\right)}}\frac{m_{s_k+a}}{m_{s_{n_{i-1}}+a}} & \text{if } i\geq 2.
    \end{cases}
    \end{equation}
The rest of the quiver decoration may be constructed in the usual way.
\end{defn}

For the complete quiver labelling in our running example of $\mathcal{F}_{2,5,6}(\mathbb{C}^8)$, see Figure \ref{fig complete QP quiver decoration F2,5,6,C8} given in Appendix \ref{append Example of complete quiver labelling}.

Before defining the ideal toric chart we generalise two results from the $G/B$ case to the $G/P$ setting, both of which combine our quiver decoration and the critical point conditions. The first is a generalisation of Proposition \ref{prop crit points, sum at vertex is nu_i}:

\begin{prop} \label{prop G/P crit points, sum at vertex is nu_i}
With the above choice of $Q_P$ quiver decoration, if the critical point conditions hold at every dot vertex $v \in \mathcal{V}^{\bullet}_P$, then the sum of outgoing arrows at each dot vertex $v_{(k,a)}$ is given in terms of the $m_i$ coordinates by
    $$\varpi\left(v_{(k,a)}\right) := \sum_{a:t(a)=v_{(k,a)}} r_a  = m_{s_k+a}.$$
\end{prop}

\begin{proof}
For all dot vertices for which (\ref{eqn coord of arrow leaving vertex k,a}) is sufficient to describe the coordinate of the arrow entering that vertex, the proof is the same as the respective proof in the full flag case (Proposition \ref{prop crit points, sum at vertex is nu_i}). It remains to consider those dot vertices directly below squares $L_i$, which are not below star vertices, that is vertices $v_{\left(n_{i-1}+2,k_i-1\right)}, \ldots, v_{\left(n_{i},1\right)}$ for each $i=1, \ldots, l$. For each such vertex there is only one outgoing arrow; this arrow is horizontal and so its coordinate is given by (\ref{eqn coord of horiz arrow leaving vertex k,a}), thus $\varpi\left(v_{(k,a)}\right) = r_{b_{(k,a)}}$.

If $i=1$ then clearly $\varpi\left(v_{(k,a)}\right) = m_{s_k+a}$, as desired. For $i=2, \ldots, l$, supposing that the critical point conditions hold, we wish to show that for $k = n_{i-1}+2, \ldots, n_i$, $a=n_i-k+1$, we have
    $$
    \varpi\left(v_{(k,a)}\right) = m_{s_k+a}
    \quad \Leftrightarrow \quad
    m_{s_k+a} = r_{b_{(k,a)}} = r_{a_{\left(n_{i-1},a\right)}}\frac{m_{s_k+a}}{m_{s_{n_{i-1}}+a}}
    \quad \Leftrightarrow \quad
    r_{a_{\left(n_{i-1},a\right)}}=m_{s_{n_{i-1}}+a}.
    $$

We note that for $i = 2, \ldots, l$ and $a=1, \ldots, k_i-1$, the vertices $v_{\left(n_{i-1},a\right)}$ are directly to the left of the square $L_i$, and strictly above the star vertex in $L_i$. So we see that there is only one incoming arrow at each vertex $v_{\left(n_{i-1},a\right)}$, namely $r_{a_{\left(n_{i-1},a\right)}}$.
In particular, the critical point condition at each of these vertices, $v_{\left(n_{i-1},a\right)}$, says that $\varpi\left(v_{\left(n_{i-1},a\right)}\right) = r_{a_{\left(n_{i-1},a\right)}}$.
However we have already shown that the statement holds at the vertex $v_{\left(n_{i-1},a\right)}$, that is, if the critical point conditions hold at every dot vertex, then $\varpi\left(v_{\left(n_{i-1},a\right)}\right) = m_{s_{n_{i-1}}+a}$. Thus we have $r_{a_{\left(n_{i-1},a\right)}} = m_{s_{n_{i-1}}+a}$, and so the proof is complete.
\end{proof}

The second of the two results mentioned above is an extension of Proposition \ref{prop weight at crit point non trop}:
\begin{prop} \label{prop G/P weight at crit point non trop}
At a critical point in the fibre over $d\in \left(T^{\vee}\right)^{W_P}$, the weight matrix is an $n\times n$ matrix $\mathrm{diag}(c, \ldots, c)$ where
    $$c^n = \prod_{i=1}^{l+1} d_i^{k_i}.
    $$
\end{prop}

\begin{proof}
By the theorem of Rietsch that the critical points of the superpotential are given by Toeplitz matrices \cite[non-T-equivariant case of Theorem 4.1]{Rietsch2008}, we see that at a critical point, the weight matrix $\mathrm{diag}(t_1, t_2, \ldots, t_n)$ is given by $\mathrm{diag}(c,c, \ldots, c)$ for some $c$. %** Holds over $\mathbb{K}$ or only over $\mathbf{K}$?**

By taking the determinant we obtain
    $$\prod_{i=1}^n t_{P,i} = \prod_{i=1}^n \frac{\Xi_{P,i}}{\Xi_{P,i+1}} = \frac{\Xi_{P,1}}{\Xi_{P,n+1}} = \Xi_{P,1}
    $$
recalling $\Xi_{n+1} =1$ by definition. This gives the desired value of $c$ as follows:
    $$c^n = \prod_{i=1}^n t_{P,i} = \Xi_{P,1} = \left( \prod_{\substack{v\in \mathcal{D}_1 \\ v \in \mathcal{V}_P}} x_v \right) \left( \prod_{\substack{v\in \mathcal{D}_1 \cap L_i \\ v \notin \mathcal{V}_P}} d_i \right) = \prod_{i=1}^{l+1} d_i^{k_i}.
    $$
% where $x_v = d_i$ for $v \in D_1 \cap L_i$.
\end{proof}

We now describe the ideal coordinate chart, which is defined in terms of the quiver decoration as follows:
    \begin{equation} \label{eqn G/P ideal toric chart defn}
    \psi_P : \left(T^{\vee}\right)^{W_P} \times \left(\mathbb{K}^*\right)^{\mathcal{V}_P^{\bullet}} \to \mathcal{M}_P \to Z_P \,, \quad
    \left(d, \boldsymbol{m}\right)
        \mapsto \left( \prod_{k=1}^{n-1} \prod_{\substack{a=1 \\ v_{(k,a)} \in \mathcal{V}_P^{\bullet}}}^{n-k} \mathbf{y}^{\vee}_a\left( \frac{1}{m_{s_k+a}} \right)\right) \gamma_P\left(\boldsymbol{x}_{\mathcal{V}_P} \right)
    \end{equation}
where $\gamma_P$ is the weight map from Section \ref{subsec The superpotential, highest weight and weight maps}, and we have used the relation $r_a=x_{h(a)}/x_{t(a)}$ between vertex and arrow coordinates. %Of note, some of the $m_i$ coordinates are redundant due to the squares $L_i$ in the quiver, however we write the toric chart $\psi_P$ as a map on $\left(T^{\vee}\right)^{W_P} \times \left(\mathbb{K}^*\right)^N$ as opposed to $\left(T^{\vee}\right)^{W_P} \times \left(\mathbb{K}^*\right)^{\mathcal{V}_P^{\bullet}}$ for ease of notation both here and later on.

\begin{rem}
The sequence of $\mathbf{y}^{\vee}_i$'s in (\ref{eqn G/P ideal toric chart defn}) may be more easily understood by looking at the $v_{(k,a)}$ labelling of the vertices of the quiver $Q_P$. In particular, starting at top dot vertex in the left-most column and working downwards, we take note of the `$a$'-component of the labelling $v_{(k,a)}$. Continuing column by column to the right until all dot vertices have been considered gives the sequence of subscripts in the above product. Moreover we see that in the $G/B$ case this returns the reduced expression $\mathbf{i}_0$ for $\bar{w}_0$, and so $\psi_P$ descends to the original ideal coordinate chart $\psi=\psi_{\mathbf{i}_0}$ on $Z$ (defined first in Section \ref{sec The ideal coordinates} and in full generality in Section \ref{subsec A family of ideal polytopes}).
\end{rem}

We wish to show that, with this quiver decoration, the quiver toric chart and the ideal toric chart return the same element of $Z_P$. We have already seen in Lemma \ref{lem b0 is gamma wt matrix} that
    $$\left[\theta_P \left( \boldsymbol{x}_{\mathcal{V}_P^*} , \boldsymbol{r}_{\mathcal{A}_{P,\hat{\mathrm{v}}}}\right) \right]_0 = \left[b_P\right]_0 = \gamma_P\left(\boldsymbol{x}_{\mathcal{V}_P} \right).
    $$
Since $\left[\psi_P(d,\boldsymbol{m})\right]_0 = \gamma_P\left(\boldsymbol{x}_{\mathcal{V}_P} \right)$ by definition, it remains to consider $\left[\psi_P(d,\boldsymbol{m})\right]_-$ and show that it is equal to $\left[b_P\right]_-$, which we do in the next theorem:

\begin{thm} \label{thm b- in terms of y_i}
We can factorise $b_P \in Z_P$ as $b_P=\left[b_P\right]_-\left[b_P\right]_0$ where
    $$\left[b_P\right]_- = \prod_{k=1}^{n-1} \prod_{\substack{a=1 \\ v_{(k,a)} \in \mathcal{V}_P^{\bullet}}}^{n-k} \mathbf{y}^{\vee}_a\left( \frac{1}{m_{s_k+a}} \right) %=: \left[\psi_P(d,\boldsymbol{m})\right]_-
    .$$
\end{thm}

We devote Sections \ref{subsec G/P Applying the Chamber Ansatz}-\ref{subsec G/P The coordinate change} to the proof of this theorem, which we do via a sequence of lemmas.

\subsection{Applying the Chamber Ansatz} \label{subsec G/P Applying the Chamber Ansatz}
\fancyhead[L]{9.2 \ \ Applying the Chamber Ansatz}
% \fancyhead[L]{10.2 \ \ Applying the Chamber Ansatz}

Similar to the structure of the the proof of Theorem \ref{thm coord change} in the $G/B$ setting, in order to prove Theorem \ref{thm b- in terms of y_i} we will first need to apply (a more general version of) the Chamber Ansatz (stated in the $G/B$ case as Theorem \ref{thm chamber ansatz}).
In particular, the analogue of the Chamber Ansatz that we will need (\cite[Theorem 7.1]{MarshRietsch2004}) may be thought of as replacing $\bar{w}_0$ with $\dot{w}_P\bar{w}_0$ in Theorem \ref{thm chamber ansatz}. With this in mind, we now state the first two lemmas we require for the proof of Theorem \ref{thm b- in terms of y_i}:

\begin{lem} \label{lem wPw0 description in dot vertices}
We can write $\dot{w}_P\bar{w}_0$ as the following product:
    $$\dot{w}_P\bar{w}_0 = \prod_{k=1}^{n-1} \prod_{\substack{a=1 \\ v_{(k,a)} \in \mathcal{V}_P^{\bullet}}}^{n-k} \bar{s}_a.
    $$
\end{lem}

\begin{lem} \label{lem b- in terms of y_i}
Let $R=l(\dot{w}_P\bar{w}_0) = N- l(w_P)$ denote the length of $\dot{w}_P\bar{w}_0$ and $\mathbf{w} = (w_{(0)}, w_{(1)}, \ldots, w_{(R)})$ be the sequence of partial products for $\dot{w}_P\bar{w}_0$ defined by its sequence of factors (as described by Lemma \ref{lem wPw0 description in dot vertices}):
    $$\left(w_{(1)}, w_{(1)}^{-1}w_{(2)}, \ldots, w_{(R-1)}^{-1}w_{(R)}\right) = (s_{i_1}, \ldots, s_{i_R}).
    $$
Then
    \begin{equation} \label{eqn G/P b- factored into y's CA application}
    \left[b_P\right]_- = \prod_{k=1}^{n-1} \prod_{\substack{a=1 \\ v_{(k,a)} \in \mathcal{V}_P^{\bullet}}}^{n-k} \mathbf{y}^{\vee}_a\left( h_r \right)
    \end{equation}
where $h_r$ is the respective coordinate given by the (generalised) Chamber Ansatz in terms of minors of $u_L$:
    $$h_r = \frac{\prod_{j\neq i_k} \Delta^{\omega^{\vee}_j}_{w_{(r)}\omega^{\vee}_j}(u_L)^{-a_{j,i_r}}}{\Delta^{\omega^{\vee}_{i_r}}_{w_{(r)}\omega^{\vee}_{i_r}}(u_L) \Delta^{\omega^{\vee}_{i_r}}_{w_{(r-1)}\omega^{\vee}_{i_r}}(u_L)}, \quad r=1, \ldots, R.
    $$
\end{lem}

\begin{proof}%[Proof of Lemma {\ref{lem b- in terms of y_i}}]
We will prove Lemma \ref{lem b- in terms of y_i} similar to the proof of the factorisation of $b$ from $u_1$ in Lemma \ref{lem form of u_1 and b factorisations} (see (\ref{eqn form of b}) and (\ref{eqn 1/m_i in terms of chamber minors})) in the full flag case.

To give the factorisation (\ref{eqn G/P b- factored into y's CA application}) we note that $b_P B^{\vee}_+ = u_L \dot{w}_P\bar{w}_0 B^{\vee}_+$ with $u_L \in U^{\vee}_+$, so we may apply the (generalised) Chamber Ansatz (\cite[Theorem 7.1]{MarshRietsch2004}).
We obtain $b_P B^{\vee}_+ = u_L \dot{w}_P\bar{w}_0 B^{\vee}_+ = \mathbf{y}_{i_1}^{\vee}(h_1) \cdots \mathbf{y}_{i_R}^{\vee}(h_R)B^{\vee}_+$ with $h_r$ given by
    $$h_r = \frac{\prod_{j\neq i_k} \Delta^{\omega^{\vee}_j}_{w_{(r)}\omega^{\vee}_j}(u_L)^{-a_{j,i_r}}}{\Delta^{\omega^{\vee}_{i_r}}_{w_{(r)}\omega^{\vee}_{i_r}}(u_L) \Delta^{\omega^{\vee}_{i_r}}_{w_{(r-1)}\omega^{\vee}_{i_r}}(u_L)}, \quad r=1, \ldots, R.
    $$
\end{proof}

When it is clear from the context which case we are working in, we will often simply write the `Chamber Ansatz' rather than the `generalised Chamber Ansatz'.

It remains to prove Lemma \ref{lem wPw0 description in dot vertices}. We first give an example to motivate our proof of this lemma:
\begin{ex} \label{ex of proof of Lem wPw0 description in dot vertices}
We will describe $\dot{w}_P\bar{w}_0$ for our running example of $\mathcal{F}_{2,5,6}(\mathbb{C}^8)$.
First, recall the reduced expression $\mathbf{i}_0$ and the superscript `$\mathrm{op}$' notation from Section \ref{subsec The Chamber Ansatz}:
    $$\begin{aligned}
    &\mathbf{i}_0 = (i_1, \ldots, i_{\binom{r}{2}}) := (1,2, \ldots, r-1, 1,2 \ldots, r-2, \ldots, 1, 2, 1) \\
    &\mathbf{i}_0^{\mathrm{op}} = (i_{\binom{r}{2}}, \ldots, i_1) := (1,2,1,3,2,1, \ldots, r-2, \ldots, 2, 1, r-1, \ldots, 2, 1) \\
    % &\mathbf{i}'_0 = (i'_1, \ldots, i'_{\binom{r}{2}}) := (r-1, r-2, \ldots, 1, r-1, r-2, \ldots, 2, \ldots, r-1, r-2, r-1) \\
    % &{\mathbf{i}'_0}^{\mathrm{op}} = (i'_{\binom{r}{2}}, \ldots, i'_1) := (r-1, r-2, r-1, \ldots, r-3, r-2, r-1).
    \end{aligned}
    $$

For $\bar{w}_0$ we take the reduced expression ${\mathbf{i}_0}^{\mathrm{op}}$ with $r=8$.
% Additionally we will write $\dot{w}_{0,r}$ for the representative in terms of $\dot{s}_i$'s of the longest element of
%     $$\left\langle s_i \ | \ i \in \{ 1, \ldots, r \} \right\rangle \subset W
%     $$
% given by the reduced expression $\mathbf{i}_0$.
Considering the quiver $Q_P$ we recall the notation $\dot{w}_{L_i}$ defined in Section \ref{subsec G/P Constructing matrices from a given quiver decoration} and Example \ref{ex F2,5,6,C8 wLi}:
    $$\dot{w}_{L_1}=\dot{s}_1, \quad \dot{w}_{L_2}=\dot{s}_4 \dot{s}_3 \dot{s}_4, \quad \dot{w}_{L_3}=1, \quad \dot{w}_{L_4}=\dot{s}_7 . %\quad \Rightarrow \quad \dot{w}_P = \dot{s}_1 \dot{s}_4 \dot{s}_3 \dot{s}_4 \dot{s}_7
    $$
Using parentheses and square braces to highlight various sub-products, we obtain the desired result as follows:
    $$\begin{aligned} \dot{w}_P\bar{w}_0
    % &= \dot{w}_{L_1}\dot{w}_{L_2}\dot{w}_{L_3}\dot{w}_{L_4} \bar{w}_0 %= \dot{s}_1 \dot{s}_4 \dot{s}_3 \dot{s}_4 \dot{s}_7 \bar{w}_0
    % \\
    &= \dot{w}_{L_1}\dot{w}_{L_2}\dot{w}_{L_3}\dot{w}_{L_4} (\bar{s}_1) (\bar{s}_2\bar{s}_1) (\bar{s}_3\bar{s}_2\bar{s}_1) (\bar{s}_4\bar{s}_3\bar{s}_2\bar{s}_1) (\bar{s}_5\bar{s}_4\bar{s}_3\bar{s}_2\bar{s}_1) (\bar{s}_6\bar{s}_5\bar{s}_4\bar{s}_3\bar{s}_2\bar{s}_1) (\bar{s}_7\bar{s}_6\bar{s}_5\bar{s}_4\bar{s}_3\bar{s}_2\bar{s}_1) \\
    % &=  (\bar{s}_1) \dot{w}_{0,k_1-1} (\bar{s}_2\bar{s}_1) (\bar{s}_3\bar{s}_2\bar{s}_1) (\bar{s}_4\bar{s}_3\bar{s}_2\bar{s}_1) \dot{w}_{0,k_3-1} (\bar{s}_5\bar{s}_4\bar{s}_3\bar{s}_2\bar{s}_1) (\bar{s}_6\bar{s}_5\bar{s}_4\bar{s}_3\bar{s}_2\bar{s}_1) (\bar{s}_7\bar{s}_6\bar{s}_5\bar{s}_4\bar{s}_3\bar{s}_2\bar{s}_1) \dot{w}_{0,k_4-1} \\
    % & \omit \hfill \text{by permuting all non-trivial $w_{L_r}$, $r=4, \ldots, 1$ to the right} \\
    &= (\bar{s}_1) \dot{s}_1 (\bar{s}_2\bar{s}_1) (\bar{s}_3\bar{s}_2\bar{s}_1) (\bar{s}_4\bar{s}_3\bar{s}_2\bar{s}_1) \dot{s}_1\dot{s}_2\dot{s}_1 (\bar{s}_5\bar{s}_4\bar{s}_3\bar{s}_2\bar{s}_1) (\bar{s}_6\bar{s}_5\bar{s}_4\bar{s}_3\bar{s}_2\bar{s}_1) (\bar{s}_7\bar{s}_6\bar{s}_5\bar{s}_4\bar{s}_3\bar{s}_2\bar{s}_1) \dot{s}_1 \\
    & \omit \hfill \text{by permuting all non-trivial $w_{L_r}$, $r=4, \ldots, 1$, to the right} \\
    &= (\bar{s}_2\bar{s}_1) (\bar{s}_3\bar{s}_2\bar{s}_1) (\bar{s}_4\bar{s}_3) \dot{s}_1 (\bar{s}_5\bar{s}_4\bar{s}_3\bar{s}_2\bar{s}_1) (\bar{s}_6\bar{s}_5\bar{s}_4\bar{s}_3\bar{s}_2\bar{s}_1) (\bar{s}_7\bar{s}_6\bar{s}_5\bar{s}_4\bar{s}_3\bar{s}_2) \\
    & \omit \hfill \text{since $\bar{s}_i\cdots \bar{s}_1=(\dot{s}_1\cdots \dot{s}_i)^{-1}$ by definition} \\
    &= [\bar{s}_2\bar{s}_3\bar{s}_4\bar{s}_5\bar{s}_6\bar{s}_7]  (\bar{s}_1) (\bar{s}_2\bar{s}_1) (\bar{s}_3) \dot{s}_1 (\bar{s}_4\bar{s}_3\bar{s}_2\bar{s}_1) (\bar{s}_5\bar{s}_4\bar{s}_3\bar{s}_2\bar{s}_1) (\bar{s}_6\bar{s}_5\bar{s}_4\bar{s}_3\bar{s}_2) \\
    & \omit \hfill \text{by permuting the first term in each sub-product marked by parentheses to the left} \\
    &= [\bar{s}_2\bar{s}_3\bar{s}_4\bar{s}_5\bar{s}_6\bar{s}_7] [\bar{s}_1\bar{s}_2\bar{s}_3\bar{s}_4\bar{s}_5\bar{s}_6] (\bar{s}_1) \dot{s}_1 (\bar{s}_3\bar{s}_2\bar{s}_1) (\bar{s}_4\bar{s}_3\bar{s}_2\bar{s}_1) (\bar{s}_5\bar{s}_4\bar{s}_3\bar{s}_2) \\
    & \omit \hfill \text{since there was no possible cancellation we have again permuted the first term} \\
    & \omit \hfill \text{in each sub-product marked by parentheses to the left} \\
    &= [\bar{s}_2\bar{s}_3\bar{s}_4\bar{s}_5\bar{s}_6\bar{s}_7] [\bar{s}_1\bar{s}_2\bar{s}_3\bar{s}_4\bar{s}_5\bar{s}_6] (\bar{s}_3\bar{s}_2\bar{s}_1) (\bar{s}_4\bar{s}_3\bar{s}_2\bar{s}_1) (\bar{s}_5\bar{s}_4\bar{s}_3\bar{s}_2) \\
    & \omit \hfill \text{since $\bar{s}_1=\dot{s}_1^{-1}$} \\
    &= [\bar{s}_2\bar{s}_3\bar{s}_4\bar{s}_5\bar{s}_6\bar{s}_7] [\bar{s}_1\bar{s}_2\bar{s}_3\bar{s}_4\bar{s}_5\bar{s}_6] [\bar{s}_3\bar{s}_4\bar{s}_5] (\bar{s}_2\bar{s}_1) (\bar{s}_3\bar{s}_2\bar{s}_1) (\bar{s}_4\bar{s}_3\bar{s}_2) \\
    & \omit \hfill \text{by permuting the first term in each sub-product marked by parentheses to the left} \\
    & \hspace{0.18cm} \vdots \\
    % &= [\bar{s}_2\bar{s}_3\bar{s}_4\bar{s}_5\bar{s}_6\bar{s}_7] [\bar{s}_1\bar{s}_2\bar{s}_3\bar{s}_4\bar{s}_5\bar{s}_6] [\bar{s}_3\bar{s}_4\bar{s}_5] [\bar{s}_2\bar{s}_3\bar{s}_4] (\bar{s}_1) (\bar{s}_2\bar{s}_1) (\bar{s}_3\bar{s}_2) \\
    % &= [\bar{s}_2\bar{s}_3\bar{s}_4\bar{s}_5\bar{s}_6\bar{s}_7] [\bar{s}_1\bar{s}_2\bar{s}_3\bar{s}_4\bar{s}_5\bar{s}_6] [\bar{s}_3\bar{s}_4\bar{s}_5] [\bar{s}_2\bar{s}_3\bar{s}_4] [\bar{s}_1\bar{s}_2\bar{s}_3] (\bar{s}_1) (\bar{s}_2) \\
    &= [\bar{s}_2\bar{s}_3\bar{s}_4\bar{s}_5\bar{s}_6\bar{s}_7] [\bar{s}_1\bar{s}_2\bar{s}_3\bar{s}_4\bar{s}_5\bar{s}_6] [\bar{s}_3\bar{s}_4\bar{s}_5] [\bar{s}_2\bar{s}_3\bar{s}_4] [\bar{s}_1\bar{s}_2\bar{s}_3] [\bar{s}_1\bar{s}_2] \\
    \end{aligned}
    $$
\end{ex}

\begin{proof}[Proof of Lemma {\ref{lem wPw0 description in dot vertices}}]
We will take the reduced expression ${\mathbf{i}_0}^{\mathrm{op}}$ for $\bar{w}_0$. %, that is
    % $$\bar{w}_0 = \prod_{j=1}^{n-1} \prod_{i=1}^{j} \bar{s}_{j-i+1}.
    % $$
This means that
    $$\dot{w}_P\bar{w}_0 = \dot{w}_{L_1} \cdots \dot{w}_{L_{l+1}} \prod_{j=1}^{n-1} \prod_{i=1}^{j} \bar{s}_{j-i+1}
    $$
where some or all of the $\dot{w}_{L_r}$ may be trivial. Additionally we will write $\dot{w}_{0,r}$ for the representative in terms of $\dot{s}_i$'s of the longest element of
    $$\left\langle s_i \ | \ i \in \{ 1, \ldots, r \} \right\rangle \subset W
    $$
given by the reduced expression $\mathbf{i}_0$. % with $n=r$.

Now for each $r\in \{1, \ldots, l+1\}$, due to choosing the reduced expression ${\mathbf{i}_0}^{\mathrm{op}}$ for $\bar{w}_0$ and since $\dot{w}_{L_r}$ is a representative of the longest element of
    $$\left\langle s_i \ | \ i \in \{ n_{r-1}+1, \ldots, n_r-1 \} \right\rangle,
    $$
% and using a similar argument to the proof of Lemma \ref{lem pass weyl group elt from square past w0},
we have the following:
    $$\begin{aligned} \dot{w}_{L_r} \left( \prod_{j=1}^{n_r-1} \prod_{i=1}^{j} \bar{s}_{j-i+1} \right)
        &= \left( \prod_{j=1}^{n_r-1} \prod_{i=1}^{j} \bar{s}_{j-i+1} \right) \dot{w}_{0,k_r-1}
            & \begin{aligned}[t] \text{noticing that } & \\
            n_{r-1}+1 = n_r-1 - (k_r-1) & \\
            \end{aligned} \\
        &= \left( \prod_{j=1}^{n_r-2} \prod_{i=1}^{j} \bar{s}_{j-i+1} \right) \left( \prod_{i=1}^{n_r-k_r} \bar{s}_{n_r-i} \right) \dot{w}_{0,k_r-2}.
    \end{aligned}
    $$
In particular, after the permutation in the first line above, the last $k_r-1$ $\bar{s}_i$ terms form the inverse of the first $k_r-1$ $\dot{s}_i$ terms
    $$\bar{s}_{k_r-1}\cdots \bar{s}_1=(\dot{s}_1\cdots \dot{s}_{k_r-1})^{-1}
    $$
and hence these terms no longer appear after the second equality.

We apply the above procedure to permute each $\dot{w}_{L_{r}}$ in $\dot{w}_P\bar{w}_0$ to the right, starting with $r=l+1$ and working until we have treated $\dot{w}_{L_{1}}$, and then we make the respective cancellations. We obtain
    \begin{align}
    \dot{w}_P\bar{w}_0 &= \left( \prod_{j=1}^{n_1-1} \prod_{i=1}^{j} \bar{s}_{j-i+1} \right) \dot{w}_{0,k_1-1}
        \left( \prod_{j=n_1}^{n_2-1} \prod_{i=1}^{j} \bar{s}_{j-i+1} \right) \dot{w}_{0,k_2-1} \cdots
        \left( \prod_{j=n_l}^{n_{l+1}-1} \prod_{i=1}^{j} \bar{s}_{j-i+1} \right) \dot{w}_{0,k_{l+1}-1} \notag \\
    &= \left( \prod_{j=1}^{n_1-2} \prod_{i=1}^{j} \bar{s}_{j-i+1} \right) \left( \prod_{i=1}^{n_1-k_1} \bar{s}_{n_1-i} \right) \dot{w}_{0,k_1-2} \cdots
        \left( \prod_{j=n_l}^{n_{l+1}-2} \prod_{i=1}^{j} \bar{s}_{j-i+1} \right) \left( \prod_{i=1}^{n_{l+1}-k_{l+1}} \bar{s}_{n_{l+1}-i} \right) \dot{w}_{0,k_{l+1}-2}. \label{eqn wPw0 after permuting wLr}
    \end{align}

% We note that, if we were to only keep track of the reduced expression for $\bar{w}_0$, ${\mathbf{i}'_0}^{\mathrm{op}}$, then for each $j=1, \ldots, l+1$ we have effectively `removed' from this expression the ($n_j-1$)-th appearance of $1$, the ($n_j-2$)-th appearance of $2$, ... and the ($n_j-k_j+1$)-th appearance of $k_j-1$.

% ***need??*** We note that if we were to only keep track of the $\bar{s}_i$'s, then for each $j=1, \ldots, l+1$ we no longer have the $\bar{s}_t$ which was originally the ($n_j-t$)-th occurrence of $\bar{s}_t$ in $\bar{w}_0$.
% ($n_j-1$)-th $\bar{s}_1$, the ($n_j-2$)-th $\bar{s}_2$, ... and the ($n_j-k_j+1$)-th $\bar{s}_{k_j-1}$.

Next we consider the products of the form
    $$\prod_{i=1}^{j} \bar{s}_{j-i+1} \quad \text{and} \quad \prod_{i=1}^{n_r-k_r} \bar{s}_{n_r-i}.
    $$
Starting with the left-most such product in (\ref{eqn wPw0 after permuting wLr}), then taking each subsequent product in succession, we permute the respective first terms as far left as possible using the relation
    \begin{equation} \label{eqn si sj = sj si relation}
    \bar{s}_i\bar{s}_j=\bar{s}_j\bar{s}_i \quad \text{when } |i-j|\geq 2.
    \end{equation}
We may do this since, by construction, all terms to the left of the $\bar{s}_j$ we are trying to permute to the beginning of the product (apart from in the first sub-product of $\bar{s}_i$'s with strictly increasing subscripts) have subscript at most $j-2$. After these permutations we see that $\dot{w}_P\bar{w}_0$ begins with the following product:
    $$ \prod_{i=k_1}^{n-1}\bar{s}_i.
    $$
Indeed, using square braces and parentheses in the same way as in Example \ref{ex of proof of Lem wPw0 description in dot vertices}, $\dot{w}_P\bar{w}_0$ is now given by the product
    $$\left[ \prod_{i=k_1}^{n-1}\bar{s}_i \right]
        \left( \prod_{j=1}^{n_1-2} \prod_{i=2}^{j} \bar{s}_{j-i+1} \right) \left( \prod_{i=2}^{n_1-k_1} \bar{s}_{n_1-i} \right) \dot{w}_{0,k_1-2} \cdots
        \left( \prod_{j=n_l}^{n_{l+1}-2} \prod_{i=2}^{j} \bar{s}_{j-i+1} \right) \left( \prod_{i=2}^{n_{l+1}-k_{l+1}} \bar{s}_{n_{l+1}-i} \right) \dot{w}_{0,k_{l+1}-2}
    $$
noting that some of these sub-products may be trivial.

If there exists some $k_r$ such that $k_r=n_r-1$, then in our new description of $\dot{w}_P\bar{w}_0$ we will be able to make a cancellation of the first $k_r-2=n_r-3$ $\dot{s}_i$ terms of $\dot{w}_{0,k_r-2}$ with the last $k_r-2$ $\bar{s}_i$ terms of the preceding product of $\bar{s}_i$'s. In Example \ref{ex of proof of Lem wPw0 description in dot vertices}, there was no possible cancellation after this first set of permutations to the left.

% Again we note that if we were to only keep track of the reduced expression for $\bar{w}_0$, then as well as the `removals' mentioned before, for each $j=1, \ldots, l+1$ such that $k_j\geq 2$ we have effectively `removed' from the original expression ${\mathbf{i}'_0}^{\mathrm{op}}$ the ($n_j-2$)-th appearance of $1$, the ($n_j-3$)-th appearance of $2$, ... and the ($n_j-(k_j-1)-1$)-th appearance of $k_j-1$.

% Again we note that if we were to only keep track of the $\bar{s}_i$'s, then for each $j=1, \ldots, l+1$ such that $k_j\geq 2$ we no longer have the ($n_j-2$)-th $\bar{s}_1$, the ($n_j-3$)-th $\bar{s}_2$, ... and the ($n_j-(k_j-1)-1$)-th $\bar{s}_{k_j-1}$, where by the $r$-th $\bar{s}_i$ we mean the $\bar{s}_i$ which was originally the $r$-th $\bar{s}_i$  in $\bar{w}_0$.

% We note that if we were to only keep track of the $\bar{s}_i$'s, then for each $j=1, \ldots, l+1$ such that $k_j\geq 2$ and each $t=1, \ldots, k_j-1$ we no longer have the $\bar{s}_t$ which was originally the ($n_j-t-1$)-th occurrence of $\bar{s}_t$ in $\bar{w}_0$.

We repeat this procedure iteratively: if our products are now of the form
    $$\prod_{i=b}^{j} \bar{s}_{j-i+1} \quad \text{and} \quad \prod_{i=b}^{n_r-k_r} \bar{s}_{n_r-i}
    $$
then starting with the left-most product and taking each subsequent product in succession, we permute the respective first terms to the left using the relation (\ref{eqn si sj = sj si relation}) until we reach the sub-products of $\bar{s}_i$'s with strictly increasing subscripts which we formed earlier.

If there exists some $k_r$ such that $k_r=n_r-b$ then we will be able to make a cancellation of the first $k_r-b-1=n_r-b-2$ $\dot{s}_i$ terms of $\dot{w}_{0,k_r-b-1}$ with the last $k_r-b-1$ $\bar{s}_i$ terms of the preceding product of $\bar{s}_i$'s. In Example \ref{ex of proof of Lem wPw0 description in dot vertices}, we showed the cancellation which occurred after the second set of permutations to the left, that is, when $b=2$.

We see that $\dot{w}_P\bar{w}_0$ now begins with the following $b$ products of $\bar{s}_i$'s for some $r$: %with strictly increasing subscripts, the last term of each being $\bar{s}_{n-b}$. **
    $$ \left[ \prod_{i=k_1}^{n-1}\bar{s}_i \right] \left[ \prod_{i=k_1-1}^{n-2}\bar{s}_i \right] \cdots \left[ \prod_{i=1}^{n-n_1}\bar{s}_i \right] \cdots \left[ \prod_{i=k_r}^{n-n_{r-1}-1}\bar{s}_i \right] \cdots \left[ \prod_{i=k_r-n_{r-1}+b-1}^{n-b}\bar{s}_i \right]
    $$
% Additionally if we were to only keep track of the $\bar{s}_i$'s, then for each $j=1, \ldots, l+1$ such that $k_j\geq b+1$ and each $t=1, \ldots, k_j-1$ we no longer have the $\bar{s}_t$ which was originally the ($n_j-t-b$)-th occurrence of $\bar{s}_t$ in $\bar{w}_0$.
In particular, repeating the above procedure yields the result in at most $n-1$ steps:
    \begin{align}
    \dot{w}_P\bar{w}_0 &= \prod_{r=1}^{l+1} \prod_{t=1}^{k_r}  \prod_{i=k_r-t+1}^{n-n_{r-1}-t} \bar{s}_i \label{eqn wPw0 increasing si} \\
        &= \prod_{k=1}^{n-1} \prod_{\substack{a=1 \\ v_{(k,a)} \in \mathcal{V}_L^{\bullet}}}^{n-k} \bar{s}_a. \notag
    \end{align}
\end{proof}

% \subsection[Chamber Ansatz minors (\texorpdfstring{$G/P$}{G/P} setting)]{Chamber Ansatz minors ($G/P$ setting)}

\newpage
\subsection{Chamber Ansatz minors}
\label{subsec G/P Chamber Ansatz minors}
\fancyhead[L]{9.3 \ \ Chamber Ansatz minors}
% \fancyhead[L]{10.3 \ \ Chamber Ansatz minors}

In order to more easily compute the Chamber Ansatz minors required for the proof of Theorem \ref{thm b- in terms of y_i}, we will need ansatz arrangements (see Section \ref{subsec The Chamber Ansatz}). We begin with an example.

\begin{ex}
In Figure \ref{fig ansatz arrangement wPw0 F2,5,6,C8} we give the ansatz arrangement for $\dot{w}_P\bar{w}_0$ in our running example of $\mathcal{F}_{2,5,6}(\mathbb{C}^8)$, using the description given in Lemma \ref{lem wPw0 description in dot vertices}.

\begin{figure}[ht]
\centering
\begin{tikzpicture}[scale=0.76]
    % pseudolines
        \draw (-0.3,8) -- (19.1,8);
        \draw (-0.3,7) -- (19.1,7);
        \draw (-0.3,6) -- (19.1,6);
        \draw (-0.3,5) -- (19.1,5);
        \draw (-0.3,4) -- (19.1,4);
        \draw (-0.3,3) -- (19.1,3);
        \draw (-0.3,2) -- (19.1,2);
        \draw (-0.3,1) -- (19.1,1);

    % crossings
        % first set
        \draw[thick, color=white] (0.4,2) -- (0.8,2);
        \draw[thick, color=white] (0.4,3) -- (0.8,3);
        \draw[line cap=round] (0.4,2) -- (0.8,3);
        \draw[line cap=round] (0.4,3) -- (0.8,2);

        \draw[thick, color=white] (1.2,3) -- (1.6,3);
        \draw[thick, color=white] (1.2,4) -- (1.6,4);
        \draw[line cap=round] (1.2,3) -- (1.6,4);
        \draw[line cap=round] (1.2,4) -- (1.6,3);

        \draw[thick, color=white] (2,4) -- (2.4,4);
        \draw[thick, color=white] (2,5) -- (2.4,5);
        \draw[line cap=round] (2,4) -- (2.4,5);
        \draw[line cap=round] (2,5) -- (2.4,4);

        \draw[thick, color=white] (2.8,5) -- (3.2,5);
        \draw[thick, color=white] (2.8,6) -- (3.2,6);
        \draw[line cap=round] (2.8,5) -- (3.2,6);
        \draw[line cap=round] (2.8,6) -- (3.2,5);

        \draw[thick, color=white] (3.6,6) -- (4,6);
        \draw[thick, color=white] (3.6,7) -- (4,7);
        \draw[line cap=round] (3.6,6) -- (4,7);
        \draw[line cap=round] (3.6,7) -- (4,6);

        \draw[thick, color=white] (4.4,7) -- (4.8,7);
        \draw[thick, color=white] (4.4,8) -- (4.8,8);
        \draw[line cap=round] (4.4,7) -- (4.8,8);
        \draw[line cap=round] (4.4,8) -- (4.8,7);

        % second set
        \draw[thick, color=white] (5.2,1) -- (5.6,1);
        \draw[thick, color=white] (5.2,2) -- (5.6,2);
        \draw[line cap=round] (5.2,1) -- (5.6,2);
        \draw[line cap=round] (5.2,2) -- (5.6,1);

        \draw[thick, color=white] (6,2) -- (6.4,2);
        \draw[thick, color=white] (6,3) -- (6.4,3);
        \draw[line cap=round] (6,2) -- (6.4,3);
        \draw[line cap=round] (6,3) -- (6.4,2);

        \draw[thick, color=white] (6.8,3) -- (7.2,3);
        \draw[thick, color=white] (6.8,4) -- (7.2,4);
        \draw[line cap=round] (6.8,3) -- (7.2,4);
        \draw[line cap=round] (6.8,4) -- (7.2,3);

        \draw[thick, color=white] (7.6,4) -- (8,4);
        \draw[thick, color=white] (7.6,5) -- (8,5);
        \draw[line cap=round] (7.6,4) -- (8,5);
        \draw[line cap=round] (7.6,5) -- (8,4);

        \draw[thick, color=white] (8.4,5) -- (8.8,5);
        \draw[thick, color=white] (8.4,6) -- (8.8,6);
        \draw[line cap=round] (8.4,5) -- (8.8,6);
        \draw[line cap=round] (8.4,6) -- (8.8,5);

        \draw[thick, color=white] (9.2,6) -- (9.6,6);
        \draw[thick, color=white] (9.2,7) -- (9.6,7);
        \draw[line cap=round] (9.2,6) -- (9.6,7);
        \draw[line cap=round] (9.2,7) -- (9.6,6);

        % third set
        \draw[thick, color=white] (10,3) -- (10.4,3);
        \draw[thick, color=white] (10,4) -- (10.4,4);
        \draw[line cap=round] (10,3) -- (10.4,4);
        \draw[line cap=round] (10,4) -- (10.4,3);

        \draw[thick, color=white] (10.8,4) -- (11.2,4);
        \draw[thick, color=white] (10.8,5) -- (11.2,5);
        \draw[line cap=round] (10.8,4) -- (11.2,5);
        \draw[line cap=round] (10.8,5) -- (11.2,4);

        \draw[thick, color=white] (11.6,5) -- (12,5);
        \draw[thick, color=white] (11.6,6) -- (12,6);
        \draw[line cap=round] (11.6,5) -- (12,6);
        \draw[line cap=round] (11.6,6) -- (12,5);

        % fourth set
        \draw[thick, color=white] (12.4,2) -- (12.8,2);
        \draw[thick, color=white] (12.4,3) -- (12.8,3);
        \draw[line cap=round] (12.4,2) -- (12.8,3);
        \draw[line cap=round] (12.4,3) -- (12.8,2);

        \draw[thick, color=white] (13.2,3) -- (13.6,3);
        \draw[thick, color=white] (13.2,4) -- (13.6,4);
        \draw[line cap=round] (13.2,3) -- (13.6,4);
        \draw[line cap=round] (13.2,4) -- (13.6,3);

        \draw[thick, color=white] (14,4) -- (14.4,4);
        \draw[thick, color=white] (14,5) -- (14.4,5);
        \draw[line cap=round] (14,4) -- (14.4,5);
        \draw[line cap=round] (14,5) -- (14.4,4);

        % fifth set
        \draw[thick, color=white] (14.8,1) -- (15.2,1);
        \draw[thick, color=white] (14.8,2) -- (15.2,2);
        \draw[line cap=round] (14.8,1) -- (15.2,2);
        \draw[line cap=round] (14.8,2) -- (15.2,1);

        \draw[thick, color=white] (15.6,2) -- (16,2);
        \draw[thick, color=white] (15.6,3) -- (16,3);
        \draw[line cap=round] (15.6,2) -- (16,3);
        \draw[line cap=round] (15.6,3) -- (16,2);

        \draw[thick, color=white] (16.4,3) -- (16.8,3);
        \draw[thick, color=white] (16.4,4) -- (16.8,4);
        \draw[line cap=round] (16.4,3) -- (16.8,4);
        \draw[line cap=round] (16.4,4) -- (16.8,3);

        % sixth set
        \draw[thick, color=white] (17.2,1) -- (17.6,1);
        \draw[thick, color=white] (17.2,2) -- (17.6,2);
        \draw[line cap=round] (17.2,1) -- (17.6,2);
        \draw[line cap=round] (17.2,2) -- (17.6,1);

        \draw[thick, color=white] (18,2) -- (18.4,2);
        \draw[thick, color=white] (18,3) -- (18.4,3);
        \draw[line cap=round] (18,2) -- (18.4,3);
        \draw[line cap=round] (18,3) -- (18.4,2);

    % dots
        \node at (0.6,2.5) {$\bullet$};
        \node at (1.4,3.5) {$\bullet$};
        \node at (2.2,4.5) {$\bullet$};
        \node at (3,5.5) {$\bullet$};
        \node at (3.8,6.5) {$\bullet$};
        \node at (4.6,7.5) {$\bullet$};

        \node at (5.4,1.5) {$\bullet$};
        \node at (6.2,2.5) {$\bullet$};
        \node at (7,3.5) {$\bullet$};
        \node at (7.8,4.5) {$\bullet$};
        \node at (8.6,5.5) {$\bullet$};
        \node at (9.4,6.5) {$\bullet$};

        \node at (10.2,3.5) {$\bullet$};
        \node at (11,4.5) {$\bullet$};
        \node at (11.8,5.5) {$\bullet$};

        \node at (12.6,2.5) {$\bullet$};
        \node at (13.4,3.5) {$\bullet$};
        \node at (14.2,4.5) {$\bullet$};

        \node at (15,1.5) {$\bullet$};
        \node at (15.8,2.5) {$\bullet$};
        \node at (16.6,3.5) {$\bullet$};

        \node at (17.4,1.5) {$\bullet$};
        \node at (18.2,2.5) {$\bullet$};

    % pseudoline labels
        \node at (-0.6,8) {$8$};
        \node at (-0.6,7) {$7$};
        \node at (-0.6,6) {$6$};
        \node at (-0.6,5) {$5$};
        \node at (-0.6,4) {$4$};
        \node at (-0.6,3) {$3$};
        \node at (-0.6,2) {$2$};
        \node at (-0.6,1) {$1$};

        \node at (19.4,8) {$2$};
        \node at (19.4,7) {$1$};
        \node at (19.4,6) {$5$};
        \node at (19.4,5) {$4$};
        \node at (19.4,4) {$3$};
        \node at (19.4,3) {$6$};
        \node at (19.4,2) {$8$};
        \node at (19.4,1) {$7$};

    % chamber labels
        \node at (2.2,7.5) {$1234567$};
        \node at (11.9,7.5) {$1345678$};

        \node at (1.8,6.5) {$123456$};
        \node at (6.6,6.5) {$134567$};
        \node at (14.3,6.5) {$345678$};

        \node at (1.4,5.5) {$12345$};
        \node at (5.8,5.5) {$13456$};
        \node at (10.2,5.5) {$34567$};
        \node at (15.5,5.5) {$34678$};

        \node at (1,4.5) {$1234$};
        \node at (5,4.5) {$1345$};
        \node at (9.4,4.5) {$3456$};
        \node at (12.6,4.5) {$3467$};
        \node at (16.7,4.5) {$3678$};

        \node at (0.6,3.5) {$123$};
        \node at (4.2,3.5) {$134$};
        \node at (8.6,3.5) {$345$};
        \node at (11.8,3.5) {$346$};
        \node at (15,3.5) {$367$};
        \node at (17.9,3.5) {$678$};

        \node at (0.1,2.5) {$12$};
        \node at (3.4,2.5) {$13$};
        \node at (9.4,2.5) {$34$};
        \node at (14.2,2.5) {$36$};
        \node at (17,2.5) {$67$};
        \node at (18.7,2.5) {$78$};

        \node at (2.6,1.5) {$1$};
        \node at (10.2,1.5) {$3$};
        \node at (16.2,1.5) {$6$};
        \node at (18.3,1.5) {$7$};
\end{tikzpicture}
\caption{The ansatz arrangement for $\dot{w}_P\bar{w}_0$ in the example of $\mathcal{F}_{2,5,6}(\mathbb{C}^8)$} \label{fig ansatz arrangement wPw0 F2,5,6,C8}
\end{figure}
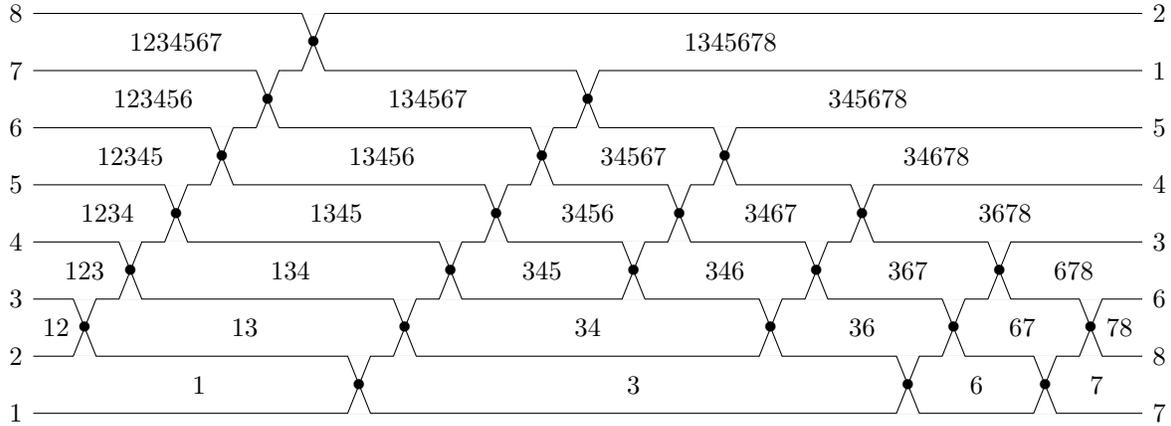
\end{ex}

We say a chamber is at height $b$ if it is found between the $b$-th and $(b+1)$-th strings when counting from the bottom of the ansatz arrangement. For example, the left-most chamber at height $b$ is between the strings with labels $b$, $b+1$ by definition.

\begin{lem} \label{lem G/P chamber label form}
We consider the ansatz arrangement corresponding to the description of $\dot{w}_P\bar{w}_0$ given in Lemma \ref{lem wPw0 description in dot vertices}. For $j=1, \ldots, l+1$ and $t=1, \ldots, k_j$, the chamber labels of this ansatz arrangement are of the form $\{1, 2, \ldots, b\}$ or
    \begin{equation} \label{eqn chamber label below string}
    \{n_{j-1}+1, \ldots, n_{j}-t\} \cup \{n_{j}+1, \ldots, n_{j-1}+b+t\}
    \end{equation}
where $b=1, \ldots, n-n_{j-1}-t$ is the height of the chamber.
\end{lem}

Of note, the only description of $\dot{w}_P\bar{w}_0$ we will use is the one given in Lemma \ref{lem wPw0 description in dot vertices}, thus we will often simply write `the ansatz arrangement for $\dot{w}_P\bar{w}_0$' as it is clear that we are referring to this particular arrangement.

\begin{proof}
Considering the description (\ref{eqn wPw0 increasing si}) for $\dot{w}_P\bar{w}_0$, each sub-product
    \begin{equation} \label{eqn subproduct 1 of si for ansatz arr proof}
    \prod_{i=k_r-t+1}^{n-n_{r-1}-t} \bar{s}_i, \quad t=1, \ldots, k_r,
    \end{equation}
in $\dot{w}_P\bar{w}_0$ corresponds to a sequence of crossings at heights $k_r-t+1, \ldots, n-n_{r-1}-t$ in succession. Since the heights of these crossings increase by $1$ sequentially, we can think of this as taking the $(k_r-t+1)$-th string counting from the bottom of the arrangement and moving it upwards until it crosses the $(n-n_{r-1}-t+1)$-th string (again counting from the bottom) but no further. The other strings remain disjoint during this procedure. For example, see Figure \ref{fig Crossings in the ansatz arrangement for wPw0}, ignoring the shading for now.

\begin{figure}[ht]
\centering
\begin{tikzpicture}[scale=0.76]
    % shading
        \filldraw[color=black!10] (1.8,-0.28) -- (1.8,3) -- (2,3) -- (2.4,4) -- (3.6,4) -- (4,5) -- (4.8,5) -- (6.03,6) -- (6.8,6) -- (7.2,7) -- (8.4,7) -- (8.8,8) -- (9,8) -- (9,-0.28) -- cycle;
        \shade[left color=white, right color=black!10] (0.8,-0.28) -- (0.8,3) -- (2,3) -- (2,-0.28) -- cycle;
        \shade[left color=black!10, right color=white] (9,-0.28) -- (9,8) -- (10,8) -- (10,-0.28) -- cycle;

    % pseudolines
        \draw (0.3,9) -- (10.5,9);
        \draw (0.3,8) -- (10.5,8);
        \draw (0.3,7) -- (10.5,7);
        \draw (0.3,6) -- (10.5,6);
        \draw (0.3,5) -- (10.5,5);
        \draw (0.3,4) -- (10.5,4);
        \draw (0.3,3) -- (10.5,3);
        \draw (0.3,2) -- (10.5,2);
        \draw (0.3,0.72) -- (10.5,0.72);
        \draw (0.3,-0.28) -- (10.5,-0.28);
        \node at (1, 1.5) {$\vdots$};
        \node at (9.8, 1.5) {$\vdots$};

    % crossings
        \draw[thick, color=black!10] (2,3) -- (2.4,3);
        \draw[thick, color=white] (2,4) -- (2.4,4);
        \draw[line cap=round] (2,3) -- (2.4,4);
        \draw[line cap=round] (2,4) -- (2.4,3);

        \draw[thick, color=black!10] (3.6,4) -- (4,4);
        \draw[thick, color=white] (3.6,5) -- (4,5);
        \draw[line cap=round] (3.6,4) -- (4,5);
        \draw[line cap=round] (3.6,5) -- (4,4);

        \draw[thick, color=black!10] (4.8,4.995) -- (6,4.995);
        \draw[thick, color=white] (4.8,6.005) -- (6,6.005);
        % \draw[line cap=round] (5.2,5) -- (5.6,6);
        % \draw[line cap=round] (5.2,6) -- (5.6,5);

        \draw[thick, color=black!10] (6.8,6) -- (7.2,6);
        \draw[thick, color=white] (6.8,7) -- (7.2,7);
        \draw[line cap=round] (6.8,6) -- (7.2,7);
        \draw[line cap=round] (6.8,7) -- (7.2,6);

        \draw[thick, color=black!10] (8.4,7) -- (8.8,7);
        \draw[thick, color=white] (8.4,8) -- (8.8,8);
        \draw[line cap=round] (8.4,7) -- (8.8,8);
        \draw[line cap=round] (8.4,8) -- (8.8,7);

    % dots
        \node at (2.2,3.5) {$\bullet$};
        \node at (3.8,4.5) {$\bullet$};
        % \node at (5.7,4.99) {$\cdots$};
        \node at (5.4,5.65) {$\iddots$};
        % \node at (5.17,5.99) {$\cdots$};
        \node at (7,6.5) {$\bullet$};
        \node at (8.6,7.5) {$\bullet$};

    % pseudoline labels
        \node at (-0.05,8.99) {$\cdots$};
        \node at (-0.05,7.99) {$\cdots$};
        \node at (-0.05,6.99) {$\cdots$};
        \node at (-0.05,5.99) {$\cdots$};
        \node at (-0.05,4.99) {$\cdots$};
        \node at (-0.05,3.99) {$\cdots$};
        \node at (-0.05,2.99) {$\cdots$};
        \node at (-0.05,1.99) {$\cdots$};
        \node at (-0.05,0.71) {$\cdots$};
        \node at (-0.05,-0.29) {$\cdots$};
        \node at (-2.75,8.5) {height $n-n_r-t+1$};
        \node at (-2,6.14) {$\vdots$};
        \node at (-2.3,3.5) {height $k_r-t+1$};
        \node at (-1.92,2.5) {height $k_r-t$};
        \node at (-1.43,0.22) {height $1$};

        \node at (10.93,8.99) {$\cdots$};
        \node at (10.93,7.99) {$\cdots$};
        \node at (10.93,6.99) {$\cdots$};
        \node at (10.93,5.99) {$\cdots$};
        \node at (10.93,4.99) {$\cdots$};
        \node at (10.93,3.99) {$\cdots$};
        \node at (10.93,2.99) {$\cdots$};
        \node at (10.93,1.99) {$\cdots$};
        \node at (10.93,0.71) {$\cdots$};
        \node at (10.93,-0.29) {$\cdots$};

    % \bar{s}_i's
        \node at (2.2,-1) {$\bar{s}_{k_r-t+1}$};
        \node at (5.4,-1) {$\cdots$};
        \node at (8.6,-1) {$\bar{s}_{n-n_r-t+1}$};
\end{tikzpicture}
\caption{Crossings in the ansatz arrangement for $\dot{w}_P\bar{w}_0$} \label{fig Crossings in the ansatz arrangement for wPw0}
\end{figure}
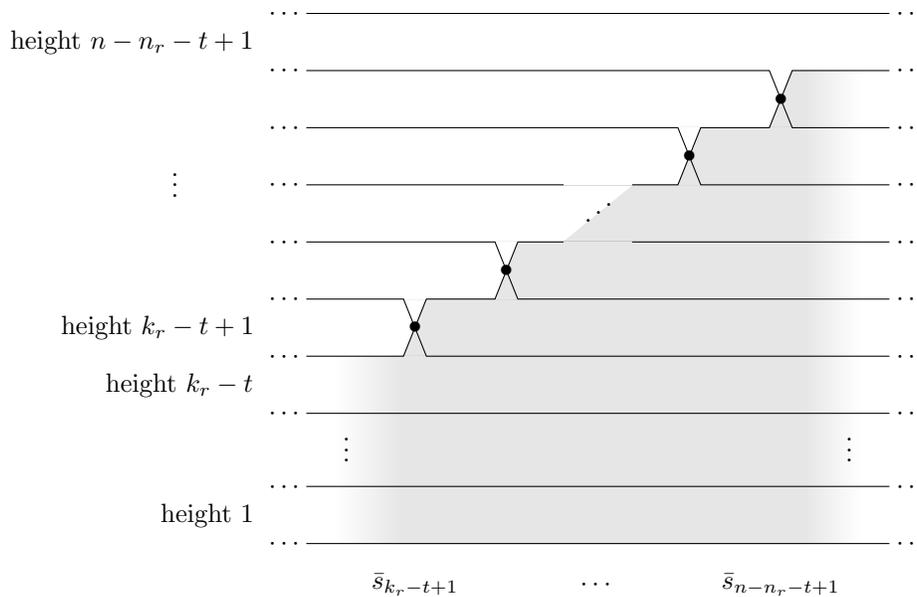
In particular, the last crossing from one of these sub-products (\ref{eqn subproduct 1 of si for ansatz arr proof}), when considered in succession in the description (\ref{eqn wPw0 increasing si}) for $\dot{w}_P\bar{w}_0$, is always with string $n$.
This holds by induction using the following four facts:
\begin{enumerate}
    \item The subscripts of the terms in the sub-products (\ref{eqn subproduct 1 of si for ansatz arr proof}) increase by $1$ with each subsequent term. %, thus when we take a particular string we only move it upwards.
    \label{fact subscripts of si subprod incr}
    \item The subscript of the last term of each of these sub-products decreases by $1$ when considered in succession in $\dot{w}_P\bar{w}_0$, see (\ref{eqn wPw0 increasing si}). \label{fact last term of si subprod decr by 1}
    \item The first such `last term' is $\bar{s}_{n-1}$, which gives the crossing of the $n$-th and $(n-1)$-th strings counting from the bottom of the ansatz arrangement.
    \item String $n$ is the $n$-th string at the left hand side of the ansatz arrangement, again counting from the bottom.
\end{enumerate}

We note that since the description (\ref{eqn wPw0 increasing si}) for $\dot{w}_P\bar{w}_0$ is factored into the sub-products (\ref{eqn subproduct 1 of si for ansatz arr proof}), our subsequent arguments will involve this idea of taking a string and moving it upwards past string $n$, keeping the other strings disjoint.

Starting from the left side of the ansatz arrangement, due to the pseudoline labelling the first chamber at each height $b$ has label $\{1, \ldots, b\}$. After a given string $i$ has been moved upwards in the diagram past string $n$, we may consider the collection of chambers in the ansatz arrangement given by the first chamber at each height $b$ after this set of crossings and below the string $i$. We will call these the \emph{first chambers below string $i$}. For example, see the shaded region of Figure \ref{fig Crossings in the ansatz arrangement for wPw0}. Of note, we always omit the $\{0\}$ label.

Now, for each $j \in \{1, \ldots, l+1\}$ we consider sub-products of $\dot{w}_P\bar{w}_0$ of the following form:
    $$
    \prod_{t=1}^{k_r}  \prod_{i=k_r-t+1}^{n-n_{r-1}-t} \bar{s}_i, \quad r=1, \ldots, j,
    $$
\begin{claim} For each $j \in \{1, \ldots, l+1\}$ we have the following:
\begin{enumerate}
    \item
        % For each $t$ in the crossings corresponding to **, the string we are taking and moving upwards until it crosses string $n$, is string $n_j-t+1$.
        After considering the crossings corresponding to
            $$\left(\prod_{r=1}^{j-1} \prod_{t=1}^{k_r}  \prod_{i=k_r-t+1}^{n-n_{r-1}-t} \bar{s}_i\right) \left(\prod_{t=1}^{t'-1}  \prod_{i=k_j-t+1}^{n-n_{r-1}-t} \bar{s}_i \right), \quad t' = 1, \ldots k_j,
            $$
        the next string we will take and move upwards until it crosses string $n$, is string $n_j-t'+1$. That is at this point in the ansatz arrangement, the $(k_j-t'+1)$-th string counting from the bottom of the arrangement, is string $n_j-t'+1$.
        \label{claim 1 ansatz arr for wPw0}

    \item After considering the crossings corresponding to
            \begin{equation} \label{eqn subproduct 2 of si for ansatz arr proof} \prod_{r=1}^{j} \prod_{t=1}^{k_r}  \prod_{i=k_r-t+1}^{n-n_{r-1}-t} \bar{s}_i
            \end{equation}
        the strings in the ansatz arrangement are
            $$n_j+1, \ldots, n, n_{j-1}+1, \ldots, n_j, n_{j-2}+1, \ldots, n_{j-1}, \ldots, n_0+1, \ldots, n_1
            $$
        read from bottom to top.
        \label{claim 2 ansatz arr for wPw0}
\end{enumerate}
\end{claim}

We will see that step by step, the proof of this claim will provide enough information to prove the lemma.
\begin{proof}[Proof of Claim]
We proceed by induction.

On the left hand side of the ansatz arrangement the strings, from bottom to top, are labelled $n_0+1=1, \ldots, n$. Thus we see that Claim \ref{claim 1 ansatz arr for wPw0} holds immediately when $j=1$ since $k_1=n_1$ and, by Fact \ref{fact subscripts of si subprod incr} above, when we take a particular string we only ever move it upwards leaving all other strings disjoint.

Now we observe by the above Fact \ref{fact last term of si subprod decr by 1}, that once a string has crossed string $n$ it will not be involved in any crossings in the remainder of the ansatz arrangement. %Thus after each sub-product of the form (\ref{eqn subproduct 1 of si for ansatz arr proof}) in the description (\ref{eqn wPw0 increasing si}) for $\dot{w}_P\bar{w}_0$, we may in practice ignore the strings above string $n$ when considering the subsequent crossings.
In particular we may combine this with the consequence of Fact \ref{fact subscripts of si subprod incr} from the previous paragraph. The result is that both during and after the crossings corresponding to (\ref{eqn subproduct 2 of si for ansatz arr proof}) for $j=1$, not only do the strings labelled $k_1+1 = n_1+1, \ldots, n$ remain disjoint, but the strings labelled $1, \ldots, k_1=n_1$ also do not cross each other. Thus Claim \ref{claim 2 ansatz arr for wPw0} holds when $j=1$.

As a consequence it is clear that the first chambers below string $n_0+1=1$ have labels of the form
    $$\{n_1+1, \ldots, n_1+b\}, \quad b=1, \ldots, n-n_1
    $$
where $b$ is the height of the chamber. In fact, by Claim \ref{claim 1 ansatz arr for wPw0} when $j=1$, we see inductively that for strings $1, \ldots, k_1$ (after moving a given string $n_1-t+1$, $t\in \{1, \ldots, k_1\}$, upwards past string $n$) the first chambers below string $n_1-t+1$ have labels of the form
    $$\{1, \ldots, n_1-t\} \cup \{n_1+1, \ldots, b+t\}, \
    $$
where $b=1, \ldots, n-t$ is the height of the chamber. We recall that we always omit the $\{0\}$ label. Of note, we choose to write the string labels $1, \ldots, k_1$ in the form $n_1-t+1$ for consistency of notation.

Now suppose that Claim \ref{claim 2 ansatz arr for wPw0} holds for $j=1, \ldots, j'$. Then after the crossings corresponding to (\ref{eqn subproduct 2 of si for ansatz arr proof}) for $j=j'$, the strings in the ansatz arrangement are given from bottom to top as follows:
    $$n_{j'}+1, \ldots, n, n_{j'-1}+1, \ldots, n_{j'}, n_{j'-2}+1, \ldots, n_{j'-1}, \ldots, n_0+1, \ldots, n_1
    $$
In particular, the $(k_{j'+1})$-th string counting from the bottom of the arrangement, is the desired string $n_{j'}+k_{j'+1} = n_{j'+1}$, satisfying Claim \ref{claim 1 ansatz arr for wPw0} with $j=j'+1$ and $t'=1$. Moreover by the consequence of Fact \ref{fact subscripts of si subprod incr} mentioned above, the strings which are currently the bottom $k_{j'+1}$ strings, will never cross. Thus we see that Claim \ref{claim 1 ansatz arr for wPw0} holds for $j=j'+1$, that is for each $t'=1, \ldots, k_{j'+1}$, and the induction is complete for Claim \ref{claim 1 ansatz arr for wPw0}.

After all of these crossings we are again in the case of Claim \ref{claim 2 ansatz arr for wPw0}, but now with $j=j'+1$. We have assumed Claim \ref{claim 2 ansatz arr for wPw0} holds when $j=j'$, we see that the case $j=j'+1$ holds by a similar argument to the $j=1$ case, but now applied the strings $n_{j'}+1, \ldots, n_{j'}+k_{j'+1} = n_{j'+1}$.

For the chamber labels, again by a similar argument to the $j=1$ case, we see inductively that for strings $n_{j'}+1, \ldots, n_{j'+1}$ (after moving string $n_{j'+1}-t+1$, $t\in \{1, \ldots, k_{j'+1}\}$, upwards past string $n$) the first chambers below string $n_{j'+1}-t+1$ have labels of the form
    $$
    \{n_{j'}+1, \ldots, n_{j'+1}-t\} \cup \{n_{j'+1}+1, \ldots, n_{j'}+b+t\}
    $$
where $b=1, \ldots, n-n_{j'}-t$ is the height of the chamber.
\end{proof}
We have now considered the labels of all remaining chambers in the ansatz arrangement for $\dot{w}_P\bar{w}_0$ and so the proof of the lemma is complete.
\end{proof}

In order to compute the minors of $u_L$ corresponding to the chamber labels we make use of graphs, like in the full flag case (see Section \ref{subsec Chamber Ansatz minors}). The graphs are constructed in exactly the same way as in the $G/B$ case, however we add some additional decoration to help us use them.

Firstly, we recall from Section \ref{subsec G/P Constructing matrices from a given quiver decoration} that to construct the matrix $u_L$ from the quiver $Q_P$, we work column by column, considering the $1$-paths leaving each dot vertex in succession. In the corresponding graph we will use dashed lines to distinguish the contributions from these columns and we will call these columns of the graph for ease. Moreover, we will group these columns in the graph by the square $L_i$ that the respective quiver columns intersect. We label these groupings by $k_i$ since this gives the respective number of columns.

The last addition to the graph is that we will draw dotted rectangles around the contribution from the last $1$-path in each column. In particular, each dotted rectangle corresponds to a factor $X^{\vee}_{i,\alpha}$ in $u_L$, noting that $X^{\vee}_{i,1}=\mathbf{x}^{\vee}_i$. By the factorisation of $X^{\vee}_{i,\alpha}(x_{v_t}/x_{v_s})=\tilde{X}^{\vee}_{i,\alpha}(r_{a_1}, \ldots, r_{a_{\alpha}})$ in terms of $\mathbf{x}^{\vee}_i$'s (see (\ref{eqn factor Xi into xi})), we see that the corresponding rectangle in the graph spans the horizontal lines
    $$i+1-\alpha, \ldots, i+1.
    $$
The factors $X^{\vee}_{i,\alpha}$ only arise when the minimal $1$-paths in $Q_P$ cross row $E_i=E_{n_j}$ for the respective values of $j$, before terminating at star vertices. Thus the dotted rectangles always have their top edges on the horizontal lines $n_j$.
It follows that for $j =1, \ldots, l$ and $\alpha=1, \ldots, k_j$, each column $k=n_{j-1}+\alpha$ in the graph contains the series of steps corresponding to the product
    $$\mathbf{x}^{\vee}_{n-1}(-) \mathbf{x}^{\vee}_{n-2}(-) \cdots \mathbf{x}^{\vee}_{n_j+1}(-) X^{\vee}_{n_j,\alpha}(-).
    $$

\begin{ex}
In our running example of $\mathcal{F}_{2,5,6}(\mathbb{C}^8)$, we begin with the factorisation of $u_L$ given in Examples \ref{ex gL temp quiver decoration} and \ref{ex uL temp quiver decoration cont}, where we have underlined the $X^{\vee}_{i,\alpha}$ factors expressed in terms of $\mathbf{x}^{\vee}_i$'s (see (\ref{eqn factor Xi into xi}) for the definition). Continuing with temporary labelling from these examples (see Figure \ref{fig gL temp quiver decoration}), we then give the corresponding graph in Figure \ref{fig temp labelling of graph for uL for F2,5,6,C8}.
    $$\begin{aligned}
    u_L &= \mathbf{x}^{\vee}_7(a) \mathbf{x}^{\vee}_6(b) \mathbf{x}^{\vee}_5(c) \mathbf{x}^{\vee}_4(d) \mathbf{x}^{\vee}_3(e) \mathbf{x}^{\vee}_2(f) \\
        &\hspace{0.47cm} \mathbf{x}^{\vee}_7(g) \mathbf{x}^{\vee}_6(h) \mathbf{x}^{\vee}_5(i) \mathbf{x}^{\vee}_4(j) \mathbf{x}^{\vee}_3(k) \underline{\big[ \mathbf{x}^{\vee}_1(f) \mathbf{x}^{\vee}_2(l) \mathbf{x}^{\vee}_1(-f) \mathbf{x}^{\vee}_2(-l) \big]} \\
        &\hspace{0.47cm} \mathbf{x}^{\vee}_7(m) \mathbf{x}^{\vee}_6(n) \mathbf{x}^{\vee}_5(p) \\
        &\hspace{0.47cm} \mathbf{x}^{\vee}_7(q) \mathbf{x}^{\vee}_6(r) \underline{\big[ \mathbf{x}^{\vee}_4(p) \mathbf{x}^{\vee}_5(s) \mathbf{x}^{\vee}_4(-p) \mathbf{x}^{\vee}_5(-s) \big]}  \\
        &\hspace{0.47cm} \mathbf{x}^{\vee}_7(t) \mathbf{x}^{\vee}_6(u) \underline{\left[ \big[ \mathbf{x}^{\vee}_3(p) \mathbf{x}^{\vee}_4(s) \mathbf{x}^{\vee}_3(-p) \mathbf{x}^{\vee}_4(-s) \big] \mathbf{x}^{\vee}_5(v) \big[ \mathbf{x}^{\vee}_3(p) \mathbf{x}^{\vee}_4(-s) \mathbf{x}^{\vee}_3(-p) \mathbf{x}^{\vee}_4(s) \big] \mathbf{x}^{\vee}_5(-v) \right]} \\
        &\hspace{0.47cm} \mathbf{x}^{\vee}_7(w) \mathbf{x}^{\vee}_6(x)
    \end{aligned}
    $$

\begin{figure}[ht]
\centering
\begin{tikzpicture}[scale=0.7]
    % pseudolines
        \draw[thick] (0,8) -- (20,8);
        \draw (0,7) -- (20,7);
        \draw[thick] (0,6) -- (20,6);
        \draw[thick] (0,5) -- (20,5);
        \draw (0,4) -- (20,4);
        \draw (0,3) -- (20,3);
        \draw[thick] (0,2) -- (20,2);
        \draw (0,1) -- (20,1);
        % first set
            \draw (0.5,7) -- (0.8,8);
            \draw (1,6) -- (1.3,7);
            \draw (1.5,5) -- (1.8,6);
            \draw (2,4) -- (2.3,5);
            \draw (2.5,3) -- (2.8,4);
            \draw (3,2) -- (3.3,3);
        \draw[dashed, color=black!60] (3.4,0.5) -- (3.4,8.5);
        \draw[densely dotted, color=black!90] (2.9,1.9) -- (2.9,3.1) -- (3.4,3.1) -- (3.4,1.9) -- cycle;
        % second set
            \draw (3.5,7) -- (3.8,8);
            \draw (4,6) -- (4.3,7);
            \draw (4.5,5) -- (4.8,6);
            \draw (5,4) -- (5.3,5);
            \draw (5.5,3) -- (5.8,4);

            \draw (6,1) -- (6.3,2);
            \draw (6.5,2) -- (6.8,3);
            \draw (7,1) -- (7.3,2);
            \draw (7.5,2) -- (7.8,3);
        \draw[dashed, color=black!60] (7.9,0.5) -- (7.9,8.5);
        \draw[densely dotted, color=black!90] (5.9,0.9) -- (5.9,3.1) -- (7.9,3.1) -- (7.9,0.9) -- cycle;
        % third set
            \draw (8,7) -- (8.3,8);
            \draw (8.5,6) -- (8.8,7);
            \draw (9,5) -- (9.3,6);
        \draw[dashed, color=black!60] (9.4,0.5) -- (9.4,8.5);
        \draw[densely dotted, color=black!90] (8.9,4.9) -- (8.9,6.1) -- (9.4,6.1) -- (9.4,4.9) -- cycle;
        % fourth set
            \draw (9.5,7) -- (9.8,8);
            \draw (10,6) -- (10.3,7);

            \draw (10.5,4) -- (10.8,5);
            \draw (11,5) -- (11.3,6);
            \draw (11.5,4) -- (11.8,5);
            \draw (12,5) -- (12.3,6);
        \draw[dashed, color=black!60] (12.4,0.5) -- (12.4,8.5);
        \draw[densely dotted, color=black!90] (10.4,3.9) -- (10.4,6.1) -- (12.4,6.1) -- (12.4,3.9) -- cycle;
        % fifth set
            \draw (12.5,7) -- (12.8,8);
            \draw (13,6) -- (13.3,7);

            \draw (13.5,3) -- (13.8,4);
            \draw (14,4) -- (14.3,5);
            \draw (14.5,3) -- (14.8,4);
            \draw (15,4) -- (15.3,5);
            \draw (15.5,5) -- (15.8,6);
            \draw (16,3) -- (16.3,4);
            \draw (16.5,4) -- (16.8,5);
            \draw (17,3) -- (17.3,4);
            \draw (17.5,4) -- (17.8,5);
            \draw (18,5) -- (18.3,6);
        \draw[dashed, color=black!60] (18.4,0.5) -- (18.4,8.5);
        \draw[densely dotted, color=black!90] (13.4,2.9) -- (13.4,6.1) -- (18.4,6.1) -- (18.4,2.9) -- cycle;
        % sixth set
            \draw (18.5,7) -- (18.8,8);
            \draw (19,6) -- (19.3,7);
        \draw[densely dotted, color=black!90] (18.9,5.9) -- (18.9,7.1) -- (19.4,7.1) -- (19.4,5.9) -- cycle;

    % below
        \draw[decorate, decoration={brace, mirror}, color=black!60] (0.48,0.35) -- (7.82,0.35) node[midway, yshift=-0.4cm]{$k_1$};
        \draw[decorate, decoration={brace, mirror}, color=black!60] (7.98,0.35) -- (18.32,0.35) node[midway, yshift=-0.4cm]{$k_2$};
        \draw[decorate, decoration={brace, mirror}, color=black!60] (18.48,0.35) -- (19.32,0.35) node[midway, yshift=-0.4cm]{$k_3$};

    % dots
        % first set
        \node at (0.5,7) {\scriptsize{$\bullet$}};
        \node at (1,6) {\scriptsize{$\bullet$}};
        \node at (1.5,5) {\scriptsize{$\bullet$}};
        \node at (2,4) {\scriptsize{$\bullet$}};
        \node at (2.5,3) {\scriptsize{$\bullet$}};
        \node at (3,2) {\scriptsize{$\bullet$}};

        \node at (0.8,8) {\scriptsize{$\bullet$}};
        \node at (1.3,7) {\scriptsize{$\bullet$}};
        \node at (1.8,6) {\scriptsize{$\bullet$}};
        \node at (2.3,5) {\scriptsize{$\bullet$}};
        \node at (2.8,4) {\scriptsize{$\bullet$}};
        \node at (3.3,3) {\scriptsize{$\bullet$}};
        % second set
        \node at (3.5,7) {\scriptsize{$\bullet$}};
        \node at (4,6) {\scriptsize{$\bullet$}};
        \node at (4.5,5) {\scriptsize{$\bullet$}};
        \node at (5,4) {\scriptsize{$\bullet$}};
        \node at (5.5,3) {\scriptsize{$\bullet$}};
        \node at (6,1) {\scriptsize{$\bullet$}};
        \node at (6.5,2) {\scriptsize{$\bullet$}};
        \node at (7,1) {\scriptsize{$\bullet$}};
        \node at (7.5,2) {\scriptsize{$\bullet$}};

        \node at (3.8,8) {\scriptsize{$\bullet$}};
        \node at (4.3,7) {\scriptsize{$\bullet$}};
        \node at (4.8,6) {\scriptsize{$\bullet$}};
        \node at (5.3,5) {\scriptsize{$\bullet$}};
        \node at (5.8,4) {\scriptsize{$\bullet$}};
        \node at (6.3,2) {\scriptsize{$\bullet$}};
        \node at (6.8,3) {\scriptsize{$\bullet$}};
        \node at (7.3,2) {\scriptsize{$\bullet$}};
        \node at (7.8,3) {\scriptsize{$\bullet$}};
        % third set
        \node at (8,7) {\scriptsize{$\bullet$}};
        \node at (8.5,6) {\scriptsize{$\bullet$}};
        \node at (9,5) {\scriptsize{$\bullet$}};

        \node at (8.3,8) {\scriptsize{$\bullet$}};
        \node at (8.8,7) {\scriptsize{$\bullet$}};
        \node at (9.3,6) {\scriptsize{$\bullet$}};
        % fourth set
        \node at (9.5,7) {\scriptsize{$\bullet$}};
        \node at (10,6) {\scriptsize{$\bullet$}};
        \node at (10.5,4) {\scriptsize{$\bullet$}};
        \node at (11,5) {\scriptsize{$\bullet$}};
        \node at (11.5,4) {\scriptsize{$\bullet$}};
        \node at (12,5) {\scriptsize{$\bullet$}};

        \node at (9.8,8) {\scriptsize{$\bullet$}};
        \node at (10.3,7) {\scriptsize{$\bullet$}};
        \node at (10.8,5) {\scriptsize{$\bullet$}};
        \node at (11.3,6) {\scriptsize{$\bullet$}};
        \node at (11.8,5) {\scriptsize{$\bullet$}};
        \node at (12.3,6) {\scriptsize{$\bullet$}};
        % fifth set
        \node at (12.5,7) {\scriptsize{$\bullet$}};
        \node at (13,6) {\scriptsize{$\bullet$}};
        \node at (13.5,3) {\scriptsize{$\bullet$}};
        \node at (14,4) {\scriptsize{$\bullet$}};
        \node at (14.5,3) {\scriptsize{$\bullet$}};
        \node at (15,4) {\scriptsize{$\bullet$}};
        \node at (15.5,5) {\scriptsize{$\bullet$}};
        \node at (16,3) {\scriptsize{$\bullet$}};
        \node at (16.5,4) {\scriptsize{$\bullet$}};
        \node at (17,3) {\scriptsize{$\bullet$}};
        \node at (17.5,4) {\scriptsize{$\bullet$}};
        \node at (18,5) {\scriptsize{$\bullet$}};

        \node at (12.8,8) {\scriptsize{$\bullet$}};
        \node at (13.3,7) {\scriptsize{$\bullet$}};
        \node at (13.8,4) {\scriptsize{$\bullet$}};
        \node at (14.3,5) {\scriptsize{$\bullet$}};
        \node at (14.8,4) {\scriptsize{$\bullet$}};
        \node at (15.3,5) {\scriptsize{$\bullet$}};
        \node at (15.8,6) {\scriptsize{$\bullet$}};
        \node at (16.3,4) {\scriptsize{$\bullet$}};
        \node at (16.8,5) {\scriptsize{$\bullet$}};
        \node at (17.3,4) {\scriptsize{$\bullet$}};
        \node at (17.8,5) {\scriptsize{$\bullet$}};
        \node at (18.3,6) {\scriptsize{$\bullet$}};
        % sixth set
        \node at (18.5,7) {\scriptsize{$\bullet$}};
        \node at (19,6) {\scriptsize{$\bullet$}};

        \node at (18.8,8) {\scriptsize{$\bullet$}};
        \node at (19.3,7) {\scriptsize{$\bullet$}};

        \node at (0,1) {\scriptsize{$\bullet$}};
        \node at (0,2) {\scriptsize{$\bullet$}};
        \node at (0,3) {\scriptsize{$\bullet$}};
        \node at (0,4) {\scriptsize{$\bullet$}};
        \node at (0,5) {\scriptsize{$\bullet$}};
        \node at (0,6) {\scriptsize{$\bullet$}};
        \node at (0,7) {\scriptsize{$\bullet$}};
        \node at (0,8) {\scriptsize{$\bullet$}};

        \node at (20,1) {\scriptsize{$\bullet$}};
        \node at (20,2) {\scriptsize{$\bullet$}};
        \node at (20,3) {\scriptsize{$\bullet$}};
        \node at (20,4) {\scriptsize{$\bullet$}};
        \node at (20,5) {\scriptsize{$\bullet$}};
        \node at (20,6) {\scriptsize{$\bullet$}};
        \node at (20,7) {\scriptsize{$\bullet$}};
        \node at (20,8) {\scriptsize{$\bullet$}};

    % pseudoline labels
        \node at (-0.3,8) {$8$};
        \node at (-0.3,7) {$7$};
        \node at (-0.3,6) {$6$};
        \node at (-0.3,5) {$5$};
        \node at (-0.3,4) {$4$};
        \node at (-0.3,3) {$3$};
        \node at (-0.3,2) {$2$};
        \node at (-0.3,1) {$1$};

        \node at (20.3,8) {$8$};
        \node at (20.3,7) {$7$};
        \node at (20.3,6) {$6$};
        \node at (20.3,5) {$5$};
        \node at (20.3,4) {$4$};
        \node at (20.3,3) {$3$};
        \node at (20.3,2) {$2$};
        \node at (20.3,1) {$1$};

    % weight labels
        % first set
        \node at (0.85,7.5) {$a$}; % +(0.05,0) from end of diagonal line at top
        \node at (1.35,6.5) {$b$};
        \node at (1.85,5.5) {$c$};
        \node at (2.35,4.5) {$d$};
        \node at (2.85,3.5) {$e$};
        \node at (3.35,2.5) {$f$};
        % second set
        \node at (3.85,7.5) {$g$};
        \node at (4.35,6.5) {$h$};
        \node at (4.85,5.5) {$i$};
        \node at (5.35,4.5) {$j$};
        \node at (5.85,3.5) {$k$};

        \node at (6.35,1.5) {$f$};
        \node at (6.85,2.5) {$l$};
        \node at (7.54,1.5) {$-f$};
        \node at (7.95,2.5) {$-l$};
        % third set
        \node at (8.42,7.5) {$m$};
        \node at (8.85,6.5) {$n$};
        \node at (9.35,5.5) {$p$};
        % fourth set
        \node at (9.85,7.5) {$q$};
        \node at (10.35,6.5) {$r$};

        \node at (10.85,4.5) {$p$};
        \node at (11.35,5.5) {$s$};
        \node at (12,4.5) {$-p$};
        \node at (12.5,5.5) {$-s$};
        % fifth set
        \node at (12.85,7.5) {$t$};
        \node at (13.35,6.5) {$u$};

        \node at (13.85,3.5) {$p$};
        \node at (14.35,4.5) {$s$};
        \node at (15,3.5) {$-p$};
        \node at (15.5,4.5) {$-s$};
        \node at (15.85,5.5) {$v$};
        \node at (16.35,3.5) {$p$};
        \node at (17,4.5) {$-s$};
        \node at (17.5,3.5) {$-p$};
        \node at (17.85,4.5) {$s$};
        \node at (18.5,5.5) {$-v$};
        % sixth set
        \node at (18.9,7.5) {$w$};
        \node at (19.35,6.5) {$x$};
\end{tikzpicture}
\caption{Graph (with temporary labelling) for computing minors of $u_L$ for $\mathcal{F}_{2,5,6}(\mathbb{C}^8)$} \label{fig temp labelling of graph for uL for F2,5,6,C8}
\end{figure}
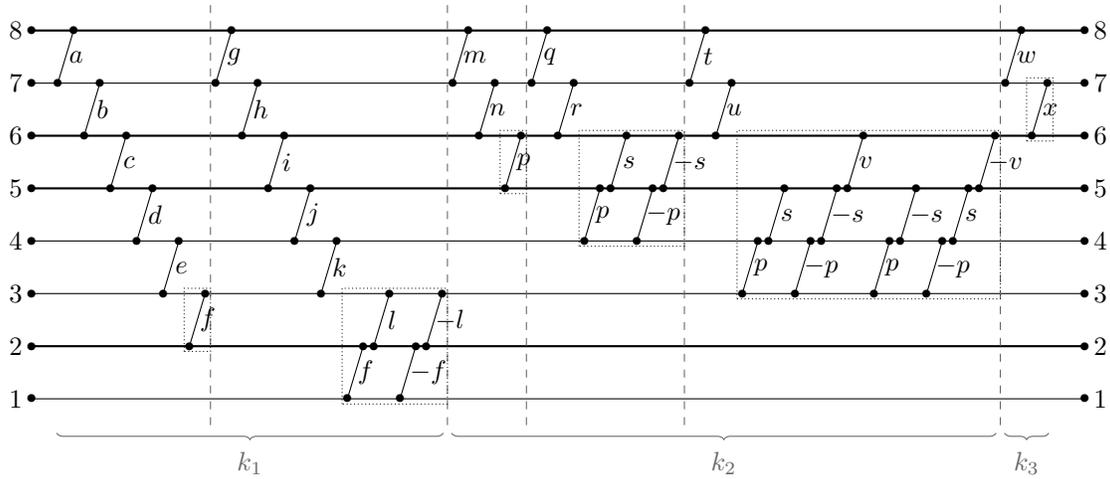

\end{ex}

\begin{lem} \label{lem minors in CA are monomial}
All minors in the application of the Chamber Ansatz (in the proof of Lemma \ref{lem b- in terms of y_i}) are monomial and consequently the resulting coordinate change is monomial.
\end{lem}

\begin{proof}
In the application of the Chamber Ansatz in the proof of Lemma \ref{lem b- in terms of y_i}, the relevant minors are those flag minors of $u_L$ with column sets given by the chamber labels of the respective ansatz arrangement for $\dot{w}_P\bar{w}_0$.

Firstly, the minors corresponding to the chamber labels $\{1, \ldots, b\}$ always take the value $1$ since $u_L \in U^{\vee}_+$. We can also see this from the graph using the result of Fomin and Zelevinsky which describes matrix minors in terms of weights of paths (\cite[Proposition 4.2]{FominZelevinsky1999}, given in our notation as Theorem \ref{thm minors from paths in graph}). Namely, the paths must be horizontal and the lack of non-trivial torus factors means that all horizontal edges have weight $1$, thus any horizontal path in this graph necessarily has weight $1$.

In order to use the graph to show that the remaining minors are monomial, we need to consider the dotted rectangles and the columns separated by dashed lines.

\begin{claim}
In each dotted rectangle, the sum of the weights of the family of paths from the bottom left hand corner to the top right hand corner is equal to the weight of the first possible such path. Moreover, the sum of the weights of the family of paths passing through dotted rectangles between any two distinct horizontal lines other than the bottom and top lines, is zero. % Thus for non-zero contributions we are forced to always rise before coming to a dotted rectangle, that is, we only ever enter a dotted rectangle from the bottom left corner. ?**
\end{claim}

\begin{proof}[Proof of Claim]
We begin by recalling that each dotted rectangle corresponds to a factor $X^{\vee}_{i,\alpha}$ in $u_L$.
The only minor of the matrix $X^{\vee}_{i,\alpha}(x_{v_t}/x_{v_s})=\tilde{X}^{\vee}_{i,\alpha}(r_{a_1}, \ldots, r_{a_{\alpha}})$ which has the possibility to take a value other than $0$ or $1$ is $\Delta^{i-\alpha+1}_{i+1}$. We see this because any submatrix other than the $1\times 1$ submatrix given by the entry in position $(i+1-\alpha, i+1)$, either has a zero row or column, or is triangular with $1$'s on the leading diagonal. This special minor is given by
    \begin{equation} \label{eqn minor re dotted rectangle i,alpha}
    \Delta^{i-\alpha+1}_{i+1}\left( \tilde{X}^{\vee}_{i,\alpha}(r_{a_1}, \ldots, r_{a_{\alpha}}) \right) = \prod_{j=1}^{\alpha} r_{a_j}.
    \end{equation}

Now we consider the graph corresponding to the matrix $X^{\vee}_{i,\alpha}(x_{v_t}/x_{v_s})=\tilde{X}^{\vee}_{i,\alpha}(r_{a_1}, \ldots, r_{a_{\alpha}})$, namely the subgraph within the respective dotted rectangle. By the result of Fomin and Zelevinsky (\cite[Proposition 4.2]{FominZelevinsky1999}, given in our notation as Theorem \ref{thm minors from paths in graph}) we see that the sum of the weights of the families of paths (in this dotted rectangle) between any two horizontal lines $l_1 \leq l_2 \in \{i+1-\alpha, \ldots, i+1\}$ is equal to the minor
    $$\Delta^{l_1}_{l_2} = \begin{cases}
    \prod_{j=1}^{\alpha} r_{a_j} & \text{if } (l_1, l_2) = (i+1-\alpha, i+1), \\
    1 & \text{if } l_1=l_2, \\
    0 & \text{otherwise}.
    \end{cases}
    $$
Moreover, for $j=1, \ldots, \alpha$ the first diagonal step from line $i+1-j$ to $i+1-j+1$ has weight $r_{a_{\alpha-j+1}}$ by construction. Thus the weight of the path in the dotted rectangle which travels upwards at the first possible opportunity is exactly equal to the minor (\ref{eqn minor re dotted rectangle i,alpha}). See Figure \ref{fig Subgraph describing paths in doted rectangles} for an example of such a path with $\alpha=3$.
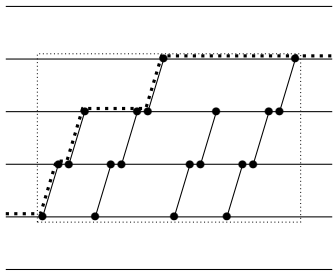
\begin{figure}[ht!]
\centering
\begin{tikzpicture}[scale=0.7]
    % pseudolines
        \draw (12.8,7) -- (19,7);
        \draw (12.8,6) -- (19,6);
        \draw (12.8,5) -- (19,5);
        \draw (12.8,4) -- (19,4);
        \draw (12.8,3) -- (19,3);
        \draw (12.8,2) -- (19,2);
        % steps upwards
            \draw (13.5,3) -- (13.8,4);
            \draw (14,4) -- (14.3,5);
            \draw (14.5,3) -- (14.8,4);
            \draw (15,4) -- (15.3,5);
            \draw (15.5,5) -- (15.8,6);
            \draw (16,3) -- (16.3,4);
            \draw (16.5,4) -- (16.8,5);
            \draw (17,3) -- (17.3,4);
            \draw (17.5,4) -- (17.8,5);
            \draw (18,5) -- (18.3,6);
        \draw[densely dotted, color=black!90] (13.4,2.9) -- (13.4,6.1) -- (18.4,6.1) -- (18.4,2.9) -- cycle;

    % path
        \draw[very thick, dotted] (12.8,3.06) -- (13.46,3.06) -- (13.76,4.06) -- (13.96,4.06) -- (14.26,5.06) -- (15.46,5.06) -- (15.76,6.06) -- (19,6.06);

    % dots
        \node at (13.5,3) {\scriptsize{$\bullet$}};
        \node at (14,4) {\scriptsize{$\bullet$}};
        \node at (14.5,3) {\scriptsize{$\bullet$}};
        \node at (15,4) {\scriptsize{$\bullet$}};
        \node at (15.5,5) {\scriptsize{$\bullet$}};
        \node at (16,3) {\scriptsize{$\bullet$}};
        \node at (16.5,4) {\scriptsize{$\bullet$}};
        \node at (17,3) {\scriptsize{$\bullet$}};
        \node at (17.5,4) {\scriptsize{$\bullet$}};
        \node at (18,5) {\scriptsize{$\bullet$}};

        \node at (13.8,4) {\scriptsize{$\bullet$}};
        \node at (14.3,5) {\scriptsize{$\bullet$}};
        \node at (14.8,4) {\scriptsize{$\bullet$}};
        \node at (15.3,5) {\scriptsize{$\bullet$}};
        \node at (15.8,6) {\scriptsize{$\bullet$}};
        \node at (16.3,4) {\scriptsize{$\bullet$}};
        \node at (16.8,5) {\scriptsize{$\bullet$}};
        \node at (17.3,4) {\scriptsize{$\bullet$}};
        \node at (17.8,5) {\scriptsize{$\bullet$}};
        \node at (18.3,6) {\scriptsize{$\bullet$}};
\end{tikzpicture}
\caption{Subgraph describing paths in doted rectangles} \label{fig Subgraph describing paths in doted rectangles}
\end{figure}

\end{proof}

A corollary of the claim is that if we are tracing a path and come to a dotted rectangle then we only need to consider two options; either we travel from the bottom left corner to the top right corner taking every opportunity to travel upwards, or the path stays horizontal through the rectangle. That is, the sum of all other contributions to the computation of the respective minor will be zero. Consequently we will often say \emph{the} path upwards through a rectangle, since we lose no information by not considering the full family of paths between the bottom and top lines.

Now by construction of $u_L$, for each $j=1, \ldots, l$ and $\alpha=1, \ldots, k_j$, we have a single dotted rectangle corresponding to $X^{\vee}_{n_j,\alpha}$, appearing in this order in the graph, namely
    $$X^{\vee}_{n_1,1}, \ X^{\vee}_{n_1,2}, \ \ldots, \ X^{\vee}_{n_1,k_1}, \ X^{\vee}_{n_2,1}, \ \ldots, \ X^{\vee}_{n_2,k_2}, \ \ldots, \ X^{\vee}_{n_l,1}, \ \ldots, \ X^{\vee}_{n_l,k_l}.
    $$
Thus we have the following three facts:
    \begin{enumerate}
        \item There is exactly one dotted rectangle beginning on each horizontal line $n_j+1-\alpha$ for $j=1, \ldots, l$ and $\alpha=1, \ldots, k_j$.
        \item If we travel along a given horizontal line, after passing horizontally through a dotted rectangle (unless on the top line) the only possible steps upwards are also within (other) dotted rectangles, but by fact 1, these rectangles must be between different pairs of lines.
        \item After the set of columns labelled $k_j$, (that is, after the first $n_j$ columns) there are no steps between the horizontal lines $n_j+1-\alpha$ and $n_j+2-\alpha$ for $\alpha=1, \ldots, k_j$.
    \end{enumerate}

For reference, we recall from Lemma \ref{lem G/P chamber label form} that for $j=1, \ldots, l+1$ and $t=1, \ldots, k_j$, we are considering the flag minors with column sets given by chamber labels of the form
    $$
    \{n_{j-1}+1, \ldots, n_{j}-t\} \cup \{n_{j}+1, \ldots, n_{j-1}+b+t\}
    $$
where $b=1, \ldots, n-n_{j-1}-t$ is the height of the chamber. We will use induction to prove that these minors are monomial by considering paths within the set of columns labelled $k_j$, that is, within the columns $n_{j-1}+1, \ldots, n_j$.

In the rest of this proof we will write `the path from $i$' to mean `the path on the horizontal line $i$ at the beginning of the $(n_{j-1}+1)$-th column', reserving the terms source and sink for the beginning and end of horizontal lines in the complete graph.

Firstly, taking $k_0:=0$, any paths from $i \leq n_{j-1}$ must remain horizontal through the columns $n_{j-1}+1, \ldots, n_j$ by the third fact above.
% We note that if, for some $i \in \{1, \ldots j-2\}$ and $t\in \{1, \ldots, k_i-1\}$, the set of incoming paths at the beginning of the $(n_{j-1}+1)$-th column is given by
%     $$\{n_{i-1}+1, \ldots, n_{i}-t\} \cup \{n_{i}+1, \ldots, n_{i-1}+b+t\},
%     $$
% then the chamber label is given by the union of two disjoint sets, the second of which begins with $n_i+1\leq n_{j-2}+1 \leq n_{j-1}$. Consequently, due to the column sets of the minors we wish to compute, the paths from $n_{i-1}+1,\ldots, n_{i-1}+b+t$ must remain horizontal for the remainder of the graph, regardless of whether or not $n_{i-1}+b+t\geq n_{j-1}$.
% COVERED BY THE INDUCTION

Paths from $n_{j-1}+1$ are either horizontal in which case the minor takes value $1$. Alternatively the paths may travel upwards in column $n_j$ through the dotted rectangle corresponding to $X^{\vee}_{n_j,k_j}$. By the claim, the sum of the weights of the family of paths between these lines is monomial, if non-zero.
Now we recall that the minors we need to consider are (unions of) sets of increasing integers beginning with $n_i+1$ for some $i$. Consequently, since all paths from $i \leq n_{j-1}$ must remain horizontal through the column set $k_j$, we see that the paths from $n_{j-1}+1$ we need to consider either stay horizontal, or travel upwards to the line $n_j+1$, that is travelling upwards at every opportunity through the rectangle. For example, see Figure \ref{fig minor 367 for F2,5,6,C8}.

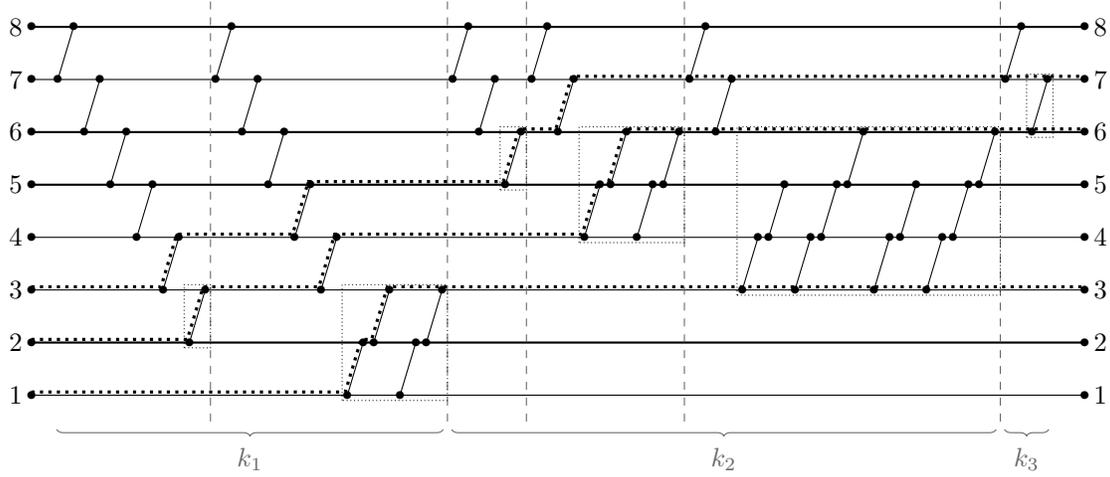
\begin{figure}[ht]
\centering
\begin{tikzpicture}[scale=0.7]
    % pseudolines
        \draw[thick] (0,8) -- (20,8);
        \draw (0,7) -- (20,7);
        \draw[thick] (0,6) -- (20,6);
        \draw[thick] (0,5) -- (20,5);
        \draw (0,4) -- (20,4);
        \draw (0,3) -- (20,3);
        \draw[thick] (0,2) -- (20,2);
        \draw (0,1) -- (20,1);
        % first set
            \draw (0.5,7) -- (0.8,8);
            \draw (1,6) -- (1.3,7);
            \draw (1.5,5) -- (1.8,6);
            \draw (2,4) -- (2.3,5);
            \draw (2.5,3) -- (2.8,4);
            \draw (3,2) -- (3.3,3);
        \draw[dashed, color=black!60] (3.4,0.5) -- (3.4,8.5);
        \draw[densely dotted, color=black!90] (2.9,1.9) -- (2.9,3.1) -- (3.4,3.1) -- (3.4,1.9) -- cycle;

        % second set
            \draw (3.5,7) -- (3.8,8);
            \draw (4,6) -- (4.3,7);
            \draw (4.5,5) -- (4.8,6);
            \draw (5,4) -- (5.3,5);
            \draw (5.5,3) -- (5.8,4);

            \draw (6,1) -- (6.3,2);
            \draw (6.5,2) -- (6.8,3);
            \draw (7,1) -- (7.3,2);
            \draw (7.5,2) -- (7.8,3);
        \draw[dashed, color=black!60] (7.9,0.5) -- (7.9,8.5);
        \draw[densely dotted, color=black!90] (5.9,0.9) -- (5.9,3.1) -- (7.9,3.1) -- (7.9,0.9) -- cycle;
        % third set
            \draw (8,7) -- (8.3,8);
            \draw (8.5,6) -- (8.8,7);
            \draw (9,5) -- (9.3,6);
        \draw[dashed, color=black!60] (9.4,0.5) -- (9.4,8.5);
        \draw[densely dotted, color=black!90] (8.9,4.9) -- (8.9,6.1) -- (9.4,6.1) -- (9.4,4.9) -- cycle;
        % fourth set
            \draw (9.5,7) -- (9.8,8);
            \draw (10,6) -- (10.3,7);

            \draw (10.5,4) -- (10.8,5);
            \draw (11,5) -- (11.3,6);
            \draw (11.5,4) -- (11.8,5);
            \draw (12,5) -- (12.3,6);
        \draw[dashed, color=black!60] (12.4,0.5) -- (12.4,8.5);
        \draw[densely dotted, color=black!90] (10.4,3.9) -- (10.4,6.1) -- (12.4,6.1) -- (12.4,3.9) -- cycle;
        % fifth set
            \draw (12.5,7) -- (12.8,8);
            \draw (13,6) -- (13.3,7);

            \draw (13.5,3) -- (13.8,4);
            \draw (14,4) -- (14.3,5);
            \draw (14.5,3) -- (14.8,4);
            \draw (15,4) -- (15.3,5);
            \draw (15.5,5) -- (15.8,6);
            \draw (16,3) -- (16.3,4);
            \draw (16.5,4) -- (16.8,5);
            \draw (17,3) -- (17.3,4);
            \draw (17.5,4) -- (17.8,5);
            \draw (18,5) -- (18.3,6);
        \draw[dashed, color=black!60] (18.4,0.5) -- (18.4,8.5);
        \draw[densely dotted, color=black!90] (13.4,2.9) -- (13.4,6.1) -- (18.4,6.1) -- (18.4,2.9) -- cycle;
        % sixth set
            \draw (18.5,7) -- (18.8,8);
            \draw (19,6) -- (19.3,7);
        \draw[densely dotted, color=black!90] (18.9,5.9) -- (18.9,7.1) -- (19.4,7.1) -- (19.4,5.9) -- cycle;

    % paths
        \draw[very thick, dotted] (0,3.06) -- (2.46,3.06) -- (2.76,4.06) -- (4.96,4.06) -- (5.26,5.06) -- (8.96,5.06) -- (9.26,6.06) -- (9.96,6.06) -- (10.26,7.06) -- (20,7.06);
        \draw[very thick, dotted] (0,2.06) -- (2.96,2.06) -- (3.26,3.06) -- (5.46,3.06) -- (5.76,4.06) -- (10.46,4.06) -- (10.76,5.06) (10.96,5.06) -- (11.26,6.06) -- (20,6.06);
        \draw[very thick, dotted] (0,1.06) -- (5.96,1.06) -- (6.26,2.06) -- (6.46,2.06) -- (6.76,3.06) -- (20,3.06);

    % below
        \draw[decorate, decoration={brace, mirror}, color=black!60] (0.48,0.35) -- (7.82,0.35) node[midway, yshift=-0.4cm]{$k_1$};
        \draw[decorate, decoration={brace, mirror}, color=black!60] (7.98,0.35) -- (18.32,0.35) node[midway, yshift=-0.4cm]{$k_2$};
        \draw[decorate, decoration={brace, mirror}, color=black!60] (18.48,0.35) -- (19.32,0.35) node[midway, yshift=-0.4cm]{$k_3$};

    % dots
        % first set
        \node at (0.5,7) {\scriptsize{$\bullet$}};
        \node at (1,6) {\scriptsize{$\bullet$}};
        \node at (1.5,5) {\scriptsize{$\bullet$}};
        \node at (2,4) {\scriptsize{$\bullet$}};
        \node at (2.5,3) {\scriptsize{$\bullet$}};
        \node at (3,2) {\scriptsize{$\bullet$}};

        \node at (0.8,8) {\scriptsize{$\bullet$}};
        \node at (1.3,7) {\scriptsize{$\bullet$}};
        \node at (1.8,6) {\scriptsize{$\bullet$}};
        \node at (2.3,5) {\scriptsize{$\bullet$}};
        \node at (2.8,4) {\scriptsize{$\bullet$}};
        \node at (3.3,3) {\scriptsize{$\bullet$}};
        % second set
        \node at (3.5,7) {\scriptsize{$\bullet$}};
        \node at (4,6) {\scriptsize{$\bullet$}};
        \node at (4.5,5) {\scriptsize{$\bullet$}};
        \node at (5,4) {\scriptsize{$\bullet$}};
        \node at (5.5,3) {\scriptsize{$\bullet$}};
        \node at (6,1) {\scriptsize{$\bullet$}};
        \node at (6.5,2) {\scriptsize{$\bullet$}};
        \node at (7,1) {\scriptsize{$\bullet$}};
        \node at (7.5,2) {\scriptsize{$\bullet$}};

        \node at (3.8,8) {\scriptsize{$\bullet$}};
        \node at (4.3,7) {\scriptsize{$\bullet$}};
        \node at (4.8,6) {\scriptsize{$\bullet$}};
        \node at (5.3,5) {\scriptsize{$\bullet$}};
        \node at (5.8,4) {\scriptsize{$\bullet$}};
        \node at (6.3,2) {\scriptsize{$\bullet$}};
        \node at (6.8,3) {\scriptsize{$\bullet$}};
        \node at (7.3,2) {\scriptsize{$\bullet$}};
        \node at (7.8,3) {\scriptsize{$\bullet$}};
        % third set
        \node at (8,7) {\scriptsize{$\bullet$}};
        \node at (8.5,6) {\scriptsize{$\bullet$}};
        \node at (9,5) {\scriptsize{$\bullet$}};

        \node at (8.3,8) {\scriptsize{$\bullet$}};
        \node at (8.8,7) {\scriptsize{$\bullet$}};
        \node at (9.3,6) {\scriptsize{$\bullet$}};
        % fourth set
        \node at (9.5,7) {\scriptsize{$\bullet$}};
        \node at (10,6) {\scriptsize{$\bullet$}};
        \node at (10.5,4) {\scriptsize{$\bullet$}};
        \node at (11,5) {\scriptsize{$\bullet$}};
        \node at (11.5,4) {\scriptsize{$\bullet$}};
        \node at (12,5) {\scriptsize{$\bullet$}};

        \node at (9.8,8) {\scriptsize{$\bullet$}};
        \node at (10.3,7) {\scriptsize{$\bullet$}};
        \node at (10.8,5) {\scriptsize{$\bullet$}};
        \node at (11.3,6) {\scriptsize{$\bullet$}};
        \node at (11.8,5) {\scriptsize{$\bullet$}};
        \node at (12.3,6) {\scriptsize{$\bullet$}};
        % fifth set
        \node at (12.5,7) {\scriptsize{$\bullet$}};
        \node at (13,6) {\scriptsize{$\bullet$}};
        \node at (13.5,3) {\scriptsize{$\bullet$}};
        \node at (14,4) {\scriptsize{$\bullet$}};
        \node at (14.5,3) {\scriptsize{$\bullet$}};
        \node at (15,4) {\scriptsize{$\bullet$}};
        \node at (15.5,5) {\scriptsize{$\bullet$}};
        \node at (16,3) {\scriptsize{$\bullet$}};
        \node at (16.5,4) {\scriptsize{$\bullet$}};
        \node at (17,3) {\scriptsize{$\bullet$}};
        \node at (17.5,4) {\scriptsize{$\bullet$}};
        \node at (18,5) {\scriptsize{$\bullet$}};

        \node at (12.8,8) {\scriptsize{$\bullet$}};
        \node at (13.3,7) {\scriptsize{$\bullet$}};
        \node at (13.8,4) {\scriptsize{$\bullet$}};
        \node at (14.3,5) {\scriptsize{$\bullet$}};
        \node at (14.8,4) {\scriptsize{$\bullet$}};
        \node at (15.3,5) {\scriptsize{$\bullet$}};
        \node at (15.8,6) {\scriptsize{$\bullet$}};
        \node at (16.3,4) {\scriptsize{$\bullet$}};
        \node at (16.8,5) {\scriptsize{$\bullet$}};
        \node at (17.3,4) {\scriptsize{$\bullet$}};
        \node at (17.8,5) {\scriptsize{$\bullet$}};
        \node at (18.3,6) {\scriptsize{$\bullet$}};
        % sixth set
        \node at (18.5,7) {\scriptsize{$\bullet$}};
        \node at (19,6) {\scriptsize{$\bullet$}};

        \node at (18.8,8) {\scriptsize{$\bullet$}};
        \node at (19.3,7) {\scriptsize{$\bullet$}};

        \node at (0,1) {\scriptsize{$\bullet$}};
        \node at (0,2) {\scriptsize{$\bullet$}};
        \node at (0,3) {\scriptsize{$\bullet$}};
        \node at (0,4) {\scriptsize{$\bullet$}};
        \node at (0,5) {\scriptsize{$\bullet$}};
        \node at (0,6) {\scriptsize{$\bullet$}};
        \node at (0,7) {\scriptsize{$\bullet$}};
        \node at (0,8) {\scriptsize{$\bullet$}};

        \node at (20,1) {\scriptsize{$\bullet$}};
        \node at (20,2) {\scriptsize{$\bullet$}};
        \node at (20,3) {\scriptsize{$\bullet$}};
        \node at (20,4) {\scriptsize{$\bullet$}};
        \node at (20,5) {\scriptsize{$\bullet$}};
        \node at (20,6) {\scriptsize{$\bullet$}};
        \node at (20,7) {\scriptsize{$\bullet$}};
        \node at (20,8) {\scriptsize{$\bullet$}};

    % pseudoline labels
        \node at (-0.3,8) {$8$};
        \node at (-0.3,7) {$7$};
        \node at (-0.3,6) {$6$};
        \node at (-0.3,5) {$5$};
        \node at (-0.3,4) {$4$};
        \node at (-0.3,3) {$3$};
        \node at (-0.3,2) {$2$};
        \node at (-0.3,1) {$1$};

        \node at (20.3,8) {$8$};
        \node at (20.3,7) {$7$};
        \node at (20.3,6) {$6$};
        \node at (20.3,5) {$5$};
        \node at (20.3,4) {$4$};
        \node at (20.3,3) {$3$};
        \node at (20.3,2) {$2$};
        \node at (20.3,1) {$1$};
\end{tikzpicture}
\caption{Graph describing the minor $\Delta^{\{1,2,3\}}_{\{3\}\cup\{6,7\}}(u_L)$ for $\mathcal{F}_{2,5,6}(\mathbb{C}^8)$} \label{fig minor 367 for F2,5,6,C8}
\end{figure}

Now suppose we have chosen some path from $n_{j-1}+1$, either horizontal or the one travelling upwards through the rectangle corresponding to $X^{\vee}_{n_j,k_j}$. We will consider the possible paths from $n_{j-1}+2$. If our path from $n_{j-1}+1$ travels upwards, then in order for our paths to remain vertex disjoint, any path from $n_{j-1}+2$ must rise before the path from $n_{j-1}+1$ does. Thus it must first rise to line $n_j+1$ (if $k_j\geq 2$) by travelling through the dotted rectangle in the $(n_j-1)$-th column, corresponding to $X^{\vee}_{n_j,k_j-1}$, and then up the unique step between lines $n_j+1$, $n_j+2$ in the $n_j$-th column. By the claim we see that the corresponding contribution to the minor is monomial.

If the path from $n_{j-1}+1$ is horizontal then (if $k_j\geq 2$) any path from $n_{j-1}+2$ may either stay horizontal or rise through the dotted rectangle corresponding to $X^{\vee}_{n_j,k_j-1}$. If it stays horizontal, then by the second fact above it will come to another dotted rectangle, in particular the one corresponding to $X^{\vee}_{n_j,k_j}$. By the claim, we see that in this case the sum of the weights of the family of paths from $n_{j-1}+2$, is non-zero only if the paths in the family stay horizontal during this second rectangle.

If instead the path from $n_{j-1}+2$ rises through the first rectangle, then, by the claim, the sum of the weights of the family of paths between these lines is monomial, if non-zero. Thus we may assume that we rise to the top of the rectangle, that is, to line $n_j+1$. From this point there is a unique step from line $n_j+1$ to $n_j+2$ in column $n_j$ however the chamber labels we are considering force us to remain horizontal. We see this from the third fact above (since the path from $n_{j-1}+1$, if horizontal through the set of columns labelled $k_j$, will have no opportunity to rise at a later point in the graph) combined with the form of the minors we need to compute, noting that we are in the case of a label given by a union of two disjoint sets. Additionally, since the line from $n_{j-1}+1$ being horizontal through the set of columns labelled $k_j$ corresponds to the appearance of $n_{j-1}+1$ in the chamber label, and the line from $n_{j-1}+2$ rises, the first term of the second set in the chamber label must be $n_j+1$. Thus once the path from $n_{j-1}+2$ has risen to the horizontal line $n_j+1$, it must remain on this line for the remainder of the graph

We may repeat a similar argument for any paths from $n_{j-1}+3, \ldots, n$. The result is the following; within the column set labelled $k_j$, either
    \begin{enumerate}
        \item each path is horizontal and so the corresponding contribution to the minor is equal to $1$, or
        \item each path takes every opportunity to rise, either from bottom to top through rectangles (giving a monomial contribution to the respective minor by the claim) or up a sequence of `single' steps, until it reaches either the horizontal line for the desired sink of the path or the end of the column set labelled $k_j$, and so the corresponding contribution to the minor is monomial.
    \end{enumerate}

We see that all the minors in our application of the Chamber Ansatz are monomial from three facts. Firstly that if a path is on a horizontal line $i\leq k_{j-1}$ at the start of the $(n_{j-1}+1)$-th column then it must remain horizontal, secondly that our minors are flag minors (so we always begin paths at sources $1, \ldots, i$) and thirdly that there is a monomial (if non-zero) contribution to the minor from the family of paths between any two distinct lines in each column set $k_j$, $j=1, \ldots, l$.

Putting these facts together we see that for a given minor, taking the following paths for $i=1, \ldots, n$ evaluates the minor from the graph and that it must be monomial: from source $i$, we travel upwards at every opportunity through each column set $k_j$ in succession, until we reach the horizontal line which leads into the desired sink, defined by the $i$-th integer in the respective chamber label.

\end{proof}

Let $j \in \{1, \ldots, l+1\}$ and $n_{j-1}\leq k \leq n_j$, then by the proof of Lemma \ref{lem G/P chamber label form} the label of the chamber to the \emph{right} of the crossing at height $a$ from the $k$-th sub-product of the form (\ref{eqn subproduct 1 of si for ansatz arr proof}) (in the description of $\dot{w}_P\bar{w}_0$ given in (\ref{eqn wPw0 increasing si})) is given by
    \begin{equation} \label{eqn orginial chamber label in terms of k,a}
    \{ n_{j-1}+1, \ldots, n_{j-1}+n_j-k \} \cup \{ n_j+1, \ldots, k+a \}.
    \end{equation}
We see this is exactly (\ref{eqn chamber label below string}) from the statement of Lemma \ref{lem G/P chamber label form}, by setting $a=b$, $k=n_{j-1}+t$.

It is then natural to assign the pair $(k,a)$ to the chamber on the right of this crossing, similar to the proof of Lemma \ref{lem coord change for b m_i in terms of p_i}. We will call these \emph{chamber pairs} to distinguish them from the chamber labels used previously, for example in Lemma \ref{lem G/P chamber label form}. The leftmost chambers are labelled consistently, taking $k=0$, and we leave the chambers above and below the pseudoline arrangement unlabelled. In Figure \ref{fig chamber pair labels wPw0 ansatz arrangement} we show the chamber pairs $(k,a)$ surrounding the crossing at height $a$ from the $k$-th sub-product of the form (\ref{eqn subproduct 1 of si for ansatz arr proof}).

\begin{figure}[ht]
\centering
\begin{tikzpicture}
    %dots and lines
        \node at (1,0.5) {$\bullet$};
        \draw[<-, >=stealth, densely dashed, rounded corners=39] (0.92,0.52) -- (-0.4, 1.25) -- (-2.9,1.75);
        \draw (-0.3,1) -- (0.75,1) -- (1.25,0) -- (2.3,0);
        \draw (-0.3,0) -- (0.75,0) -- (1.25,1) -- (2.3,1);
    %\labels
        \node at (1,-0.3) {{$(k, a-1)$}};
        \node at (-0.2,0.5) {{$(k-1, a)$}};
        \node at (1.9,0.5) {{$(k, a)$}};
        \node at (1,1.3) {{$(k-1, a+1)$}};
        \node at (1,2) %{{$(s_k+a)$-th crossing}};
            {crossing at height $a$ from the $k$-th sub-product in $\dot{w}_P\bar{w}_0$ of the form (\ref{eqn subproduct 1 of si for ansatz arr proof})};
\end{tikzpicture}
\caption{Labelling of chamber pairs $(k,a)$ in ansatz arrangement for $\dot{w}_P\bar{w}_0$} \label{fig chamber pair labels wPw0 ansatz arrangement}
\end{figure}
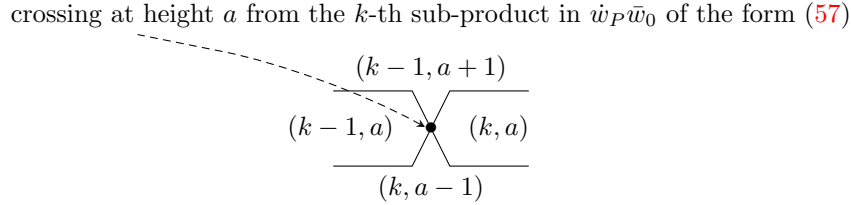

\begin{rem}
In the $G/B$ case, the crossing in Figure \ref{fig chamber pair labels wPw0 ansatz arrangement} is exactly the $(s_k+a)$-th crossing in the ansatz arrangement for $\bar{w}_0$ given by the reduced expression $\mathbf{i}_0$. Moreover, both the new chamber pairs $(k,a)$ and the original labels they correspond to, are the same as their respective labels in the full flag case, defined in the proof of Lemma \ref{lem coord change for b m_i in terms of p_i}.
\end{rem}

% We demonstrate this labelling with an example:
\begin{ex} \label{ex chamber labels k,a pairs}
In Figure \ref{fig k,a labels ansatz arrangement wPw0 F2,5,6,C8} we give a new labelling of the ansatz arrangement for $\dot{w}_P\bar{w}_0$ in our running example of $\mathcal{F}_{2,5,6}(\mathbb{C}^8)$. For the pairs $(k,a)$, if $k=n_j$ for some $j$ then we have written $k$ in bold. Following the labelling given above, the chambers with two crossings directly above them have two labels.

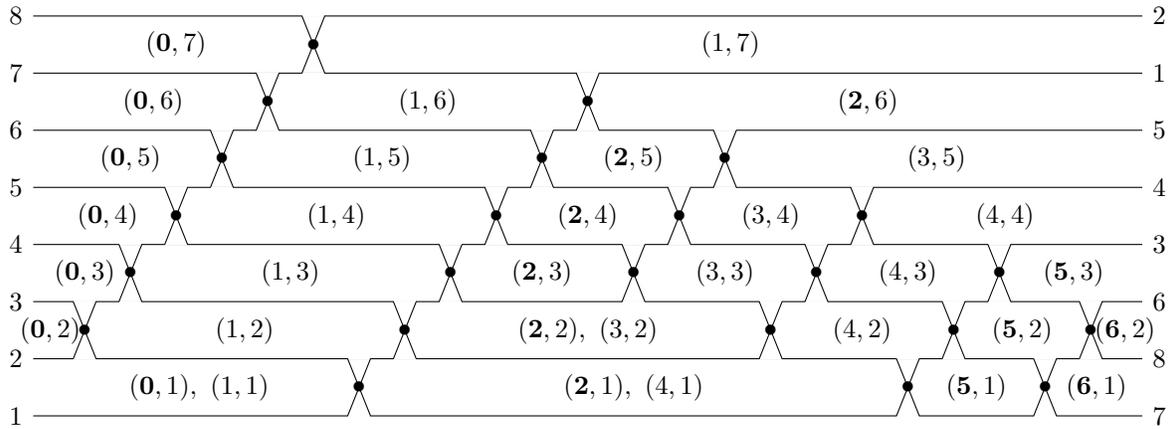
\begin{figure}[ht]
\centering
\begin{tikzpicture}[scale=0.76]
    % pseudolines
        \draw (-0.3,8) -- (19.1,8);
        \draw (-0.3,7) -- (19.1,7);
        \draw (-0.3,6) -- (19.1,6);
        \draw (-0.3,5) -- (19.1,5);
        \draw (-0.3,4) -- (19.1,4);
        \draw (-0.3,3) -- (19.1,3);
        \draw (-0.3,2) -- (19.1,2);
        \draw (-0.3,1) -- (19.1,1);

    % crossings
        % first set
        \draw[thick, color=white] (0.4,2) -- (0.8,2);
        \draw[thick, color=white] (0.4,3) -- (0.8,3);
        \draw[line cap=round] (0.4,2) -- (0.8,3);
        \draw[line cap=round] (0.4,3) -- (0.8,2);

        \draw[thick, color=white] (1.2,3) -- (1.6,3);
        \draw[thick, color=white] (1.2,4) -- (1.6,4);
        \draw[line cap=round] (1.2,3) -- (1.6,4);
        \draw[line cap=round] (1.2,4) -- (1.6,3);

        \draw[thick, color=white] (2,4) -- (2.4,4);
        \draw[thick, color=white] (2,5) -- (2.4,5);
        \draw[line cap=round] (2,4) -- (2.4,5);
        \draw[line cap=round] (2,5) -- (2.4,4);

        \draw[thick, color=white] (2.8,5) -- (3.2,5);
        \draw[thick, color=white] (2.8,6) -- (3.2,6);
        \draw[line cap=round] (2.8,5) -- (3.2,6);
        \draw[line cap=round] (2.8,6) -- (3.2,5);

        \draw[thick, color=white] (3.6,6) -- (4,6);
        \draw[thick, color=white] (3.6,7) -- (4,7);
        \draw[line cap=round] (3.6,6) -- (4,7);
        \draw[line cap=round] (3.6,7) -- (4,6);

        \draw[thick, color=white] (4.4,7) -- (4.8,7);
        \draw[thick, color=white] (4.4,8) -- (4.8,8);
        \draw[line cap=round] (4.4,7) -- (4.8,8);
        \draw[line cap=round] (4.4,8) -- (4.8,7);

        % second set
        \draw[thick, color=white] (5.2,1) -- (5.6,1);
        \draw[thick, color=white] (5.2,2) -- (5.6,2);
        \draw[line cap=round] (5.2,1) -- (5.6,2);
        \draw[line cap=round] (5.2,2) -- (5.6,1);

        \draw[thick, color=white] (6,2) -- (6.4,2);
        \draw[thick, color=white] (6,3) -- (6.4,3);
        \draw[line cap=round] (6,2) -- (6.4,3);
        \draw[line cap=round] (6,3) -- (6.4,2);

        \draw[thick, color=white] (6.8,3) -- (7.2,3);
        \draw[thick, color=white] (6.8,4) -- (7.2,4);
        \draw[line cap=round] (6.8,3) -- (7.2,4);
        \draw[line cap=round] (6.8,4) -- (7.2,3);

        \draw[thick, color=white] (7.6,4) -- (8,4);
        \draw[thick, color=white] (7.6,5) -- (8,5);
        \draw[line cap=round] (7.6,4) -- (8,5);
        \draw[line cap=round] (7.6,5) -- (8,4);

        \draw[thick, color=white] (8.4,5) -- (8.8,5);
        \draw[thick, color=white] (8.4,6) -- (8.8,6);
        \draw[line cap=round] (8.4,5) -- (8.8,6);
        \draw[line cap=round] (8.4,6) -- (8.8,5);

        \draw[thick, color=white] (9.2,6) -- (9.6,6);
        \draw[thick, color=white] (9.2,7) -- (9.6,7);
        \draw[line cap=round] (9.2,6) -- (9.6,7);
        \draw[line cap=round] (9.2,7) -- (9.6,6);

        % third set
        \draw[thick, color=white] (10,3) -- (10.4,3);
        \draw[thick, color=white] (10,4) -- (10.4,4);
        \draw[line cap=round] (10,3) -- (10.4,4);
        \draw[line cap=round] (10,4) -- (10.4,3);

        \draw[thick, color=white] (10.8,4) -- (11.2,4);
        \draw[thick, color=white] (10.8,5) -- (11.2,5);
        \draw[line cap=round] (10.8,4) -- (11.2,5);
        \draw[line cap=round] (10.8,5) -- (11.2,4);

        \draw[thick, color=white] (11.6,5) -- (12,5);
        \draw[thick, color=white] (11.6,6) -- (12,6);
        \draw[line cap=round] (11.6,5) -- (12,6);
        \draw[line cap=round] (11.6,6) -- (12,5);

        % fourth set
        \draw[thick, color=white] (12.4,2) -- (12.8,2);
        \draw[thick, color=white] (12.4,3) -- (12.8,3);
        \draw[line cap=round] (12.4,2) -- (12.8,3);
        \draw[line cap=round] (12.4,3) -- (12.8,2);

        \draw[thick, color=white] (13.2,3) -- (13.6,3);
        \draw[thick, color=white] (13.2,4) -- (13.6,4);
        \draw[line cap=round] (13.2,3) -- (13.6,4);
        \draw[line cap=round] (13.2,4) -- (13.6,3);

        \draw[thick, color=white] (14,4) -- (14.4,4);
        \draw[thick, color=white] (14,5) -- (14.4,5);
        \draw[line cap=round] (14,4) -- (14.4,5);
        \draw[line cap=round] (14,5) -- (14.4,4);

        % fifth set
        \draw[thick, color=white] (14.8,1) -- (15.2,1);
        \draw[thick, color=white] (14.8,2) -- (15.2,2);
        \draw[line cap=round] (14.8,1) -- (15.2,2);
        \draw[line cap=round] (14.8,2) -- (15.2,1);

        \draw[thick, color=white] (15.6,2) -- (16,2);
        \draw[thick, color=white] (15.6,3) -- (16,3);
        \draw[line cap=round] (15.6,2) -- (16,3);
        \draw[line cap=round] (15.6,3) -- (16,2);

        \draw[thick, color=white] (16.4,3) -- (16.8,3);
        \draw[thick, color=white] (16.4,4) -- (16.8,4);
        \draw[line cap=round] (16.4,3) -- (16.8,4);
        \draw[line cap=round] (16.4,4) -- (16.8,3);

        % sixth set
        \draw[thick, color=white] (17.2,1) -- (17.6,1);
        \draw[thick, color=white] (17.2,2) -- (17.6,2);
        \draw[line cap=round] (17.2,1) -- (17.6,2);
        \draw[line cap=round] (17.2,2) -- (17.6,1);

        \draw[thick, color=white] (18,2) -- (18.4,2);
        \draw[thick, color=white] (18,3) -- (18.4,3);
        \draw[line cap=round] (18,2) -- (18.4,3);
        \draw[line cap=round] (18,3) -- (18.4,2);

    % dots
        \node at (0.6,2.5) {$\bullet$};
        \node at (1.4,3.5) {$\bullet$};
        \node at (2.2,4.5) {$\bullet$};
        \node at (3,5.5) {$\bullet$};
        \node at (3.8,6.5) {$\bullet$};
        \node at (4.6,7.5) {$\bullet$};

        \node at (5.4,1.5) {$\bullet$};
        \node at (6.2,2.5) {$\bullet$};
        \node at (7,3.5) {$\bullet$};
        \node at (7.8,4.5) {$\bullet$};
        \node at (8.6,5.5) {$\bullet$};
        \node at (9.4,6.5) {$\bullet$};

        \node at (10.2,3.5) {$\bullet$};
        \node at (11,4.5) {$\bullet$};
        \node at (11.8,5.5) {$\bullet$};

        \node%[color=red]
            at (12.6,2.5) {$\bullet$};
        \node at (13.4,3.5) {$\bullet$};
        \node at (14.2,4.5) {$\bullet$};

        \node%[color=red]
            at (15,1.5) {$\bullet$};
        \node%[color=orange]
            at (15.8,2.5) {$\bullet$};
        \node at (16.6,3.5) {$\bullet$};

        \node%[color=orange]
            at (17.4,1.5) {$\bullet$};
        \node%[color=orange]
            at (18.2,2.5) {$\bullet$};

    % pseudoline labels
        \node at (-0.6,8) {$8$};
        \node at (-0.6,7) {$7$};
        \node at (-0.6,6) {$6$};
        \node at (-0.6,5) {$5$};
        \node at (-0.6,4) {$4$};
        \node at (-0.6,3) {$3$};
        \node at (-0.6,2) {$2$};
        \node at (-0.6,1) {$1$};

        \node at (19.4,8) {$2$};
        \node at (19.4,7) {$1$};
        \node at (19.4,6) {$5$};
        \node at (19.4,5) {$4$};
        \node at (19.4,4) {$3$};
        \node at (19.4,3) {$6$};
        \node at (19.4,2) {$8$};
        \node at (19.4,1) {$7$};

    % chamber labels
        \node at (2.2,7.5) {$(\boldsymbol{0},7)$};
        \node at (11.9,7.5) {$(1,7)$};

        \node at (1.8,6.5) {$(\boldsymbol{0},6)$};
        \node at (6.6,6.5) {$(1,6)$};
        \node at (14.3,6.5) {$(\boldsymbol{2},6)$};

        \node at (1.4,5.5) {$(\boldsymbol{0},5)$};
        \node at (5.8,5.5) {$(1,5)$};
        \node at (10.2,5.5) {$(\boldsymbol{2},5)$};
        \node at (15.5,5.5) {$(3,5)$};

        \node at (1,4.5) {$(\boldsymbol{0},4)$};
        \node at (5,4.5) {$(1,4)$};
        \node at (9.4,4.5) {$(\boldsymbol{2},4)$};
        \node at (12.6,4.5) {$(3,4)$};
        \node at (16.7,4.5) {$(4,4)$};

        \node at (0.6,3.5) {$(\boldsymbol{0},3)$};
        \node at (4.2,3.5) {$(1,3)$};
        \node at (8.6,3.5) {$(\boldsymbol{2},3)$};
        \node at (11.8,3.5) {$(3,3)$};
        \node at (15,3.5) {$(4,3)$};
        \node at (17.9,3.5) {$(\boldsymbol{5},3)$};

        \node at (0,2.5) {$(\boldsymbol{0},2)$};
        \node at (3.4,2.5) {$(1,2)$};
        \node at (9.4,2.5) {$(\boldsymbol{2},2), \ (3,2)$};
        \node at (14.2,2.5) {$(4,2)$};
        \node at (17,2.5) {$(\boldsymbol{5},2)$};
        \node at (18.8,2.5) {$(\boldsymbol{6},2)$};

        \node at (2.6,1.5) {$(\boldsymbol{0},1), \ (1,1)$};
        \node at (10.2,1.5) {$(\boldsymbol{2},1), \ (4,1)$};
        \node at (16.2,1.5) {$(\boldsymbol{5},1)$};
        \node at (18.3,1.5) {$(\boldsymbol{6},1)$};
\end{tikzpicture}
\caption{Alternate labelling of ansatz arrangement for $\dot{w}_P\bar{w}_0$ in the example of $\mathcal{F}_{2,5,6}(\mathbb{C}^8)$} \label{fig k,a labels ansatz arrangement wPw0 F2,5,6,C8}
\end{figure}
\end{ex}

As mentioned in Example \ref{ex chamber labels k,a pairs}, the chambers with two crossings directly above them are labelled by two pairs $(k,a)$. By the proof of Lemma \ref{lem G/P chamber label form} (see Claim \ref{claim 1 ansatz arr for wPw0}), the maximum number of $(k,a)$ pairs a chamber can have is two, and the first of these has $k=n_{j-1}$ for some $j \in \{1, \ldots, l+1\}$. Moreover, if the height of the chamber is $a$, then the second of the two pairs has $k=n_j-a$. This follows from the fact that the second crossing directly above this chamber is, for some $t$, the $(k_j-t+1)$-th string (counting from the bottom) crossing the string above it. Thus the crossing is at height $a+1=k_j-t+1$ and is part of the set of crossings from the $(n_{j-1}+t)$-th sub-product of the form (\ref{eqn subproduct 1 of si for ansatz arr proof}) in the description (\ref{eqn wPw0 increasing si}) for $\dot{w}_P\bar{w}_0$. % by claim 1

Let $j \in \{1, \ldots, l+1\}$ and suppose a chamber is labelled by the two pairs $(n_{j-1},a)$ and $(n_j-a,a)$. We will show that these two pairs are equivalent, namely they both correspond to the same original chamber label as given in Lemma \ref{lem G/P chamber label form}, or equivalently in (\ref{eqn orginial chamber label in terms of k,a}).

Using the chamber label description (\ref{eqn orginial chamber label in terms of k,a}) given in terms of $(k,a)$ pairs, the chamber pair $(n_{j-1},a)$ corresponds to the original label
    $$\{ n_{j-1}+1, \ldots, n_j \} \cup \{ n_j+1, \ldots, n_{j-1}+a \}=\{ n_{j-1}+1, \ldots, n_{j-1}+a \}.
    $$
Similarly the chamber pair $(n_j-a,a)$ corresponds to the original label
    $$\{ n_{j-1}+1, \ldots, n_{j-1}+a \} \cup \{ n_j+1, \ldots, n_j \}=\{ n_{j-1}+1, \ldots, n_{j-1}+a \}
    $$
as required, since the second set on the left hand side is empty. %; it is not possible to have a set given by strictly increasing integers from an intger $x$ to $x-1$.

% \subsection[The coordinate change (\texorpdfstring{$G/P$}{G/P} setting)]{The coordinate change ($G/P$ setting)}

\subsection{The coordinate change} \label{subsec G/P The coordinate change}
\fancyhead[L]{9.4 \ \ The coordinate change}
% \fancyhead[L]{10.4 \ \ The coordinate change}

We now ready to prove Theorem \ref{thm b- in terms of y_i}, the statement of which we recall here:
\begin{thm*}
We can factorise $b_P$ as $b_P = \left[b_P\right]_- \left[b_P\right]_0$ where
    $$\left[b_P\right]_- = \prod_{k=1}^{n-1} \prod_{\substack{a=1 \\ v_{(k,a)} \in \mathcal{V}_P^{\bullet}}}^{n-k} \mathbf{y}^{\vee}_a\left( \frac{1}{m_{s_k+a}} \right) %=: \left[\psi_P(d,\boldsymbol{m})\right]_-
    .$$
\end{thm*}

\begin{proof}[Proof (of Theorem {\ref{thm b- in terms of y_i}})]
By Lemma \ref{lem b- in terms of y_i} we know that $\left[b_P\right]_-$ may be factored into a product of $\mathbf{y}^{\vee}_a$'s with the desired sequence of subscripts. It remains to show that the arguments of these terms, given by the Chamber Ansatz in Lemma \ref{lem b- in terms of y_i} as
    $$h_r = \frac{\prod_{j\neq i_k} \Delta^{\omega^{\vee}_j}_{w_{(r)}\omega^{\vee}_j}(u_L)^{-a_{j,i_r}}}{\Delta^{\omega^{\vee}_{i_r}}_{w_{(r)}\omega^{\vee}_{i_r}}(u_L) \Delta^{\omega^{\vee}_{i_r}}_{w_{(r-1)}\omega^{\vee}_{i_r}}(u_L)}, \quad r=1, \ldots, R,
    $$
are exactly the respective inverted coordinates $1/m_{s_k+a}$. Since this expression is quite unpleasant, we will instead work diagrammatically, using the ansatz arrangement for $\dot{w}\bar{w}_0$ and the graph for $u_L$ to evaluate the necessary quotients of minors.

We recall from Section \ref{subsec The Chamber Ansatz} that if $A_k$, $B_k$, $C_k$ and $D_k$ are the minors corresponding to the chambers surrounding the $k$-th singular point in an ansatz arrangement, with $A_k$ and $D_k$ above and below it and $B_k$ and $C_k$ to the left and right, then the Chamber Ansatz gives
    % \vspace{-0.1cm}
    $$
\begin{tikzpicture}[baseline=0.43cm]
    %dots and stars
    \node at (1,0.5) {$\bullet$};

    \draw (0,1) -- (0.75,1) -- (1.25,0) -- (2,0);
    \draw (0,0) -- (0.75,0) -- (1.25,1) -- (2,1);
    %\labels
    \node at (1,-0.25) {\scriptsize{$D_k$}};
    \node at (0.3,0.5) {\scriptsize{$B_k$}};
    \node at (1.7,0.5) {\scriptsize{$C_k$}};
    \node at (1,1.25) {\scriptsize{$A_k$}};
\end{tikzpicture}
    \qquad \quad
    t_k = \frac{A_k D_k}{B_k C_k}.
    $$

We will write $t_{(k,a)}$ for the coordinate given by the Chamber Ansatz which corresponds to the crossing in Figure \ref{fig chamber pair labels wPw0 ansatz arrangement}, that is, the crossing at height $a$ from the $k$-th sub-product in $\dot{w}_P\bar{w}_0$ of the form (\ref{eqn subproduct 1 of si for ansatz arr proof}). The chamber label to the right of this crossing is given either by (\ref{eqn orginial chamber label in terms of k,a}) or by the chamber pair $(k,a)$, where $n_{j-1}\leq k \leq n_j$ for some $j$. Thus to simplify notation we will write $\Delta^{\{1,\ldots, a\}}_{(k,a)}(u_L)$ in place of
    $$ \Delta^{\{1,\ldots, a\}}_{\{ n_{j-1}+1, \ldots, n_{j-1}+n_j-k \} \cup \{ n_j+1, \ldots, k+a \}}(u_L).
    $$
With this notation, we use Figure \ref{fig chamber pair labels wPw0 ansatz arrangement} to obtain the coordinate $t_{(k,a)}$ given by the Chamber Ansatz:
    $$t_{(k,a)} = \frac{\Delta^{\{1,\ldots, a+1\}}_{(k-1,a+1)}(u_L) \Delta^{\{1,\ldots, a-1\}}_{(k,a-1)}(u_L)}{\Delta^{\{1,\ldots, a\}}_{(k-1,a)}(u_L) \Delta^{\{1,\ldots, a\}}_{(k,a)}(u_L)}.
    $$

Now viewing minors in terms of paths in the respective graphs, we see by the proof of Lemma \ref{lem minors in CA are monomial} that in the quotient of minors
    $$\frac{\Delta^{\{1,\ldots, a\}}_{(k,a)}(u_L)}{\Delta^{\{1,\ldots, a-1\}}_{(k,a-1)}(u_L)}
    $$
the contributions from most paths in the graph for $u_L$ cancel. Indeed this minor is equal to weight of a single path, namely the path which begins at source $a$ and travels upwards at every opportunity until it reaches the horizontal line for the sink $k+a$. Thus
    \begin{equation} \label{eqn t(k,a) in terms of two paths}
    t_{(k,a)} = \frac{\left(\begin{gathered}\text{weight of the path which begins at source $a+1$ and travels upwards} \\ \text{at every opportunity until it reaches the horizontal line for the sink $k+a$}\end{gathered}\right)}{\left(\begin{gathered}\text{weight of the path which begins at source $a$ and travels upwards} \\ \text{at every opportunity until it reaches the horizontal line for the sink $k+a$}\end{gathered}\right)}.
    \end{equation}
It remains to prove that $t_{(k,a)}=\frac{1}{m_{s_k+a}}$, which we will split into two cases. The first case will be when the two paths in (\ref{eqn t(k,a) in terms of two paths}) defining $t_{(k,a)}$ either pass through no dotted rectangles or through dotted rectangles where $\alpha=1$ (that is, spanning exactly two lines, with a single diagonal step). The second case will be when at least one of the paths passes through a dotted rectangle which has $\alpha\geq2$.
Of note, since the crossing in the ansatz arrangement giving the coordinate $t_{(k,a)}$ is directly to the left of the chamber labelled by the chamber pair $(k,a)$, we see that we do not need to consider $k=0$.

Before treating these two cases we recall that the weights of the diagonal line segments in the graph are given by arrow coordinates from the quiver $Q_P$. In particular, by construction of the graph we see that for a step in the $k$-th column between lines $b$, $b+1$ which is not contained in a dotted rectangle, the weight of this step is given by the arrow coordinate $r_{a_{(k,b-k)}}$. For the steps within a dotted square, we only need to consider the label of the first step between each pair of consecutive horizontal lines, by the claim in the proof of Lemma \ref{lem minors in CA are monomial}. In particular, in the dotted square corresponding to some $X_{n_j,\alpha}$ factor (that is, spanning the horizontal lines $n_j+1-\alpha, \ldots, n_j+1$) we see that the first step between lines $b$ and $b+1$ has weight
    $$\begin{cases}
    r_{a_{(n_{j-1}+1, k_j-1)}} & \text{if } b=n_j+1-\alpha \text{ (that is, the first step at the lowest height)} \\
    r_{b_{(n_{j-1}+1+c, k_j-c)}} & \text{for } b=n_j+1-\alpha+c \text{ where } c=1,\ldots, \alpha-1.
    \end{cases}
    $$
For example, see Figure \ref{fig graph for computing minors of uL for F2,5,6,C8} for this labelling in our running example.

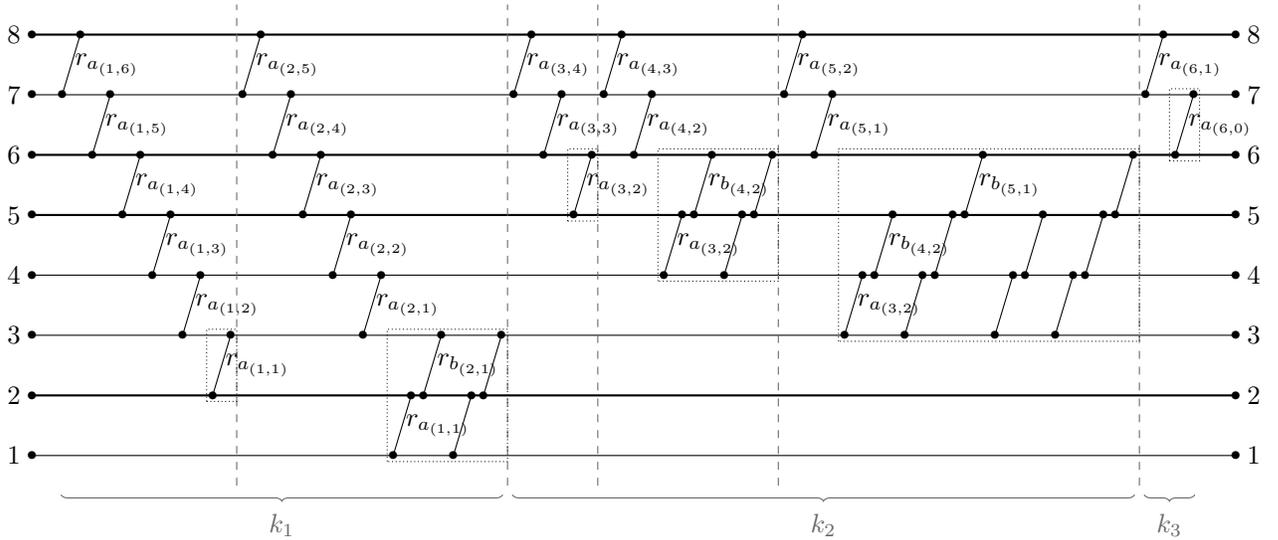
\begin{figure}[ht]
\centering
\begin{tikzpicture}[scale=0.8]
    % pseudolines
        \draw[thick] (0,8) -- (20,8);
        \draw (0,7) -- (20,7);
        \draw[thick] (0,6) -- (20,6);
        \draw[thick] (0,5) -- (20,5);
        \draw (0,4) -- (20,4);
        \draw (0,3) -- (20,3);
        \draw[thick] (0,2) -- (20,2);
        \draw (0,1) -- (20,1);
        % first set
            \draw (0.5,7) -- (0.8,8);
            \draw (1,6) -- (1.3,7);
            \draw (1.5,5) -- (1.8,6);
            \draw (2,4) -- (2.3,5);
            \draw (2.5,3) -- (2.8,4);
            \draw (3,2) -- (3.3,3);
        \draw[dashed, color=black!60] (3.4,0.5) -- (3.4,8.5);
        \draw[densely dotted, color=black!90] (2.9,1.9) -- (2.9,3.1) -- (3.4,3.1) -- (3.4,1.9) -- cycle;
        % second set
            \draw (3.5,7) -- (3.8,8);
            \draw (4,6) -- (4.3,7);
            \draw (4.5,5) -- (4.8,6);
            \draw (5,4) -- (5.3,5);
            \draw (5.5,3) -- (5.8,4);

            \draw (6,1) -- (6.3,2);
            \draw (6.5,2) -- (6.8,3);
            \draw (7,1) -- (7.3,2);
            \draw (7.5,2) -- (7.8,3);
        \draw[dashed, color=black!60] (7.9,0.5) -- (7.9,8.5);
        \draw[densely dotted, color=black!90] (5.9,0.9) -- (5.9,3.1) -- (7.9,3.1) -- (7.9,0.9) -- cycle;
        % third set
            \draw (8,7) -- (8.3,8);
            \draw (8.5,6) -- (8.8,7);
            \draw (9,5) -- (9.3,6);
        \draw[dashed, color=black!60] (9.4,0.5) -- (9.4,8.5);
        \draw[densely dotted, color=black!90] (8.9,4.9) -- (8.9,6.1) -- (9.4,6.1) -- (9.4,4.9) -- cycle;
        % fourth set
            \draw (9.5,7) -- (9.8,8);
            \draw (10,6) -- (10.3,7);

            \draw (10.5,4) -- (10.8,5);
            \draw (11,5) -- (11.3,6);
            \draw (11.5,4) -- (11.8,5);
            \draw (12,5) -- (12.3,6);
        \draw[dashed, color=black!60] (12.4,0.5) -- (12.4,8.5);
        \draw[densely dotted, color=black!90] (10.4,3.9) -- (10.4,6.1) -- (12.4,6.1) -- (12.4,3.9) -- cycle;
        % fifth set
            \draw (12.5,7) -- (12.8,8);
            \draw (13,6) -- (13.3,7);

            \draw (13.5,3) -- (13.8,4);
            \draw (14,4) -- (14.3,5);
            \draw (14.5,3) -- (14.8,4);
            \draw (15,4) -- (15.3,5);
            \draw (15.5,5) -- (15.8,6);
            \draw (16,3) -- (16.3,4);
            \draw (16.5,4) -- (16.8,5);
            \draw (17,3) -- (17.3,4);
            \draw (17.5,4) -- (17.8,5);
            \draw (18,5) -- (18.3,6);
        \draw[dashed, color=black!60] (18.4,0.5) -- (18.4,8.5);
        \draw[densely dotted, color=black!90] (13.4,2.9) -- (13.4,6.1) -- (18.4,6.1) -- (18.4,2.9) -- cycle;
        % sixth set
            \draw (18.5,7) -- (18.8,8);
            \draw (19,6) -- (19.3,7);
        \draw[densely dotted, color=black!90] (18.9,5.9) -- (18.9,7.1) -- (19.4,7.1) -- (19.4,5.9) -- cycle;

    % below
        \draw[decorate, decoration={brace, mirror}, color=black!60] (0.48,0.35) -- (7.82,0.35) node[midway, yshift=-0.4cm]{$k_1$};
        \draw[decorate, decoration={brace, mirror}, color=black!60] (7.98,0.35) -- (18.32,0.35) node[midway, yshift=-0.4cm]{$k_2$};
        \draw[decorate, decoration={brace, mirror}, color=black!60] (18.48,0.35) -- (19.32,0.35) node[midway, yshift=-0.4cm]{$k_3$};

    % dots
        % first set
        \node at (0.5,7) {\scriptsize{$\bullet$}};
        \node at (1,6) {\scriptsize{$\bullet$}};
        \node at (1.5,5) {\scriptsize{$\bullet$}};
        \node at (2,4) {\scriptsize{$\bullet$}};
        \node at (2.5,3) {\scriptsize{$\bullet$}};
        \node at (3,2) {\scriptsize{$\bullet$}};

        \node at (0.8,8) {\scriptsize{$\bullet$}};
        \node at (1.3,7) {\scriptsize{$\bullet$}};
        \node at (1.8,6) {\scriptsize{$\bullet$}};
        \node at (2.3,5) {\scriptsize{$\bullet$}};
        \node at (2.8,4) {\scriptsize{$\bullet$}};
        \node at (3.3,3) {\scriptsize{$\bullet$}};
        % second set
        \node at (3.5,7) {\scriptsize{$\bullet$}};
        \node at (4,6) {\scriptsize{$\bullet$}};
        \node at (4.5,5) {\scriptsize{$\bullet$}};
        \node at (5,4) {\scriptsize{$\bullet$}};
        \node at (5.5,3) {\scriptsize{$\bullet$}};
        \node at (6,1) {\scriptsize{$\bullet$}};
        \node at (6.5,2) {\scriptsize{$\bullet$}};
        \node at (7,1) {\scriptsize{$\bullet$}};
        \node at (7.5,2) {\scriptsize{$\bullet$}};

        \node at (3.8,8) {\scriptsize{$\bullet$}};
        \node at (4.3,7) {\scriptsize{$\bullet$}};
        \node at (4.8,6) {\scriptsize{$\bullet$}};
        \node at (5.3,5) {\scriptsize{$\bullet$}};
        \node at (5.8,4) {\scriptsize{$\bullet$}};
        \node at (6.3,2) {\scriptsize{$\bullet$}};
        \node at (6.8,3) {\scriptsize{$\bullet$}};
        \node at (7.3,2) {\scriptsize{$\bullet$}};
        \node at (7.8,3) {\scriptsize{$\bullet$}};
        % third set
        \node at (8,7) {\scriptsize{$\bullet$}};
        \node at (8.5,6) {\scriptsize{$\bullet$}};
        \node at (9,5) {\scriptsize{$\bullet$}};

        \node at (8.3,8) {\scriptsize{$\bullet$}};
        \node at (8.8,7) {\scriptsize{$\bullet$}};
        \node at (9.3,6) {\scriptsize{$\bullet$}};
        % fourth set
        \node at (9.5,7) {\scriptsize{$\bullet$}};
        \node at (10,6) {\scriptsize{$\bullet$}};
        \node at (10.5,4) {\scriptsize{$\bullet$}};
        \node at (11,5) {\scriptsize{$\bullet$}};
        \node at (11.5,4) {\scriptsize{$\bullet$}};
        \node at (12,5) {\scriptsize{$\bullet$}};

        \node at (9.8,8) {\scriptsize{$\bullet$}};
        \node at (10.3,7) {\scriptsize{$\bullet$}};
        \node at (10.8,5) {\scriptsize{$\bullet$}};
        \node at (11.3,6) {\scriptsize{$\bullet$}};
        \node at (11.8,5) {\scriptsize{$\bullet$}};
        \node at (12.3,6) {\scriptsize{$\bullet$}};
        % fifth set
        \node at (12.5,7) {\scriptsize{$\bullet$}};
        \node at (13,6) {\scriptsize{$\bullet$}};
        \node at (13.5,3) {\scriptsize{$\bullet$}};
        \node at (14,4) {\scriptsize{$\bullet$}};
        \node at (14.5,3) {\scriptsize{$\bullet$}};
        \node at (15,4) {\scriptsize{$\bullet$}};
        \node at (15.5,5) {\scriptsize{$\bullet$}};
        \node at (16,3) {\scriptsize{$\bullet$}};
        \node at (16.5,4) {\scriptsize{$\bullet$}};
        \node at (17,3) {\scriptsize{$\bullet$}};
        \node at (17.5,4) {\scriptsize{$\bullet$}};
        \node at (18,5) {\scriptsize{$\bullet$}};

        \node at (12.8,8) {\scriptsize{$\bullet$}};
        \node at (13.3,7) {\scriptsize{$\bullet$}};
        \node at (13.8,4) {\scriptsize{$\bullet$}};
        \node at (14.3,5) {\scriptsize{$\bullet$}};
        \node at (14.8,4) {\scriptsize{$\bullet$}};
        \node at (15.3,5) {\scriptsize{$\bullet$}};
        \node at (15.8,6) {\scriptsize{$\bullet$}};
        \node at (16.3,4) {\scriptsize{$\bullet$}};
        \node at (16.8,5) {\scriptsize{$\bullet$}};
        \node at (17.3,4) {\scriptsize{$\bullet$}};
        \node at (17.8,5) {\scriptsize{$\bullet$}};
        \node at (18.3,6) {\scriptsize{$\bullet$}};
        % sixth set
        \node at (18.5,7) {\scriptsize{$\bullet$}};
        \node at (19,6) {\scriptsize{$\bullet$}};

        \node at (18.8,8) {\scriptsize{$\bullet$}};
        \node at (19.3,7) {\scriptsize{$\bullet$}};

        \node at (0,1) {\scriptsize{$\bullet$}};
        \node at (0,2) {\scriptsize{$\bullet$}};
        \node at (0,3) {\scriptsize{$\bullet$}};
        \node at (0,4) {\scriptsize{$\bullet$}};
        \node at (0,5) {\scriptsize{$\bullet$}};
        \node at (0,6) {\scriptsize{$\bullet$}};
        \node at (0,7) {\scriptsize{$\bullet$}};
        \node at (0,8) {\scriptsize{$\bullet$}};

        \node at (20,1) {\scriptsize{$\bullet$}};
        \node at (20,2) {\scriptsize{$\bullet$}};
        \node at (20,3) {\scriptsize{$\bullet$}};
        \node at (20,4) {\scriptsize{$\bullet$}};
        \node at (20,5) {\scriptsize{$\bullet$}};
        \node at (20,6) {\scriptsize{$\bullet$}};
        \node at (20,7) {\scriptsize{$\bullet$}};
        \node at (20,8) {\scriptsize{$\bullet$}};

    % pseudoline labels
        \node at (-0.3,8) {$8$};
        \node at (-0.3,7) {$7$};
        \node at (-0.3,6) {$6$};
        \node at (-0.3,5) {$5$};
        \node at (-0.3,4) {$4$};
        \node at (-0.3,3) {$3$};
        \node at (-0.3,2) {$2$};
        \node at (-0.3,1) {$1$};

        \node at (20.3,8) {$8$};
        \node at (20.3,7) {$7$};
        \node at (20.3,6) {$6$};
        \node at (20.3,5) {$5$};
        \node at (20.3,4) {$4$};
        \node at (20.3,3) {$3$};
        \node at (20.3,2) {$2$};
        \node at (20.3,1) {$1$};

    % weight labels
        % first set
        \node at (1.25,7.5) {$r_{a_{(1,6)}}$}; %a % +(0.05,0) from end of diagonal line at top
        \node at (1.75,6.5) {$r_{a_{(1,5)}}$}; %b
        \node at (2.25,5.5) {$r_{a_{(1,4)}}$}; %c
        \node at (2.75,4.5) {$r_{a_{(1,3)}}$}; %d
        \node at (3.25,3.5) {$r_{a_{(1,2)}}$}; %e
        \node at (3.75,2.5) {$r_{a_{(1,1)}}$}; %f
        % second set
        \node at (4.25,7.5) {$r_{a_{(2,5)}}$}; %g
        \node at (4.75,6.5) {$r_{a_{(2,4)}}$}; %h
        \node at (5.25,5.5) {$r_{a_{(2,3)}}$}; %i
        \node at (5.75,4.5) {$r_{a_{(2,2)}}$}; %j
        \node at (6.25,3.5) {$r_{a_{(2,1)}}$}; %k

        \node at (6.75,1.5) {$r_{a_{(1,1)}}$}; %f
        \node at (7.25,2.5) {$r_{b_{(2,1)}}$}; %l
        % \node at (8,1.5) {$-r_{a_{(1,1)}}$}; %-f
        % \node at (8.45,2.5) {$-r_{b_{(2,1)}}$}; %-l
        % third set
        \node at (8.75,7.5) {$r_{a_{(3,4)}}$}; %m
        \node at (9.25,6.5) {$r_{a_{(3,3)}}$}; %n
        \node at (9.75,5.5) {$r_{a_{(3,2)}}$}; %p
        % fourth set
        \node at (10.25,7.5) {$r_{a_{(4,3)}}$}; %q
        \node at (10.75,6.5) {$r_{a_{(4,2)}}$}; %r

        \node at (11.25,4.5) {$r_{a_{(3,2)}}$}; %p
        \node at (11.75,5.5) {$r_{b_{(4,2)}}$}; %s
        % \node at (12.5,4.5) {$-r_{a_{(3,2)}}$}; %-p
        % \node at (12.95,5.5) {$-r_{b_{(4,2)}}$}; %-s
        % fifth set
        \node at (13.25,7.5) {$r_{a_{(5,2)}}$}; %t
        \node at (13.75,6.5) {$r_{a_{(5,1)}}$}; %u

        \node at (14.25,3.5) {$r_{a_{(3,2)}}$}; %p
        \node at (14.75,4.5) {$r_{b_{(4,2)}}$}; %s
        % \node at (15.45,3.5) {$-r_{a_{(3,2)}}$}; %-p
        % \node at (15.95,4.5) {$-r_{b_{(4,2)}}$}; %-s
        \node at (16.25,5.5) {$r_{b_{(5,1)}}$}; %v
        % \node at (16.75,3.5) {$r_{a_{(3,2)}}$}; %p
        % \node at (17.45,4.5) {$-r_{b_{(4,2)}}$}; %-s
        % \node at (17.95,3.5) {$-r_{a_{(3,2)}}$}; %-p
        % \node at (18.25,4.5) {$r_{b_{(4,2)}}$}; %s
        % \node at (18.95,5.5) {$-r_{b_{(5,1)}}$}; %-v
        % sixth set
        \node at (19.25,7.5) {$r_{a_{(6,1)}}$}; %w
        \node at (19.75,6.5) {$r_{a_{(6,0)}}$}; %x
\end{tikzpicture}
\caption{Graph for computing minors of $u_L$ for $\mathcal{F}_{2,5,6}(\mathbb{C}^8)$} \label{fig graph for computing minors of uL for F2,5,6,C8}
\end{figure}

We also recall the quiver decoration for reference. The coordinates of the vertical arrows are given in (\ref{eqn coord of vert arrow leaving vertex k,a}) by
    \begin{equation*}
    r_{a_{(k,a-1)}}=\begin{cases}
        m_a & \text{if } k=1 \\
        r_{a_{(k-1,a)}} \frac{m_{s_{k}+a}}{m_{s_{k-1}+a}} & \text{if } k\geq 2.
        \end{cases}
    \end{equation*}
In particular we have the following relation:
    \begin{equation} \label{eqn vert arrow coord relation k,a and k+1,a-1}
    \frac{r_{a_{(k,a)}}}{r_{a_{(k+1,a-1)}}} = \frac{r_{a_{(k,a)}}}{r_{a_{(k,a)}} \left(\frac{m_{s_{k+1}+a}}{m_{s_k+a}}\right)}
        = \frac{m_{s_k+a}}{m_{s_{k+1}+a}}.
    \end{equation}
The coordinates of the horizontal arrows $r_{b_{(k,a)}}$, where $k\in \{ n_{r-1}+1, \ldots, n_r \}$ and $a=n_r-k+1$, are given in (\ref{eqn coord of horiz arrow leaving vertex k,a}) by
    \begin{equation*}
        r_{b_{(k,a)}} =\begin{cases}
        m_{s_k+a} & \text{if } r=1 \\
        r_{a_{(n_{r-1},a)}}\frac{m_{s_k+a}}{m_{s_{n_{r-1}}+a}} & \text{if } r\geq 2.
        \end{cases}
    \end{equation*}

We are now ready to prove that $t_{(k,a)}=\frac{1}{m_{s_k+a}}$. In the first case, when the two paths in (\ref{eqn t(k,a) in terms of two paths}) defining $t_{(k,a)}$ either pass through no dotted rectangles or through dotted rectangles where $\alpha=1$, we have
    $$\begin{aligned}
    t_{(k,a)} &= \frac{\prod_{i=1}^{k-1}r_{a_{(i,a)}}}{\prod_{i=1}^{k}r_{a_{(i,a-1)}}} & \text{by %(\ref{eqn t(k,a) in terms of two paths}) and
            the graph labelling}\\
        &= \frac{1}{r_{a_{(1,a-1)}}}\prod_{i=1}^{k-1}\frac{r_{a_{(i,a)}}}{r_{a_{(i+1,a-1)}}} \\
        &= \frac{1}{m_{s_1+a}}\prod_{i=1}^{k-1}\frac{m_{s_i+a}}{m_{s_{i+1}+a}} & \text{by (\ref{eqn vert arrow coord relation k,a and k+1,a-1})} \\
        % &= \frac{\prod_{i=1}^{k-1}m_{s_i+a}}{\prod_{i=1}^{k}m_{s_i+a}} \\
        &= \frac{1}{m_{s_k+a}} & \text{since the product is telescopic.}
    \end{aligned}
    $$
    % $$t_{(k,a)} = \frac{\prod_{i=1}^{k-1}r_{a_{(i,a)}}}{\prod_{i=1}^{k}r_{a_{(i,a-1)}}} \\
    %     = \frac{1}{r_{a_{(1,a-1)}}}\prod_{i=1}^{k-1}\frac{r_{a_{(i,a)}}}{r_{a_{(i+1,a-1)}}} \\
    %     = \frac{1}{m_{s_1+a}}\prod_{i=1}^{k-1}\frac{m_{s_i+a}}{m_{s_{i+1}+a}} \\
    %     % = \frac{\prod_{i=1}^{k-1}m_{s_i+a}}{\prod_{i=1}^{k}m_{s_i+a}} \\
    %     = \frac{1}{m_{s_k+a}}.
    % $$

The second case is when at least one of the paths in (\ref{eqn t(k,a) in terms of two paths}) defining $t_{(k,a)}$ passes through a dotted rectangle with $\alpha\geq2$. We note that if only one of these paths passes through a dotted rectangle with $\alpha\geq 2$, then it will be the path from source $a$, since both paths must end at the same sink $k+a$. We also see that our two paths may be considered in segments where they do or do not pass through dotted rectangles with $\alpha\geq2$.

If for some $j$ and $\alpha\geq 2$, the path from source $a$ passes through the dotted rectangle corresponding to $X^{\vee}_{n_j,\alpha}$ (that is, spanning the horizontal lines $n_j+1-\alpha, \ldots, n_j+1$), then the path from source $a+1$ necessarily passes through the dotted rectangle corresponding to $X^{\vee}_{n_j,\alpha-1}$. This follows from the proof of Lemma \ref{lem minors in CA are monomial}. In particular, the quotient of the contributions from the paths which travel upwards at every opportunity within these two rectangles is
    \begin{align}
    \frac{r_{a_{(n_{j-1}+1, k_j-1)}}\prod_{c=1}^{\alpha-2}r_{b_{(n_{j-1}+1+c, k_j-c)}}}{r_{a_{(n_{j-1}+1, k_j-1)}}\prod_{c=1}^{\alpha-1}r_{b_{(n_{j-1}+1+c, k_j-c)}}}
        &= \frac{1}{r_{b_{(n_{j-1}+\alpha, k_j-\alpha+1)}}} \notag \\
        &=\begin{cases}
        \frac{1}{m_{s_{n_{j-1}+\alpha}+k_j-\alpha+1}} & \text{if } j=1 \\
        \frac{1}{r_{a_{(n_{j-1},k_j-\alpha+1)}}}\frac{m_{s_{n_{j-1}}+k_j-\alpha+1}}{m_{s_{n_{j-1}+\alpha}+k_j-\alpha+1}} & \text{if } j\geq 2
        \end{cases} \notag \\
        &=\begin{cases}
        \frac{1}{m_{s_{\alpha}+k_1-\alpha+1}} & \text{if } j=1 \\
        \frac{1}{r_{a_{(n_{j-1},k_j-\alpha+1)}}}\frac{m_{s_{n_{j-1}}+k_j-\alpha+1}}{m_{s_{n_{j-1}+\alpha}+k_j-\alpha+1}} & \text{if } j\geq 2.
        \end{cases} \label{eqn dotted rectangle contributions to t(k,a)}
    \end{align}

Additionally, we note that before these two dotted rectangles, the last step upwards from line $n_j+1-\alpha$ to $n_j+2-\alpha$ has weight $r_{a_{(n_{j-1},k_j-\alpha+1)}}$ if $j\geq2$, and if $j=1$ then no such step exists. In particular, since we are only considering paths which travel upwards at every opportunity, if $j\geq2$ then the path from source $a+1$ must have travelled up this step. Thus the quotient of the contributions from the two paths between the lines $n_j+1-\alpha$ and $n_j+1$ is
    \begin{equation} \label{eqn step and dotted rectangle contributions to t(k,a)}
    \frac{r_{a_{(n_{j-1},k_j-\alpha+1)}}}{r_{b_{(n_{j-1}+\alpha, k_j-\alpha+1)}}}
        % = r_{a_{(n_{j-1},k_j-\alpha+1)}}\frac{1}{r_{a_{(n_{j-1},k_j-\alpha+1)}}}\frac{m_{s_{n_{j-1}}+k_j-\alpha+1}}{m_{s_{n_{j-1}+\alpha}+k_j-\alpha+1}}
        = \frac{m_{s_{n_{j-1}}+k_j-\alpha+1}}{m_{s_{n_{j-1}+\alpha}+k_j-\alpha+1}}.
    \end{equation}

\begin{claim}
We are never required to only travel part way through a dotted rectangle.
\end{claim}

\begin{proof}[Proof of Claim]
We first observing the form of the chamber pairs $(k,a)$ such that we \emph{would} need to stop part way through one of these rectangles on a path which begins at source $a$, travels upwards at every opportunity and ends at sink $k+a$. These chamber pairs $(k,a)$ must satisfy the following conditions:
    $$n_{j-1} + 1 < k+a < n_j +1 \ \text{for some} \ j=1, \ldots, l, \ \text{and} \ n_{j-1} < k.
    $$
Explicitly, by the second condition we see that by the end of the $n_{j-1}$-th column, the path will have taken $n_{j-1}$ diagonal steps upwards, however by the first condition the path will have not yet reached the line $k+a>n_{j-1}+1$ (which it will have done by the end of the $n_j$-th column).
% These conditions mean that the paths in questions are forced to take diagonal steps upwards in the set of columns labelled $k_j$, but may never reach the top of the diagonal squares in these columns.

These chamber pairs only appear the second pairs for chambers that have two (if they appear at all), that is, the chamber pair with the larger value of $k$ defined relative to the crossing on the right of the chamber (see Figure \ref{fig chamber pair labels wPw0 ansatz arrangement}). Finally we note that we never need to consider the paths from source $a$ to sink $k+a$ for the chamber pairs $(k,a)$ above. This is because the crossing in the ansatz arrangement giving the coordinate $t_{(k,a)}$, is directly to the left of the chamber labelled by the chamber pair $(k,a)$, that is, we only consider the first chamber pair if a given chamber has two.
\end{proof}

To finish proving the second case of the form of $t_{(k,a)}$, we consider chamber pairs $(k,a)$ such that $n_{j-1}< k\leq n_j$ for $j=1, \dots, l$, and proceed by induction on $j$, making use of the columns in the graph.

We recall that in the first case we treated the situation when the two paths in (\ref{eqn t(k,a) in terms of two paths}) defining $t_{(k,a)}$ either pass through no dotted rectangles or through dotted rectangles with $\alpha=1$. Thus in order to complete the foundation for the induction, we only need to consider chamber pairs $(k,a)$ with $1\leq k \leq n_1$, such that the path from source $a$ passes through a dotted rectangle corresponding to $X^{\vee}_{n_1, \alpha}$, with $\alpha\geq 2$. In particular, this restricts the possible values of $a$ to $\{1, \ldots, n_1\}$.
Of note, if $n_1=1$ then this case doesn't exist and so the base case of the induction is already complete. By the claim, we are required to always travel to the top of dotted rectangles, so we see that the only possibility is to have $k+a = n_1+1$. Moreover, the bottom of the rectangle lies on the line $a=n_1+1-\alpha$. Thus $\alpha=k$, and by (\ref{eqn dotted rectangle contributions to t(k,a)}) with $j=1$ we have
    $$t_{(k,a)}=\frac{1}{m_{s_{n_1-a+1}+k_1-n_1+a}} = \frac{1}{m_{s_{k}+a}}.
    $$

For the inductive step, let $j\in\{1, \ldots, l\}$ and suppose we have shown that
    $$t_{(k,a)}=\frac{1}{m_{s_k+a}}
    $$
for all chamber pairs $(k,a)$ with $k = n_{j-1}$, by using (\ref{eqn t(k,a) in terms of two paths}) and thus the paths from sources $a$ and $a+1$ which travel upwards at every opportunity until they reach the horizontal line for the sink $k+a$. We need to prove that
    $$t_{(k+c,a)}=\frac{1}{m_{s_{k+c}+a}}
    $$
for $1\leq c \leq k_j$ such that $(k+c,a)$ appears as the first or only chamber pair of its respective chamber.

In order to compute $t_{(k+c,a)}$, the desired paths from the sources $a$ and $a+1$ to the sink $k+c+a$, must begin with the steps upwards which are found in the paths needed to compute $t_{(k,a)}$. Thanks to the claim we will assume that $(k+c,a)$ appears as the first or only chamber pair of its respective chamber, since we may ignore this chamber pair if not. We split our consideration into two cases:
\begin{enumerate}
    \item If $k+a \geq n_j$ then the paths from sources $a$ and $a+1$ will not travel through any dotted rectangles with $\alpha\geq 2$ in the set of columns labelled $k_j$. In particular we have
    % If $1\leq c \leq n_j-k$, in other words we reach the line $k+c+a$ by the end of the $n_j$-th column, then we have
        $$\begin{aligned}
        t_{(k+c,a)} &= \frac{1}{m_{s_k+a}} \frac{\prod_{i=k}^{k+c-1}r_{a_{(i,a)}}}{\prod_{i=k+1}^{k+c}r_{a_{(i,a-1)}}} & \text{by the inductive hypothesis and the graph labelling} \\
            &= \frac{1}{m_{s_k+a}} \prod_{i=k}^{k+c-1}\frac{r_{a_{(i,a)}}}{r_{a_{(i+1,a-1)}}} \\
            &= \frac{1}{m_{s_k+a}} \prod_{i=k}^{k+c-1}\frac{m_{s_i+a}}{m_{s_{i+1}+a}} & \text{by (\ref{eqn vert arrow coord relation k,a and k+1,a-1})} \\
            &= \frac{1}{m_{s_{k+c}+a}} & \text{since the product is telescopic.}
        \end{aligned}
        $$
    \item If $n_{j-1}+1 \leq k+a< n_j-1$ then, since $c\geq 1$ and we must travel all the way to the top of dotted rectangles, we see that we must have $k+c+a=n_j+1$. Thus the path from source $a$ must pass through a dotted rectangle corresponding to $X^{\vee}_{n_j,\alpha}$ with $\alpha\geq 2$. Since the bottom of this dotted rectangle is on the line $k+a=n_j+1-\alpha$ and $k=n_{j-1}$, we see that $\alpha = k_j-a+1$ and so we have
        $$t_{(k+c,a)} = t_{(n_j-a+1,a)}
        = \frac{1}{m_{s_{n_{j-1}}+a}} \frac{m_{s_{n_{j-1}}+a}}{m_{s_{n_{j-1}+k_j-a+1}+a}}
        = \frac{1}{m_{s_{n_j-a+1}+a}}
        = \frac{1}{m_{s_{k+c}+a}}
        $$
    where the second equality is by the inductive hypothesis and (\ref{eqn step and dotted rectangle contributions to t(k,a)}) with $\alpha =k_{j+1}-a+1$.

\end{enumerate}
% After passing through the dotted square corresponding to $X_{n_j,k_j}$, we know from the proof of Lemma \ref{lem minors in CA are monomial} that any path must remain horizontal until it reaches the dotted square corresponding to $X_{n_{j+1},k_{j+1}}$. However, if our path travels through a dotted square corresponding to $X_{n_j,\alpha}$ for $\alpha < k_j$, then it must rise up the next available step before reaching another dotted square. % corresponding to $X_{n_{j+1},k_{j+1}-k_j+\alpha}$. This step has weight $r_{a_{(n_{j-1}+\alpha+1, k_j-\alpha)}}$.
%
% Now suppose our paths from sources $a$ $a+1$ have, for some $j$ and $\alpha$, travelled through the dotted squares corresponding to $X_{n_j,\alpha}$, $X_{n_j,\alpha-1}$ respectively. Then either these paths take $h$, $h+1$ steps upwards respectively (with $h\geq0$) until they reach the next dotted square(s), or they both take $h$ steps upwards until they reach the horizontal line for sink $k+a$.
\end{proof}

% \section[The tropical viewpoint (\texorpdfstring{$G/P$}{G/P} setting)]{The tropical viewpoint ($G/P$ setting)}

\section{The tropical viewpoint} \label{sec G/P The tropical viewpoint}
\fancyhead[L]{10 \ \ The tropical viewpoint}
% \fancyhead[L]{11 \ \ The tropical viewpoint}

In this section we generalise the definition of superpotential polytopes associated to a given highest weight (given in Section \ref{subsec Constructing polytopes} using tropicalisation) to the setting of partial flag varieties. These polytopes will depend on the choice of positive toric chart, for which we focus our attention on the ideal coordinates. In the $G/B$ case, for each choice of highest weight, the associated critical point of the superpotential gives rise to a point inside the respective polytope, which is Judd's tropical critical point \cite{Judd2018}. By work of Judd and Rietsch \cite{JuddRietsch2019} an analogous statement holds in the $G/P$ setting. It is natural to also wish for an analogous statement of the description of the tropical critical point in terms ideal fillings. Indeed we generalise the notion of ideal fillings and show that it may be used to describe tropical critical points in the $G/P$ case. We conclude with an interpretation of these more general ideal fillings using Toeplitz matrices over generalised Puiseux series.

% \subsection[Constructing polytopes (\texorpdfstring{$G/P$}{G/P} setting)]{Constructing polytopes ($G/P$ setting)}

\newpage
\subsection{Constructing polytopes} \label{subsec G/P Constructing polytopes}
\fancyhead[L]{10.1 \ \ Constructing polytopes}
% \fancyhead[L]{11.1 \ \ Constructing polytopes}

In this section we generalise Section \ref{subsec Constructing polytopes} to the $G/P$ setting. Similar to the $G/B$ case, we begin with $Z_P$, and consider it over the field of generalised Puiseux series, $\mathbf{K}$. Again we have a well-defined notion of the totally positive part of $Z_P(\mathbf{K})$, denoted by $Z_P(\mathbf{K}_{>0})$. It is defined, for a given torus chart on $Z_P(\mathbf{K})$, by the subset where the characters take values in $\mathbf{K}_{>0}$. Moreover, both the quiver and ideal torus charts defined in Sections \ref{subsec G/P The quiver torus and toric chart} and \ref{subsec G/P Quiver decoration} respectively, give isomorphisms
    \begin{equation} \label{eqn G/P toric chart isom}
    T^{\vee}(\mathbf{K}_{>0})^{W_P} \times (\mathbf{K}_{>0})^{\mathcal{V}_P^{\bullet}} \xrightarrow{\sim} Z_P(\mathbf{K}_{>0})
    \end{equation}
where we consider $T^{\vee}(\mathbf{K}_{>0})^{W_P}$ to be the highest weight torus.

We will now restrict our attention to a fibre of the highest weight map $\mathrm{hw}_P$ (see Section \ref{sec G/P Landau-Ginzburg models}). To do so, we first define $t^{\lambda} \in \left(T^{\vee}(\mathbf{K}_{>0})\right)^{W_P}$ via the condition $\chi(t^{\lambda}) = t^{\langle \chi, \lambda\rangle}$ for $\chi \in X^*\left(\left(T^{\vee}\right)^{W_P}\right)$. This allows us to define
    $$Z_{P,t^\lambda}(\mathbf{K}):= \left \{b_P \in Z_P(\mathbf{K}) \ | \ \mathrm{hw}_P\left(b_P\right)=t^\lambda \right\}.
    $$
We denote the restriction of the superpotential to this fibre by
    $$\mathcal{W}_{P,t^{\lambda}}: Z_{P,t^\lambda}(\mathbf{K}) \to \mathbf{K}.
    $$

For a fixed element $t^{\lambda} \in T^{\vee}(\mathbf{K}_{>0})^{W_P}$ of the highest weight torus, the isomorphism (\ref{eqn G/P toric chart isom}) for the ideal toric chart restricts to
    $$\phi_{P,t^{\lambda},\boldsymbol{m}} : (\mathbf{K}_{>0})^{\mathcal{V}_P^{\bullet}} \to Z_{P,t^{\lambda}}(\mathbf{K}_{>0}),
    $$
with $m_i$ coordinates.
This toric chart may be considered as defining a positive atlas for $Z_{P,t^{\lambda}}(\mathbf{K}_{>0})$.
We denote the composition of $\phi_{P,t^{\lambda},\boldsymbol{m}}$ with the superpotential $\mathcal{W}_{P,t^{\lambda}}$, by
    $$\mathcal{W}_{P,t^{\lambda},\boldsymbol{m}} : (\mathbf{K}_{>0})^{\mathcal{V}_P^{\bullet}} \to \mathbf{K}_{>0}
    $$
and observe that it is a positive rational map. We denote its tropicalisation by
    $$\mathrm{Trop}\left(\mathcal{W}_{P,t^{\lambda},\boldsymbol{m}}\right) : \mathbb{R}^{\mathcal{V}_P^{\bullet}}_{\boldsymbol{\mu}} \to \mathbb{R}.
    $$

We may associate a convex polytope to our tropical superpotential, defined as follows:
    $$\mathcal{P}_{P,\lambda,\boldsymbol{\mu}} := \left\{ \boldsymbol{\alpha} \in \mathbb{R}^{\mathcal{V}_P^{\bullet}}_{\boldsymbol{\mu}} \ \big\vert \ \mathrm{Trop}\left(\mathcal{W}_{P,t^{\lambda},\boldsymbol{m}}\right)(\boldsymbol{\alpha}) \geq 0 \right\}.
    $$

% \subsection[Tropical critical points and the weight map (\texorpdfstring{$G/P$}{G/P} setting)]{Tropical critical points and the weight map ($G/P$ setting)}

\newpage
\subsection{Tropical critical points and the weight map} \label{subsec G/P Tropical critical points and the weight map}
\fancyhead[L]{10.2 \ \ Tropical critical points and the weight map}
% \fancyhead[L]{11.2 \ \ Tropical critical points and the weight map}

We recall from Section \ref{subsec Tropical critical points and the weight map} that in the $G/B$ case there is a unique critical point of $\mathcal{W}$ that lies in $Z_{t^{\lambda}}(\mathbf{K}_{>0})$, which we called the positive critical point of $\mathcal{W}_{t^{\lambda}}$, denoted $p_{\lambda}$. This was a consequence of work by Judd in \cite[Section 5]{Judd2018}, and also follows from the more general result of Judd and Rietsch in \cite{JuddRietsch2019}. In the $G/P$ setting, again working over $\mathbf{K}_{>0}$, we use the same result of Judd and Rietsch to see that $\mathcal{W}_P$ has a unique critical point in each fibre $Z_{P,t^\lambda}(\mathbf{K}_{>0})$, which we call the positive critical point of $\mathcal{W}_{P,t^{\lambda}}$, denoted $p_{P,\lambda}$.

As before, this critical point $p_{P,\lambda} \in Z_{P,t^{\lambda}}(\mathbf{K}_{>0})$ defines a point $p_{P,\lambda}^{\mathrm{trop}} \in \mathrm{Trop}(Z_{P,t^{\lambda}})$, called the tropical critical point of $\mathcal{W}_{P,t^{\lambda}}$. Explicitly, using a positive chart (such as $\phi_{P,t^{\lambda},\boldsymbol{m}}$) we apply the valuation $\mathrm{Val}_{\mathbf{K}}$ to every coordinate of $p_{P,\lambda}$. This gives rise to the corresponding point $p_{P,\lambda, \boldsymbol{\mu}}^{\mathrm{trop}}$ in the associated tropical chart $\mathrm{Trop}(Z_{P,t^{\lambda}}) \to \mathbb{R}^{\mathcal{V}_P^{\bullet}}$.
Moreover, for a choice of positive chart the tropical critical point lies in the interior of the respective superpotential polytope, for example, $p_{P,\lambda, \boldsymbol{\mu}}^{\mathrm{trop}} \in \mathcal{P}_{P,\lambda,\boldsymbol{\mu}}$. We recall that an explicit statement in full generality is given by Judd and Rietsch in \cite[Theorem 1.2]{JuddRietsch2019}.

In the $G/B$ case we showed that the image of the tropical critical point under the weight projection, $\mathrm{Trop}(\mathrm{wt})\left(p_{\lambda}^{\mathrm{trop}}\right)$, is the centre of mass of the weight polytope (Corollary \ref{cor weight matrix at trop crit point}). In order to give the analogous result in the $G/P$ case we make the following notational convention:
    $$\left\langle \lambda, \epsilon_k^{\vee} \right\rangle = \lambda_r \quad \text{where } n_{r-1}+1 \leq k \leq n_r \text{ for some } r \in \{1\ldots, l+1\}.
    $$
For example if $G/P=\mathcal{F}_{2,5,6}(\mathbb{C}^8)$ then we write $\lambda = \left( \lambda_1, \lambda_1, \lambda_2, \lambda_2, \lambda_2, \lambda_3, \lambda_4, \lambda_4 \right)$ for some $\lambda_1, \ldots, \lambda_4$.

With this notation we give a new definition of $\ell$ as follows:
    \begin{equation} \label{eqn defn G/P ell}
    \ell := \frac1n \sum_{i=1}^{l+1} k_i\lambda_i
    \end{equation}
observing that this descends to the original definition in the $G/B$ case. We are now ready to generalise Corollary \ref{cor weight matrix at trop crit point} to the $G/P$ case:

\begin{cor}[Corollary of Proposition {\ref{prop G/P weight at crit point non trop}}] \label{cor G/P weight matrix at trop crit point}
Given $\lambda=(\lambda_1\geq \lambda_2 \geq \ldots \geq \lambda_n)$, the weight matrix at the critical point in the fibre over $t^{\lambda} \in \left(T^{\vee}(\mathbf{K}_{>0})\right)^{W_P}$ is an $n\times n$ matrix $\mathrm{diag}\left(t^{\ell}, \ldots, t^{\ell}\right)$ where
    $$\ell = \frac{1}{n}\sum_{i=1}^{l+1} k_i \lambda_i.
    $$
\end{cor}
\hfill\qedsymbol

% \subsection[Ideal fillings (\texorpdfstring{$G/P$}{G/P} setting)]{Ideal fillings ($G/P$ setting)}

\newpage
\subsection{Ideal fillings} \label{subsec G/P Ideal fillings}
\fancyhead[L]{10.3 \ \ Ideal fillings}
% \fancyhead[L]{11.3 \ \ Ideal fillings}

In order to generalise our description of ideal fillings from the full flag case to that of partial flag varieties, we first recall that fillings in the $G/B$ case (see Section \ref{subsec Ideal fillings}) are described by the assignment of non-negative real numbers to boxes in upper triangular form, with a one-to-one correspondence between these boxes and the dot vertices in the respective quiver.  It follows that in the $G/P$ case, the assignment will be to boxes in block upper triangular form, with an approach analogous to the construction of the $Q_P$ quiver given in Section \ref{subsec Constructing quivers}.

We begin, as in Section \ref{subsec Constructing quivers}, with an $n \times n$ square containing the $k_i \times k_i$ squares, $L_i$, on the leading diagonal. The boxes for our filling are given by the unit squares strictly above the $L_i$ squares. These are filled with numbers $n_{ij}$, with the labelling similar to the entries of an $n\times n$ matrix. For example see Figure \ref{fig filling for 2,5,6 n8}. The filling is called integral if all the $n_{ij}$ are integral.

The filling is said to be ideal if $n_{ij}=\max\{n_{i+1 , j}, n_{i , j-1} \} $ for $j-i\geq 2$. Note that if $n_{i+1 , j}$ doesn't appear in the filling, then the condition degenerates to $n_{ij}=n_{i \; j-1} $, and similarly if $n_{i , j-1}$ doesn't appear. For example see Figure \ref{fig shaded filling for 2,5,6 n8}, where we have shaded neighbouring boxes identically if they automatically take the same value due to the degeneration of ideal filling conditions.

\begin{figure}[ht]
\centering
\begin{minipage}[t]{0.42\textwidth}
\centering
\begin{tikzpicture}[scale=0.8]
    %squares
        \draw[dotted, thick, color=black!50] (0.5,0.5) -- (0.5,8.5) -- (8.5,8.5) -- (8.5,0.5) -- cycle;
        \draw[dotted, thick, color=black!50] (2.5,3.5) -- (2.5,8.5);
        \draw[dotted, thick, color=black!50] (5.5,2.5) -- (5.5,6.5);
        \draw[dotted, thick, color=black!50] (6.5,0.5) -- (6.5,3.5);
        \draw[dotted, thick, color=black!50] (0.5,6.5) -- (5.5,6.5);
        \draw[dotted, thick, color=black!50] (2.5,3.5) -- (6.5,3.5);
        \draw[dotted, thick, color=black!50] (5.5,2.5) -- (8.5,2.5);

        \draw (2.5,8.5) -- (8.5,8.5) -- (8.5,2.5) -- (6.5,2.5) -- (6.5,3.5) -- (5.5,3.5) -- (5.5,6.5) -- (2.5,6.5) -- cycle;
        \draw (3.5,8.5) -- (3.5,6.5);
        \draw (4.5,8.5) -- (4.5,6.5);
        \draw (5.5,8.5) -- (5.5,6.5);
        \draw (6.5,8.5) -- (6.5,3.5);
        \draw (7.5,8.5) -- (7.5,2.5);

        \draw (2.5,7.5) -- (8.5,7.5);
        \draw (5.5,6.5) -- (8.5,6.5);
        \draw (5.5,5.5) -- (8.5,5.5);
        \draw (5.5,4.5) -- (8.5,4.5);
        \draw (6.5,3.5) -- (8.5,3.5);

    %square labels
        \node[color=black!50] at (1.5,7.5) {$L_1$};
        \node[color=black!50] at (4,5) {$L_2$};
        \node[color=black!50] at (6,3) {$L_3$};
        \node[color=black!50] at (7.5,1.5) {$L_4$};

    %filling labels
        \node at (3,8) {$n_{13}$};
        \node at (4,8) {$n_{14}$};
        \node at (5,8) {$n_{15}$};
        \node at (6,8) {$n_{16}$};
        \node at (7,8) {$n_{17}$};
        \node at (8,8) {$n_{18}$};

        \node at (3,7) {$n_{23}$};
        \node at (4,7) {$n_{24}$};
        \node at (5,7) {$n_{25}$};
        \node at (6,7) {$n_{26}$};
        \node at (7,7) {$n_{27}$};
        \node at (8,7) {$n_{28}$};

        \node at (6,6) {$n_{36}$};
        \node at (7,6) {$n_{37}$};
        \node at (8,6) {$n_{38}$};

        \node at (6,5) {$n_{46}$};
        \node at (7,5) {$n_{47}$};
        \node at (8,5) {$n_{48}$};

        \node at (6,4) {$n_{56}$};
        \node at (7,4) {$n_{57}$};
        \node at (8,4) {$n_{58}$};

        \node at (7,3) {$n_{67}$};
        \node at (8,3) {$n_{68}$};
    \end{tikzpicture}
\caption{Filling for $G/P=\mathcal{F}_{2,5,6}(\mathbb{C}^8)$} \label{fig filling for 2,5,6 n8}
\end{minipage}
\hspace{1.4cm}
\begin{minipage}[t]{0.42\textwidth}
\centering
\begin{tikzpicture}[scale=0.8]
    %squares
        \draw[dotted, thick, color=black!50] (0.5,0.5) -- (0.5,8.5) -- (8.5,8.5) -- (8.5,0.5) -- cycle;
        \draw[dotted, thick, color=black!50] (2.5,3.5) -- (2.5,8.5);
        \draw[dotted, thick, color=black!50] (5.5,2.5) -- (5.5,6.5);
        \draw[dotted, thick, color=black!50] (6.5,0.5) -- (6.5,3.5);
        \draw[dotted, thick, color=black!50] (0.5,6.5) -- (5.5,6.5);
        \draw[dotted, thick, color=black!50] (2.5,3.5) -- (6.5,3.5);
        \draw[dotted, thick, color=black!50] (5.5,2.5) -- (8.5,2.5);

        \fill[black!5] (2.5,8.5) -- (5.5,8.5) -- (5.5,6.5) -- (2.5, 6.5) -- cycle;
        \fill[black!15] (5.5,6.5) -- (6.5,6.5) -- (6.5,3.5) -- (5.5, 3.5) -- cycle;
        \fill[black!25] (6.5,3.5) -- (8.5,3.5) -- (8.5,2.5) -- (6.5, 2.5) -- cycle;

        \fill[black!35] (5.5,8.5) -- (6.5,8.5)-- (6.5,6.5) -- (5.5,6.5)  -- cycle;

        \fill[black!45] (6.5,6.5) -- (8.5,6.5)-- (8.5,3.5) -- (6.5,3.5)  -- cycle;
        \fill[black!55] (6.5,8.5) -- (8.5,8.5)-- (8.5,6.5) -- (6.5,6.5)  -- cycle;

        \draw (2.5,8.5) -- (8.5,8.5) -- (8.5,2.5) -- (6.5,2.5) -- (6.5,3.5) -- (5.5,3.5) -- (5.5,6.5) -- (2.5,6.5) -- cycle;
        \draw (3.5,8.5) -- (3.5,6.5);
        \draw (4.5,8.5) -- (4.5,6.5);
        \draw (5.5,8.5) -- (5.5,6.5);
        \draw (6.5,8.5) -- (6.5,3.5);
        \draw (7.5,8.5) -- (7.5,2.5);

        \draw (2.5,7.5) -- (8.5,7.5);
        \draw (5.5,6.5) -- (8.5,6.5);
        \draw (5.5,5.5) -- (8.5,5.5);
        \draw (5.5,4.5) -- (8.5,4.5);
        \draw (6.5,3.5) -- (8.5,3.5);

    %square labels
        \node[color=black!50] at (1.5,7.5) {$L_1$};
        \node[color=black!50] at (4,5) {$L_2$};
        \node[color=black!50] at (6,3) {$L_3$};
        \node[color=black!50] at (7.5,1.5) {$L_4$};

    %filling labels
        \node at (3,8) {$n_{13}$};
        \node at (4,8) {$n_{14}$};
        \node at (5,8) {$n_{15}$};
        \node at (6,8) {$n_{16}$};
        \node at (7,8) {$n_{17}$};
        \node at (8,8) {$n_{18}$};

        \node at (3,7) {$n_{23}$};
        \node at (4,7) {$n_{24}$};
        \node at (5,7) {$n_{25}$};
        \node at (6,7) {$n_{26}$};
        \node at (7,7) {$n_{27}$};
        \node at (8,7) {$n_{28}$};

        \node at (6,6) {$n_{36}$};
        \node at (7,6) {$n_{37}$};
        \node at (8,6) {$n_{38}$};

        \node at (6,5) {$n_{46}$};
        \node at (7,5) {$n_{47}$};
        \node at (8,5) {$n_{48}$};

        \node at (6,4) {$n_{56}$};
        \node at (7,4) {$n_{57}$};
        \node at (8,4) {$n_{58}$};

        \node at (7,3) {$n_{67}$};
        \node at (8,3) {$n_{68}$};
    \end{tikzpicture}
\caption{Filling for $G/P=\mathcal{F}_{2,5,6}(\mathbb{C}^8)$ (with shading)} \label{fig shaded filling for 2,5,6 n8}
\end{minipage}
\end{figure}

\begin{defn}\label{def ideal filling for lambda G/P version}
We say that an ideal filling $\{n_{ij}\}$ is an ideal filling for $\lambda$ if
    $$\sum_{(i,j) \ : \ v_{ji}\in \mathcal{V}_P^{\bullet}} n_{ij}\alpha_{ij} + \ell \sum \epsilon_i  = \lambda,
    $$
where we recall the definition
    \begin{equation*}
    \ell := \frac1n \sum_{i=1}^{l+1} k_i\lambda_i
    \end{equation*}
given in (\ref{eqn defn G/P ell}).
\end{defn}

Of note, this definition descends to Definition \ref{def ideal filling for lambda GLn version} in the $G/B$ case.

With the above definition in mind, we generalise Proposition \ref{prop GLn version of Jamie's 6.2} to the $G/P$ setting:
\begin{prop}%[Generalisation of {$GL_n/B$} version of Proposition {\ref{prop GLn version of Jamie's 6.2}} to the {$G/P$} case]% {\cite[Proposition 6.2]{Judd2018}}]
\label{prop G/P version of Jamie's 6.2}
% Let $\lambda$ be a dominant weight and $\ell := \frac1n \sum_{i=1}^{l+1} k_i\lambda_i$.
For any solution $(\rho_a)$ to the tropical critical point conditions for $\lambda$ in the quiver $Q_P$, (that is, (\ref{eqn trop crit point conds}) taken over $\mathcal{V}^{\bullet}_P$), the formula
    $$n_{ij}= \min_{a:h(a)=v_{ji}}\{\rho_a\} = \min_{a:t(a)=v_{ji}}\{\rho_a\}
    $$
defines an ideal filling for $\lambda$ (see Definition \ref{def ideal filling for lambda G/P version}), and every ideal filling arises in this way. In particular we see that for a given $\lambda$, the ideal filling for $\lambda$ exists and is unique.
% We have a bijective correspondence:
%     $$\begin{Bmatrix} \text{solutions to the tropical critical}\\ \text{conditions in } Q_P \text{ with highest weight } \lambda \end{Bmatrix} \leftrightarrow \begin{Bmatrix} \text{ideal fillings } \{n_{ij}\}_{(i,j) \, : \, v_{ji} \in \mathcal{V}^{\bullet}_P} \text{ for }\lambda, \\ \text{i.e. such that } \sum_{(i,j) \, : \, v_{ji}\in \mathcal{V}_P^{\bullet}} n_{ij}\alpha_{ij} + \ell \sum \epsilon_i  = \lambda \end{Bmatrix}
    % $$
\end{prop}

\begin{proof} The proof is similar to the $G/B$ case, in which we give a pair of maps between two sets which are inverse to each other. In this setting we will form a bijective correspondence:
    $$\begin{Bmatrix} \text{solutions to the tropical critical}\\ \text{conditions in } Q_P \text{ with highest weight } \lambda \end{Bmatrix} \leftrightarrow \begin{Bmatrix} \text{ideal fillings } \{n_{ij}\}_{(i,j) \, : \, v_{ji} \in \mathcal{V}^{\bullet}_P} \text{ for }\lambda, \\ \text{i.e. such that } \sum_{(i,j) \, : \, v_{ji}\in \mathcal{V}_P^{\bullet}} n_{ij}\alpha_{ij} + \ell \sum \epsilon_i  = \lambda \end{Bmatrix}
    $$
It suffices to simply give an outline of the proof in the $G/P$ setting, which highlights the necessary generalisations.

\medskip
\noindent
\textbf{Map from ideal fillings for $\lambda$ to solutions to the tropical critical conditions.}

Let $\{n_{ij}\}$ be an ideal filling for $\lambda$. We note that for $n_{ji}$ to appear in the filling, the pair $(i,j)$ must satisfy $1\leq i<j\leq n$ as before, but we also now require that $v_{ji} \in \mathcal{V}_P^{\bullet}$.

For each pair $(i,j)$ such that $1\leq i\leq j\leq n$, we define (as in the $G/B$ case) two sums of entries of the ideal filling; those $n_{il}$ strictly to the right of $n_{ij}$ and those $n_{lj}$ strictly above $n_{ij}$ respectively:
    \begin{equation} \label{eqn defn Hh and Hv partial flag case}
    H^h_{ij}:= \sum_{\substack{j<l\leq n \\ v_{li} \in \mathcal{V}_P^{\bullet} }} n_{il}, \quad
    H^v_{ij}:= \sum_{\substack{1\leq l<i \\ v_{jl} \in \mathcal{V}_P^{\bullet} }} n_{lj}.
    \end{equation}
We observe that although this definition is very similar to the one in the $G/B$ case, some of the $H^h_{ij}$ or $H^v_{ij}$ will automatically be equal due to the structure of ideal fillings and quivers for partial flag varieties; for example, if $k_1\geq 2$ then $H^h_{11} = H^h_{12} = H^h_{22}$. In general we see that for $r \in \{1\ldots, l+1\}$, all $H^h_{ij}$ with $n_{r-1}+1 \leq i \leq j \leq n_r$ are equal, and similarly for $H^v_{ij}$.

Taking our new definition of $\ell$, that is $\ell := \frac1n \sum_{i=1}^{l+1} k_i\lambda_i$, we again define a map from ideal fillings for $\lambda$ to tropical vertex coordinates of the quiver as follows:
    $$\delta_{v_{ji}}:=H^h_{ij}-H^v_{ij}+\ell.
    $$
We need to show that this defines a solution to the tropical critical conditions for $\lambda$.

As before, the addition of $\ell$ in the above definition doesn't affect the tropical arrow coordinates. Indeed, the vertical arrow coordinates for $1\leq i \leq j < n$, and the horizontal arrow coordinates for $1\leq i < j \leq n$, are respectively
    $$\delta_{v_{ji}}-\delta_{v_{j+1,i}}=H^v_{i+1,j+1} - H^v_{ij}
    \quad \text{and} \quad
    \delta_{v_{ji}}-\delta_{v_{j,i+1}}=H^h_{i,j-1} - H^h_{i+1,j}.
    $$
Both of these are $\geq 0$, so it follows that the point lies in $\left\{\mathrm{Trop}(\mathcal{W}_{t^{\lambda}})\geq 0\right\}$. Additionally, we see it will lie in the fibre over $\lambda$ as follows: for $\epsilon_k^{\vee} \in X^*(T^{\vee})$ with $n_{r-1}+1 \leq k \leq n_r$ for some $r \in \{1\ldots, l+1\}$, we have
    $$\begin{aligned}
    \lambda_r = \left\langle \lambda, \epsilon_k^{\vee} \right\rangle
    &= \bigg\langle \sum_{\substack{1\leq i < j \leq n \\ v_{ji} \in \mathcal{V}_P^{\bullet} }} n_{ij}(\epsilon_i-\epsilon_j) + \ell \sum_{1\leq i \leq n}\epsilon_i , \  \epsilon_k^{\vee} \ \bigg\rangle & \text{since } \{n_{ij}\} \text{ is an ideal filling for } \lambda \\
    &= \sum_{\substack{k < j \leq n \\ v_{jk} \in \mathcal{V}_P^{\bullet} }} n_{kj} - \sum_{\substack{1\leq i < k \\ v_{ki} \in \mathcal{V}_P^{\bullet} }} n_{ik} + \ell \\
    &= H^h_{kk}-H^v_{kk}+\ell \\
    &= H^h_{n_{r-1}+1, n_r}-H^v_{n_{r-1}+1, n_r}+\ell = \delta_{v_{n_r, n_{r-1}+1}}
    \end{aligned}
    $$
and we note that $v_{n_r, n_{r-1}+1}$ is exactly the star vertex in the square $L_r$, as required. %** Need to state the requirement..?! **

It remains to show that the point we have defined satisfies the tropical critical point conditions. As in the $G/B$ case, we require a lemma:
\begin{lem}%[{\cite[Lemma 6.7]{Judd2018}}]
For $1 \leq i < j \leq n$, write $\bar{H}^h_{ij}:= H^h_{ij}+n_{ij}$, $\bar{H}^v_{ij}:= H^v_{ij} + n_{ij}$.
Then if $j-i\geq 1$, either
    $$\bar{H}^v_{i,j+1}=\bar{H}^v_{ij} \quad \text{or} \quad \bar{H}^h_{i,j+1}=\bar{H}^h_{i+1,j+1}
    $$
or both are true. Hence we have $\min\left\{ \bar{H}^v_{i,j+1} - \bar{H}^v_{ij}, \bar{H}^h_{i,j+1}-\bar{H}^h_{i+1,j+1} \right\} = 0$.
\end{lem}
This is the same lemma as in the previous setting with its statement unaffected by the alteration to our definitions of $H^h_{ij}$ and $H^v_{ij}$ in (\ref{eqn defn Hh and Hv partial flag case}). However, in this more general setting we need to give a sight adjustment of Judd's proof:

\begin{proof}
By the ideal filling conditions we have $\bar{H}^v_{ij} \leq \bar{H}^v_{i,j+1}$ and $\bar{H}^h_{i+1,j+1} \leq \bar{H}^h_{i,j+1}$. If $n_{r-1}+1 \leq i \leq n_r$ and $n_r+1 \leq j \leq n_r-1$, for $r \in \{ 1, \ldots, l \}$, then the ideal filling conditions become degenerate and we have $\bar{H}^v_{ij} = \bar{H}^v_{i,j+1}$. Similarly if $n_{r-1}+1 \leq i \leq n_r-1$ and $n_r+1 \leq j \leq n_r$, for $r \in \{ 1, \ldots, l \}$, then automatically $\bar{H}^h_{i+1,j+1} = \bar{H}^h_{i,j+1}$.

It remains to consider when $\bar{H}^v_{ij} < \bar{H}^v_{i,j+1}$. For this to be true there must exist some $l$ such that $1\leq l\leq i$ and $n_{lj} < n_{l,j+1}$. Hence we see that
    $$\max_{\substack{l\leq k \leq j-1 \\ v_{k+1,k \in \mathcal{V}^{\bullet}}}}\{ n_{k,k+1} \} < \max_{\substack{l\leq k \leq j \\ v_{k+1,k \in \mathcal{V}^{\bullet}}}}\{ n_{k,k+1} \} = n_{j,j+1}.
    $$
In particular, since $l\leq i \leq j-1$ we must have $n_{i,i+1} < n_{j,j+1} $. Then for $r\geq j+1$ we have
    $$n_{ir} = \max_{\substack{i\leq k \leq r-1 \\ v_{k+1,k \in \mathcal{V}^{\bullet}}}}\{ n_{k,k+1} \} = \max_{\substack{i+1\leq k \leq r-1 \\ v_{k+1,k \in \mathcal{V}^{\bullet}}}}\{ n_{k,k+1} \} = n_{i+1,r}.
    $$
It follows that $\bar{H}^h_{i+1,j+1} = \bar{H}^h_{i,j+1}$, and so the proof is complete.
\end{proof}

We first consider the minimum over incoming arrow coordinates at dot vertices. Let $v_{ji} \in \mathcal{V}_P^{\bullet}$ be a dot vertex with two incoming arrows. Then the minimum over incoming arrow coordinates at $v_{ji}$ is
    $$\min\left\{ H^v_{i+1,j+1} - H^v_{ij}, H^h_{i,j-1}-H^h_{i+1,j} \right\} = n_{ij} + \min\left\{ \bar{H}^v_{i,j+1} - \bar{H}^v_{ij}, \bar{H}^h_{i,j+1}-\bar{H}^h_{i+1,j+1} \right\} = n_{ij}.
    $$
There are two cases where a dot vertex $v_{ij}$ has only one incoming arrow. The first is when $v_{ji}$ lies on the bottom wall of the quiver, with the incoming arrow coordinate given by $H^h_{i,n-1}-H^h_{i+1,n}=n_{in}$ as desired. Secondly, if $v_{ji}$ lies directly to the left of a square $L_r$, in a row without a star vertex, that is $n_{r-1}+1 \leq j \leq n_r -1 $ and $i=n_{r-1}$, then the single incoming arrow has coordinate
    $$H^v_{i+1,j+1} - H^v_{ij} = n_{r-1}n_{ij} -(n_{r-1}-1)n_{ij} = n_{ij}
    $$
where the first equality is a consequence of degenerate ideal filling conditions.

Now considering outgoing arrows, let $v_{ji} \in \mathcal{V}_P^{\bullet}$ be a dot vertex with two outgoing arrows. Then the minimum over outgoing arrow coordinates at $v_{ji}$ is
    $$\min\left\{ H^v_{i+1,j} - H^v_{i,j-1}, H^h_{i-1,j-1}-H^h_{ij} \right\} = n_{ij} + \min\left\{ \bar{H}^v_{i-1,j} - \bar{H}^v_{i-1,j-1}, \bar{H}^h_{i-1,j}-\bar{H}^h_{ij} \right\} = n_{ij}.
    $$
There are similarly two cases where a dot vertex $v_{ij}$ has only one outgoing arrow. The first is when $v_{ji}$ lies on the left wall of the quiver, with the outgoing arrow coordinate given by $H^v_{2,j}-H^v_{1,j-1}=n_{1j}$ as desired. Secondly, if $v_{ji}$ lies directly below a square $L_r$, in a column without a star vertex, that is $j=n_r+1$ and $n_{r-1}+2 \leq i \leq n_r$, then the single outgoing arrow has coordinate
    $$H^h_{i-1,j-1} - H^h_{ij} = (n-n_r)n_{ij} -(n-n_r-1)n_{ij} = n_{ij}
    $$
where the first equality is a consequence of degenerate ideal filling conditions.

Thus the tropical critical point conditions are satisfied at all dot vertices and so our point is indeed a tropical critical point for $\lambda$, as required.

\medskip

\noindent
\textbf{Map from solutions to the tropical critical conditions to ideal fillings.}

Suppose $(\rho_a)_{\ \in \mathcal{A}}$ is a solution to the tropical critical conditions for $\lambda$. Then for $v \in \mathcal{V}_P^{\bullet}$ we consider the map
    $$\pi : \mathcal{V}_P^{\bullet} \to \mathbb{R}, \quad \pi(v):=\min_{a:h(a)=v}\{\rho_a\}% = \min_{a:t(a)=v}\{\rho_a\}
    .
    $$
We also recall a lemma from the $G/B$ case:
\begin{lem} %\label{lem trop crit pt is an ideal filling}
At a tropical critical point, the filling $\{n_{ij} = \pi(v_{ji})\}$ is an ideal filling. That is, if we have the following sub-diagram
\begin{center}
\begin{tikzpicture}
    %dots and stars
    \node (31) at (0,0) {$\bullet$};
        \node at (-0.2,-0.2) {\scriptsize{$v$}};
    \node (21) at (0,1) {$\bullet$};
        \node at (-0.2,1.2) {\scriptsize{$w$}};
    \node (32) at (1,0) {$\bullet$};
        \node at (1.2,-0.2) {\scriptsize{$u$}};

    %arrows
    \draw[->] (31) -- (21);
    \draw[->] (32) -- (31);

    %arrow labels
    \node at (-0.2,0.5) {\scriptsize{$a$}};
    \node at (0.5,-0.2) {\scriptsize{$b$}};
\end{tikzpicture}
\end{center}
then we must have $\pi(v)=\max\{ \pi(u), \pi(w)\}$.
\end{lem}
This lemma and its proof hold in both the $G/B$ and $G/P$ cases so we omit the proof here.

We will show that $\{ n_{ij}=\pi(v_{ij}) \}$ is an ideal filling for $\lambda$. To do this, we need the vertex coordinates of the quiver at the tropical critical point, which we denote by $(\delta_v)_{v \in \mathcal{V}}$. In particular we notice that at the bottom left vertex we have
    \begin{align*}
    \delta_{v_{n1}}
    &= \mathrm{Val}_{\mathbf{K}}\left(\frac{\Xi_{P,n}}{\Xi_{P,n+1}}\right) &\text{by definition of } \Xi_{P,i} \text{ given in (\ref{eqn G/P xi for wt map defn})
    , and noting } \Xi_{P,n+1}=1 \\
    &= \mathrm{Val}_{\mathbf{K}}(t_{P,n}) & \text{recalling the $t_{P,i}$ defined in (\ref{eqn G/P defn gamma and t})}
    \\
    &= \ell & \text{by Corollary \ref{cor G/P weight matrix at trop crit point}.}
    \end{align*}
We recall another lemma from the $G/B$ case:
\begin{lem}%[Generalisation of {\cite[Lemma 6.9]{Judd2018}}] \label{lem Generalisation of Judd's Lemma 6.9}
For $v \in \mathcal{V}_P$ we write $\mathrm{bel}(v)$ and $\mathrm{lef}(v)$ for the sets of vertices directly below and directly to the left of $v$ respectively. Then at a tropical critical point we have
    $$\delta_v=\sum_{w \in \mathrm{bel}(v)} \pi(w) - \sum_{w \in \mathrm{lef}(v)} \pi(w) +\ell.
    $$
\end{lem}
The proof of this lemma given in the $G/B$ case also holds in the $G/P$ case.

Using this lemma we see that at a tropical critical point, the ideal filling $\{ n_{ij}=\pi(v_{ij}) \}$ is an ideal filling for $\lambda$ as follows: for $\epsilon_k^{\vee} \in X^*(T^{\vee})$ with $n_{r-1}+1 \leq k \leq n_r$ for some $r \in \{1\ldots, l+1\}$, we have
    $$\begin{aligned}
    \bigg\langle \sum_{\substack{ 1\leq i < j \leq n \\ v_{ji} \in \mathcal{V}_P^{\bullet} }} n_{ij}\alpha_{ij} +\ell \sum_{1\leq i \leq n}\epsilon_i, \ \epsilon^{\vee}_k \bigg\rangle
    % &= \langle \sum_{\substack{ 1\leq i < j \leq n \\ v_{ji} \in \mathcal{V}_P^{\bullet} }} n_{ij}(\epsilon_i-\epsilon_j) +\ell \sum_{1\leq i \leq n}\epsilon_i, \epsilon^{\vee}_k \ \rangle  \\
    &=  \sum_{\substack{ k < j \leq n \\ v_{jk} \in \mathcal{V}_P^{\bullet} }} n_{kj} - \sum_{\substack{ 1\leq i < k \\ v_{ki} \in \mathcal{V}_P^{\bullet} }} n_{ik} +\ell \\
    &= \sum_{\substack{ k < j \leq n \\ v_{jk} \in \mathcal{V}_P^{\bullet} }} \pi(v_{jk}) - \sum_{\substack{ 1\leq i < k \\ v_{ki} \in \mathcal{V}_P^{\bullet} }} \pi(v_{ki}) + \ell \\
    &= \sum_{w \in \mathrm{bel}(v_{n_r, n_{r-1}+1})} \pi(w) - \sum_{w \in \mathrm{lef}(v_{n_r, n_{r-1}+1})} \pi(w) + \ell \\
    &= \delta_{v_{n_r, n_{r-1}+1}} = \lambda_r
    \end{aligned}$$
with the last equality a consequence of the fact that $v_{n_r, n_{r-1}+1}$ is the star vertex in the square $L_r$.

To complete the proof of Proposition \ref{prop G/P version of Jamie's 6.2}, we note that the maps defined above are inverse to each other by construction. %, and preserve integrality exactly when $\ell \in \mathbb{Z}$.
\end{proof}

The following corollary is the $G/P$ analogue of Corollary \ref{cor nu'_i coords ideal filling for lambda}:

\begin{cor} \label{cor G/P nu'_i coords ideal filling for lambda}
% The $m_i$ coordinates on $Z_P$ give rise to ideal fillings for $\lambda$ at critical points.
Let the positive critical point $p_{P,\lambda} \in Z_{P,t^{\lambda}}(\mathbf{K}_{>0})$ of $\mathcal{W}_{P,t^{\lambda}}$ be written in the ideal coordinates $\boldsymbol{m}$. Then the valuations $\mu_k=\mathrm{Val}_{\mathbf{K}}(m_k)$ defining the tropical critical point, $p_{P,\lambda, \boldsymbol{\mu}}^{\mathrm{trop}}$, give rise to an ideal filling $\left\{n_{ij}= \mu_{s_i+j-i} \right\}_{(i,j) \, : \, v_{ji} \in \mathcal{V}^{\bullet}_P}$ for $\lambda$ (where we recall the definition of $s_i$ given in Section \ref{sec G/P The ideal coordinates}).
\end{cor}

\begin{proof}
By Proposition \ref{prop G/P crit points, sum at vertex is nu_i}, at a critical point we have
    $$ \varpi(v_{ji}) = \sum_{a:t(a)=v_{ji}} r_a = m_{s_i +j-i}.$$
Thus by Proposition \ref{prop G/P version of Jamie's 6.2} we see that
    $$n_{ij} = \pi(v_{ji}) =\mathrm{Val}_{\mathbf{K}}(\varpi(v_{ji})) = \mathrm{Val}_{\mathbf{K}}(m_{s_i +j-i})
    $$
defines an ideal filling for $\lambda$.
\end{proof}

The next result generalises Proposition \ref{prop trop crit pt independent of i} from the $G/B$ case, with the same proof:
\begin{prop} \label{prop G/P ideal filling indep of red expr}
For a given highest weight $\lambda$, the ideal filling for $\lambda$ %, and consequently the tropical critical point $p_{P, \lambda, \boldsymbol{\mu}'}^{\mathrm{trop}}$,
is independent of the choice of reduced expression $\mathbf{i}$ for $w_Pw_0$.
\end{prop}

We conclude with the following theorem which generalises the example %\ref{ex Toeplitz mat and ideal fillings} %Theorem \ref{thm Toeplitz mat and ideal fillings}
from the introduction. It gives an interpretation of ideal fillings using Toeplitz matrices over generalised Puiseux series.
\begin{thm} \label{thm G/P Toeplitz mat and ideal fillings}
Let $\phi_i: SL_2 \to GL_n^{\vee}$ be the homomorphism corresponding to the $i$-th simple root of $GL_n^{\vee}$ and take
    $$\mathbf{y}^{\vee}_i(z) = \phi_i\begin{pmatrix} 1 & 0 \\ z & 1 \end{pmatrix}.
    $$
Let $\mathbf{i}=(i_1, \ldots, i_M)$ stand for an arbitrary reduced expression $s_{i_1} \cdots s_{i_M}$ for $w_Pw_0$. Then we have an ordering on the set of positive roots $R^P_+$ % $R_+=\{\alpha_{ij}=\epsilon_i-\epsilon_j \ | \ 1\leq i<j \leq n \}$
given by
    $$\alpha^{\mathbf{i}}_j=\begin{cases}
    \alpha_{i_1, i_1+1} & \text{for } j=1, \\
    s_{i_1}\cdots s_{i_{j-1}}\alpha_{i_j, i_j+1} & \text{for } j=2, \ldots, M.
    \end{cases}
    $$
Now take $m_{\alpha^{\mathbf{i}}_1},\ldots, m_{\alpha^{\mathbf{i}}_M}$ to be generalised Puiseux series with positive leading coefficients and non-negative valuations $\mu_{\alpha}=\mathrm{Val}_{\mathbf{K}}(m_{\alpha})$ (defined in Section \ref{subsec The basics of tropicalisation}). If the product
$\mathbf{y}^{\vee}_{i_1}\left(m_{\alpha^{\mathbf{i}}_1}^{-1}\right)\cdots  \mathbf{y}^{\vee}_{i_M}\left(m_{\alpha^{\mathbf{i}}_M}^{-1}\right)$
is a Toeplitz matrix, then the valuations $\mu_{\alpha}$ form an ideal filling:
\begin{center}
\begin{tikzpicture}[scale=1.3]
    %squares
    %horizontal
        \draw (0.5,8.5) -- (1.6,8.5);
        \draw (2,8.5) -- (4.2,8.5);
        \draw (4.6,8.5) -- (5.8,8.5);
        \draw (6.2,8.5) -- (7.3,8.5);
        \draw (0.5,7.5) -- (1.6,7.5);
        \draw (2,7.5) -- (4.2,7.5);
        \draw (4.6,7.5) -- (5.8,7.5);
        \draw (6.2,7.5) -- (7.3,7.5);
            \node at (1,7.3) {\scriptsize{$\vdots$}};
            \node at (2.6,7.3) {\scriptsize{$\vdots$}};
            \node at (3.6,7.3) {\scriptsize{$\vdots$}};
            \node at (5.2,7.3) {\scriptsize{$\vdots$}};
            \node at (6.8,7.3) {\scriptsize{$\vdots$}};
        \draw (0.5,6.9) -- (1.6,6.9);
        \draw (2,6.9) -- (4.2,6.9);
        \draw (4.6,6.9) -- (5.8,6.9);
        \draw (6.2,6.9) -- (7.3,6.9);
        \draw (0.5,5.9) -- (1.6,5.9);
        \draw (2,5.9) -- (4.2,5.9);
        \draw (4.6,5.9) -- (5.8,5.9);
        \draw (6.2,5.9) -- (7.3,5.9);
            \node at (3.6,5.7) {\scriptsize{$\vdots$}};
            \node at (5.2,5.7) {\scriptsize{$\vdots$}};
            \node at (6.8,5.7) {\scriptsize{$\vdots$}};
        \draw (3.1,5.3) -- (4.2,5.3);
        \draw (4.6,5.3) -- (5.8,5.3);
        \draw (6.2,5.3) -- (7.3,5.3);
        \draw (3.1,4.3) -- (4.2,4.3);
        \draw (4.6,4.3) -- (5.8,4.3);
        \draw (6.2,4.3) -- (7.3,4.3);
            \node at (6.8,4.1) {\scriptsize{$\vdots$}};
        \draw (6.2,3.7) -- (7.3,3.7);
        \draw (6.3,2.7) -- (7.3,2.7);

    %vertical
        \draw (0.5,8.5) -- (0.5,7.4);
        \draw (0.5,7) -- (0.5,5.9);
        \draw (1.5,8.5) -- (1.5,7.4);
        \draw (1.5,7) -- (1.5,5.9);
            \node at (1.8,8) {\scriptsize{$\cdots$}};
            \node at (1.8,6.4) {\scriptsize{$\cdots$}};
        \draw (2.1,8.5) -- (2.1,7.4);
        \draw (2.1,7) -- (2.1,5.9);
        \draw (3.1,8.5) -- (3.1,7.4);
        \draw (3.1,7) -- (3.1,5.8);
        \draw (3.1,5.4) -- (3.1,4.3);
        \draw (4.1,8.5) -- (4.1,7.4);
        \draw (4.1,7) -- (4.1,5.8);
        \draw (4.1,5.4) -- (4.1,4.3);
            \node at (4.4,8) {\scriptsize{$\cdots$}};
            \node at (4.4,6.4) {\scriptsize{$\cdots$}};
            \node at (4.4,4.8) {\scriptsize{$\cdots$}};
        \draw (4.7,8.5) -- (4.7,7.4);
        \draw (4.7,7) -- (4.7,5.8);
        \draw (4.7,5.4) -- (4.7,4.3);
        \draw (5.7,8.5) -- (5.7,7.4);
        \draw (5.7,7) -- (5.7,5.8);
        \draw (5.7,5.4) -- (5.7,4.2);
            \node at (6,8) {\scriptsize{$\cdots$}};
            \node at (6,6.4) {\scriptsize{$\cdots$}};
            \node at (6,4.8) {\scriptsize{$\cdots$}};
            \node at (6,4.1) {\scriptsize{$\ddots$}};
        \draw (6.3,8.5) -- (6.3,7.4);
        \draw (6.3,7) -- (6.3,5.8);
        \draw (6.3,5.4) -- (6.3,4.2);
        \draw (6.3,3.8) -- (6.3,2.7);
        \draw (7.3,8.5) -- (7.3,7.4);
        \draw (7.3,7) -- (7.3,5.8);
        \draw (7.3,5.4) -- (7.3,4.2);
        \draw (7.3,3.8) -- (7.3,2.7);

    %filling labels
        \node at (1.02,8) {\scriptsize{$\mu_{\alpha_{1,n_1+1}}$}};
        \node at (2.62,8) {\scriptsize{$\mu_{\alpha_{1,n_2}}$}};
        \node at (3.62,8) {\scriptsize{$\mu_{\alpha_{1,n_2+1}}$}};
        \node at (5.22,8) {\scriptsize{$\mu_{\alpha_{1,n_3}}$}};
        \node at (6.82,8) {\scriptsize{$\mu_{\alpha_{1,n}}$}};

        \node at (1.02,6.4) {\scriptsize{$\mu_{\alpha_{n_1,n_1+1}}$}};
        \node at (2.62,6.4) {\scriptsize{$\mu_{\alpha_{n_1,n_2}}$}};
        \node at (3.62,6.4) {\scriptsize{$\mu_{\alpha_{n_1,n_2+1}}$}};
        \node at (5.22,6.4) {\scriptsize{$\mu_{\alpha_{n_1,n_3}}$}};
        \node at (6.82,6.4) {\scriptsize{$\mu_{\alpha_{n_1,n}}$}};

        \node at (3.62,4.8) {\scriptsize{$\mu_{\alpha_{n_2,n_2+1}}$}};
        \node at (5.22,4.8) {\scriptsize{$\mu_{\alpha_{n_2,n_3}}$}};
        \node at (6.82,4.8) {\scriptsize{$\mu_{\alpha_{n_2,n}}$}};

        \node at (6.82,3.2) {\scriptsize{$\mu_{\alpha_{n_l,n}}$}};
    \end{tikzpicture}
\end{center}
Moreover every ideal filling arises in this way.
\end{thm}

% This result follows by combining Proposition \ref{prop G/P version of Jamie's 6.2}, Corollary \ref{cor G/P nu'_i coords ideal filling for lambda} and Proposition \ref{prop G/P ideal filling indep of red expr}, with the theorem of Rietsch that the critical points of the superpotential are given by Toeplitz matrices \cite[non-T-equivariant case of Theorem 4.1]{Rietsch2008}.

% This result follows by combining the theorem of Rietsch that the critical points of the superpotential are given by Toeplitz matrices \cite[non-T-equivariant case of Theorem 4.1]{Rietsch2008}, with our result that the factorisations of the positive critical point lead to ideal fillings (combination of Proposition \ref{prop G/P version of Jamie's 6.2}, Corollary \ref{cor G/P nu'_i coords ideal filling for lambda} and Proposition \ref{prop G/P ideal filling indep of red expr}).

\begin{proof}
We begin by taking the reduced expression for $w_Pw_0$ given in Lemma \ref{lem wPw0 description in dot vertices}.
If the product
$\mathbf{y}^{\vee}_{i_1}\left(m_{\alpha^{\mathbf{i}}_1}^{-1}\right)\cdots  \mathbf{y}^{\vee}_{i_M}\left(m_{\alpha^{\mathbf{i}}_M}^{-1}\right)$
is a Toeplitz matrix, then, by \cite[non-T-equivariant case of Theorem 4.1]{Rietsch2008} and the conditions on the $m_{\alpha}$, the coordinates $m_{\alpha}$ are the coordinates of a positive critical point of the superpotential for some $\lambda$.
Thus the valuations of the coordinates of this point define a tropical critical point, which corresponds to an ideal filling by Proposition \ref{prop G/P version of Jamie's 6.2}. Moreover, we see that the entries of the ideal filling are exactly the valuations of the $m_{\alpha}$ coordinates (in the desired ordering thanks to the choice of reduced expression). Finally, since the ideal filling is independent of the choice of reduced expression by Proposition \ref{prop G/P ideal filling indep of red expr}, the result holds for all choices of reduced expression.
\end{proof}

%%%%%%%%%%%%%%%%%%%%%%%%%%%%%%%%%%%%%%%%%%%%%%%%%%%%%%%%%%%%%%%%%%%%%

%%%%%%%%%%%%%%%%%%%%%%%%%%%%%%%%%%%%%%%%%%%%%%%%%%%%%%%%%%%%%%%%%%%%%

\newpage

\appendix
\chapter{Recovering our coordinates} \label{append Recovering our coordinates}
\fancyhead[L]{A \ \ Recovering our coordinates}
\fancyhead[R]{Full flag varieties}

We recall the matrix $\mathbf{y}_{\mathbf{i}_0}^{\vee}\left(\frac{1}{m_1}, \frac{1}{m_2}, \ldots, \frac{1}{m_N} \right)=:Y$ from the $G/B$ setting, where
    $$\mathbf{i}_0 = (i_1, \ldots, i_N) :=
    (1,2, \ldots, n-1, 1,2 \ldots, n-2, \ldots, 1, 2, 1).
    $$
In this section we will use minors of $Y$ to recover the $m_i$ coordinates.

\begin{prop4} \label{prop recover nu' coords from Y}
Recall that $\Delta^{J}_K$ denotes the minor with row set defined by $J$ and column set given by $K$. Then writing $m_{jk}:=m_{s_k+j-k}$, for $1\leq i<j \leq n$ we have
    \begin{equation} \label{eqn recover nu' coords from Y}
    m_{ji}= \begin{cases}
        \frac{\Delta^{\{n\}}_{\{n-i+1\}}(Y)}{\Delta^{\{n\}}_{\{n-i\}}(Y)} & \text{for } j=n \vspace{0.2cm}\\
        \frac{\Delta^{[j+1,n]}_{[j-i+1,n-i]}(Y) \Delta^{[j,n]}_{[j-i+1,n-i+1]}(Y)}{\Delta^{[j+1,n]}_{[j-i+2,n-i+1]}(Y) \Delta^{[j,n]}_{[j-i,n-i]}(Y)} & \text{for } j<n
        \end{cases}
    \end{equation}
\end{prop4}

The following Corollary is immediate: %(using Corollary \ref{cor nu'_i coords ideal filling for lambda}):
\begin{cor4} \label{cor ideal filling from Y}
% The filling $\{\mathrm{val}(\hat{n}_{ij})\}_{1\leq i<j \leq n}$ is an ideal filling for $\lambda$.
The filling $\left\{n_{ij}=\mathrm{Val}_{\mathbf{K}}(m_{ji})\right\}_{1\leq i<j \leq n}$ given in terms of matrix minors by ({\ref{eqn recover nu' coords from Y}}), is an ideal filling for $\lambda$ when Y is a Toeplitz matrix.
\end{cor4}
\hfill\qedsymbol

\begin{proof}[Proof of Proposition {\ref{prop recover nu' coords from Y}}]

The proof of this proposition will rely on the proof of Lemma \ref{lem coord change for b m_i in terms of p_i}.

Firstly we recall the planar graphs which we used to compute Chamber Ansatz minors. These were introduced in Section \ref{subsec Chamber Ansatz minors}) using the rule in Figure \ref{Labelled line segments in graphs for computing Chamber Ansatz minors}. To compute minors of $Y$ we require labelled line segments corresponding to $\mathbf{y}_{i_k}^{\vee}(z_k)$ factors, which are described in Figure \ref{fig Labelled line segments in graphs for computing CA minors y_i factors}. These are again found at height $i_k$ and oriented left to right.
\begin{figure}[ht!]
\centering
    \begin{tikzpicture}[scale=0.85]
        \node at (0.75,2) {$\bullet$};
        \node at (1.25,1) {$\bullet$};
        \draw (0,3) -- (2,3);
        \draw (0,2) -- (2,2);
        \draw (0,1) -- (2,1);
        \draw (0,0) -- (2,0);
        \draw (0.75,2) -- (1.25,1);
        \node at (1.35,1.5) {$z_k$};
        \node at (1,-0.6) {For factors $\mathbf{y}^{\vee}_{i_k}(z_k)$};
    \end{tikzpicture}
\caption{Labelled line segments in graphs for computing Chamber Ansatz minors, $\mathbf{y}_{i}^{\vee}$ factors} \label{fig Labelled line segments in graphs for computing CA minors y_i factors}
\end{figure}

Now we notice the relation between the following two reduced expressions for $w_0$:
    $$\begin{gathered}
    \mathbf{i}_0 = (i_1, \ldots, i_N) := (1,2, \ldots, n-1, 1,2 \ldots, n-2, \ldots, 1, 2, 1), \\
    \mathbf{i}'_0 = (i'_1, \ldots, i'_N) := (n-1, n-2, \ldots, 1, n-1, n-2, \ldots, 2, \ldots, n-1, n-2, n-1).
    \end{gathered}
    $$
In particular we see that the graph for $Y = \mathbf{y}_{\mathbf{i}_0}^{\vee}\left(\frac{1}{m_1}, \frac{1}{m_2}, \ldots, \frac{1}{m_N} \right)$ is simply a reflection in the horizontal axis of the graph for $X:= \mathbf{x}_{\mathbf{i}'_0}^{\vee}\left(\frac{1}{m_1}, \frac{1}{m_2}, \ldots, \frac{1}{m_N} \right)$, where we fix the origin at the centre of the graph. Of note, during this reflection we keep the diagonal line segments labelled as before, but change the horizontal line labellings to reflect their new heights in the graph. For example see Figure \ref{fig The graphs for X and Y when n=4}.

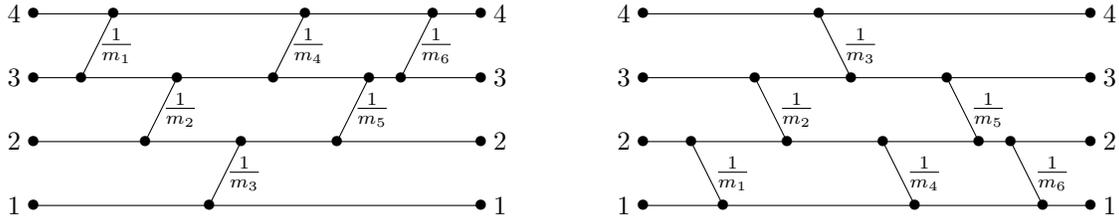
\begin{figure}[hb!]
\centering
\begin{minipage}[b]{0.48\linewidth}
    \centering
    \begin{tikzpicture}[scale=0.85]
        % pseudolines
        \draw (0,4) -- (7,4);
        \draw (0,3) -- (7,3);
        \draw (0,2) -- (7,2);
        \draw (0,1) -- (7,1);

        \draw (0.75,3) -- (1.25,4);
        \draw (1.75,2) -- (2.25,3);
        \draw (2.75,1) -- (3.25,2);
        \draw (3.75,3) -- (4.25,4);
        \draw (4.75,2) -- (5.25,3);
        \draw (5.75,3) -- (6.25,4);
        % dots and stars
        \node at (0,1) {$\bullet$};
        \node at (0,2) {$\bullet$};
        \node at (0,3) {$\bullet$};
        \node at (0,4) {$\bullet$};
        \node at (7,1) {$\bullet$};
        \node at (7,2) {$\bullet$};
        \node at (7,3) {$\bullet$};
        \node at (7,4) {$\bullet$};

        \node at (0.75,3) {$\bullet$};
        \node at (1.25,4) {$\bullet$};
        \node at (1.75,2) {$\bullet$};
        \node at (2.25,3) {$\bullet$};
        \node at (2.75,1) {$\bullet$};
        \node at (3.25,2) {$\bullet$};

        \node at (3.75,3) {$\bullet$};
        \node at (4.25,4) {$\bullet$};
        \node at (4.75,2) {$\bullet$};
        \node at (5.25,3) {$\bullet$};
        \node at (5.75,3) {$\bullet$};
        \node at (6.25,4) {$\bullet$};

        % pseudoline labels
        \node at (-0.3,4) {$4$};
        \node at (-0.3,3) {$3$};
        \node at (-0.3,2) {$2$};
        \node at (-0.3,1) {$1$};

        \node at (7.3,4) {$4$};
        \node at (7.3,3) {$3$};
        \node at (7.3,2) {$2$};
        \node at (7.3,1) {$1$};

        % weight labels
        \node at (1.3,3.5) {$\frac{1}{m_1}$};
        \node at (2.3,2.5) {$\frac{1}{m_2}$};
        \node at (3.3,1.5) {$\frac{1}{m_3}$};
        \node at (4.3,3.5) {$\frac{1}{m_4}$};
        \node at (5.3,2.5) {$\frac{1}{m_5}$};
        \node at (6.3,3.5) {$\frac{1}{m_6}$};
    \end{tikzpicture}
\end{minipage}
\hfill
\begin{minipage}[b]{0.48\linewidth}
    \centering
    \begin{tikzpicture}[scale=0.85]
        % pseudolines
        \draw (0,4) -- (7,4);
        \draw (0,3) -- (7,3);
        \draw (0,2) -- (7,2);
        \draw (0,1) -- (7,1);

        \draw (0.75,2) -- (1.25,1);
        \draw (1.75,3) -- (2.25,2);
        \draw (2.75,4) -- (3.25,3);
        \draw (3.75,2) -- (4.25,1);
        \draw (4.75,3) -- (5.25,2);
        \draw (5.75,2) -- (6.25,1);
        % dots and stars
        \node at (0,1) {$\bullet$};
        \node at (0,2) {$\bullet$};
        \node at (0,3) {$\bullet$};
        \node at (0,4) {$\bullet$};
        \node at (7,1) {$\bullet$};
        \node at (7,2) {$\bullet$};
        \node at (7,3) {$\bullet$};
        \node at (7,4) {$\bullet$};

        \node at (0.75,2) {$\bullet$};
        \node at (1.25,1) {$\bullet$};
        \node at (1.75,3) {$\bullet$};
        \node at (2.25,2) {$\bullet$};
        \node at (2.75,4) {$\bullet$};
        \node at (3.25,3) {$\bullet$};

        \node at (3.75,2) {$\bullet$};
        \node at (4.25,1) {$\bullet$};
        \node at (4.75,3) {$\bullet$};
        \node at (5.25,2) {$\bullet$};
        \node at (5.75,2) {$\bullet$};
        \node at (6.25,1) {$\bullet$};

        % pseudoline labels
        \node at (-0.3,4) {$4$};
        \node at (-0.3,3) {$3$};
        \node at (-0.3,2) {$2$};
        \node at (-0.3,1) {$1$};

        \node at (7.3,4) {$4$};
        \node at (7.3,3) {$3$};
        \node at (7.3,2) {$2$};
        \node at (7.3,1) {$1$};

        % weight labels
        \node at (1.4,1.5) {$\frac{1}{m_1}$};
        \node at (2.4,2.5) {$\frac{1}{m_2}$};
        \node at (3.4,3.5) {$\frac{1}{m_3}$};
        \node at (4.4,1.5) {$\frac{1}{m_4}$};
        \node at (5.4,2.5) {$\frac{1}{m_5}$};
        \node at (6.4,1.5) {$\frac{1}{m_6}$};
    \end{tikzpicture}
\end{minipage}
\caption{The respective graphs for $X$ and $Y$ when $n=4$} \label{fig The graphs for X and Y when n=4}
\end{figure}

In terms of matrices we see that $\mathbf{x}_i^{\vee}(z)^T=\mathbf{y}_i^{\vee}(z)$, and changing between the reduced expressions $\mathbf{i}_0$ and $\mathbf{i}_0'$ corresponds to taking the anti-transpose (that is, transposing over the anti-diagonal). Thus taking the anti-transpose of $X^T$ we get $Y$, which we could have equivalently have obtained by sending row $r$ in $X$ to row $n-r+1$ and then column $s$ to $n-s+1$. In particular, exchanging the rows and then the columns in this way is exactly the reflection of the graph for $X$ that we have just done. This gives the following relation between matrix minors:
    $$\Delta^{[a,b]}_{[c,d]}(Y) = \Delta^{[n-b+1,n-a+1]}_{[n-d+1,n-c+1]}(X).
    $$
In particular, the quotients of minors we wish to compute become
    \begin{equation} \label{eqn appendix minors of Y in terms of X}
    \frac{\Delta^{\{n\}}_{\{n-i+1\}}(Y)}{\Delta^{\{n\}}_{\{n-i\}}(Y)}
        = \frac{\Delta^{\{1\}}_{\{i\}}(X)}{\Delta^{\{1\}}_{\{i+1\}}(X)} \ , \quad
    \frac{\Delta^{[j+1,n]}_{[j-i+1,n-i]}(Y) \Delta^{[j,n]}_{[j-i+1,n-i+1]}(Y)}{\Delta^{[j+1,n]}_{[j-i+2,n-i+1]}(Y) \Delta^{[j,n]}_{[j-i,n-i]}(Y)}
        = \frac{\Delta^{[1,n-j]}_{[i+1,n-j+i]}(X) \Delta^{[1,n-j+1]}_{[i,n-j+i]}(X)}{\Delta^{[1,n-j]}_{[i,n-j+i-1]}(X) \Delta^{[1,n-j+1]}_{[i+1,n-j+i+1]}(X)}.
    \end{equation}

Now in the proof of Lemma \ref{lem coord change for b m_i in terms of p_i} we carefully studied the graph of $u_1 = \mathbf{x}_{i'_1}^{\vee}(p_1) \cdots \mathbf{x}_{i'_N}^{\vee}(p_N)$. Moreover we used it to compute the following minors:
    $$\Delta^{\{1, \ldots, a \}}_{\{k+1, \ldots, k+a\}}(u_1)
        = \prod_{\substack{r=1,\ldots, k \\ b=1, \ldots, a}} p_{s_{r+1}-b+1}, \quad \text{for} \quad k=1,\ldots, n-1, \quad a=1,\ldots, n-k.
    $$
We see that $u_1$ is exactly the matrix $X$ if we take $p_i=\frac{1}{m_i}$, and indeed all the minors of $X$ in (\ref{eqn appendix minors of Y in terms of X}) take the form
    $$\Delta^{\{1, \ldots, a \}}_{\{k+1, \ldots, k+a\}}(X) \quad \text{for} \quad k \in \{1,\ldots, n-1\}, \quad a \in \{1,\ldots, n-k\}.
    $$
Consequently we may use the computations of minors from the proof of Lemma \ref{lem coord change for b m_i in terms of p_i} to compute the necessary minors of $X$, and therefore of $Y$.

In the first case we obtain the desired result as follows:
    $$\frac{\Delta^{\{n\}}_{\{n-i+1\}}(Y)}{\Delta^{\{n\}}_{\{n-i\}}(Y)}
        = \frac{\Delta^{\{1\}}_{\{i\}}(X)}{\Delta^{\{1\}}_{\{i+1\}}(X)}
        = \frac{ \prod_{r=1,\ldots, i-1} \frac{1}{m_{s_{r+1}}} }{ \prod_{r=1,\ldots, i} \frac{1}{m_{s_{r+1}}} }
        = m_{s_{i+1}} = m_{s_i+n-i} = m_{ni}
    $$
where we have used the definitions of $s_k = \sum_{j=1}^{k-1} (n-j)$ (from Section \ref{sec The ideal coordinates}) and $m_{jk}=m_{s_k+j-k}$ for the last two equalities respectively.
In the second case we have
    $$\begin{aligned}
    \frac{\Delta^{[j+1,n]}_{[j-i+1,n-i]}(Y) \Delta^{[j,n]}_{[j-i+1,n-i+1]}(Y)}{\Delta^{[j+1,n]}_{[j-i+2,n-i+1]}(Y) \Delta^{[j,n]}_{[j-i,n-i]}(Y)}
        &= \frac{\Delta^{[1,n-j]}_{[i+1,n-j+i]}(X) \Delta^{[1,n-j+1]}_{[i,n-j+i]}(X)}{\Delta^{[1,n-j]}_{[i,n-j+i-1]}(X) \Delta^{[1,n-j+1]}_{[i+1,n-j+i+1]}(X)} \\
        &= \frac{ \prod\limits_{\substack{r=1,\ldots, i \\ b=1, \ldots, n-j}} \frac{1}{m_{s_{r+1}-b+1}} \prod\limits_{\substack{r=1,\ldots, i-1 \\ b=1, \ldots, n-j+1}} \frac{1}{m_{s_{r+1}-b+1}} }{ \prod\limits_{\substack{r=1,\ldots, i-1 \\ b=1, \ldots, n-j}} \frac{1}{m_{s_{r+1}-b+1}} \prod\limits_{\substack{r=1,\ldots, i \\ b=1, \ldots, n-j+1}} \frac{1}{m_{s_{r+1}-b+1}} } \\
        &= \frac{ \prod\limits_{r=1,\ldots, i-1 } \frac{1}{m_{s_{r+1}-(n-j)}} }{ \prod\limits_{r=1,\ldots, i} \frac{1}{m_{s_{r+1}-(n-j)}} } \\
        &= m_{s_{i+1}-(n-j)} \\
        &= m_{s_i+n-i-(n-j)} \\
        &= m_{s_i+j-i} = m_{ji}
    \end{aligned}
    $$
as desired.
\end{proof}

\chapter{Complete quiver labelling example} \label{append Example of complete quiver labelling}
\fancyhead[L]{B \ \ Example of complete quiver labelling}
\fancyhead[R]{Partial flag varieties}

\begin{ex4}[Weight matrix for {$\mathcal{F}_{2,5,6}(\mathbb{C}^8)$}] \label{ex wt matrix in full F2,5,6,C8}

We recall the definition of the weight matrix elements given in (\ref{eqn 2nd G/P defn gamma and t}) as
    $$t_{P,i} = x_{v_{n,n-i+1}} \prod\limits_{v \in \mathcal{D}_{i+1} \cap \mathcal{V}^{\bullet}} \prod\limits_{a \in p(v)} r_a
    $$
where we set $x_{v_{n,n-i+1}}=x_{v_{n,n_l+1}}$, the star vertex coordinate in the last square $L_{l+1}$, if $v_{n,n-i+1}$ is not present in the quiver.

We use the quiver decoration in Figure \ref{fig complete QP quiver decoration F2,5,6,C8} to compute these $t_{P,i}$ in the case of $\mathcal{F}_{2,5,6}(\mathbb{C}^8)$.

Marking the paths with parentheses, each arrow coordinate within a given path by a new fraction, and the vertex coordinate $x_{v_{n,n-i+1}}$ with square braces, we have
    $$\begin{aligned}
    t_{P,1} &= x_{v_{8,7}} \prod\limits_{v \in \mathcal{D}_2 \cap \mathcal{V}^{\bullet}} \prod\limits_{a \in p(v)} r_a \\
        &= \left[d'_4\right] \left(m_8 m_2 \right) \left( \frac{m_5 m_{11} m_{23}}{m_4 m_{10}} \frac{m_4 m_{10} m_{20}}{m_3 m_{9}} \frac{m_5 m_{11} m_{16}}{m_4 m_{10}} \right) \left(\frac{m_6 m_{12} m_{17} m_{21} m_{24} m_{26}}{m_5 m_{11} m_{16} m_{20} m_{23}}  \right) \\
        &= d'_4 \frac{m_2 m_5 m_6 m_8 m_{11} m_{12} m_{17} m_{21} m_{24} m_{26}}{m_3 m_4 m_9 m_{10}} \\
    % \end{aligned}
    % $$
    % $$\begin{aligned}
    t_{P,2} &= x_{v_{8,7}} \prod\limits_{v \in \mathcal{D}_3 \cap \mathcal{V}^{\bullet}} \prod\limits_{a \in p(v)} r_a \\
        &= \left[d'_4\right]
            \left( m_2 \right) \left( \frac{m_3 m_{9}}{m_2} \right) \left( \frac{m_4 m_{10} m_{20}}{m_3 m_{9}} \frac{m_5 m_{11} m_{16}}{m_4 m_{10}} \right) \left(\frac{m_6 m_{12} m_{17} m_{21} m_{24}}{m_5 m_{11} m_{16} m_{20}} \right) \left(\frac{m_7 m_{13} m_{18} m_{22} m_{25} m_{27}}{m_6 m_{12} m_{17} m_{21} m_{24}}  \right) \\
        &= d'_4 m_7 m_{13} m_{18} m_{22} m_{25} m_{27}
    \end{aligned}
    $$
and so on for $t_{P,3}, \ldots, t_{P,8}$.

\newpage
\begin{figure}[!hb]
\centering
% \makebox[\textwidth][c]{
\makebox[\textwidth][c]{
\begin{tikzpicture}[scale=2.4]
    %squares
        \draw[dotted, thick, color=black!50] (0.5,0.5) -- (0.5,8.5) -- (8.5,8.5) -- (8.5,0.5) -- cycle;
        \draw[dotted, thick, color=black!50] (1.5,0.5) -- (1.5,6.5);
        \draw[dotted, thick, color=black!50] (2.5,0.5) -- (2.5,8.5);
        \draw[dotted, thick, color=black!50] (3.5,0.5) -- (3.5,3.5);
        \draw[dotted, thick, color=black!50] (4.5,0.5) -- (4.5,3.5);
        \draw[dotted, thick, color=black!50] (5.5,0.5) -- (5.5,6.5);
        \draw[dotted, thick, color=black!50] (6.5,0.5) -- (6.5,3.5);

        \draw[dotted, thick, color=black!50] (0.5,6.5) -- (5.5,6.5);
        \draw[dotted, thick, color=black!50] (0.5,5.5) -- (2.5,5.5);
        \draw[dotted, thick, color=black!50] (0.5,4.5) -- (2.5,4.5);
        \draw[dotted, thick, color=black!50] (0.5,3.5) -- (6.5,3.5);
        \draw[dotted, thick, color=black!50] (0.5,2.5) -- (8.5,2.5);
        \draw[dotted, thick, color=black!50] (0.5,1.5) -- (6.5,1.5);

    %square labels
        \node[color=black!50] at (2.7,8.2) {\small{$L_1$}};
        \node[color=black!50] at (5.7,6.2) {\small{$L_2$}};
        \node[color=black!50] at (6.7,3.2) {\small{$L_3$}};
        \node[color=black!50] at (8.2,2.7) {\small{$L_4$}};

    %dots and stars + arrow labels
        \node (21) at (1,7) {$\boldsymbol{*}$};
            \node at (1.15,7.05) {\tiny{$d_1'$}};
        \node (31) at (1,6) {$\bullet$};
            \node at (0.88,6.7) {\tiny{$m_2$}}; %f
        \node (41) at (1,5) {$\bullet$};
            \node at (0.88,5.7) {\tiny{$m_3$}}; %e
        \node (51) at (1,4) {$\bullet$};
            \node at (0.88,4.7) {\tiny{$m_4$}}; %d
        \node (61) at (1,3) {$\bullet$};
            \node at (0.88,3.7) {\tiny{$m_5$}}; %c
        \node (71) at (1,2) {$\bullet$};
            \node at (0.88,2.7) {\tiny{$m_6$}}; %b
        \node (81) at (1,1) {$\bullet$};
            \node at (0.88,1.7) {\tiny{$m_7$}}; %a

        \node (32) at (2,6) {$\bullet$};
            \node at (1.55,6.1) {\tiny{$m_8$}}; %l
        \node (42) at (2,5) {$\bullet$};
            \node at (1.8,5.7) {\tiny{$\frac{m_3m_9}{m_2}$}}; %k
            \node at (1.5,5.1) {\tiny{$\frac{m_8m_9}{m_2}$}};
        \node (52) at (2,4) {$\bullet$};
            \node at (1.8,4.7) {\tiny{$\frac{m_4m_{10}}{m_3}$}}; %j
            \node at (1.5,4.1) {\tiny{$\frac{m_8m_9m_{10}}{m_2m_3}$}};
        \node (62) at (2,3) {$\bullet$};
            \node at (1.8,3.7) {\tiny{$\frac{m_5m_{11}}{m_4}$}}; %i
            \node at (1.5,3.1) {\tiny{$\frac{m_8m_9m_{10}m_{11}}{m_2m_3m_4}$}};
        \node (72) at (2,2) {$\bullet$};
            \node at (1.8,2.7) {\tiny{$\frac{m_6m_{12}}{m_5}$}}; %h
            \node at (1.5,2.1) {\tiny{$\frac{m_8m_9m_{10}m_{11}m_{12}}{m_2m_3m_4m_5}$}};
        \node (82) at (2,1) {$\bullet$};
            \node at (1.8,1.7) {\tiny{$\frac{m_7m_{13}}{m_6}$}}; %g
            \node at (1.47,1.1) {\tiny{$\frac{m_8m_9m_{10}m_{11}m_{12}m_{13}}{m_2m_3m_4m_5m_6}$}};

        % \node (33) at (3,6) {$\boldsymbol{*}$};
        % \node (43) at (3,5) {$\bullet$};
        \node (53) at (3,4) {$\boldsymbol{*}$};
            \node at (3.15,4.05) {\tiny{$d_2'$}};
            \node at (2.5,4.15) {\tiny{$\frac{d'_1}{d'_2}\frac{1}{m_4m_8m_9m_{10}}$}};
        \node (63) at (3,3) {$\bullet$};
            \node at (2.7,3.7) {\tiny{$\frac{m_5m_{11}m_{16}}{m_4m_{10}}$}}; %p
            \node at (2.5,3.15) {\tiny{$\frac{d'_1}{d'_2}\frac{m_{16}}{m_4m_8m_9m_{10}^2}$}};
        \node (73) at (3,2) {$\bullet$};
            \node at (2.7,2.7) {\tiny{$\frac{m_6m_{12}m_{17}}{m_5m_{11}}$}}; %n
            \node at (2.5,2.15) {\tiny{$\frac{d'_1}{d'_2}\frac{m_{16}m_{17}}{m_4m_8m_9m_{10}^2m_{11}}$}};
        \node (83) at (3,1) {$\bullet$};
            \node at (2.7,1.7) {\tiny{$\frac{m_7m_{13}m_{18}}{m_6m_{12}}$}}; %m
            \node at (2.5,0.85) {\tiny{$\frac{d'_1}{d'_2}\frac{m_{16}m_{17}m_{18}}{m_4m_8m_9m_{10}^2m_{11}m_{12}}$}};

        % \node (44) at (4,5) {$\boldsymbol{*}$};
        % \node (54) at (4,4) {$\bullet$};
        \node (64) at (4,3) {$\bullet$};
            \node at (3.5,3.1) {\tiny{$\frac{m_4m_{10}m_{20}}{m_3m_9}$}}; %s
        \node (74) at (4,2) {$\bullet$};
            \node at (3.62,2.7) {\tiny{$\frac{m_6m_{12}m_{17}m_{21}}{m_5m_{11}m_{16}}$}}; %r
            \node at (3.5,2.1) {\tiny{$\frac{m_4m_{10}m_{20}m_{21}}{m_3m_9m_{16}}$}};
        \node (84) at (4,1) {$\bullet$};
            \node at (3.62,1.7) {\tiny{$\frac{m_7m_{13}m_{18}m_{22}}{m_6m_{12}m_{17}}$}}; %q
            \node at (3.5,1.1) {\tiny{$\frac{m_4m_{10}m_{20}m_{21}m_{22}}{m_3m_9m_{16}m_{17}}$}};

        % \node (55) at (5,4) {$\boldsymbol{*}$};
        \node (65) at (5,3) {$\bullet$};
            \node at (4.5,3.1) {\tiny{$\frac{m_5m_{11}m_{23}}{m_4m_{10}}$}}; %v
        \node (75) at (5,2) {$\bullet$};
            \node at (4.52,2.7) {\tiny{$\frac{m_6m_{12}m_{17}m_{21}m_{24}}{m_5m_{11}m_{16}m_{20}}$}}; %u
            \node at (4.5,2.1) {\tiny{$\frac{m_5m_{11}m_{23}m_{24}}{m_4m_{10}m_{20}}$}};
        \node (85) at (5,1) {$\bullet$};
            \node at (4.52,1.7) {\tiny{$\frac{m_7m_{13}m_{18}m_{22}m_{25}}{m_6m_{12}m_{17}m_{21}}$}}; %t
            \node at (4.5,0.85) {\tiny{$\frac{m_5m_{11}m_{23}m_{24}m_{25}}{m_4m_{10}m_{20}m_{21}}$}};

        \node (66) at (6,3) {$\boldsymbol{*}$};
            \node at (6.15,3.05) {\tiny{$d_3'$}};
            \node at (5.5,3.15) {\tiny{$\frac{d'_2}{d'_3}\frac{m_3m_4m_9m_{10}}{m_5^2m_{11}^2m_{16}m_{20}m_{23}}$}};
        \node (76) at (6,2) {$\bullet$};
            \node at (6.58,2.7) {\tiny{$\frac{m_6m_{12}m_{17}m_{21}m_{24}m_{26}}{m_5m_{11}m_{16}m_{20}m_{23}}$}}; %x
            \node at (5.5,2.15) {\tiny{$\frac{d'_2}{d'_3}\frac{m_3m_4m_9m_{10}m_{26}}{m_5^2m_{11}^2m_{16}m_{20}m_{23}^2}$}};
        \node (86) at (6,1) {$\bullet$};
            \node at (6.58,1.7) {\tiny{$\frac{m_7m_{13}m_{18}m_{22}m_{25}m_{27}}{m_6m_{12}m_{17}m_{21}m_{24}}$}}; %w
            \node at (5.59,1.15) {\tiny{$\frac{d'_2}{d'_3}\frac{m_3m_4m_9m_{10}m_{26}m_{27}}{m_5^2m_{11}^2m_{16}m_{20}m_{23}^2m_{24}}$}};

        % \node (77) at (7,2) {$\boldsymbol{*}$};
        \node (87) at (7,1) {$\boldsymbol{*}$};
            \node at (7.15,1.05) {\tiny{$d_4'$}};
            \node at (6.6,0.8) {\tiny{$\frac{d'_3}{d'_4}\frac{m_5m_{11}m_{16}m_{20}m_{23}}{m_7m_{13}m_{18}m_{22}m_{25}m_{26}m_{27}}$}};

        % \node (88) at (8,1) {$\boldsymbol{*}$};

    %Row E_i
        % \node at (-0.6,1.5) {\small{Row $E_7$}};
        % \node at (-0.6,2.5) {\small{Row $E_6$}};
        % \node at (-0.6,3.5) {\small{Row $E_5$}};
        % \node at (-0.6,4.5) {\small{Row $E_4$}};
        % \node at (-0.6,5.5) {\small{Row $E_3$}};
        % \node at (-0.6,6.5) {\small{Row $E_2$}};
        % \node at (-0.6,7.5) {\small{Row $E_1$}};

    %s_i's
        \node[draw, circle, minimum size=4mm, inner sep=1pt] at (1,7.5) {\tiny{$\dot{s}_1$}};
        \node[draw, circle, minimum size=4mm, inner sep=1pt] at (3,5.5) {\tiny{$\dot{s}_3$}};
        \node[draw, circle, minimum size=4mm, inner sep=1pt] at (3,4.5) {\tiny{$\dot{s}_4$}};
        \node[draw, circle, minimum size=4mm, inner sep=1pt] at (4,4.5) {\tiny{$\dot{s}_4$}};
        \node[draw, circle, minimum size=4mm, inner sep=1pt] at (7,1.5) {\tiny{$\dot{s}_7$}};

    %vertical arrows
        \draw[->] (81) -- (71);
        \draw[->] (71) -- (61);
        \draw[->] (61) -- (51);
        \draw[->] (51) -- (41);
        \draw[->] (41) -- (31);
        \draw[->] (31) -- (21);
        % \draw[->] (21) -- (11);

        \draw[->] (82) -- (72);
        \draw[->] (72) -- (62);
        \draw[->] (62) -- (52);
        \draw[->] (52) -- (42);
        \draw[->] (42) -- (32);
        % \draw[->] (32) -- (22);

        \draw[->] (83) -- (73);
        \draw[->] (73) -- (63);
        \draw[->] (63) -- (53);
        % \draw[->] (53) -- (43);
        % \draw[->] (43) -- (33);

        \draw[->] (84) -- (74);
        \draw[->] (74) -- (64);
        % \draw[->] (64) -- (54);
        % \draw[->] (54) -- (44);

        \draw[->] (85) -- (75);
        \draw[->] (75) -- (65);
        % \draw[->] (65) -- (55);

        \draw[->] (86) -- (76);
        \draw[->] (76) -- (66);

        % \draw[->] (87) -- (77);

    %horizontal arrows
        % \draw[->] (22) -- (21);
        \draw[->] (32) -- (31);
        \draw[->] (42) -- (41);
        \draw[->] (52) -- (51);
        \draw[->] (62) -- (61);
        \draw[->] (72) -- (71);
        \draw[->] (82) -- (81);

        % \draw[->] (33) -- (32);
        % \draw[->] (43) -- (42);
        \draw[->] (53) -- (52);
        \draw[->] (63) -- (62);
        \draw[->] (73) -- (72);
        \draw[->] (83) -- (82);

        % \draw[->] (44) -- (43);
        % \draw[->] (54) -- (53);
        \draw[->] (64) -- (63);
        \draw[->] (74) -- (73);
        \draw[->] (84) -- (83);

        % \draw[->] (55) -- (54);
        \draw[->] (65) -- (64);
        \draw[->] (75) -- (74);
        \draw[->] (85) -- (84);

        \draw[->] (66) -- (65);
        \draw[->] (76) -- (75);
        \draw[->] (86) -- (85);

        % \draw[->] (77) -- (76);
        \draw[->] (87) -- (86);

        % \draw[->] (88) -- (87);
    \end{tikzpicture}}
\caption{Complete $Q_P$ quiver decoration for $\mathcal{F}_{2,5,6}(\mathbb{C}^8)$} \label{fig complete QP quiver decoration F2,5,6,C8}
\end{figure}

The weight matrix $\gamma_P \in T^{\vee}$ we obtain is
\renewcommand{\arraystretch}{2}
    $$
    \begingroup % keep the change local
    \setlength\arraycolsep{-15pt}
    \begin{pmatrix}
    d'_4 \frac{m_2 m_5 m_6 m_8 m_{11} m_{12} m_{17} m_{21} m_{24} m_{26}}{m_3 m_4 m_9 m_{10}} & & & & & & & \\
    & d'_4 m_7 m_{13} m_{18} m_{22} m_{25} m_{27} & & & & & & \\
    & & d'_3 \frac{m_5 m_{11} m_{16} m_{20} m_{23}}{m_{26} m_{27}} & & & & & \\
    & & & d'_2 \frac{m_3 m_{4} m_{9} m_{10}}{m_5 m_{11} m_{23} m_{24} m_{25}} & & & & \\
    & & & & d'_2 \frac{m_3 m_{9}}{m_{20} m_{21} m_{22}} & & & \\
    & & & & & d'_2 \frac{m_4 m_{10}}{m_{16} m_{17} m_{18}} & & \\
    & & & & & & d'_1 \frac{1}{m_8 m_9 m_{10} m_{11} m_{12}} & \\
    & & & & & & & d'_1 \frac{1}{m_2 m_3 m_4 m_5 m_6 m_7}
    \end{pmatrix}
    \endgroup
    $$
\renewcommand{\arraystretch}{1}

\end{ex4}

\newpage

\fancyhead[L]{}
\fancyhead[R]{References}

\renewcommand\bibname{References}
\addcontentsline{toc}{chapter}{References}
\bibliographystyle{plain}
\bibliography{thesis_teresa_with_corrections}

\begin{thebibliography}{10}

\bibitem{ZainabMastersThesis}
Zainab Al-Sultani.
\newblock Totally positive toeplitz matrices.
\newblock {\em King's College London MSc Thesis}, 2013.

\bibitem{BatyrevEtAl2000}
Victor~V. Batyrev, Ionu\c{t} Ciocan-Fontanine, Bumsig Kim, and Duco van
  Straten.
\newblock Mirror symmetry and toric degenerations of partial flag manifolds.
\newblock {\em Acta Math.}, 184(1):1--39, 2000.

\bibitem{BerensteinFominZelevinsky1996}
Arkady Berenstein, Sergey Fomin, and Andrei Zelevinsky.
\newblock Parametrizations of canonical bases and totally positive matrices.
\newblock {\em Adv. Math.}, 122(1):49--149, 1996.

\bibitem{BerensteinKazhdan2000}
Arkady Berenstein and David Kazhdan.
\newblock Geometric and unipotent crystals.
\newblock Number Special Volume, Part I, pages 188--236. 2000.
\newblock GAFA 2000 (Tel Aviv, 1999).

\bibitem{BerensteinKazhdan2007}
Arkady Berenstein and David Kazhdan.
\newblock Geometric and unipotent crystals. {II}. {F}rom unipotent bicrystals
  to crystal bases.
\newblock In {\em Quantum groups}, volume 433 of {\em Contemp. Math.}, pages
  13--88. Amer. Math. Soc., Providence, RI, 2007.

\bibitem{BourbakiCh1-3}
Nicolas Bourbaki.
\newblock {\em Lie groups and {L}ie algebras. {C}hapters 1--3}.
\newblock Elements of Mathematics (Berlin). Springer-Verlag, Berlin, 1998.
\newblock Translated from the French, Reprint of the 1989 English translation.

\bibitem{ChhaibiThesis13}
Reda Chhaibi.
\newblock Littelmann path model for geometric crystals, whittaker functions on
  lie groups and brownian motion.
\newblock 2013.
\newblock arXiv:1302.0902.

\bibitem{EguchiHoriXiong1997}
Tohru Eguchi, Kentaro Hori, and Chuan-Sheng Xiong.
\newblock Gravitational quantum cohomology.
\newblock {\em Internat. J. Modern Phys. A}, 12(9):1743--1782, 1997.

\bibitem{FockGoncharov2006}
Vladimir Fock and Alexander Goncharov.
\newblock Moduli spaces of local systems and higher {T}eichm\"{u}ller theory.
\newblock {\em Publ. Math. Inst. Hautes \'{E}tudes Sci.}, (103):1--211, 2006.

\bibitem{FockGoncharov2009}
Vladimir~V. Fock and Alexander~B. Goncharov.
\newblock Cluster ensembles, quantization and the dilogarithm.
\newblock {\em Ann. Sci. \'{E}c. Norm. Sup\'{e}r. (4)}, 42(6):865--930, 2009.

\bibitem{FominZelevinsky1999}
Sergey Fomin and Andrei Zelevinsky.
\newblock Double {B}ruhat cells and total positivity.
\newblock {\em J. Amer. Math. Soc.}, 12(2):335--380, 1999.

\bibitem{GelfandTsetlin1950}
I.M. Gelfand and M.L. Tsetlin.
\newblock {Finite-dimensional representations of the group of unimodular
  matrices}.
\newblock {\em Dokl. Akad. Nauk SSSR}, 71:825--828, 1950.

\bibitem{Givental1997}
Alexander Givental.
\newblock Stationary phase integrals, quantum {T}oda lattices, flag manifolds
  and the mirror conjecture.
\newblock In {\em Topics in singularity theory}, volume 180 of {\em Amer. Math.
  Soc. Transl. Ser. 2}, pages 103--115. Amer. Math. Soc., Providence, RI, 1997.

\bibitem{JoeKim2003}
Dosang Joe and Bumsig Kim.
\newblock Equivariant mirrors and the {V}irasoro conjecture for flag manifolds.
\newblock {\em Int. Math. Res. Not.}, (15):859--882, 2003.

\bibitem{Judd2018}
Jamie Judd.
\newblock Tropical critical points of the superpotential of a flag variety.
\newblock {\em J. Algebra}, 497:102--142, 2018.

\bibitem{JuddRietsch2019}
Jamie Judd and Konstanze Rietsch.
\newblock The tropical critical point and mirror symmetry.
\newblock 2019.
\newblock arXiv:1911.04463v2.

\bibitem{Kashiwara1991}
Masaki Kashiwara.
\newblock On crystal bases of the {$Q$}-analogue of universal enveloping
  algebras.
\newblock {\em Duke Math. J.}, 63(2):465--516, 1991.

\bibitem{KazhdanLustig1980}
David Kazhdan and George Lusztig.
\newblock Schubert varieties and {P}oincar\'{e} duality.
\newblock In {\em Geometry of the {L}aplace operator ({P}roc. {S}ympos. {P}ure
  {M}ath., {U}niv. {H}awaii, {H}onolulu, {H}awaii, 1979)}, Proc. Sympos. Pure
  Math., XXXVI, pages 185--203. Amer. Math. Soc., Providence, R.I., 1980.

\bibitem{Littelmann1998}
Peter Littelmann.
\newblock Cones, crystals, and patterns.
\newblock {\em Transform. Groups}, 3(2):145--179, 1998.

\bibitem{Ludenbach2022}
Teresa L\"udenbach.
\newblock Ideal polytopes for representations of {$GL_n(\mathbb{C})$}.
\newblock 2022.
\newblock arXiv:2206.10522.

\bibitem{Luszig1990}
George Lusztig.
\newblock Canonical bases arising from quantized enveloping algebras.
\newblock {\em J. Amer. Math. Soc.}, 3(2):447--498, 1990.

\bibitem{Lusztig1990_2}
George Lusztig.
\newblock Canonical bases arising from quantized enveloping algebras. {II}.
\newblock Number 102, pages 175--201 (1991). 1990.
\newblock Common trends in mathematics and quantum field theories (Kyoto,
  1990).

\bibitem{Lusztig1994}
George Lusztig.
\newblock Total positivity in reductive groups.
\newblock In {\em Lie theory and geometry}, volume 123 of {\em Progr. Math.},
  pages 531--568. Birkh\"{a}user Boston, Boston, MA, 1994.

\bibitem{MarshRietsch2004}
R.~J. Marsh and K.~Rietsch.
\newblock Parametrizations of flag varieties.
\newblock {\em Represent. Theory}, 8:212--242, 2004.

\bibitem{MarshRietsch2020}
R.~J. Marsh and K.~Rietsch.
\newblock The {$B$}-model connection and mirror symmetry for {G}rassmannians.
\newblock {\em Adv. Math.}, 366:107027, 131, 2020.

\bibitem{Rietsch2008}
Konstanze Rietsch.
\newblock A mirror symmetric construction of {$qH^\ast_T(G/P)_{(q)}$}.
\newblock {\em Adv. Math.}, 217(6):2401--2442, 2008.

\bibitem{RietschWilliams2019}
Konstanze Rietsch and Lauren Williams.
\newblock Newton-{O}kounkov bodies, cluster duality, and mirror symmetry for
  {G}rassmannians.
\newblock {\em Duke Math. J.}, 168(18):3437--3527, 2019.

\bibitem{SpringerLAG}
T.A. Springer.
\newblock {\em Linear algebraic groups}.
\newblock Modern Birkh\"{a}user Classics. Birkh\"{a}user Boston, Inc., Boston,
  MA, second edition, 2009.

\end{thebibliography}

\end{document}